\documentclass{article}
\usepackage{latexsym, amscd, amsfonts, eucal, mathrsfs, amsmath, amssymb, amsthm, xypic,xr, makeidx, stmaryrd}

\def\Z{\mathbf{Z}}
\def\C{\mathbf{C}}

\newcommand{\OC}{{\mathcal O}{\mathcal C}}

\DeclareMathOperator{\non}{nc}
\DeclareMathOperator{\Cone}{Cn}
\DeclareMathOperator{\Brd}{Brd}
\DeclareMathOperator{\Adj}{Adj}
\DeclareMathOperator{\Tang}{Tang}
\DeclareMathOperator{\ZS}{\Phi}
\DeclareMathOperator{\Gr}{Gr}
\newcommand{\vect}{\vec}
\DeclareMathOperator{\Alg}{Alg}
\DeclareMathOperator{\CAlg}{CAlg}
\DeclareMathOperator{\frun}{ff}
\DeclareMathOperator{\Bimod}{Bimod}
\DeclareMathOperator{\Groth}{Groth}
\newcommand{\h}[1]{\rm{h} \! #1}
\newcommand{\htt}[1]{\rm{h}_2 \! #1}
\DeclareMathOperator{\Fam}{Fam}
\DeclareMathOperator{\MTSO}{MTSO}
\DeclareMathOperator{\MTO}{MTO}

\DeclareMathOperator{\BSO}{BSO}
\DeclareMathOperator{\Loc}{Loc}
\DeclareMathOperator{\Emb}{Emb}

\DeclareMathOperator{\Diff}{Diff}
\DeclareMathOperator{\EO}{EO}
\DeclareMathOperator{\BO}{BO}
\DeclareMathOperator{\BDiff}{BDiff}
\DeclareMathOperator{\EDiff}{EDiff}
\DeclareMathOperator{\Sub}{Sub}
\DeclareMathOperator{\fd}{fd}

\DeclareMathOperator{\MSO}{MSO}

\DeclareMathOperator{\Nerve}{N}

\newcommand{\Cob}[3]{{\text {\bf Cob}}^{#3}_{#1}(#2)}
\DeclareMathOperator{\Vect}{{\bf Vect}}
\DeclareMathOperator{\ev}{ev}
\DeclareMathOperator{\HochCoh}{HH}
\DeclareMathOperator{\coev}{coev}
\DeclareMathOperator{\End}{End}
\DeclareMathOperator{\fr}{fr}
\DeclareMathOperator{\tr}{tr}

\DeclareMathOperator{\OHom}{\bHom}

\DeclareMathOperator{\nCob}{{ \bf nCob}}
\DeclareMathOperator{\gVect}{grVect}
\DeclareMathOperator{\Chain}{{ \bf Chain}}
\DeclareMathOperator{\Bord}{{\bf Bord}}
\DeclareMathOperator{\untCob}{{\bf SemiCob}}
\DeclareMathOperator{\SuntCob}{{ \bf PreCob}}
\DeclareMathOperator{\OSuntCob}{{ \bf PreCob}^{\degree}}
\DeclareMathOperator{\unBord}{{\bf PreBord}}
\DeclareMathOperator{\untBord}{{\bf PBord}}
\DeclareMathOperator{\tCob}{{\bf Cob}_{t}}
\DeclareMathOperator{\bdCob}{{\bf Cob}^{un}_{\bd}}
\DeclareMathOperator{\tunCob}{{\bf Cob}_{t}^{un}}
\DeclareMathOperator{\ori}{or}
\DeclareMathOperator{\tChain}{{\bf Chain}_{t}}

\newcommand{\bfA}{{\mathbf A}}
\newcommand{\bfS}{{\mathbf S}}
\newcommand{\degree}{\text{o}}

\newcommand{\bigdot}{\bullet}

\DeclareMathOperator{\CP}{ { \mathbf CP} }

\DeclareMathOperator{\PL}{PL}

\DeclareMathOperator{\bHom}{{Map}}

\DeclareMathOperator{\Cat}{\mathcal{C}at}

\DeclareMathOperator{\Q}{\mathbf{Q}}

\DeclareMathOperator{\OO}{O}
\DeclareMathOperator{\SO}{SO}

\DeclareMathOperator{\R}{\mathbb R}

\DeclareMathOperator{\Sing}{Sing}
\DeclareMathOperator{\cDelta}{{\bf \Delta}}

\DeclareMathOperator{\Top}{Top}

\DeclareMathOperator{\bd}{\partial}

\DeclareMathOperator{\calA}{\mathcal{A}}

\DeclareMathOperator{\calE}{\mathcal{E}}
\DeclareMathOperator{\calZ}{\mathcal{Z}}

\DeclareMathOperator{\calB}{\mathcal{B}}

\DeclareMathOperator{\calF}{\mathcal{F}}

\newcommand{\hn}[1]{{\text h}_{n} \! #1}
\DeclareMathOperator{\Hom}{Hom} 
\DeclareMathOperator{\HH}{H} 
\DeclareMathOperator{\id}{id} \DeclareMathOperator{\Fun}{Fun}
\DeclareMathOperator{\calC}{\mathcal{C}}

\DeclareMathOperator{\op}{op}
\DeclareMathOperator{\calD}{\mathcal{D}}

\DeclareMathOperator{\calO}{\mathcal{O}}

\DeclareMathOperator{\FrFun}{Fun^{\fr}}
\newtheorem{theorem}{Theorem}[subsection]
\newtheorem{lemma}[theorem]{Lemma}
\newtheorem{claim}[theorem]{Claim}
\newtheorem{proposition}[theorem]{Proposition}
\newtheorem{corollary}[theorem]{Corollary}

\newtheorem{thesis}[theorem]{Thesis}

\theoremstyle{definition}
\newtheorem{definition}[theorem]{Definition}
\newtheorem{construction}[theorem]{Construction}
\newtheorem{variant}[theorem]{Variant}
\newtheorem{protodefinition}[theorem]{Definition Sketch}
\newtheorem{idefinition}[theorem]{Incorrect Definition}
\newtheorem{convention}[theorem]{Convention}
\newtheorem{example}[theorem]{Example}
\newtheorem{notation}[theorem]{Notation}

\newtheorem{remark}[theorem]{Remark}
\newtheorem{warning}[theorem]{Warning}
\newtheorem{exercise}[theorem]{Exercise}

\begin{document}

\title{On the Classification of Topological Field Theories (Draft)}
\author{Jacob Lurie}
\maketitle

Our goal in this article is to give an expository account of some recent work on the classification of topological field theories. More specifically, we will outline the proof of a version of the {\it cobordism hypothesis} conjectured by Baez and Dolan in \cite{baezdolan}.

\tableofcontents

\section*{}

\subsection*{Terminology}

Unless otherwise specified, we will use the word {\it manifold} to refer to a compact smooth manifold $M$, possibly with boundary (or with corners). If $M$ is a manifold, we will denote its boundary
by $\bd M$. We will say that $M$ is {\it closed} if the boundary $\bd M$ is empty. For
a brief description of how the ideas of this paper generalize to manifolds which are not smooth, we refer the reader to Remark \ref{skilj}.

Throughout this paper, we will make informal use of the language of higher category theory.
We will always use the term {\it $n$-category} to refer to what is sometimes called a 
{\it weak} $n$-category: that is, a collection of objects $\{ X, Y, Z, \ldots \}$ together
with an $(n-1)$-category $\OHom(X,Y)$ for every pair of objects $X$ and $Y$, which
are equipped with a notion of composition which is associative up to coherent isomorphism.
We refer the reader to \ref{higt} for an informal discussion and \ref{bigseg} for the outline of a more precise definition.

If $\calC$ is a category (or a higher category) equipped with an associative and unital tensor product
$\otimes$, we will let ${\bf 1}$ denote the unit object of $\calC$.

Let $V$ be a finite-dimensional real vector space. By an {\it inner product} on $V$ we will mean a symmetric bilinear form $b: V \times V \rightarrow \R$ which is positive-definite
(so that $b(v,v) > 0$ for $v \neq 0$). More generally, if $X$ is a topological
space and $\zeta$ is a real vector bundle on $X$, then by an {\it inner product} on
$\zeta$ we will mean an inner product on each fiber $\zeta_{x}$, which depends continuously
on the point $x \in X$.

\subsection*{Disclaimer}
Our objective in this paper is to give an informal account of some ideas relating to the classification of topological field theories. In many instances, we have not attempted to give precise definitions, let alone careful proofs. A more detailed account of the ideas and methods described in this paper will appear elsewhere.

\subsection*{Acknowledgements}

I would like to thank David Ben-Zvi, Kevin Costello, Chris Douglas, Dan Freed, Dennis Gaitsgory, Soren Galatius, Andre Henriques, Kiyoshi Igusa, David Nadler, Chris Schommer-Pries, Stefan Stolz, Peter Teichner, Constantin Teleman, and Ulrike Tillmann for helpful conversations about the subject matter of this paper. I would also like to acknowledge a great intellectual debt to the work of John Baez and James Dolan (who proposed the cobordism and tangle hypotheses in \cite{baezdolan} and thereby
brought into focus the ideas we discuss here) and the work of Kevin Costello (whose beautiful
paper \cite{costello} inspired the present work). Special thanks are due to Mike Hopkins: he has been a collaborator through various stages of the project described here, though he has declined to be named as a coauthor. Finally, I would like to thank the American Institute of Mathematics for supporting me during the time this paper was written.

\section{Topological Field Theories and Higher Categories}\label{glob1}

The starting point for this paper is Atiyah's definition of a {\it topological field theory}, which we will
review in \S \ref{kob}. This notion is fairly concrete, and it is not difficult to explicitly classify
topological field theories of dimensions $\leq 2$. In \S \ref{secex}, we will discuss some of the difficulties that we encounter when attempting to generalize this classification to the higher-dimensional setting.
To address these difficulties, we will introduce the notion of an {\it extended} topological field theory.
We will then formulate a version of the {\it Baez-Dolan cobordism hypothesis} (Theorem \ref{swisher}), which provides an elegant classification of extended topological field theories.

The notion of an extended topological field theory and the cobordism hypothesis itself are most naturally expressed using the language of higher category theory, which we will review informally in \S \ref{higt}. 
This language can also be used to introduce a more refined version of topological field theory
which takes into account the homotopy types of diffeomorphism groups of manifolds; we will discuss this definition in \S \ref{sugar}, and formulate an appropriate generalization of the cobordism hyothesis
(Theorem \ref{swisher2}).

\subsection{Classical Definitions}\label{kob}

In this section, we will review the notion of a topological field theory as axiomatized by
Atiyah; for details, we refer the reader to \cite{atiyah}.

\begin{definition}\label{slovak}
Let $n$ be a positive integer. We define a category $\Cob{}{n}{}$ as follows:
\begin{itemize}
\item[$(1)$] An object of $\Cob{}{n}{}$ is a closed oriented $(n-1)$-manifold $M$.

\item[$(2)$] Given a pair of objects $M, N \in \Cob{}{n}{}$, a morphism from $M$ to $N$ in
$\Cob{}{n}{}$ is a bordism from $M$ to $N$: that is, an oriented $n$-dimensional manifold
$B$ equipped with an orientation-preserving diffeomorphism $\bd B \simeq \overline{M} \coprod N$.
Here $\overline{M}$ denotes the manifold $M$ equipped with the opposite orientation.
We regard two bordisms $B$ and $B'$ as defining the same morphism in $\Cob{}{n}{}$ if there
is an orientation-preserving diffeomorphism $B \simeq B'$ which extends the evident diffeomorphism
$\bd B \simeq \overline{M} \coprod N \simeq \bd B'$ between their boundaries.

\item[$(3)$] For any object $M \in \Cob{}{n}{}$, the identity map $\id_{M}$ is represented by the product bordism $B = M \times [0,1]$. 

\item[$(4)$] Composition of morphisms in $\Cob{}{n}{}$ is given by gluing bordisms together. More precisely,
suppose we are given a triple of objects $M, M', M'' \in \Cob{}{n}{}$, and a pair of bordisms
$B: M \rightarrow M'$, $B': M' \rightarrow M''$, the composition $B' \circ B$ is defined to be the morphism represented by the manifold $B \coprod_{M'} B'$.
\end{itemize}
\end{definition}

\begin{remark}\label{jirik}
The composition law for bordisms described in Definition \ref{slovak} is potentially ambiguous, because we did not explain how to endow the manifold $B \coprod_{M'} B'$ with a smooth structure. To do so, we need to make some auxiliary choices (for example, the choice of a smooth collar around $M'$
inside of $B$ and $B'$), which ultimately turn out to be irrelevant (different choices of collar lead to different smooth structures on $B \coprod_{M'} B'$, but the resulting bordisms are nevertheless diffeomorphic). We will not press this technical point any further here; later, we will introduce more elaborate versions of Definition \ref{slovak} in which the issue does not arise.
\end{remark}

Recall that a {\it symmetric monoidal} category is a category $\calC$ equipped with a functor
$\otimes: \calC \times \calC \rightarrow \calC$ and a unit object ${\bf 1}_{\calC} \in \calC$, together with
isomorphisms
$$ C \otimes {\bf 1}_{\calC} \simeq C$$
$$ C \otimes D \simeq D \otimes C$$
$$ C \otimes (D \otimes E) \simeq (C \otimes D) \otimes E.$$
which express the idea that $\otimes$ is a commutative and associative product on $\calC$
(with unit by ${\bf 1}_{\calC}$). These isomorphisms should be required to satisfy a list of coherence conditions which we do not recall here; see \cite{maclane} for a complete definition.

\begin{example}
For each $n > 0$, the category $\Cob{}{n}{}$ can be endowed with the structure of a symmetric monoidal category, where the tensor product operation $\otimes: \Cob{}{n}{} \times \Cob{}{n}{} \rightarrow \Cob{}{n}{}$ is
given by the disjoint union of manifolds. The unit object of $\Cob{}{n}{}$ is the empty set (regarded as a manifold of dimension $(n-1)$). 
\end{example}

\begin{example}
Let $k$ be a field. Then the category $\Vect(k)$ of vector spaces over $k$ can be regarded as a symmetric monoidal category with respect to the usual tensor product functor $\otimes: \Vect(k) \times \Vect(k) \rightarrow \Vect(k)$. The unit object of $\Vect(k)$ is the vector space $k$ itself.
\end{example}

Given a pair of symmetric monoidal categories $\calC$ and $\calD$, a {\it symmetric monoidal
functor} from $\calC$ to $\calD$ is a functor $F: \calC \rightarrow \calD$ together with
a collection of isomorphisms
$$ F(C \otimes C') \simeq F(C) \otimes F(C') \quad \quad F( {\bf 1}_{\calC} ) \simeq { \bf 1}_{\calD}.$$
These isomorphisms are required to be compatible with the commutativity and associativity
constraints on the tensor products in $\calC$ and $\calD$; we refer the reader again
to \cite{maclane} for a more complete discussion.

\begin{definition}[Atiyah]\label{atdef}
Let $k$ be a field. A {\it topological field theory} of dimension $n$ is a symmetric monoidal
functor $Z: \Cob{}{n}{} \rightarrow \Vect(k)$.
\end{definition}

Unwinding Definition \ref{atdef}, we see that a topological field theory $Z$ 
of dimension $n$ is given by the following data:

\begin{itemize}
\item[$(a)$] For every oriented closed manifold $M$ of dimension $(n-1)$, a vector
space $Z(M)$.
\item[$(b)$] For every oriented bordism $B$ from an $(n-1)$-manifold $M$ to another
$(n-1)$-manifold $N$, a linear map of vector spaces $Z(B): Z(M) \rightarrow Z(N)$.
\item[$(c)$] A collection of isomorphisms 
$$Z(\emptyset) \simeq k \quad \quad Z(M \coprod N) \simeq Z(M) \otimes Z(N).$$
\end{itemize}

Moreover, these data are required to satisfy a number of natural coherence properties which
we will not make explicit.

\begin{remark}\label{getnum}
Let $M$ be a closed oriented manifold of dimension $n$. Then we can regard $M$ as a bordism from the empty $(n-1)$-manifold to itself. In this way, $M$ determines a morphism $\emptyset \rightarrow \emptyset$ in the category $\Cob{}{n}{}$. If $Z$ is a topological field theory of dimension $n$,
then $M$ determines a map $Z(M): Z(\emptyset) \rightarrow Z(\emptyset)$. Since $Z$ is a
tensor functor, it preserves unit objects: that is, $Z(\emptyset)$ is {\em canonically} isomorphic
to the ground field $k$. Consequently, we can think of $Z(M)$ as an element of the endomorphism ring
$\Hom_{ \Vect(k)}(k,k)$: that is, as an element of $k$. In other words, the functor $Z$ assigns a 
{\em number} to every closed oriented manifold of dimension $n$.
\end{remark}

\begin{remark}
Let $B$ be an oriented $n$-manifold with boundary $\bd B$. Then $B$ can usually be interpreted
as a morphism in $\Cob{}{n}{}$ in many different ways: one for every decomposition of the boundary
$\bd B$ as a disjoint union of two components.

For example, let us suppose that $M$ is a closed oriented manifold of dimension $(n-1)$, and let
$\overline{M}$ denote the same manifold with the opposite orientation. The product manifold $M \times [0,1]$ has boundary $\overline{M} \coprod M$. It therefore determines a bordism from $M$ to itself: when so regarded, it represents the {\em identity} map $\id_{M}$ in the category $\Cob{}{n}{}$. However,
there are several other ways to view $M \times [0,1]$ as a morphism in $\Cob{}{n}{}$, corresponding
to other decompositions of the boundary $\bd (M \times [0,1])$. For example:
\begin{itemize}
\item[$(a)$] We can regard $M \times [0,1]$ as a bordism from $\overline{M}$ to itself; it then
represents the identity map $\id_{ \overline{M} }$ in the category $\Cob{}{n}{}$.
\item[$(b)$] We can regard $M \times [0,1]$ as a bordism from $\overline{M} \coprod M$ to the
empty set. In this case, $M \times [0,1]$ represents a morphism
$\overline{M} \coprod M \rightarrow \emptyset$ in $\Cob{}{n}{}$, which we will denote by
$\ev_{M}$ and refer to as the {\it evaluation map for $M$}.
\item[$(c)$] We can regard $M \times [0,1]$ as a bordism from the empty set to
$M \coprod \overline{M}$. It then represents a morphism $\emptyset \rightarrow M \coprod \overline{M}$
in the category $\Cob{}{n}{}$, which we will denote by $\coev_{M}$ and refer to as the
{\it coevaluation map for $M$}.
\end{itemize}
\end{remark}

Suppose now that $Z$ is a topological field theory of dimension $n$, and let
$M$ be a closed oriented $(n-1)$-manifold. Applying the functor $Z$ to the evaluation map
$\ev_{M}$, we obtain a map of vector spaces
$$ Z( \overline{M} ) \otimes Z(M) \simeq Z( \overline{M} \coprod M) \stackrel{ Z(\ev_M)}{\rightarrow} Z(\emptyset) \simeq k.$$
In other words, there is a canonical bilinear pairing of $Z(M)$ with $Z( \overline{M})$.

\begin{proposition}\label{scunnel}
Let $Z$ be a topological field theory of dimension $n$. Then for every
closed $(n-1)$-manifold $M$, the vector space $Z(M)$ is finite dimensional, and the
pairing $Z( \overline{M}) \otimes Z(M) \rightarrow k$ is perfect: that is, it induces an isomorphism
$\alpha$ from $Z( \overline{M} )$ to the dual space of $Z(M)$.
\end{proposition}

The proof is completely formal: we can use the coevaluation map of $M$ to explicitly construct an inverse to $\alpha$. More precisely, let $Z(M)^{\vee}$ denote the dual space to $Z(M)$. Applying
$Z$ to the coevaluation map $\coev_{M}$, we obtain a map
$$ k \simeq Z(\emptyset) \stackrel{ Z( \coev_{M})}{\longrightarrow} Z( M \coprod \overline{M} ) \simeq 
Z(M) \otimes Z( \overline{M}).$$
Tensoring this map with $Z(M)^{\vee}$ and composing with the natural pairing of $Z(M)^{\vee}$ with
$Z(M)$, we get a map
$$ \beta: Z(M)^{\vee} \rightarrow Z(M)^{\vee} \otimes Z(M) \otimes Z(\overline{M}) \rightarrow Z(\overline{M}).$$
By judiciously applying the axioms for a topological field theory, one can deduce that $\beta$
is an inverse to $\alpha$: this proves that $\alpha$ is an isomorphism. Because every element
in the tensor product $Z(M) \otimes Z(\overline{M})$ belongs to $Z(M) \otimes V$ for some finite
dimensional subspace $V$ of $Z( \overline{M})$, the image of the map $\beta$ is necessarily finite dimensional; since $\beta$ is an isomorphism, we conclude that $Z(M)^{\vee}$ is finite dimensional
(so that $Z(M)$ is also finite dimensional).

In low dimensions, it is possible to describe topological field theories very explicitly.

\begin{example}[Field Theories in Dimension 1]\label{1dim}
Let $Z$ be a $1$-dimensional topological field theory.
Then $Z$ assigns a vector space $Z(M)$ to every closed oriented $0$-manifold $M$.
A zero dimensional manifold $M$ is simply a finite set of points. An orientation of $M$
determines a decomposition $M = M_{+} \coprod M_{-}$ of $M$ into ``positively oriented''
and ``negatively oriented'' points. In particular, there are {\em two} oriented manifolds which
consist of only a single point, up to orientation-preserving diffeomorphism. Let us denote these manifolds by $P$ and $Q$. Applying the functor $Z$, we obtain vector spaces $Z(P)$ and $Z(Q)$.
However, these vector spaces are related to one another: according to Proposition \ref{scunnel}, we can write $Z(P) = V$ and $Z(Q) = V^{\vee}$, for some finite-dimensional vector space $V$.

Once we have specified $V$, the remainder of the field theory is uniquely determined (up to isomorphism). For example, the value of $Z$ on any oriented $0$-manifold $M$
is canonically isomorphic to the tensor product
$$ (\bigotimes_{x \in M_{+}} V) \otimes ( \bigotimes_{y \in M_{-}} V^{\vee}).$$
Of course, this does not yet determine $Z$: we must also specify the behavior of $Z$ on
$1$-manifolds $B$ with boundary. However, since $Z$ is a symmetric monoidal functor,
it suffices to specify $Z(B)$ when $B$ is {\em connected}. In this case, the $1$-manifold $B$ is diffeomorphic either to a closed interval $[0,1]$ or to a circle $S^1$. There are five cases
to consider, depending on how we decompose $\bd B$ into ``incoming'' and ``outgoing'' pieces:
\begin{itemize}
\item[$(a)$] Suppose that $B = [0,1]$, regarded as a bordism from $P$ to itself. Then
$Z(B)$ coincides with the identity map $\id: V \rightarrow V$.
\item[$(b)$] Suppose that $B = [0,1]$, regarded as a bordism from $Q$ to itself. Then
$Z(B)$ coincides with the identity map $\id: V^{\vee} \rightarrow V^{\vee}$.
\item[$(c)$] Suppose that $B= [0,1]$, regarded as a bordism from $P \coprod Q$ to the empty set.
Then $Z(B)$ is a linear map from $V \otimes V^{\vee}$ into the ground field $k$: namely, the evaluation
map $(v, \lambda) \mapsto \lambda(v)$.
\item[$(d)$] Suppose that $B = [0,1]$, regarded as a bordism from the empty set to $P \coprod Q$.
Then $Z(B)$ is a linear map from $k$ to $Z(P \coprod Q) \simeq V \otimes V^{\vee}$. Under
the canonical isomorphism $V \otimes V^{\vee} \simeq \End(V)$, this linear map is given by
$x \mapsto x \id_{V}$.
\item[$(e)$] Suppose that $B = S^1$, regarded as a bordism from the empty set to itself.
Then $Z(B)$ is a linear map from $k$ to itself, which we can identify with an element of $k$
(Remark \ref{getnum}). To compute this element, it is convenient to decompose the circle
$S^1 \simeq \{ z \in \C: |z| = 1 \}$ into two intervals
$$ S^{1}_{-} = \{ z \in \C: (|z| = 1) \wedge {\text Im}(z) \leq 0 \} \quad \quad S^{1}_{+} = \{ z \in \C: (|z| = 1) \wedge {\text Im}(z) \geq 0 \},$$ 
meeting in the subset 
$$S^{1}_{-} \cap S^{1}_{+} = \{ \pm 1 \} \subseteq S^1.$$
It follows that $Z(S^1)$ is given as the composition of the maps
$$ k \simeq Z( \emptyset) \stackrel{ Z( S^{1}_{-})}{\longrightarrow} Z( \pm 1)
\stackrel{ Z(S^{1}_{+})}{\longrightarrow} Z( \emptyset) \simeq k.$$
These maps were described by $(c)$ and $(d)$ above. Under the identification of
$Z( \pm 1)$ with $V \otimes V^{\vee} \simeq \End(V)$, the map $Z( S^{1}_{-}): k \rightarrow \End(V)$ is given by $x \mapsto x \id_{V}$, while $Z(S^{1}_{+}): \End(V) \rightarrow k$ is given by 
$A \mapsto \tr(A)$. Consequently, $Z(S^1)$ is given by the trace of the identity map from $V$ to itself: in other words, the dimension of $V$.
\end{itemize}
\end{example}

\begin{remark}
Example \ref{1dim} illustrates some central themes which will reappear (in a more sophisticated form) later in this paper. {\it A priori}, the specification of a topological field theory $Z$ involves
a large quantity of data: one must give a vector space $Z(M)$ for every closed oriented manifold $M$ of the appropriate dimension, together with a number of linear maps satisfying various conditions. However, the field theory $Z$ is often determined by only a tiny fragment of this data, given
by evaluating $Z$ on a small class of manifolds (in Example \ref{1dim}, the entire field theory $Z$ is determined by the single vector space $V = Z(P)$). 
Nevertheless, it can be interesting to consider the values of $Z$ on arbitrary manifolds. In Example \ref{1dim}, the value $Z( S^1)$ recovers the {\em dimension} of the vector space $V$.
This is the most important numerical invariant of $V$, and is in some sense the only invariant: any two vector spaces with the same (integer) dimension are isomorphic to one another.
\end{remark}

\begin{example}[Field Theories in Dimension 2]\label{coming}
Let $Z$ be a $2$-dimensional topological field theory. Then $Z$ assigns a vector space $Z(M)$ to every closed, oriented $1$-manifold $M$. In particular, $Z$ determines a vector space $A = Z(S^1)$.
Since $Z$ is a symmetric monoidal functor, the values of $Z$ on objects are determined by $A$:
every closed $1$-manifold $M$ is a disjoint union of $n$ circles for some $n \geq 0$, so that
$Z(M) \simeq A^{\otimes n}$.

Evaluating the field theory $Z$ on bordisms between $1$-manifolds, we obtain some algebraic structure on the vector space $A$. For example, let $B$ denote a pair of pants, regarded as a bordism
from two copies of $S^1$ to a third copy of $S^1$. Then $Z(B)$ determines a linear map
$$A \otimes A \simeq Z(S^1 \coprod S^1) \stackrel{Z(B)}{\longrightarrow} Z(S^1) = A$$
which we will denote by $m$. We can view $m$ as endowing $A$ with a bilinear multiplication.
It follows easily from the definition that this multiplication is commutative and associative: for example,
the commutativity results from the observation that there is a diffeomorphism of $B$ which
permutes the two ``incoming'' boundary circles and restricts to the identity on the third.

There is also a unit for the multiplication on $A$: namely, the image of $1 \in k$ under the linear map
$$Z(D): k \simeq Z(\emptyset) \rightarrow Z(S^1) = A,$$ where we regard
the disk $D^2 = \{ z \in \C: |z| \leq 1 \}$ as a bordism from the empty set to the boundary circle
$S^1 = \bd D^2 = \{ z \in \C: |z| = 1 \}$. We can also interpret $D^2$ as a bordism from $S^1$ to
the empty set, in which case it determines a linear map $\tr: A \rightarrow k$. The composition
$$ A \otimes A \stackrel{m}{\rightarrow} A \stackrel{\tr}{\rightarrow} k$$
is the linear map associated to the cylinder $S^1 \times [0,1]$, and therefore determines a perfect
pairing of $A$ with itself (Proposition \ref{scunnel}). (Note that the $1$-sphere $S^1$ admits an orientation-reversing diffeomorphism, so that $S^1 \simeq \overline{S^1}$.)

It is convenient to summarize the analysis up to this point by introducing a definition.

\begin{definition}
Let $k$ be a field. A {\it commutative Frobenius algebra} over $k$ is a finite dimensional commutative
$k$-algebra $A$, together with a linear map $\tr: A \rightarrow k$ such that the 
the bilinear form $(a,b) \mapsto \tr(ab)$ is nondegenerate.
\end{definition}

The above analysis shows that if $Z$ is a $2$-dimensional topological field theory, then
the vector space $A = Z(S^1)$ is naturally endowed with the structure of a commutative Frobenius algebra over $k$. In fact, the converse is true as well: given a commutative Frobenius algebra $A$,
one can construct a $2$-dimensional topological field theory $Z$ such that $A = Z(S^1)$
and the multiplication and trace on $A$ are given by evaluating $Z$ on a pair of pants and a disk, respectively. Moreover, $Z$ is determined up to unique isomorphism: in other words, the category
of $2$-dimensional topological field theories is {\em equivalent} to the category of commutative Frobenius algebras.
\end{example}

\subsection{Extending Down: Lower Dimensional Manifolds}\label{secex}

In \S \ref{kob}, we analyzed the structure of an $n$-dimensional topological field theory
$Z$ for $n=1$ and $n=2$. In both cases, we accomplished this by emphasizing the value of the field theory $Z$ on closed manifolds of dimension $n-1$. However, it is possible to proceed differently:
according to Remark \ref{getnum}, for every closed oriented manifold $M$ of dimension $n$ we can identify the value $Z(M)$ with an element of the ground field $k$. In other words, a topological field theory gives rise to a {\em diffeomorphism invariant} for closed manifolds of dimension $n$. Suppose we take the point of view that these diffeomorphism invariants are the
main objects of interest. Of course, a topological field theory provides more data: we can evaluate $Z$ not only on closed manifolds of dimension $n$, but also on manifolds with boundary
and manifolds of dimension $(n-1)$. This data can be viewed as supplying a set of rules which allow
us to compute the invariant $Z(M)$ associated to a closed manifold $M$ by breaking $M$ up into pieces.

\begin{example}\label{cugh}
Let $A$ be a commutative Frobenius algebra over a field $k$. According to Example \ref{coming}, 
the algebra $A$ determines a $2$-dimensional topological field theory $Z$. In particular,
we can evaluate $Z$ on closed oriented $2$-manifolds $M$. Such manifolds 
are classified (up to orientation-preserving diffeomorphism) by a single invariant $g$, the {\it genus}, which ranges over the nonnegative integers. Consequently, for each $g \geq 0$, we can evaluate
$Z$ on a closed surface $\Sigma_g$ of genus $g$, to obtain an element $Z( \Sigma_{g}) \in k$.
Let us compute the value of $Z( \Sigma_{g})$ for small values of $g$.

\begin{itemize}
\item Suppose that $g =0$. In this case, $\Sigma_{g}$ is diffeomorphic to a $2$-sphere $S^2$, which we can
view as obtained by gluing together hemispheres $S^2_{+}$ and $S^{2}_{-}$ along the
equator $S^{2}_{+} \cap S^{2}_{-} \simeq S^1$. Consequently, $Z(S^2)$ is obtained by composing the linear maps
$$ k \simeq Z(\emptyset) \stackrel{ Z(S^2_{-})}{\longrightarrow} Z(S^1) \stackrel{ Z(S^2_{+})}{\longrightarrow} Z(\emptyset) \simeq k.$$
Note that $Z(S^1)$ coincides with the Frobenius algebra $A$, $Z(S^{2}_{-}): k \rightarrow A$ corresponds to the inclusion of the identity element of $A$, and $Z(S^2_{+}): A \rightarrow k$ is the trace map
$\tr$. It follows that the invariant $Z( \Sigma_g )$ is given by $\tr(1) \in k$.

\item Suppose that $g=1$. In this case, $\Sigma_{g}$ is diffeomorphic to a torus $S^1 \times S^1$, which we
can decompose into cylinders $S^1_{+} \times S^1$ and $S^1_{-} \times S^1$ meeting
in the pair of circles 
$$(S^{1}_{+} \times S^1) \cap (S^1_{-} \times S^1) \simeq (S^1_{+} \cap S^1_{-}) \times S^1
\simeq \{ \pm 1 \} \times S^1.$$
It follows that $Z(\Sigma_g)$ is given by composing the linear maps
$$ k \simeq Z(\emptyset) \stackrel{ Z( S^{1}_{-} \times S^1)}{\rightarrow}
Z( \{ \pm 1\} \times S^1) \stackrel{ Z(S^1_{+}) \times S^1)}{\rightarrow} Z(\emptyset) \simeq k.$$
As in Example \ref{1dim}, we can identify $Z( \{ \pm 1\} \times S^1)$ with the tensor product
$Z( S^1) \otimes Z( \overline{S}^1) \simeq A \otimes A^{\vee} = \End(A)$ (of course,
$A$ is isomorphic to its dual, since the trace pairing $(a,b) \mapsto \tr(ab)$ is a nondegenerate bilinear form on $A$, but we will not use this observation). In terms of this identification, 
the map $Z( S^{1}_{-}): k \rightarrow \End(A)$ corresponds to the inclusion of the identity element $\id_A$, while $Z(S^1_{+}): \End(A) \rightarrow k$ is given by the trace (on the matrix ring $\End(A)$, which
is unrelated to the trace on $A$). It follows that $Z( \Sigma_{g} )$ is equal to the trace of
$\id_{A}$: in other words, the dimension of the Frobenius algebra $A$.
\end{itemize}

It is possible to continue this analysis, and to compute all of the invariants $\{ Z( \Sigma_{g} ) \}_{g \geq 0}$ in terms of the structure constants for the multiplication and trace on $A$; we leave the details to the interested reader.
\end{example}

We can attempt to use the reasoning of Example \ref{cugh} in any dimension. Suppose
that $Z$ is a topological field theory of dimension $n$. For every oriented $n$-manifold
$M$, we can regard $M$ as a bordism from the empty set to $\bd M$, so that
$Z(M): Z(\emptyset) \rightarrow Z(\bd M)$ can be regarded as an element of the vector space
$Z( \bd M)$. The requirement that $Z$ be a {\em functor} can be translated as follows:
suppose that we are given a closed $(n-1)$-dimensional submanifold $N \subseteq M$ which partitions
$M$ into two pieces $M_0$ and $M_1$. Then
$Z(M)$ is the image of 
$$Z(M_0) \otimes Z(M_1) \in Z( \bd M_0) \otimes Z(\bd M_1) \simeq
Z( \bd M) \otimes Z(N) \otimes Z( \overline{N})$$
under the map $Z(\bd M) \otimes Z(N) \otimes Z( \overline{N}) \rightarrow Z(\bd M)$ induced
by the perfect pairing $Z(N) \otimes Z( \overline{N}) \rightarrow k$ of Proposition \ref{scunnel}.
In other words, Definition \ref{atdef} provides a rule for computing the invariant $Z(M)$ in terms of any decomposition $M = M_0 \coprod_{N} M_1$ along a {\em closed} submanifold $N$ of codimension $1$.
We might now ask: is it possible to break $M$ up into ``simple'' pieces by means of the above procedure?
To put the question another way, is it possible to specify a short list of ``simple'' $n$-manifolds with boundary $\{ M_{\alpha} \}$, such that {\em any} $n$-manifold can be assembled by gluing
together manifolds appearing in the list $\{ M_{\alpha} \}$ along components of their boundaries?
When $n=2$, this question has an affirmative answer: every oriented surface $\Sigma$ can be obtained by gluing together disks, cylinders, and pairs of pants. 

Unfortunately, the above method becomes increasingly inadequate as the dimension $n$ grows. If $M$ is a manifold of large dimension, then it is generally not possible to simplify $M$ very much by cutting along closed submanifolds of codimension $1$ (and these submanifolds
are {\em themselves} very complicated objects when $n \gg 0$). What we would really like to do is to chop $M$ up into very small pieces, say, by choosing a triangulation of $M$. We might then hope to somehow recover the invariant $Z(M)$ in terms of the combinatorics of the triangulation. A triangulation of an $n$-manifold $M$ allows us to write $M$ as a union $\bigcup_{\alpha} \Delta^n_{\alpha}$ of finitely many $n$-simplices, which we can regard as a {\em very} simple type of $n$-manifold with boundary. In other words, $M$ can be obtained by gluing together a collection of simplices. However, the nature of the gluing is somewhat more complicated in this case: in general, we must allow ourselves to glue along submanifolds which
are not closed, but which themselves have boundary. Definition \ref{atdef} makes no provision for
this sort of generalized gluing, which requires us to contemplate not only manifolds of dimension
$n$ and $n-1$, but also manifolds of lower dimension. For this reason, various authors have proposed refinements of Definition \ref{atdef}, such as the following:

\begin{protodefinition}\label{prono}
Let $k$ be a field. A topological field theory $Z$ gives rise to the following data:
\begin{itemize}
\item[$(a)$] For every closed oriented $n$-manifold $M$, an element $Z(M) \in k$.

\item[$(b)$] For every closed oriented $(n-1)$-manifold $M$, a $k$-vector space $Z(M)$.
When $M$ is empty, the vector space $Z(M)$ is canonically isomorphic to $k$.

\item[$(c)$] For every oriented $n$-manifold $M$, an element $Z(M)$ of the vector space
$Z( \bd M)$. In the special case where $M$ is closed, this should coincide with the element
specified by $(a)$ under the isomorphism $Z(\bd M) = Z(\emptyset) \simeq k$ of $(b)$.
\end{itemize}
A {\it $2$-extended topological field theory} consists of data $(a)$ through $(c)$ as above, together with the following:
\begin{itemize}
\item[$(d)$] For every closed oriented $(n-2)$-manifold $M$, a {\it $k$-linear category}
$Z(M)$. That is, $Z(M)$ is a category such that for every pair of objects $x,y \in Z(M)$, the
set of morphisms $\Hom_{Z(M)}(x,y)$ has the structure of a $k$-vector space, and composition
of morphisms is given by bilinear maps. Furthermore, when $M$ is empty, the category
$Z(M)$ should be (canonically equivalent to) the category $\Vect(k)$ of vector spaces over $k$.
\item[$(e)$] For every oriented $(n-1)$-manifold $M$, an object $Z(M)$ of the
$k$-linear category $Z( \bd M)$. In the special case where $M$ is closed, $Z(M)$ should coincide
with the vector space specified by $(b)$ under the equivalence $Z(\bd M) = Z(\emptyset) \simeq \Vect(k)$ of $(d)$.
\end{itemize}
\end{protodefinition}


Definition \ref{prono} is very much incomplete: it specifies a large number of invariants, but does not say very much about how they should be related to one another. For example, if $Z$ is an ordinary topological field theory and $B$ is a bordism from an oriented $(n-1)$-manifold $M$
to another oriented $(n-1)$-manifold $N$, then $Z$ associates a linear map
$Z(B): Z(M) \rightarrow Z(N)$; clause $(c)$ of Definition \ref{prono} describes only the case where
$M$ is assumed to be empty. Similarly, we should demand that if $B$ is a bordism from
an oriented $(n-2)$-manifold $M$ to another oriented $(n-2)$-manifold $N$, then
$B$ determines a functor $Z(B): Z(M) \rightarrow Z(N)$ (the data described in $(e)$ is reduces to the
special case where $M$ is empty, so that $Z(M) \simeq \Vect(k)$, and we evaluate the functor
on the vector space $k \in \Vect(k)$). Of course, this is only the tip of the iceberg: we should also demand coherence conditions which describe the behavior of the invariant $Z(B)$ when $B$ is obtained by gluing bordisms together, or as a disjoint union of bordisms, or varies by a bordism between
$(n-1)$-manifolds with boundary (to properly formulate the relevant structure, we need to contemplate $n$-manifolds {\em with corners}). In the non-extended case, the language of category theory allowed us to summarize all of this data in a very succinct way: a topological field theory is simply a symmetric monoidal functor from the category $\Cob{}{n}{}$ to the category $\Vect(k)$. There is an analogous picture in the $2$-extended situation, but it requires us to introduce the language of $2$-categories.

\begin{definition}\label{caper}
A {\it strict $2$-category} is a category enriched over categories. In other words, a strict $2$-category $\calC$ consists of the following data:
\begin{itemize}
\item A collection of objects, denoted by $X, Y, Z, \ldots$
\item For every pair of objects $X,Y \in \calC$, a category $\OHom_{\calC}(X,Y)$.
\item For every object $X \in \calC$, a distinguished object $\id_{X} \in \OHom_{\calC}(X,Y)$.
\item For every triple of objects $X,Y, Z \in \calC$, a {\it composition functor}
$$ \OHom_{\calC}(X,Y) \times \OHom_{\calC}(Y,Z) \rightarrow \OHom_{\calC}(X,Z).$$
\item The objects $\{ \id_{X} \}_{X \in \calC}$ are units with respect to composition:
in other words, for every pair of objects $X,Y \in \calC$, the functors
$$ \OHom_{\calC}(X,Y) \rightarrow \OHom_{\calC}(X,Y) \quad \quad \OHom_{\calC}(Y,X)
\rightarrow \OHom_{\calC}(Y,X)$$
given by composition with $\id_{X}$ are the identity.
\item Composition is strictly associative: that is, for every quadruple of objects
$W,X,Y,Z \in \calC$, the diagram of functors
$$ \xymatrix{ \OHom_{\calC}(W,X) \times \OHom_{\calC}(X,Y) \times
\OHom_{\calC}(Y,Z) \ar[r] \ar[d] & \OHom_{\calC}(W,Y) \times \OHom_{\calC}(Y,Z) \ar[d] \\
\OHom_{\calC}(W,X) \times \OHom_{\calC}(X,Z) \ar[r] & \OHom_{\calC}(W,Z)}$$
is commutative.
\end{itemize}
\end{definition}

\begin{example}\label{cabble}
Let $k$ be a field. There is a strict $2$-category $\Vect_{2}(k)$, which may be described as follows:
\begin{itemize}
\item The objects of $\Vect_{2}(k)$ are {\em cocomplete} $k$-linear categories: that is,
$k$-linear categories $\calC$ which are closed under the formation of direct sums
and cokernels.
\item Given a pair of objects $\calC, \calD \in \Vect_2(k)$, we define
$\OHom_{\Vect_2(k)}(\calC, \calD)$ to be the category of {\em cocontinuous $k$-linear functors} from
$\calC$ to $\calD$: that is, functors $F: \calC \rightarrow \calD$ which preserve cokernels and direct sums, and such that for every pair of objects
$x,y \in \calC$, the induced map $\Hom_{\calC}(x,y) \rightarrow \Hom_{\calD}(Fx, Fy)$ is
$k$-linear.
\item Composition and identity morphisms in $\Vect_{2}(k)$ are defined in the obvious way.
\end{itemize}
\end{example}

\begin{example}\label{cobble}
For every nonnegative integer $n \geq 2$, we can attempt to define a strict
$2$-category $\Cob{2}{n}{}$ as follows:
\begin{itemize}
\item The objects of $\Cob{2}{n}{}$ are closed oriented manifolds of dimension $(n-2)$.
\item Given a pair of objects $M,N \in \Cob{2}{n}{}$, we define a category
$\calC = \OHom_{ \Cob{2}{n}{}}(M,N)$ as follows. The {\em objects} of $\calC$ are
bordisms from $M$ to $N$: that is, oriented $(n-1)$-manifolds $B$ equipped with a
diffeomorphism $\bd B \simeq \overline{M} \coprod N$. Given a pair of objects
$B, B' \in \calC$, we let $\Hom_{\calC}(B,B')$ denote the collection of all (oriented) diffeomorphism
classes of (oriented) bordisms $X$ from $B$ to $B'$. Here we require that $X$ reduce to the
trivial bordism along the common boundary $\bd B \simeq \overline{M} \coprod N \simeq \bd B'$, so that
$X$ can be regarded as an $n$-manifold with boundary
$$\overline{B} \coprod_{ M \coprod \overline{N} } (( \overline{M} \coprod N) \times [0,1])
\coprod_{ \overline{M} \coprod N} B'.$$
\end{itemize}
Unfortunately, this definition does not (immediately) yield a strict $2$-category, because it is difficult to define a strictly associative composition law
$$ c: \OHom_{ \Cob{2}{n}{}}( M, M') \times \OHom_{ \Cob{2}{n}{}}(M', M'') \rightarrow
\OHom_{ \Cob{2}{n}{}}(M, M'').$$
Roughly speaking, given a bordism $B$ from $M$ to $M'$ and another bordism
$B'$ from $M'$ to $M''$, we would like to define $c(B,B')$ to be the bordism obtained
by gluing $B$ to $B'$ along $M'$. We encounter two difficulties:
\begin{itemize}
\item[$(i)$] In order to endow $c(B,B')$ with a smooth structure, we need to make
some additional choices, such as a smooth collar neighborhood of $M'$ in both $B$ and $B'$.
These choices were irrelevant in Definition \ref{slovak}, because we were only interested in
the bordism $c(B,B')$ up to diffeomorphism. However, to fit into the mold of Definition \ref{caper}, we need $c(B,B')$ to be defined on the nose.
\item[$(ii)$] Given a triple of composable bordisms
$B: M \rightarrow M'$, $B': M' \rightarrow M''$, and $B'': M'' \rightarrow M'''$, the associative law
of Definition \ref{caper} requires that we have an {\em equality} of bordisms
$$(B \coprod_{M'} B') \coprod_{M''} B'' = B \coprod_{M'} (B' \coprod_{M''} B'').$$ This may be difficult
to arrange: what we see in practice is a canonical {\em homeomorphism} between the right and left hand sides (which can be promoted to a diffeomorphism, provided that we have correctly dealt with
problem $(i)$).
\end{itemize}
\end{example}

For the purposes of studying $2$-extended topological field theories, it is vitally
important that our categorical formalism should incorporate Example \ref{cobble}. 
There are two possible means by which we might accomplish this:
\begin{itemize}
\item[$(a)$] Adjust the definition of $\Cob{2}{n}{}$ so that issues $(i)$ and $(ii)$ do not arise. For example, we can address problem $(i)$ by introducing a more complicated notion of bordism
which keeps track of collar neighborhoods of the boundary. Issue $(ii)$ is a bigger nuisance: though it is possible to ``rectify'' the composition law on $\Cob{2}{n}{}$ to make it strictly associative, it is somewhat
painful and technically inconvenient to do so.

\item[$(b)$] Adjust Definition \ref{caper} so that it incorporates Example \ref{cobble} more easily.
This can be accomplished by introducing the definition of a (nonstrict) {\it $2$-category}
(also called a {\it weak $2$-category} or a {\it bicategory}), where we do not require composition
to be associative on the nose but only up to coherent isomorphism.
\end{itemize}

We will adopt approach $(b)$, and work with the $2$-categories rather than strict $2$-categories.
The advantage of this approach is that it more easily accomodates Example \ref{cobble} and variations thereof. The disadvantage is that the definition of a $2$-category is more complicated than Definition \ref{caper}, because we need to define ``up to coherent isomorphism'' precisely. We will defer a more precise discussion until \S \ref{bigseg}; for the moment, we will simply take for granted that there is a good theory of $2$-categories which incorporates the examples above and use it to give a more complete formulation of Definition \ref{prono}:

\begin{definition}\label{prono2}
Let $k$ be a field. An {\it $2$-extended topological field theory} of dimension $n$ is a symmetric monoidal functor $Z: \Cob{2}{n}{} \rightarrow \Vect_{2}(k)$ between $2$-categories. 
\end{definition}

\begin{remark}
In order to make sense of Definition \ref{prono2}, we need to understand $\Cob{2}{n}{}$ and
$\Vect_{2}(k)$ not only as $2$-categories, but as {\it symmetric monoidal} $2$-categories.
In the case of $\Cob{2}{n}{}$, this is straightforward: the tensor product operation is simply given by disjoint unions of manifolds, just as for $\Cob{}{n}{}$. The tensor product on $\Vect_{2}(k)$ is a bit more subtle.
To describe it, let us first recall how to define the tensor product of a pair of vector spaces $U$ and $V$ over $k$. Given a third $k$-vector space $W$, we can define the notion of a {\it bilinear map}
from $U \times V$ into $W$: this is a map $b: U \times V \rightarrow W$ which is linear separately in each variable (in other words, we require that for each $u \in U$ the map $v \mapsto b(u,v)$ is linear,
and similarly for each $v \in V$ the map $u \mapsto b(u,v)$ is linear). The tensor product
$U \otimes V$ is defined to be the recipient of a {\em universal} bilinear map
$U \times V \rightarrow U \otimes V$. In other words, $U \otimes V$ is characterized by the following universal property: giving a linear map from $U \otimes V$ into another $k$-vector space $W$ is
equivalent to giving a bilinear map $U \times V \rightarrow W$.

We can apply the same reasoning to define a tensor product operation in the setting of
(cocomplete) $k$-linear categories. We begin by defining the analogue of the notion of a bilinear map:
given a triple of cocomplete $k$-linear categories $\calC$, $\calD$, and $\calE$, we will say that a functor $F: \calC \times \calD \rightarrow \calE$ is {\it $k$-bilinear} if for every object $C \in \calC$
the functor $D \mapsto F(C,D)$ is cocontinuous and $k$-linear, and for every object $D \in \calD$
the functor $C \mapsto F(C,D)$ is cocontinuous and $k$-linear. We can then attempt to
define a tensor product $\calC \otimes \calD$ by demanding the following universal property:
for every cocomplete $k$-linear category $\calE$, there is an equivalence between the category
of cocontinuous $k$-linear functors $\calC \otimes \calD \rightarrow \calE$ with the category
of $k$-bilinear functors $\calC \times \calD \rightarrow \calE$. Of course, it takes some effort to prove that
the tensor product $\calC \otimes \calD$ exists (and a bit more work to show that it is associative); we will not dwell on this point, since the strict $2$-category $\Vect_{2}(k)$ will soon disappear from our discussion of topological field theories.
\end{remark}

\begin{remark}
Definition \ref{prono2} should be regarded as a more elaborate version of Definition \ref{atdef}.
To explain this, we note that if $\calC$ is an arbitrary symmetric monoidal $2$-category,
then we can extract a symmetric monoidal category $\Omega \calC = \bHom_{ \calC}( {\bf 1}, {\bf 1})$
of morphisms from the unit object to itself in $\calC$. Applying this construction
in the situations of Example \ref{cabble} and \ref{cobble}, we obtain equivalences
$$ \Omega \Vect_{2}(k) \simeq \Vect(k) \quad \quad \Omega \Cob{2}{n}{} \simeq \Cob{}{n}{}.$$
Consequently, any $2$-extended field theory $Z: \Cob{2}{n}{} \rightarrow \Vect_{2}(k)$
determines a symmetric monoidal functor $\Omega Z: \Cob{}{n}{} \rightarrow \Vect(k)$, which
we can regard as an $n$-dimensional topological field theory in the sense of Definition \ref{atdef}.
\end{remark}

In general, a $2$-extended topological field theory $Z: \Cob{2}{n}{} \rightarrow \Vect_{2}(k)$
contains a great deal more information than its underlying topological field theory
$\Omega Z$, which can be useful in performing calculations. For example,
suppose that we wish to compute $Z(M) = (\Omega Z)(M)$, where $M$ is a closed oriented
$n$-manifold. Knowing that $Z(M)$ is the value of a topological field theory $\Omega Z$ on
$M$ allows us to compute $Z(M)$ by cutting $M$ along closed submanifolds of codimension
$1$. The $2$-extended field theory $Z$ itself gives us more flexibility: we can cut $M$ along $(n-1)$-manifolds with boundary. However, this freedom is still somewhat limited: given a decomposition $M = M_0 \coprod_{N} M_1$,
we can try to reconstruct $Z(M)$ in terms of the constituents $Z(M_0)$, $Z(M_1)$, and $Z(N)$. 
In particular, we need to understand $Z(N)$, where $N$ has dimension $(n-1)$. If $n$ is large,
we should expect $N$ to be quite complicated. It is therefore natural to try to break $N$ into simpler pieces. Definition \ref{prono2} gives us a limited amount of freedom to do so: given a decomposition
$N = N_0 \coprod_{P} N_1$, where $P$ is a {\em closed} $(n-2)$-manifold, we can compute
$Z(N)$ in terms of $Z(N_0)$, $Z(P)$, and $Z(N_1)$. However, we cannot generally simplify
$N$ very much by cutting along closed submanifolds: it is again necessary to allow cutting along
$(n-2)$-manifolds with boundary. For this, we need to consider topological field theories which are even more ``extended''. To make sense of these ideas, we need to introduce a bit more terminology.

\begin{definition}\label{sncat}
Let $n$ be a nonnegative integer. We define the notion of a {\it strict $n$-category} by induction on $n$:
\begin{itemize}
\item[$(a)$] If $n=0$, then a strict $n$-category is a set.
\item[$(b)$] If $n > 0$, then a strict $n$-category is a category enriched over
strict $(n-1)$-categories. In other words, a strict $n$-category $\calC$ consists of the following data:
\begin{itemize}
\item[$(i)$] A collection of objects $X,Y, Z, \ldots$
\item[$(ii)$] For every pair of objects $X,Y \in \calC$, a strict $(n-1)$-category
$\OHom_{\calC}(X,Y)$.
\item[$(iii)$] Identity objects $\id_{X} \in \OHom_{\calC}(X,X)$ and composition maps
$$\OHom_{\calC}(X,Y) \times \OHom_{\calC}(Y,Z) \rightarrow \OHom_{\calC}(X,Z),$$ satisfying
the usual unit and associativity conditions.
\end{itemize}
\end{itemize}
\end{definition}

Let us take a moment to unwind Definition \ref{sncat}. A strict $n$-category $\calC$ is
a mathematical structure in which one has:
\begin{itemize}
\item A collection of objects $X,Y, Z, \ldots$
\item For every pair of objects $X,Y \in \calC$, a collection of morphisms from $X$ to $Y$,
called {\it $1$-morphisms}.
\item For every pair of objects $X$ and $Y$ and every pair of morphisms $f,g: X \rightarrow Y$,
a collection of morphisms from $f$ to $g$, called {\it $2$-morphisms}.
\item For every pair of objects $X$ and $Y$, every pair of morphisms $f,g: X \rightarrow Y$, and
every pair of $2$-morphisms $\alpha, \beta: f \rightarrow g$, a collection of
morphisms from $\alpha$ to $\beta$, called {\it $3$-morphisms}.
\item \ldots
\end{itemize}
Moreover, these morphisms are equipped with various notions of composition, which are strictly associative at every level.

\begin{warning}
When $n=1$, Definition \ref{sncat} recovers the usual notion of category. When
$n=2$, it reduces to Definition \ref{caper}. For $n > 2$, Definition \ref{sncat} is poorly behaved.
However, there is a related notion of {\it $n$-category} (or {\it weak $n$-category}), where one requires composition to be associative only up to isomorphism, rather than ``on the nose''. Most of the examples of $n$-categories which arise naturally (such as the example we will discuss below) are {\em not} equivalent to strict $n$-categories. We will review the theory of $n$-categories in \S \ref{higt}, and
sketch a more precise definition in \S \ref{bigseg}.
\end{warning}

\begin{example}\label{slipwell}
Suppose given a pair of nonnegative integers $k \leq n$. Then there
exists a $k$-category which we will denote by $\Cob{k}{n}{}$, which can
be described informally as follows:
\begin{itemize}
\item The objects of $\Cob{k}{n}{}$ are closed oriented $(n-k)$-manifolds
\item Given a pair of objects $M,N \in \Cob{k}{n}{}$, a $1$-morphism from $M$ to $N$
is a bordism from $M$ to $N$: that is, a $(n-k+1)$-manifold $B$ equipped with a diffeomorphism
$\bd B \simeq \overline{M} \coprod N$.
\item Given a pair of objects $M, N \in \Cob{k}{n}{}$ and a pair of bordisms $B,B': M \rightarrow N$,
a $2$-morphism from $B$ to $B'$ is a bordism from $B$ to $B'$, which is required to be trivial along the boundary: in other words, a manifold with boundary 
$$\overline{B} \coprod_{ M \coprod \overline{N} } (( \overline{M} \coprod N) \times [0,1])
\coprod_{ \overline{M} \coprod N} B'.$$
\item \ldots
\item A $k$-morphism in $\Cob{k}{n}{}$ is an $n$-manifold $X$ with corners, where
the structure of $\bd X$ is determined by the source and target of the morphism.
Two $n$-manifolds with (specified) corners $X$ and $Y$ determine the same $n$-morphism
in $\Cob{k}{n}{}$ if they differ by an orientation-preserving diffeomorphism, relative to their boundaries.
\item Composition of morphisms (at all levels) in $\Cob{k}{n}{}$ is given by gluing of bordisms.
\end{itemize}
Here we encounter the same issues as in Definition \ref{cobble}, but they are somewhat more serious. For $n > 2$, it is {\em not} possible to massage the above definition to produce a strict $n$-category:
gluing of bordisms is, at best, associative up to diffeomorphism. Nevertheless, $\Cob{k}{n}{}$ is
a perfectly respectable example of a (nonstrict) $n$-category: we will sketch a more precise definition of it
in \S \ref{swugg}. 
\end{example}

\begin{remark}
When $k=1$, the category $\Cob{k}{n}{}$ described in Example \ref{slipwell}
is just the usual bordism category $\Cob{}{n}{}$ of Definition \ref{slovak}. When
$k=0$, we can identify $\Cob{k}{n}{}$ with the {\em set} of diffeomorphism
classes of closed, oriented $n$-manifolds.
\end{remark}

We might now attempt to define an extended topological field theory to be a symmetric monoidal functor from the $n$-category $\Cob{n}{n}{}$ of Example \ref{slipwell} into a suitable
$n$-categorical generalization of $\Vect(k)$. Of course, there are many possible candidates for such a generalization. We will skirt the issue by adopting the following more general definition:

\begin{definition}\label{exdef}
Let $\calC$ be a symmetric monoidal $n$-category. An {\it extended $\calC$-valued topological field theory of dimension $n$} is a symmetric monoidal functor
$$ Z: \Cob{n}{n}{} \rightarrow \calC.$$
\end{definition}

At a first glance, Definition \ref{exdef} appears much more complicated than its more classical
counterpart, Definition \ref{atdef}. First of all, it is phrased in the language of $n$-categories,
which we have not yet introduced. Second, an extended topological field theory supplies a great deal more data than
that of an ordinary topological field theory: we can evaluate an extended field theory $Z$ on manifolds (with corners) of arbitrary dimension, rather than simply on closed $(n-1)$-manifolds and $n$-manifolds with boundary. Finally, the values of $Z$ on manifolds of low dimension are typically invariants of a very abstract and higher-categorical nature.

Nevertheless, one can argue that the notion an of extended field theory $Z$ should be quite a bit {\em simpler} than its non-extended counterpart.
Optimistically, one might hope to interpret the statement that $Z$ is a functor between $n$-categories
as saying that we have a complete toolkit which will allow us to compute the value $Z(M)$ given any decomposition of $M$ into pieces. While a closed $n$-manifold $M$ might look very complicated {\em globally}, it is {\em locally} very simple: by definition, every point $x \in M$ has a neighborhood which is diffeomorphic to Euclidean space $\R^n$. Consequently, we might hope to compute $Z(M)$ by breaking
$M$ up into elemental bits of manifold such as points, disks, or simplices.
If this were possible, then $Z$ would be determined by a very small amount of data. Indeed, we have already seen that this is exactly what happens in the case $n=1$: a $1$-dimensional topological field theory $Z$ is completely determined by a single vector space, given by evaluating $Z$ at a point (Example \ref{1dim}). (Note that when $n=1$, the distinction between extended topological field theories and ordinary topological field theories evaporates.)

Motivated by the $1$-dimensional case, we might hope to prove in general that an extended topological field theory $Z$ is determined by its value on a single point: in other words, that extended topological field theories with values in $\calC$ can be identified with objects of $\calC$. This hope turns out to be a bit too naive, for two reasons:

\begin{itemize}
\item[$(1)$] As we explained above, if $M$ is a closed manifold of dimension $n$, then for every
point $x \in M$ there is a diffeomorphism of $\R^{n}$ with an open neighborhood
of $x$ in $M$. However, this diffeomorphism is not unique. More canonically, we can say that
$x$ admits a neighborhood which is diffeomorphic to an open ball in the tangent space $T_{M,x}$ of $M$ at $x$. (This diffeomorphism is still not uniquely determined, but is determined up to a contractible space of choices if we require its derivative at $x$ to reduce to the identity map from $T_{M,x}$ to itself. Alternatively, if we choose a Riemannian metric on $M$, we can obtain a canonical diffeomorphism using the spray associated to the exponential flow on the tangent bundle $T_M$.) If $n=1$, then an orientation on $M$ allows us to trivialize the tangent bundle $T_M$, so the issue of noncanonicality does not arise. However, for $n > 1$ the potential
nontriviality of $T_M$ will play an important role in the classification of (extended) topological field theories.
\item[$(2)$] Even in the case where $n=1$ and $\calC = \Vect(k)$, it is not true that giving
of a $\calC$-valued topological field theory is equivalent to giving an object of $\calC$.
In Example \ref{1dim}, we saw that a $1$-dimensional topological field theory
$Z$ was uniquely determined by a single vector space $V = Z( \ast)$, which was required to
be finite dimensional (Proposition \ref{scunnel}). In the general case, the best we can expect is that $\calC$-valued extended topological field theories should be classified by objects of $\calC$
which satisfy suitable finiteness conditions, which generalize the condition that a vector space be finite-dimensional.
\end{itemize}

To address objection $(1)$, it is convenient to replace the oriented bordism $n$-category $\nCob$ by its framed analogue:

\begin{variant}\label{barvar}
Let $M$ be an $m$-manifold. A {\it framing} of $M$ is trivialization of the tangent bundle of $M$:
that is, an isomorphism $T_M \simeq \underline{\R}^m$ of vector bundles over $M$; here
$\underline{\R}^m$ denotes the trivial bundle with fiber $\R^m$. More generally,
if $m \leq n$, we define an {\it $n$-framing of $M$} to be a trivialization of the stabilized tangent bundle
$T_M \oplus \underline{\R}^{n-m}$. 

The {\it framed bordism $n$-category} $\Cob{n}{n}{\fr}$ is defined in the same way as $\Cob{n}{n}{}$
(see Example \ref{slipwell}), except that we require that all manifolds be equipped with an $n$-framing.
If $\calC$ is a symmetric monoidal $n$-category, then a {\it framed extended topological field theory} with values in $\calC$ is a symmetric monoidal functor of $n$-categories $\Cob{n}{n}{\fr} \rightarrow \calC$.
\end{variant}

\begin{remark}
There is an evident functor $\Cob{n}{n}{\fr} \rightarrow \Cob{n}{n}{}$, which discards framings and retains only the underlying orientations. By composing with this forgetful functor, every extended topological field theory determines a framed extended topological field theory. We will give a more precise account of the relationship between the framed and oriented field theories in \S \ref{ONACT}.
\end{remark}

To address objection $(2)$, we need to introduce the notion of a {\em fully dualizable} object
of a symmetric monoidal $n$-category $\calC$. We will defer the precise definition until
\S \ref{capertown}: for the moment, we note only that full dualizability is a natural finiteness condition in the $n$-categorical setting. Moreover, when $\calC$ is the category $\Vect(k)$ of vector spaces
over a field $k$ (so that $n=1$), an object $V \in \calC$ is fully dualizable if and only if it is a finite dimensional vector space.

The main objective of this paper is to sketch a proof of the following result (and various generalizations thereof):

\begin{theorem}[Baez-Dolan Cobordism Hypothesis]\label{swisher}
Let $\calC$ be a symmetric monoidal $n$-category. Then the evaluation functor
$$Z \mapsto Z(\ast)$$ determines a bijective correspondence between $($isomorphism classes of$)$
framed extended $\calC$-valued topological field theories and
$($isomorphism classes of$)$ fully dualizable objects of $\calC$.
\end{theorem}

Theorem \ref{swisher} asserts that for every fully dualizable object $C$ of a symmetric monoidal $n$-category $\calC$,
there is an essentially unique symmetric monoidal functor $Z_C: \Cob{n}{n}{\fr} \rightarrow \calC$
such that $Z_C(\ast) \simeq C$. In other words, the symmetric monoidal $\Cob{n}{n}{\fr}$ is {\em freely generated} by a single fully dualizable object: namely, the object consisting of a single point.

\begin{remark}
A version of Theorem \ref{swisher} was originally conjectured by Baez and
Dolan; we refer the reader to \cite{baezdolan} for the original statement, which differs
in some respects from the formulation presented here. For a proof of Theorem \ref{swisher} and some variations in the case $n=2$, we refer the reader to \cite{csp}.
\end{remark}

\subsection{Higher Category Theory}\label{higt}

In \S \ref{secex}, we introduced the notion of an {\it extended} topological field theory, and argued that this notion is best described using the language of higher category theory.
Our objective in this section is to give a brief informal introduction to higher categories; we will give a more precise account in \S \ref{bigseg}.

Roughly speaking, we would like to obtain the theory of $n$-categories by means of the inductive description given in Definition \ref{sncat}: an $n$-category $\calC$ consists of a set of objects
$X,Y, Z, \ldots$, together with an $(n-1)$-category $\OHom_{\calC}(X,Y)$ for every pair of
objects $X,Y \in \calC$. These $(n-1)$-categories should be equipped with an associative and
unital composition law. Definition \ref{sncat} requires that composition be associative on
the nose: that is, for every quadruple of objects $W,X,Y,Z \in \calC$, the diagram
$$ \xymatrix{ \OHom_{\calC}(W,X) \times \OHom_{\calC}(X,Y) \times
\OHom_{\calC}(Y,Z) \ar[r]^-{c_{W,X,Y}} \ar[d]^-{c_{X,Y,Z}} & \OHom_{\calC}(W,Y) \times \OHom_{\calC}(Y,Z) \ar[d]^{c_{W,Y,Z} }\\
\OHom_{\calC}(W,X) \times \OHom_{\calC}(X,Z) \ar[r]^-{c_{W,X,Z}} & \OHom_{\calC}(W,Z)}$$
is required to commute. We have already met examples (arising from the theory of bordisms
between manifolds) which {\em almost} fit this pattern: however, the above diagram is only commutative up to isomorphism. To accomodate these examples, it is natural to replace the
commutativity requirement by the assumption that there exists an isomorphism
$$\alpha_{W,X,Y,Z}: c_{W,X,Z} \circ (\id \times c_{X,Y,Z}) \simeq
c_{W,Y,Z} \circ (c_{W,X,Y} \times \id).$$
Moreover, we should not merely assume that this isomorphism exist: we should take it
as part of the data defining our $n$-category $\calC$. Furthermore, the isomorphisms
$\{ \alpha_{W,X,Y,Z} \}_{W,X,Y,Z \in \calC}$ must themselves be required to satisfy appropriate
``associativity'' conditions, at least up to isomorphism. These isomorphisms should themselves be specified, and subject to further associativity conditions. To properly spell out all of
the relevant structure is no small feat: it is possible to do this directly for small values of $n$, but
even for $n=3$ the definition is prohibitively complicated (see \cite{tricat}).

\begin{example}\label{tableop}
Let $X$ be a topological space. We can define a category $\pi_{\leq 1} X$, the
{\it fundamental groupoid} of $X$, as follows:
\begin{itemize}
\item The objects of $\pi_{\leq 1} X$ are the points of $X$.
\item Given a pair of points $x,y \in X$, a morphism from $x$ to $y$ in $\pi_{\leq 1} X$ is
a homotopy class of paths in $X$ which start at $x$ and end at $y$.
\end{itemize}
The fundamental groupoid is a basic invariant of the topological space $X$: note that it
determines the set $\pi_0 X$ of path components of $X$ (these are precisely the isomorphism classes of objects in $\pi_{\leq 1} X$), and also the fundamental group $\pi_1 (X,x)$ of
$X$ at each point $x \in X$ (this is the automorphism group of the object $x$ in
$\pi_{\leq 1}(X)$). 

The fundamental groupoid $\pi_{\leq 1} X$ does not retain any
other information about the homotopy type of $X$, such as the higher homotopy groups
$\{ \pi_{n}(X,x) \}_{n \geq 2}$. We can attempt to remedy the situation using higher
category theory. For each $n \geq 0$, one can define an $n$-category $\pi_{\leq n} X$, called
the {\it fundamental $n$-groupoid of $X$}. Informally, this $n$-category can be described as follows:
\begin{itemize}
\item The objects of $\pi_{\leq n} X$ are the points of $X$.
\item Given a pair of objects $x,y \in X$, a $1$-morphism in $\pi_{ \leq n} X$ from $x$ to $y$
is a path in $X$ from $x$ to $y$.
\item Given a pair of objects $x,y \in X$ and a pair of $1$-morphisms $f,g: x \rightarrow y$,
a $2$-morphism from $f$ to $g$ in $\pi_{\leq n} X$ is a homotopy of paths in $X$
(which is required to be fixed at the common endpoints $x$ and $y$).
\item \ldots
\item An $n$-morphism in $\pi_{\leq n} X$ is given by a homotopy between homotopies between \ldots between paths between points of $X$. Two such homotopies determined the same
$n$-morphism in $\pi_{\leq n} X$ if they are homotopic to one another (via a homotopy which is fixed on the common boundaries).
\end{itemize}
\end{example}

We say that $\pi_{\leq n} X$ is an {\it $n$-groupoid} because all of its $k$-morphisms are
invertible for $1 \leq k \leq n$. For example, every $1$-morphism $f: x \rightarrow y$ in $\pi_{\leq n} X$ is given by a path $p: [0,1] \rightarrow X$ such that $p(0) = x$ and $p(1) = y$. The path
$t \mapsto p(1-t)$ then determines a morphism from $y$ to $x$, which can be regarded as
an inverse to $f$ (at least up to isomorphism). 

There is a converse to Example \ref{tableop}, which is a generally-accepted principle of higher category theory:

\begin{thesis}\label{cope1}
Let $\calC$ be an {\it $n$-groupoid} $($that is, an $n$-category in which every $k$-morphism
is assumed to be invertible, for $0 < k \leq n${}$)$. Then $\calC$ is equivalent to
$\pi_{\leq n} X$ for some topological space $X$.
\end{thesis}

Of course, the topological space $X$ is not at all unique: for example, any two simply-connected spaces have equivalent fundamental groupoids. To eliminate this ambiguity, we recall
the following definition from classical homotopy theory:

\begin{definition}
A topological space $X$ is called an {\it $n$-type} if the homotopy groups
$\pi_k(X,x)$ vanish for all $x \in X$ and all $k > n$.
\end{definition}

For every topological $X$, one can construct an $n$-type $Y$ and a map
$f: X \rightarrow Y$ which is an isomorphism on homotopy groups in degrees
$\leq n$; the construction proceeds by attaching cells to ``kill'' the homotopy groups
of $X$ in degrees larger than $n$. The space $Y$ is uniquely determined up
to (weak) homotopy equivalence, and the induced map on fundamental $n$-groupoids
$\pi_{\leq n} X \rightarrow \pi_{\leq n} Y$ is an equivalence of $n$-categories.
Consequently, if we are only interested in studying the fundamental $n$-groupoids of topological spaces, there is no loss of generality in assuming that the spaces are $n$-types.
We can now formulate a refinement of Thesis \ref{cope1}:

\begin{thesis}\label{cope2}
The construction $X \mapsto \pi_{\leq n} X$ establishes a bijective correspondence between
$n$-types $($up to weak homotopy equivalence$)$ and $n$-groupoids $($up to equivalence$)$.
\end{thesis}

We refer to this assertion as a thesis, rather than a theorem, because we have not yet defined the notion of an $n$-category. Thesis \ref{cope2} should be regarded as a basic requirement that any 
reasonable definition of $n$-category must satisfy: when we restrict our attention to $n$-categories in which all morphisms are assumed to be invertible, then we should recover the classical homotopy theory of $n$-types. This makes $n$-groupoids much easier to work with than $n$-categories in
general: we can describe them in reasonably concrete terms without giving an inductive
description in the style of Definition \ref{sncat}, and without ever contemplating any ``higher associativity'' conditions.

Between the theory of $n$-categories in general (which are difficult to describe) and the
theory of $n$-groupoids (which are easy to describe) there are various intermediate levels of complexity.

\begin{definition}\label{spaid}
Suppose we are given a pair of nonnegative integers $m \leq n$. An
{\it $(n,m)$-category} is an $n$-category in which all $k$-morphisms are assumed to be invertible, for $m < k \leq n$.
\end{definition}

\begin{example}
An $(n,0)$-category is an $n$-groupoid; an
$(n,n)$-category is an $n$-category.
\end{example}

\begin{variant}
In Definition \ref{spaid}, it is convenient to allow the case $n = \infty$: in this case,
an $(n,m)$-category has morphisms of all orders, but all $k$-morphisms are assumed to be
invertible for $k > m$. It is possible to allow $m = \infty$ as well, but this case
will play no role in this paper.
\end{variant}

Taking $n$ to $\infty$ in the formulation of Thesis \ref{cope2}, we obtain the following:

\begin{thesis}\label{cope3}
There is a construction $X \mapsto \pi_{\leq \infty} X$ which establishes a bijection between
topological spaces $($up to weak homotopy equivalence$)$ and $(\infty,0)$-categories $($up to equivalence$)$.
\end{thesis}

One approach to the theory of higher categories is to turn Thesis \ref{cope3} into a definition:

\begin{definition}\label{splut}
An $(\infty,0)$-category is a topological space.
\end{definition}

\begin{convention}
Throughout the remainder of this paper, we will use Definition \ref{splut} implicitly and will often
not distinguish between the notions of $(\infty,0)$-category and topological space. In particular,
we will view topological spaces $X$ as special kinds of higher categories, so that it makes
sense to talk about functors $X \rightarrow \calC$, where $\calC$ is an $(\infty,n)$-category.
\end{convention}

We can now try to mimic the recursion of Definition \ref{sncat}, starting with
$(\infty,0)$-categories rather than sets.

\begin{protodefinition}\label{splut2}
For $n > 0$, an $(\infty,n)$-category $\calC$ consists of the following data:
\begin{itemize}
\item[$(1)$] A collection of objects $X, Y, Z, \ldots$
\item[$(2)$] For every pair of objects $X,Y \in \calC$, an
$(\infty,n-1)$-category $\OHom_{\calC}(X,Y)$ of $1$-morphisms.
\item[$(3)$] An composition law for $1$-morphisms which is associative
(and unital) up to coherent isomorphism.
\end{itemize}
\end{protodefinition}

At a first glance, Definition \ref{splut2} seems to suffer from the same defects of
Definition \ref{sncat}. Do we require the composition of morphisms to be associative
on the nose, or only up to isomorphism? If the latter, what sorts of coherence conditions do we need to require? However, it is slightly easier to address these questions in the
case of $(\infty,n)$-categories than in the case of ordinary $n$-categories:

\begin{itemize}
\item[$(a)$] Let $n=1$. If we require strict associativity in Definition \ref{splut}, then we recover
the notion of a {\em topological category}: that is, a category $\calC$ in which all morphism spaces $\Hom_{\calC}(X,Y)$ are equipped with topologies, and all of the composition maps
$c_{X,Y,Z}: \Hom_{\calC}(X,Y) \times \Hom_{\calC}(Y,Z) \rightarrow \Hom_{\calC}(X,Z)$ are
continuous. In this case, it is also possible to demand only a weak form of associativity,
in which the diagrams
$$ \xymatrix{ \Hom_{\calC}(W,X) \times \Hom_{\calC}(X,Y) \times
\Hom_{\calC}(Y,Z) \ar[r]^-{c_{W,X,Y}} \ar[d]^{c_{X,Y,Z}} & \Hom_{\calC}(W,Y) \times \Hom_{\calC}(Y,Z) \ar[d]^{c_{W,Y,Z}} \\
\Hom_{\calC}(W,X) \times \Hom_{\calC}(X,Z) \ar[r]^-{c_{W,X,Z}} & \Hom_{\calC}(W,Z)}$$
are required to commute only up to (specified) homotopy. However, this turns out to be unnecessary: 
every composition law which is associative ``up to coherent homotopy'' can be replaced by an equivalent composition law which is strictly associative. Consequently, the theory of
topological categories can be regarded as a version of the theory of $(\infty,1)$-categories.
However, this version is sometimes inconvenient, as we will see in \S \ref{sugar}; we will present
a more useful definition in \S \ref{bigseg}.

\item[$(b)$] One of the main obstacles to formulating a weak version of
Definition \ref{sncat} is that the notion of ``associative up to isomorphism'' is itself
a higher-categorical idea. In the case $n=2$, we are forced to 
consider diagrams of categories
$$ \xymatrix{ \OHom_{\calC}(W,X) \times \OHom_{\calC}(X,Y) \times
\OHom_{\calC}(Y,Z) \ar[r]^-{c_{W,X,Y}} \ar[d]^{c_{X,Y,Z}} & \OHom_{\calC}(W,Y) \times \OHom_{\calC}(Y,Z) \ar[d]^{c_{W,Y,Z}} \\
\OHom_{\calC}(W,X) \times \OHom_{\calC}(X,Z) \ar[r]^-{c_{W,X,Z}} & \OHom_{\calC}(W,Z)}$$
which should be required to commute up a natural isomorphism
$\alpha_{W,X,Y,Z}$. In this diagram,
each corner represents a category, and each arrow represents a functor. In other words,
we can regard the above as determining a diagram in the category
$\Cat$ of categories. However, if we want to formulate the idea that this diagram
commutes up to {\em isomorphism} (rather than ``on the nose''), then it is not enough
to think about $\Cat$ as a category: we need to contemplate not just categories and functors, but also natural transformations. In other words, we need to think of $\Cat$ as a $2$-category.
This skirts dangerously close to circular reasoning: we are trying to introduce the definition
in a $2$-category, so we should probably avoid giving a definition which already presupposes that we understand the $2$-category $\Cat$. 

However, we do not need to understand {\em all} natural transformations in order to
contemplate diagrams of categories which commute up to isomorphism: we only need
to consider {\em invertible} natural transformations. In other words, we need to think about
$\Cat$ as a $(2,1)$-category, where the objects are categories, the $1$-morphisms are functors,
and the $2$-morphisms are invertible natural transformations. We can therefore avoid
circularity provided that we have a good theory of $(2,1)$-categories (which is
a special case of the theory of $(\infty,1)$-categories provided by $(a)$). For one
implementation of this strategy we refer the reader to \cite{bicat}.
\end{itemize}

\subsection{Extending Up: Diffeomorphism Groups}\label{sugar}

In \S \ref{secex}, we saw that the language of higher category theory is useful for formulating
the notion of an {\it extended} topological field theory. Our goal in this section is to describe a different application of higher-categorical ideas to the study of topological field theories. We begin by discussing an example.

\begin{example}\label{trug}
Let $X$ be a smooth projective algebraic variety defined over a field $k$.
We can naturally associate to $X$ a graded algebra over
the field $k$, called the {\it Hochschild cohomology} of $X$ and denoted by
$\HochCoh^{\ast}(X)$. The algebra $\HochCoh^{\ast}(X)$ is {\it commutative in the graded sense}:
that is, for every pair of homogeneous elements $x \in \HochCoh^{p}(X)$,
$y \in \HochCoh^{q}(X)$, we have $xy = (-1)^{pq} yx$. In other words, 
$\HochCoh^{\ast}(X)$ is a commutative algebra in the category $\gVect(k)$ of
{\em $\Z/ 2 \Z$-graded} vector spaces over $k$. 

In the special case where $X$ is a Calabi-Yau variety of even dimension (that is, when the canonical bundle $\Omega^{\dim X}_{X}$ is trivial), there is a canonical trace map $\HochCoh^{\ast}(X) \rightarrow k$ (here we require $X$ to have even dimension in order to guarantee that this map
does not shift degree; this hypothesis is not really important). This trace is nondegenerate, and therefore endows $\HochCoh^{\ast}(X)$ with the structure of a (graded) commutative Frobenius algebra. A graded analogue of the reasoning
described in Example \ref{coming} will then tell us that $\HochCoh^{\ast}(X)$ determines a
$2$-dimension topological field theory $Z_X$ (taking values in graded $k$-vector spaces)
such that $Z_X(S^1) = \HochCoh^{\ast}(X)$; this field theory is sometimes called the {\it $B$-model with target $X$}. 
\end{example}

Example \ref{trug} illustrates a feature which is common to many examples of topological field theories $Z$: the vector spaces $Z(M)$ are naturally given as some kind of homology or cohomology. In other words, there is a more basic invariant $\overline{Z}(M)$, which
takes values not in vector spaces but in {\em chain complexes} of vector spaces, such that
$Z(M)$ is obtained from $\overline{Z}(M)$ by passing to homology. We can attempt to axiomatize the situation by introducing the notion of a {\it chain-complex valued} topological field theory:

\begin{idefinition}\label{spout}
Let $k$ be a field. A {\it chain-complex valued topological field theory} of dimension
$n$ is a symmetric monoidal functor $\overline{Z}: \Cob{}{n}{} \rightarrow \Chain(k)$, where
$\Chain(k)$ denotes the category of chain complexes of $k$-vector spaces
$$ \ldots V_2 \rightarrow V_1 \rightarrow V_0 \rightarrow V_{-1} \rightarrow V_{-2} \rightarrow \ldots$$
\end{idefinition}

To get a feeling for why Definition \ref{spout} is unreasonable, let us suppose that we begin with
an $n$-dimensional topological field theory $Z$ which takes values in graded vector spaces, such as the $B$-model described in Example \ref{trug}. We might then ask if it is possible to promote $Z$ to a symmetric monoidal functor $\overline{Z}: \Cob{}{n}{} \rightarrow \Chain(k)$.
In particular, this would mean the following:
\begin{itemize}
\item[$(i)$] For every closed oriented manifold $M$ of dimension $(n-1)$, the graded vector
space $Z(M)$ is obtained as the homology of a chain complex $\overline{Z}(M)$.
\item[$(ii)$] For every oriented bordism $B$ from a closed $(n-1)$-manifold $M$ to
another closed $(n-1)$-manifold $N$, the map $Z(B): Z(M) \rightarrow Z(N)$ is
obtained from a map of chain complexes $\overline{Z}(B): \overline{Z}(M) \rightarrow
\overline{Z}(N)$ by passing to homology.
\item[$(iii)$] Suppose that $B$ and $B'$ are two oriented bordisms from one closed
$(n-1)$-manifold $M$ to another closed $(n-1)$-manifold $N$. If there exists an orientation-preserving diffeomorphism $\phi: B \rightarrow B'$ which reduces to the identity on $\bd B \simeq \overline{M} \coprod N \simeq \bd B'$, then $\overline{Z}(B) = \overline{Z}(B')$.
\end{itemize}

Conditions $(i)$ and $(ii)$ are generally satisfied in practice, but $(iii)$ is not.
For example, the chain map $\overline{Z}(B)$ might be defined only after choosing some
additional data on $B$, like a Riemannian metric, which is not diffeomorphism invariant.
However, all is not lost: because $Z$ is assumed to be a topological field theory in the usual sense, we know that $Z(B) = Z(B')$, so that the maps $\overline{Z}(B)$ and $\overline{Z}(B')$ induce the same map after passing to homology. In fact, this generally happens for a reason: for example, because there exists a chain homotopy $h$ relating $\overline{Z}(B)$ and $\overline{Z}(B')$.
This homotopy $h$ is generally nonzero, and depends on a choice of diffeomorphism
$\phi: B \rightarrow B'$. To discuss the situation more systematically, it is useful to introduce some terminology.

\begin{notation}
Let $M$ and $N$ be closed, oriented $(n-1)$-manifolds. We let
$\calB(M,N)$ denote a {\em classifying space} for bordisms from $M$ to $N$.
More precisely, consider the category $\calC$ whose objects are oriented bordisms $B$
from $M$ to $N$, where the morphisms are given by (orientation-preserving) diffeomorphisms
which reduce to the identity on $M$ and $N$. This is naturally a {\em topological} category:
that is, for every pair of bordisms $B$ and $B'$, the collection of diffeomorphisms
$\Hom_{\calC}(B,B')$ has a topology (the topology of uniform convergence of all derivatives)
such that the composition maps are continuous. We can then define $\calB(M,N)$ to be the
{\it classifying space} $B \calC$. 

Alternatively, we can characterize the space $\calB(M,N)$ up to homotopy equivalence
by the following property: there exists a fiber bundle $p: E \rightarrow \calB(M,N)$ whose fibers
are (smooth) bordisms from $M$ to $N$. This fiber bundle is universal in the following sense:
for any reasonable space $S$, pullback of $E$ determines a bijective correspondence between
homotopy classes of maps from $S$ into $\calB(M,N)$ and fiber bundles
$E' \rightarrow S$ whose fibers are (smooth) bordisms from $M$ to $N$.
In particular (taking $S$ to consist of a single point), we deduce that the set of path components
$\pi_0 \calB(M,N)$ can be identified with the collection of {\em diffeomorphism classes}
of bordisms from $M$ to $N$. In other words, we have a bijection
$\pi_0 \calB(M,N) \simeq \Hom_{ \Cob{}{n}{}}(M,N)$.
\end{notation}

Let us now return to our analysis of Definition \ref{spout}. Suppose that
$Z$ is an $n$-dimensional topological field theory and that we are attempting
to lift $Z$ to a chain-complex valued field theory $\overline{Z}$ which satisfies $(i)$ and
$(ii)$. It is not reasonable to demand condition $(iii)$ as stated, but we expect that it least
holds up to homotopy: that is, that $\overline{Z}$ determines a well-defined map
$$ \alpha: \pi_0 \calB(M,N) = \Hom_{ \Cob{}{n}{}}(M,N) \rightarrow [ \overline{Z}(M), \overline{Z}(N) ].$$
Here $[ \overline{Z}(M), \overline{Z}(N) ]$ denotes the collection of chain homotopy
classes of maps from $\overline{Z}(M)$ to $\overline{Z}(N)$. Because
$[ \overline{Z}(M), \overline{Z}(N) ]$ has the structure of a vector space over $k$, the map
$\alpha$ determines a $k$-linear map $\HH_0( \calB(M,N); k) \rightarrow [ \overline{Z}(M), \overline{Z}(N) ]$ (here we invoke the fact that the homology group $\HH_0( \calB(M,N); k)$ can be
identified with the free $k$-vector space generated by the set $\pi_0 \calB(M,N)$).
Note that $[ \overline{Z}(M), \overline{Z}(N) ]$ can itself be identified with $0$th homology
group of a certain chain complex: namely, the chain complex
$\bHom( \overline{Z}(M), \overline{Z}(N) )$ described by the formula
$$ \bHom( \overline{Z}(M), \overline{Z}(N) )_{i} = \prod_{n} \Hom( \overline{Z}(M)_n,
\overline{Z}(N)_{n+i}).$$
It is therefore natural to propose the following replacement for conditions $(ii)$ and $(iii)$:
\begin{itemize}
\item[$(iii')$] For every pair of closed oriented $(n-1)$-manifolds $M$ and $N$, there
is a map of chain complexes
$$ \gamma: C_{\ast}( \calB(M,N); k) \rightarrow \bHom( \overline{Z}(M), \overline{Z}(N) ).$$
Here $C_{\ast}( \calB(M,N); k)$ denotes the complex of singular $k$-valued chains
on the topological space $\calB(M,N)$.
\end{itemize}

Let us take a moment to unwind the structure described by $(iii')$. 
First of all, we get a map at the level of $0$-chains
$$ C_0( \calB(M,N); k) \rightarrow \bHom( \overline{Z}(M), \overline{Z}(N) )_0.$$
On the left hand side, every $0$-chain is automatically a $0$-cycle (since there
are no nonzero $(-1)$-chains); on the right hand side, the $0$-cycles are precisely the
{\em chain maps} from $\overline{Z}(M)$ to $\overline{Z}(N)$. We therefore obtain a map
$\gamma_0: C_0( \calB(M,N); k) \rightarrow \Hom( \overline{Z}(M), \overline{Z}(N) )$.
The left hand side can be identified with the free vector space generated by the points
in the classifying space $\calB(M,N)$. We may therefore interpret the map $\gamma$
as associating to every point $x \in \calB(M,N)$ a map of chain complexes
$\gamma_0(x): \overline{Z}(M) \rightarrow \overline{Z}(N)$. Since giving a point of
the classifying space $\calB(M,N)$ is essentially the same thing as giving a bordism
from $M$ to $N$, this is equivalent to the data described in $(ii)$.

Let us now consider the induced map at the level of $1$-chains
$$ \gamma_1: C_1( \calB(M,N); k ) \rightarrow \Hom( \overline{Z}(M), \overline{Z}(N) ).$$
The domain of $\gamma_1$ can be identified with the free $k$-vector space generated
by the set of paths $p: [0,1] \rightarrow \calB(M,N)$. Every such path begins at a point
$x = p(0)$ and ends at a point $y = p(1)$. The requirement that $\gamma$ be a map of chain complexes translates into the assertion that $\gamma_1(p)$ is a {\em chain homotopy}
between the chain maps $\gamma_0(x), \gamma_0(y): \overline{Z}(M) \rightarrow \overline{Z}(N)$.
Giving a path $p$ from $x$ to $y$ is essentially the same data as giving a {\em diffeomorphism}
between the bordisms determined by the points $x$ and $y$. Consequently, we can
regard the map $\gamma$ at the level of $1$-chains as an efficient way of encoding the structure described earlier: diffeomorphic bordisms from $M$ to $N$ give rise to chain homotopic maps
from $\overline{Z}(M)$ to $\overline{Z}(N)$, via a chain homotopy which depends on a choice of
diffeomorphism $\phi$. The requirement that $\gamma$ to be defined also in higher degrees translates into the requirement that this dependence is in some sense continuous in $\phi$.

\begin{remark}
Giving a chain map $\gamma: C_{\ast}( \calB(M,N); k) \rightarrow
\bHom( \overline{Z}(M), \overline{Z}(N) )$ is equivalent to giving a chain map
$$\delta: C_{\ast}( \calB(M,N); k) \otimes \overline{Z}(M) \rightarrow \overline{Z}(N).$$
Passing to the level of homology, we get a $k$-linear map $\HH_{\ast}( \calB(M,N); k) \otimes Z(M) \rightarrow Z(N).$ If we restrict our attention to the $0$-dimensional homology of
$\calB(M,N)$, we obtain a map $\HH_0( \calB(M,N); k) \otimes Z(M) \rightarrow Z(N)$: this
simply encodes the fact that every oriented bordism $B$ from $M$ to $N$ determines
a map $Z(B): Z(M) \rightarrow Z(N)$. However, $\gamma$ also determines maps
$\HH_{n}( \calB(M,N); k) \otimes Z(M) \rightarrow Z(N)$ for $n > 0$, which
are not determined by the original topological field theory $Z$. This can
be interesting from multiple points of view. For example, if we are primarily interested
in understanding the topological field theory $Z$, then every lifting
$\overline{Z}$ of $Z$ satisfying $(i)$ and $(iii')$ gives rise to additional operations
on the vector spaces $Z(M)$, which are parametrized by the (higher) homology
of the classifying spaces $\calB(M,N)$. Alternatively, can use these operations
as means to investigate the structure of the classifying spaces $\calB(M,N)$ themselves.
\end{remark}

We would like to give another description of the data posited by assumption $(iii')$.
For this, we need to embark on a mild digression. Suppose we are given a topological space $X$ and a chain complex $V_{\ast}$; we would like to better understand the collection of chain maps from
$C_{\ast}(X;k)$ to $V_{\ast}$. In practice, we are interested in the case where
$X$ is a classifying space $\calB(M,N)$ and $V_{\ast} = \bHom( \overline{Z}(M), \overline{Z}(N) )$. However, as a warm-up exercise, let us first consider the simplest nontrivial case where
$$ V_{m} = \begin{cases} k & \text{if } m = n \\
0 & \text{if } m \neq n. \end{cases}$$
In this case, a chain map from $C_{\ast}(X;k)$ into $V_{\ast}$ can be identified with
a $k$-valued {\it $n$-cocycle} on $X$, and two such chain maps are homotopic if and only if
they differ by a coboundary. The set of chain homotopy classes of maps from
$C_{\ast}(X;k)$ into $V_{\ast}$ can therefore be identified with the cohomology group
$\HH^{n}(X;k)$.

If $X$ is a sufficiently nice topological space, then the cohomology group
$\HH^{n}(X;k)$ can be described in another way: it is the set of homotopy classes of maps
$[ X, K(k,n) ]$. Here $K(k,n)$ denotes an {\it Eilenberg-MacLane space}: it is characterized
up to homotopy equivalence by its homotopy groups, which are given by
$$ \pi_{m} K(k,n) = \begin{cases} k & \text{if } m=n \\
0 & \text{otherwise.} \end{cases}$$

It is possible to give a similar description in the general case: for any chain complex
$V_{\ast}$, the set of chain homotopy classes of maps from $C_{\ast}(X;k)$ into
$V_{\ast}$ can be identified with the set of homotopy classes of maps from
$X$ into a certain topological space $K(V_{\ast})$, at least provided that $X$ is sufficiently
nice. Here $K(V_{\ast})$ denotes a {\em generalized} Eilenberg-MacLane space whose homotopy groups are given by the formula $\pi_{m} K(V_{\ast}) \simeq \HH_{m}( V_{\ast} )$. 
The space $K(V_{\ast})$ is generally not characterized by this formula, but it is determined
up to homotopy equivalence by the universal property stated above.

We can now reformulate assumption $(iii')$ as follows:
\begin{itemize}
\item[$(iii'')$] For every pair of closed oriented $(n-1)$-manifolds $M$ and $N$,
there is a map of topological spaces
$$ \gamma_{M,N}: \calB(M,N) \rightarrow K( \bHom( \overline{Z}(M), \overline{Z}(N) )).$$
\end{itemize}

Of course, we do not want to stop with $(iii'')$. It is not enough to specify
the maps $\gamma_{M,N}$ separately for every pair of manifolds $M,N \in \Cob{}{n}{}$:
we should also say how these maps are related to one another. This leads us to propose the following revised version of Definition \ref{spout}:

\begin{idefinition}\label{spout2}
Let $k$ be a field. A {\it chain-complex valued topological field theory}
of dimension $n$ is a continuous symmetric monoidal functor 
$$ \overline{Z}: \tCob(n) \rightarrow \tChain(k)$$ between topological categories.
Here the topological categories $\tCob(n)$ and $\tChain(n)$ can be described as follows:
\begin{itemize}
\item The objects of $\tCob(n)$ are closed oriented manifolds of dimension $(n-1)$.
\item Given a pair of objects $M, N \in \tCob(n)$, we let
$\Hom_{ \tCob(n)}(M,N)$ denote the classifying space $\calB(M,N)$ of bordisms from $M$ to $N$.
\item The objects of $\tChain(k)$ are chain complexes of $k$-vector spaces.
\item Given a pair of chain complexes $V_{\ast}$ and $W_{\ast}$, we define
$\Hom_{ \tChain(k)}( V_{\ast}, W_{\ast} )$ to be the generalized
Eilenberg-MacLane space $K( \bHom(V_{\ast}, W_{\ast}))$. 
\end{itemize}
\end{idefinition}

Definition \ref{spout2} is a vast improvement over Definition \ref{spout}, but still not quite
adequate:

\begin{itemize}
\item[$(a)$] Our definition of $\tCob(n)$ is incomplete because we
did not explain how to compose morphisms. Unwinding the definitions,
we see that $\Hom_{ \tCob(n) }( M,M')$ is the classifying of the (topological) category
$\calC_{M,M'}$ of oriented bordisms from $M$ to $M'$, where the morphisms
are given by diffeomorphisms. We would like to say that for a triple of
objects $M, M', M'' \in \tCob(n)$, the composition law
$$\Hom_{ \tCob(n)}(M,M') \times \Hom_{ \tCob(n)}(M', M'') \rightarrow \Hom_{ \tCob(n)}(M,M'')$$
is induced by a functor $\calC_{M,M'} \times \calC_{M',M''} \rightarrow \calC_{M, M''}$
given by ``gluing along $M'$''. We encounter a minor technicality having to do with
smoothness: given a pair of bordisms $B: M \rightarrow M'$ and $B': M' \rightarrow M''$, the
coproduct $B \coprod_{M'} B'$ does not inherit a smooth structure. However, this problem
can be avoided by giving more careful definitions: namely, we should require every
bordism from $M$ to $M'$ to come equipped with distinguished smooth collars
near $M$ and $M'$. 

Another issue is that the coproduct $B \coprod_{M'} B'$ is only well-defined up to isomorphism. This does not prevent us from defining a gluing functor
$\calC_{M,M'} \times \calC_{M',M''} \rightarrow \calC_{M, M''}$, but it does
mean that this functor is only well-defined up to isomorphism. Consequently, the diagram
which encodes the associativity of composition
$$ \xymatrix{ \calC_{M,M'} \times \calC_{M',M''} \times \calC_{M'',M'''} \ar[r] \ar[d] & \calC_{M, M''} \times \calC_{M'', M'''} \ar[d] \\
\calC_{M,M'} \times \calC_{M', M'''} \ar[r] & \calC_{M,M'''} }$$
can only be expected to commute up to isomorphism, so that induced diagram
of classifying spaces will commute only up to homotopy.
This problem can again be avoided in an ad hoc way by giving sufficiently careful definitions. We will not pursue the details any further, since these difficulties will disappear when we use the more sophisticated formalism of \S \ref{bigseg}. 

\item[$(b)$] Our definition of $\tChain(k)$ is also incomplete. According to Definition \ref{spout2}, if
$V_{\ast}$ and $W_{\ast}$ are chain complexes of vector spaces, then the space of morphisms
$\Hom_{ \tChain(k)}( V_{\ast}, W_{\ast} )$ is a generalized Eilenberg-MacLane space
$K( \bHom(V_{\ast}, W_{\ast}) )$ associated to the mapping complex $\bHom( V_{\ast}, W_{\ast})$.
The discussion above shows that the generalized Eilenberg-MacLane space 
$K( \bHom(V_{\ast}, W_{\ast}) )$ is well-defined up to homotopy equivalence. In order to define the topological category, we need to choose a specific construction
for the generalized Eilenberg-MacLane space $K( U_{\ast} )$ associated to a complex $U_{\ast}$.
Moreover, we need this construction to be functorial in $U_{\ast}$ and to behave well with respect to tensor products. This is again possible (and not very difficult), but we do not want to dwell on the details here.

\item[$(c)$] Even if we take the trouble to construct $\tCob(n)$ and $\tChain(k)$ as topological categories, the notion of {\em continuous functor} appearing in Definition \ref{spout2} is 
too strict. For example, suppose that $\overline{Z}: \tCob(n) \rightarrow
\tChain(k)$ is a continuous functor, and that we are given a triple of objects
$M, M', M'' \in \tCob(n)$. Functoriality guarantees us that the diagram of topological spaces
$$ \xymatrix{ \Hom_{ \tCob(n)}(M,M') \times \Hom_{ \tCob(n)}(M', M'') \ar[r] \ar[d] &
\Hom_{ \tCob(n)}(M, M'') \ar[d] \\
\Hom_{ \tChain(k)}( \overline{Z}(M), \overline{Z}(M') ) \times
\Hom_{ \tChain(k)}( \overline{Z}(M'), \overline{Z}(M'') ) \ar[r] &\Hom_{ \tChain(k) }( \overline{Z}(M), \overline{Z}(M'')) }$$
is commutative. All of the spaces in this diagram are (products of) classifying spaces of manifolds and
generalized Eilenberg-MacLane spaces: in other words, they are characterized up to homotopy equivalence by some universal property. In this context, it is somewhat unnatural to demand
such a diagram to be commutative: one should instead require that it commute up to a specified homotopy.
\end{itemize}

To address these objections, we recall from \S \ref{higt} that the theory of topological categories
can be regarded as one approach to the study of $(\infty,1)$-categories: that is, higher categories
in which all $k$-morphisms are assumed to invertible for $k > 1$. This approach is conceptually very simple (it is very easy to describe what a topological category is) but technically very inconvenient, essentially because of difficulties like those described above. We can circumvent them
by reformulating Definition \ref{spout2} in terms of a better theory of $(\infty,1)$-categories,
which we will present in \S \ref{bigseg}.

To close this section, let us make a few remarks about how the higher-categorical issues
of this section relate to those described in \S \ref{secex}. The topological category
$\tCob(n)$ of Definition \ref{spout2} should really be regarded
as an $(\infty,1)$-category, which may be described more informally as follows:
\begin{itemize}
\item The objects of $\tCob(n)$ are closed, oriented $(n-1)$-manifolds.
\item The $1$-morphisms of $\tCob(n)$ are oriented bordisms.
\item The $2$-morphisms of $\tCob(n)$ are orientation-preserving diffeomorphisms.
\item The $3$-morphisms of $\tCob(n)$ are isotopies between diffeomorphisms.
\item \ldots 
\end{itemize}
Like the $n$-category $\Cob{n}{n}{}$ of Example \ref{slipwell}, we can regard $\tCob(n)$
as a higher-categorical version of the usual bordism category $\Cob{}{n}{}$. However,
these versions are related to $\Cob{}{n}{}$ in different ways:
\begin{itemize}
\item[$(1)$] Objects and morphisms of $\Cob{}{n}{}$ can be regarded as
$(n-1)$-morphisms and $n$-morphisms of $\Cob{n}{n}{}$. We may therefore regard $\Cob{n}{n}{}$ as
an elaboration of $\Cob{}{n}{}$ obtained by considering also ``lower'' morphisms
corresponding to manifolds of dimension $< n-1$.
\item[$(2)$] The objects of $\Cob{}{n}{}$ and $\tCob(n)$ are the same, and morphisms
in $\Cob{}{n}{}$ are simply the isomorphism classes of $1$-morphisms in $\tCob(n)$.
We may therefore regard $\tCob(n)$ as an elaboration of $\Cob{}{n}{}$ obtained by
allowing higher morphisms which keep track of the diffeomorphism groups of $n$-manifolds, rather than simply identifying diffeomorphic $n$-manifolds.
\end{itemize}
These variations on the definition of $\Cob{}{n}{}$ are logically independent of one another, but
the formalism of higher category theory allows us to combine them in a natural way:

\begin{protodefinition}\label{cabber2}
Let $n$ be a nonnegative integer. The $(\infty,n)$-category $\Bord_{n}$ is described informally as follows:
\begin{itemize}
\item The objects of $\Bord_{n}$ are $0$-manifolds.
\item The $1$-morphisms of $\Bord_{n}$ are bordisms between $0$-manifolds.
\item The $2$-morphisms of $\Bord_{n}$ are bordisms between bordisms between
$0$-manifolds.
\item \ldots
\item The $n$-morphisms of $\Bord_{n}$ are bordisms between bordisms between
$\ldots$ between bordisms between $0$-manifolds (in other words, $n$-manifolds with corners).
\item The $(n+1)$-morphisms of $\Bord_{n}$ are diffeomorphisms
(which reduce to the identity on the boundaries of the relevant manifolds).
\item The $(n+2)$-morphisms of $\Bord_{n}$ are isotopies of diffeomorphisms.
\item \ldots
\end{itemize}
\end{protodefinition}

\begin{remark}
The $(\infty,n)$-category $\Bord_{n}$ is endowed with a symmetric monoidal structure, given by disjoint unions of manifolds.
\end{remark}

\begin{variant}
In Definition \ref{cabber2}, we can consider manifolds equipped with various
structures such as orientations and $n$-framings (see Variant \ref{barvar}); in these cases we obtain
variants on the $(\infty,n)$-category $\Bord_{n}$ which we will denote by
$\Bord_{n}^{\ori}$ and $\Bord_{n}^{\fr}$. We will discuss other variations on this theme in \S \ref{ONACT}.
\end{variant}

We now formulate an $(\infty,n)$-categorical version of the cobordism hypothesis:

\begin{theorem}[Cobordism Hypothesis: $(\infty,n)$-Categorical Version]\label{swisher2}
Let $\calC$ be a symmetric monoidal $(\infty,n)$-category. The evaluation functor
$Z \mapsto Z(\ast)$ determines a bijection between $($isomorphism classes of$)$
symmetric monoidal functors $\Bord_{n}^{\fr} \rightarrow \calC$ and $($isomorphism classes of$)$
fully dualizable objects of $\calC$.
\end{theorem}

\begin{remark}\label{camblis}
If $\calD$ is any $(\infty,n)$-category, then we can define an $n$-category
$\hn{\calD}$ as follows:
\begin{itemize}
\item[$(i)$] For $k < n$, the $k$-morphisms of $\hn{\calD}$ are the $k$-morphisms
of $\calD$.
\item[$(ii)$] The $n$-morphisms of $\hn{\calD}$ are given by isomorphism classes
of $n$-morphisms in $\calD$. 
\end{itemize}
This construction can be characterized by the following universal property:
let $\calC$ be an $n$-category, which we can regard as an $(\infty,n)$-category which has only identity $k$-morphisms for $k > n$. Then functors (of $n$-categories) from
$\hn{\calD}$ to $\calC$ can be identified with functors (of $(\infty,n)$-categories) from
$\calD$ to $\calC$. We call $\hn{\calD}$ the {\it homotopy $n$-category} of $\calD$.

If $\calD = \Bord_{n}^{\fr}$, then the homotopy $n$-categry $\hn{\calD}$ can be
identified with the $n$-category $\Cob{n}{n}{\fr}$ described in Variant \ref{barvar}. It follows from
the above universal property that our original formulation of the cobordism hypothesis (Theorem \ref{swisher}) is equivalent to a special case of Theorem \ref{swisher2}: namely, the special case in which we assume that $\calC$ is an ordinary $n$-category.
\end{remark}

\begin{remark}\label{cultis}
Though the original formulation of the cobordism hypothesis (Theorem \ref{swisher}) may appear
to be simpler than Theorem \ref{swisher2}, it is actually essential to our proof that we work in the
the more general setting of $(\infty,n)$-categories. This is because the proof proceeds by induction on $n$: in order to understand the $(n+1)$-morphisms in $\Bord_{n+1}$, we will need to understand
the $(n+1)$-morphisms in $\Bord^{\fr}_{n}$, which are forgotten by passing from
$\Bord^{\fr}_{n}$ to the $n$-category $\Cob{n}{n}{\fr}$.
\end{remark}

\section{Formulation of the Cobordism Hypothesis}\label{glob2}

In \S \ref{glob1}, we gave an informal introduction to the language of
higher category theory and used that language to formulate a version of the Baez-Dolan cobordism hypothesis (Theorem \ref{swisher2}) which posits a classification of extended topological field theories. Before we can describe this classification in precise mathematical terms, we need to answer a number of questions:

\begin{itemize}
\item[$(a)$] What is an $(\infty,n)$-category?
\item[$(b)$] What is a functor between $(\infty,n)$-categories?
\item[$(c)$] What is a symmetric monoidal structure on an $(\infty,n)$-category, and what
does it mean for a functor to be symmetric monoidal?
\item[$(d)$] What is the $(\infty,n)$-category $\Bord_{n}^{\fr}$?
\item[$(e)$] What does it mean for an object of a symmetric monoidal $(\infty,n)$-category to be fully dualizable?
\end{itemize}

To properly address all of these questions would require a more thorough discussion than we have space to give here. Nevertheless, we would like to convey some of the flavor of the mathematics that provides the answers (and to dispel any sense that the basic objects of higher category theory
are ill-defined). We will therefore devote \S \ref{bigseg} to describing a rigorous approach
to the study of $(\infty,n)$-categories, using Rezk's theory of
{\it complete Segal spaces} (and its higher-dimensional analogue, due to Barwick). In
\S \ref{swugg}, we will address $(d)$ by giving a construction of $\Bord_{n}^{\fr}$ using
the language of complete Segal spaces. These sections are somewhat technical, and can safely be omitted by the reader who wishes to avoid the details: once we have given a precise definition for the notion of an $(\infty,n)$-category, we will promptly ignore it and return to the somewhat informal approach of \S \ref{glob1}. In the interest of space, we will gloss over $(b)$ and
$(c)$ (for an extensive discussion of $(c)$ in the case $n=1$ we refer the reader to \cite{dag3}). 

In \S \ref{capertown}, we will address question $(e)$ by studying various finiteness conditions in the setting of higher category theory. This will allow us to reformulate Theorem \ref{swisher2} as follows:
$\Bord_{n}^{\fr}$ is the free symmetric monoidal $(\infty,n)$-category with duals generated by a single object. In \S \ref{ONACT} we will present this formulation, deduce some of its consequences, and
explain how it can be generalized to the case of manifolds which are not framed. In the special
case of topological field theories taking values in a Picard $\infty$-groupoid, this generalization
reduces to a homotopy-theoretic statement which was proven by Galatius, Madsen, Tillmann, and Weiss. In \S \ref{mumm} we will briefly review their work and its connection with the cobordism hypothesis presented here.

\subsection{Complete Segal Spaces}\label{bigseg}

In \S \ref{bigseg}, we argued that it is most natural to describe topological field theories using the language of higher category theory. This theory has a reputation for being a thorny and technical subject. This is largely due to the fact that it is very
easy to give definitions which are incorrect or poorly behaved (some of which have already appeared earlier in this paper). However, there are many (equivalent) ways to give reasonable definitions which generate a well-behaved theory. Our first goal in this section is to describe such an approach in
the setting of $(\infty,1)$-categories: Rezk's theory of {\it complete Segal spaces}. We are interested in this approach primarily for two reasons:

\begin{itemize}
\item[$(1)$] The $(\infty,1)$-category $\tCob(n)$, which we
struggled to describe as a topological category in Definition \ref{spout2}, arises
much more naturally in the language of complete Segal spaces, as we will see in \S \ref{swugg}. 
\item[$(2)$] The notion of a complete Segal space can be generalized to produce a good theory
of $(\infty,n)$-categories for each $n \geq 0$. This generalization is due originally to Barwick, and will be sketched below; for more details we refer the reader to \cite{bicat}.
\end{itemize}

To explain the basic idea, let us pretend for the moment that we already
have a good theory of $(\infty,1)$-categories, and that we would like to describe
this theory in concrete terms.
According to Thesis \ref{cope3}, the theory of $(\infty,0)$-categories is ``easy'': it is equivalent to the homotopy theory of topological spaces. The general case is more complicated, because a general $(\infty,1)$-category $\calC$ might contain noninvertible $1$-morphisms. However, we can always simplify $\calC$ by throwing those $1$-morphisms away.
Namely, we can extract an $(\infty,0)$-category $\calC_0$, which
can be described roughly as follows:
\begin{itemize}
\item The objects of $\calC_0$ are the object of $\calC$.
\item The $1$-morphisms of $\calC_0$ are the invertible $1$-morphisms of $\calC$.
\item The $2$-morphisms of $\calC_0$ are the $2$-morphisms between invertible $1$-morphisms of $\calC$.
\item $\ldots$
\end{itemize}
Since all of the morphisms in $\calC_0$ are invertible, Thesis \ref{cope3} allows us to identify
$\calC_0$ with a topological space $X_0$. The space $X_0$
can be regarded as an invariant of $\calC$: we will sometimes refer to it as a {\it classifying space} for objects
of $\calC$ (for example, the path components of $X_0$ are in bijection with
the isomorphism classes of objects of $\calC$).

If $\calC$ is an $(\infty,0)$-category, then $\calC$ is determined (up to equivalence) by the topological space $X_0$. However, it generally is not: the space $X_0$ does not contain any
information about the noninvertible $1$-morphisms in $\calC$. We can think of
a morphism in $\calC$ as a functor $[1] \rightarrow \calC$, where $[1]$ denotes the ordinary
category associated to the linearly ordered set $\{ 0 < 1 \}$. The collection of all such functors
is naturally organized into another $(\infty,1)$-category, which we will denote by $\Fun( [1], \calC)$.
We can now repeat the process described above: let $\calC_1$ denote the $(\infty,0)$-category
obtained from $\Fun( [1], \calC)$ by discarding the noninvertible $1$-morphisms. According
to Thesis \ref{cope3}, we should be able to identify $\calC_1$ with the fundamental $\infty$-groupoid of another topological space, which we will denote by $X_1$. We can think of $X_1$ as a {\it classifying space} for $1$-morphisms in $\calC$.

The topological space $X_1$ remembers a little bit more about the $(\infty,1)$-category $\calC$:
namely, the class of $1$-morphisms in $\calC$. But we still do not have enough information to reconstruct $\calC$, because $X_0$ and $X_1$ do not remember anything about {\em compositions} between noninvertible $1$-morphisms in $\calC$. To correct this defect, let us consider the
collection of all {\em pairs} of composable $1$-morphisms $X \stackrel{f}{\rightarrow} Y \stackrel{g}{\rightarrow} Z$ in $\calC$. Such a pair of morphisms can be identified with a functor
$[2] \rightarrow \calC$, where $[2]$ denotes the (ordinary) category associated to the linearly
ordered set $\{ 0 < 1 < 2 \}$. More generally, we can consider for every nonnegative integer $n$ the linearly ordered set $[n] = \{ 0 < 1 < \ldots < n \}$, which we regard as an ordinary category. The collection of all functors $[n] \rightarrow \calC$ is naturally organized into an $\infty$-category
$\Fun( [n], \calC)$. We can then discard the noninvertible $1$-morphisms to obtain
an $(\infty,0)$-category $\calC_n$, which we can identify with the fundamental $\infty$-groupoid of a topological space $X_n$. 

We might now ask the following questions:
\begin{itemize}
\item[$(Q1)$] How are the spaces $X_n$ related to one another? In other words, what
sort of mathematical object is the totality $\{ X_n \}_{n \geq 0}$?
\item[$(Q2)$] What special features, if any, does this mathematical object possess?
\item[$(Q3)$] To what extent does the sequence $\{ X_n \}_{n \geq 0}$ determine
the original $(\infty,1)$-category $\calC$?
\end{itemize}

The short answers to these questions can be summarized as follows:

\begin{thesis}\label{cope5}

\begin{itemize}
\item[$(A1)$] The topological spaces $\{ X_n \}_{n \geq 0}$ are naturally organized
into a {\em simplicial space}.
\item[$(A2)$] The simplicial space $\{ X_{n} \}_{n \geq 0}$ associated to an
$(\infty,1)$-category is always a {\it complete Segal space} $($see Definitions
\ref{psabat} and \ref{psbbat} below$)$.
\item[$(A3)$] An $(\infty,1)$-category $\calC$ is determined $($up to equivalence$)$ by the
complete Segal space $\{ X_n \}_{n \geq 0}$. Moreover, every complete Segal space
arises from an $(\infty,1)$-category in this way.
\end{itemize}
\end{thesis}

\begin{remark}
We refer to Thesis \ref{cope5} as a thesis, rather than a theorem, because it should really be regarded as a test that any definition of $(\infty,1)$-category must pass in order to be considered reasonable. In other words, it should become a theorem as soon as a suitable definition has been given. 
Alternatively, we can use Thesis \ref{cope5} to prescribe a definition which passes this test automatically: that is, we can {\em define} an $(\infty,1)$-category to be a complete Segal space.
\end{remark}

Our next goal is to explain answers $(A1)$ through
$(A3)$ in more detail. We begin with a brief review of the formalism of {\it simplicial objects}.

\begin{definition}\label{cle}
The category $\cDelta$ of {\it combinatorial simplices} is defined as follows:
\begin{itemize}
\item The objects of $\cDelta$ are the nonnegative integers. For each $n \geq 0$, we let
$[n]$ denote the corresponding object of $\cDelta$.
\item Given a pair of integers $m,n \geq 0$, we define
$\Hom_{ \cDelta}( [m], [n] )$ to be the set of nonstrictly increasing maps
$f: \{ 0 < 1 < \ldots < m \} \rightarrow \{ 0 < 1 < \ldots < n \}$.
\end{itemize}
Let $\bfA$ be an arbitrary category. A {\it simplicial object} of $\bfA$ is a 
functor from $\cDelta^{op}$ into $\bfA$.
\end{definition}

\begin{remark}
We will typically let $A_{\bigdot}$ denote a simplicial object of a category $\bfA$, and
$A_{n}$ the value of the functor $A_{\bigdot}$ when evaluated at the object $[n] \in \cDelta$.
\end{remark}

The most important special case of Definition \ref{cle} is the following:

\begin{definition}
A {\it simplicial set} is a simplicial object in the category of sets.
\end{definition}

\begin{remark}\label{aplet}
The theory of simplicial sets was originally introduced as tool for investigating the homotopy theory of topological spaces using combinatorial means. To every topological space $X$, one can associate a simplicial
set $\Sing_{\bigdot} X$ called the {\it singular complex} of $X$, by means of the formula
$$ \Sing_{n} X = \Hom( \Delta^n, X),$$
where $\Delta^n$ denotes the topological $n$-simplex $\{ x_0, x_1, \ldots, x_n \in \R : x_0 + \ldots + x_n = 1 \}$. The functor $X \mapsto \Sing_{n} X$ has a left adjoint $A_{\bigdot} \mapsto | A_{\bigdot} |$, called
the {\it geometric realization} functor. For every topological space $X$, the counit map
$| \Sing_{\bigdot} X| \rightarrow X$ is a weak homotopy equivalence. Consequently, passing from
a topological space $X$ to its singular complex $\Sing_{\bigdot} X$ entails no loss of ``homotopy
invariant'' information. In fact, it is possible to develop the theory of algebraic topology in an entirely combinatorial way, using simplicial sets as surrogates for topological spaces. 
\end{remark}

There is a close relationship between the theory of simplicial sets and classical category
theory. To every ordinary category $\calC$, we can associate a simplicial set
$\Nerve(\calC)_{\bigdot}$, called the {\it nerve} of $\calC$, by setting $\Nerve(\calC)_{n} = \Fun( [n], \calC)$. More concretely, we let $\Nerve(\calC)_{n}$ denote the set of all composable sequences of
morphisms
$$ C_0 \stackrel{f_1}{\rightarrow} C_1 \stackrel{f_2}{\rightarrow} \ldots \stackrel{f_n}{\rightarrow} C_n$$ in $\calC$ having length $n$.

The nerve $\Nerve(\calC)_{\bigdot}$ of a category $\calC$ determines $\calC$ up to isomorphism.
Indeed, the objects of $\calC$ are the elements of $\Nerve(\calC)_0$, and the morphisms
of $\calC$ are the elements of $\Nerve(\calC)_1$. To recover the composition law for morphisms
in $\calC$ from the nerve $\Nerve(\calC)_{\bigdot}$, we need to introduce a bit of terminology.
For every sequence $0 \leq i_0 \leq i_1 \leq \ldots \leq i_m \leq n$, let $p_{i_0, i_1, \ldots, i_m}$ denote
the corresponding map from $\{ 0 < 1 < \ldots < m \}$ to $\{ 0 < 1 < \ldots < n \}$, and
$p_{i_0, i_1, \ldots, i_m}^{\ast}: \Nerve(\calC)_{n} \rightarrow \Nerve(\calC)_{m}$ the
associated map. We have a pullback diagram of sets
$$ \xymatrix{ \Nerve(\calC)_{2} \ar[r]^{p_{0,1}^{\ast}} \ar[d]^{ p_{1,2}^{\ast}} &
\Nerve(\calC)_{1} \ar[d]^{ p_{1}^{\ast}} \\
\Nerve(\calC)_{1} \ar[r]^{ p_{0}^{\ast} }& \Nerve(\calC)_0. }$$
This diagram determines an isomorphism of sets
$$q: \Nerve(\calC)_2 \rightarrow \Nerve(\calC)_1 \times_{ \Nerve(\calC)_0} \Nerve(\calC)_1.$$
A pair of composable morphisms $f: C \rightarrow D$ and $g: D \rightarrow E$
in $\calC$ can be identified with an element $(f,g)$ in the fiber product
$\Nerve(\calC)_{1} \times_{ \Nerve(\calC)_{0} } \Nerve(\calC)_1$.
The composition $g \circ f$ is then given by $p_{0,2}^{\ast} q^{-1} (f,g) \in \Nerve(\calC)_{1}$:
in particular, it can be described entirely in terms of the structure of $\Nerve(\calC)_{\bigdot}$
as a simplicial set. 

Of course, not every simplicial set arises as the nerve of a category. Given an arbitrary
simplicial set $X_{\bigdot}$, we might attempt to recover a category whose objects
are the elements of $X_0$ and whose morphisms are elements of $X_1$. However,
we encounter difficulties when trying to define a composition law on morphisms.
As above, we have a commutative diagram
$$ \xymatrix{ X_{2} \ar[r]^{p_{0,1}^{\ast}} \ar[d]^{ p_{1,2}^{\ast}} &
X_1 \ar[d]^{ p_{1}^{\ast}} \\
X_{1} \ar[r]^{ p_{0}^{\ast}} & X_0 }$$
which induces a map $q: X_2 \rightarrow X_1 \times_{X_0} X_1$.
However, the diagram is not necessarily a pullback square, so that $q$ is not
necessarily an isomorphism. It turns out that this is essentially the only problem:

\begin{exercise}\label{fif}
Let $X_{\bigdot}$ be a simplicial set. Then $X$ is isomorphic to the nerve of
a category $\calC$ if and only if, for every pair of integers $m,n \geq 0$, the diagram
$$ \xymatrix{ X_{m+n} \ar[r]^{p_{0,1,\ldots,m}^{\ast}} \ar[d]_{p_{m,m+1,\ldots,m+n}^{\ast}}
& X_{m} \ar[d]^{ p_{m}^{\ast} } \\
X_{n} \ar[r]^{ p_0^{\ast}} & X_0 }$$
is a pullback square; in other words, if and only if the canonical map
$X_{m+n} \rightarrow X_{m} \times_{X_0} X_{n}$ is bijective.
\end{exercise}

Let us now return to assertion $(A1)$ of Thesis \ref{cope5}.
Let $\calC$ be an $(\infty,1)$-category. For each $n \geq 0$, we
let $X_{n}$ denote a space whose fundamental $\infty$-groupoid
coincides with the $(\infty,0)$-category obtained from
$\Fun( [n], \calC)$ by discarding the noninvertible $1$-morphism.
Observe that $X_{n}$ depends functorially on the linearly
ordered set $\{ 0 < 1 < \ldots < n \}$: given a nonstrictly increasing function $f: \{ 0 < 1 < \ldots < m \} \rightarrow \{ 0 < 1 < \ldots < n \}$, composition with $f$ determines  functor
from $\Fun( [n], \calC)$ to $\Fun( [m], \calC)$, which should (after discarding noninvertible $1$-morphisms) give rise to a map of topological spaces $X_{n} \rightarrow X_{m}$. To describe
the situation more systematically, let us consider another special case of Definition \ref{cle}:

\begin{definition}
A {\it simplicial space} is a simplicial object of the category of topological spaces.
\end{definition}

We can now summarize the above discussion as follows: given an
$(\infty,1)$-category $\calC$, the collection of topological spaces
$\{ X_n \}_{n \geq 0}$ should be organized into a simplicial space $X_{\bigdot}$.

\begin{warning}\label{culper}
The construction $\calC \mapsto X_{\bigdot}$ (which we have described
informally when $\calC$ is an $(\infty,1)$-category) is quite similar to the construction
$\calC \mapsto \Nerve(\calC)_{\bigdot}$ (which we have defined precisely when
$\calC$ is an ordinary category). In both cases, the $n$th term of the relevant simplicial object
parametrizes functors from $[n]$ into $\calC$. However, these constructions do not agree
when $\calC$ is an ordinary category. For example, the topological space $X_0$
is a classifying space for the underlying groupoid of $\calC$; in particular, the connected components of $X_0$ are in bijection with isomorphism classes of objects in $\calC$. On the other hand,
$\Nerve(\calC)_0$ is defined to be the (discrete) set of objects of $\calC$; in particular, it 
takes no account of whether or not two objects in $\calC$ are isomorphic.
\end{warning}

In spite of Warning \ref{culper}, the simplicial space $X_{\bigdot}$ extracted
from an $(\infty,1)$-category $\calC$ behaves much like the nerve of an ordinary category.
In particular, it is natural to expect that it should satisfy some analogue of
condition described in Exercise \ref{fif}. To formulate this condition, we need
to recall a bit of homotopy theory.

\begin{definition}\label{silman}
Let $f: X \rightarrow Z$ and $g: Y \rightarrow Z$ be continuous maps of topological spaces.
The {\it homotopy fiber product} of $X \times^{R}_{Z} Y$ is the topological space
$$ X \times_{Z} Z^{[0,1]} \times_{Z} Y$$
whose points consist of triples $(x,y,p)$, where $x \in X$, $y \in Y$, and
$p: [0,1] \rightarrow Z$ is a continuous path from $p(0) = f(x)$ to $p(1) = g(y)$.
\end{definition}

\begin{remark}\label{tabler}
The construction of Definition \ref{silman} should be regarded as a homotopy-theoretic
(or right derived) version of the ordinary fiber product. It has the feature of being a
{\it homotopy invariant}: given a commutative diagram of topological spaces
$$ \xymatrix{ X \ar[r] \ar[d] & Z \ar[d] & Y \ar[l] \ar[d] \\
X' \ar[r] & Z' & Y' \ar[l] }$$
in which the vertical maps are weak homotopy equivalences, the induced map
$$ X \times^{R}_{Z} Y \rightarrow X' \times^{R}_{Z'} Y'$$ is
again a weak homotopy equivalence. Moreover, the weak homotopy type
of a homotopy fiber product $X \times^{R}_{Z} Y$ does not change if we
replace the continuous maps $f: X \rightarrow Z$ and $g: Y \rightarrow Z$ by
homotopic maps. Both of these assertions fail dramatically if we replace
the homotopy fiber product $X \times^{R}_{Z} Y$ with the ordinary fiber product
$X \times_{Z} Y$.
\end{remark}

\begin{remark}
For any pair of continuous maps $f: X \rightarrow Z$, $g: Y \rightarrow Z$, there is
a canonical map from the ordinary fiber product $X \times_{Z} Y$ to the homotopy
fiber product $X \times_{Z}^{R} Y$; it carries a point $(x,y) \in X \times_{Z} Y$ to
the point $(x,y,p) \in X \times_{Z}^{R} Y$, where $p: [0,1] \rightarrow Z$ is the constant
path from $f(x) = g(y) \in Z$ to itself.
\end{remark}

\begin{definition}\label{cmale}
Suppose given a commutative diagram of topological spaces
$$ \xymatrix{ W \ar[r] \ar[d] & X \ar[d] \\
Y \ar[r] & Z. }$$
We say that this diagram is a {\it homotopy pullback square} (or a {\it homotopy Cartesian diagram})
if the composite map
$$ W \rightarrow X \times_{Z} Y \rightarrow X \times_{Z}^{R} Y$$
is a weak homotopy equivalence.
\end{definition}

\begin{remark}
Suppose we are given a commutative diagram of topological spaces
$$ \xymatrix{ W \ar[r] \ar[d] & X \ar[d]^{f} \\
Y \ar[r]^{g} & Z. }$$
In general, the condition that this diagram be a pullback square and the condition that
it be a homotopy pullback square are {\em independent}: neither implies the other.
Suppose, however, that we can somehow guarantee that the inclusion
$X \times_{Z} Y \rightarrow X \times_{Z}^{R} Y$ is a weak homotopy equivalence
(this is always true if $f$ or $g$ is a Serre fibration, for example). In this
case, if the above diagram is a pullback square, then it is a homotopy pullback square.
\end{remark}

We are now ready to formulate the homotopy-theoretic counterpart
to the condition of Exercise \ref{fif}:

\begin{definition}\label{psabat}
Let $X_{\bigdot}$ be a simplicial space. We say that $X_{\bigdot}$ is a
{\it Segal space} if the following condition is satisfied:
\begin{itemize}
\item[$(\ast)$] For every pair of integers $m, n \geq 0$, the diagram
$$ \xymatrix{ X_{m+n} \ar[r] \ar[d]
& X_{m} \ar[d] \\
X_{n} \ar[r] & X_0 }$$
is a homotopy pullback square.
\end{itemize}
\end{definition}

\begin{warning}
Definition \ref{psabat} is not completely standard. Some authors impose the additional
requirement that the simplicial space $X_{\bigdot}$ be {\it Reedy fibrant}: this is a harmless technical condition
which guarantees, among other things, that each of the maps in the diagram
$$ \xymatrix{ X_{n+m} \ar[r] \ar[d]
& X_{m} \ar[d] \\
X_{n} \ar[r] & X_0 }$$
is a Serre fibration of topological spaces. If we assume this condition, then
$X_{\bigdot}$ is a Segal space if and only if each of the maps
$X_{n+m} \rightarrow X_{n} \times_{X_0} X_{m}$ is a weak homotopy equivalence.
\end{warning}

Returning now to our discussion of Thesis \ref{cope5}, we observe that if
$\calC$ is an $(\infty,1)$-category, then it is natural to suppose that the associated
simplicial space $X_{\bigdot}$ is a Segal space. This simply encodes the idea that
giving a chain of composable morphisms
$$ C_0 \stackrel{f_1}{\rightarrow} C_1 \stackrel{f_2}{\rightarrow} \ldots \stackrel{f_n}{\rightarrow} C_n$$
is equivalent to giving a pair of chains
$$ C_0 \stackrel{f_1}{\rightarrow} \ldots \stackrel{f_m}{\rightarrow} C_m
\quad \quad C_m \stackrel{f_{m+1}}{\rightarrow} \ldots \stackrel{f_n}{\rightarrow} C_n$$
such that the final term of the first chain (the object $C_m \in \calC$) agrees with the initial
term of the second chain.

In fact, even more is true: according to Thesis \ref{cope5}, the $(\infty,1)$-category
$\calC$ is {\em determined} up to equivalence by the associated Segal space
$X_{\bigdot}$. Indeed, we can attempt recover $\calC$ from $X_{\bigdot}$ as follows:

\begin{construction}\label{jumpa}
Let $X_{\bigdot}$ be a Segal space. We can extract from $X_{\bigdot}$ an
$(\infty,1)$-category $\calC$ which may be described informally as follows:
\begin{itemize}
\item The objects of $\calC$ can be identified with points of the topological
space $X_0$.
\item Given a pair of points $x,y \in X_0$, the mapping space
$\OHom_{\calC}(x,y)$ is defined to be the iterated homotopy fiber product $\{x\} \times^{R}_{X_0} X_1 \times^{R}_{X_0} \{y\}$. 
\item Given a triple of points $x,y,z \in X_0$, the composition law
$$\OHom_{\calC}(x,y) \times \OHom_{\calC}(y,z) \rightarrow \OHom_{\calC}(x,z)$$
can be recovered as the composition
\begin{eqnarray*} ( \{x\} \times^{R}_{X_0} X_1 \times^{R}_{X_0} \{y\} )
\times ( \{ y\} \times^{R}_{X_0} X_1 \times^{R}_{X_0} \{z\}) & \rightarrow
& \{x\} \times^{R}_{X_0} X_1 \times^{R}_{X_0} X_1 \times^{R}_{X_0} \{z\} \\
& \stackrel{\phi}{\simeq} & \{x\} \times^{R}_{X_0} X_2 \times_{X_0}^{R} \{z\} \\
& \rightarrow & \{x\} \times^{R}_{X_0} X_2 \times^{R}_{X_0} \{z\}.
\end{eqnarray*}
Here the map $\phi$ really goes in the opposite direction, but our assumption that
$X_{\bigdot}$ is a Segal space implies that $\phi$ is a weak homotopy equivalence,
and is therefore invertible in the homotopy category.
\item The remaining data of the simplicial space $X_{\bigdot}$ $($and other Segal conditions$)$
guarantees that the above composition law is associative up to $($coherent$)$ homotopy.
\end{itemize}
\end{construction}

We have now sketched constructions in both directions which relate the (as yet undefined)
notion of $(\infty,1)$-category with the (well-defined) notion of a Segal space. However,
these constructions are not quite inverse to one another.

\begin{example}\label{sunny}
Let $\calC$ be an ordinary category. We can regard the nerve
$\Nerve(\calC)_{\bigdot}$ as a simplicial space, in which each set
$\Nerve(\calC)_{n}$ is endowed with the discrete topology. This
simplicial space is a Segal space. Moreover, if we apply the above
construction to $\Nerve(\calC)_{n}$, we recover the original category
$\calC$. However, the Segal space $X_{\bigdot}$ associated to $\calC$
(viewed as an $(\infty,1)$-category) does {\em not} coincide with
$\Nerve(\calC)_{\bigdot}$, as we have already seen in Warning \ref{culper}:
the space $X_0$ is usually not discrete (even up to homotopy), since its
fundamental groupoid is equivalent to the underlying groupoid of $\calC$.
\end{example}

\begin{remark}
Let $X_{\bigdot}$ be a Segal space. We can modify Construction \ref{jumpa} to
extract a more concrete invariant of $X_{\bigdot}$: an ordinary category which
we call the {\it homotopy category} of $X_{\bigdot}$ and denote by $\h{X_{\bigdot}}$.
This category can be described informally as follows:
\begin{itemize}
\item[$(1)$] The objects of $\h{X_{\bigdot}}$ are the points of the space $X_0$.
\item[$(2)$] Given a pair of points $x,y \in X_0$, we let $\Hom_{ \h{X_{\bigdot}}}( x,y)$
be the set of path components $$\pi_0( \{x\} \times_{X_0}^{R} X_1 \times_{X_0}^{R} \{y\}).$$
\end{itemize}
\end{remark}

The correspondence between
Segal spaces and $(\infty,1)$-categories is generally {\em many-to-one}: 
a given $(\infty,1)$-category $\calC$ can be obtained from many
different Segal spaces via Construction \ref{jumpa}. 
Example \ref{sunny} illustrates the origin of this difficulty. 
Suppose that we begin with a Segal space $Y_{\bigdot}$,
and use it to construct an $(\infty,1)$-category $\calC$. We can then
extract from $\calC$ a new Segal space $X_{\bigdot}$.
We can then think of (the fundamental $\infty$-groupoid of) $X_0$
as the $(\infty,0)$-category obtained from $\calC$ by discarding the
noninvertible $1$-morphisms. This $(\infty,0)$-category receives a map
from (the fundamental $\infty$-groupoid of) $Y_0$, but this map is not
necessarily an equivalence: for example, there could be
invertible $1$-morphisms in $\calC$ which do not arise from paths
in the space $Y_0$. We can rule out this phenomenon by
introducing an additional assumption on the Segal space $Y_{\bigdot}$.

\begin{definition}
Let $X_{\bigdot}$ be a Segal space, and let $f \in X_1$ be a point.
Let $x = p_0^{\ast}(f)$ and $y = p_{1}^{\ast}(f)$, so that the points
$x,y \in X_0$ can be identified with objects of the homotopy
category $\h{ X_{\bigdot} }$. The composite map
$$ \{ f \} \rightarrow \{x\} \times_{X_0} X_1 \times_{X_0} \{y\}
\rightarrow \{x\} \times_{X_0}^{R} X_1 \times_{X_0}^{R} \{y\} $$
determines a morphism
$$[f] \in \Hom_{ \h{X_{\bigdot}}}(x,y) = \pi_0( \{x\} \times_{X_0}^{R} X_1 \times_{X_0}^{R} \{y\} ).$$
We will say that $f$ is {\it invertible} if $[f]$ is an isomorphism in the homotopy category
$\h{ X_{\bigdot} }$.
\end{definition}

\begin{example}\label{hughlick}
Let $X_{\bigdot}$ be a Segal space, and let $\delta: X_0 \rightarrow X_1$ be the
``degeneracy map'' induced by the unique nondecreasing functor $\{0,1\} \rightarrow \{0\}$.
For every point $x$ in $X_0$, the morphism $[ \delta(x) ]$ in the homotopy
category $\h{ X_{\bigdot} }$ coincides with the identity map $\id_{x}: x \rightarrow x$.
In particular, $\delta(x)$ is invertible for each $x \in X_0$.
\end{example}

\begin{definition}\label{psbbat}
Let $X_{\bigdot}$ be a Segal space, and let $Z \subseteq X_1$ denote the subset
consisting of the invertible elements (this is a union of path components in $X_1$;
we will consider $Z$ as endowed with the subspace topology). We will say
that $X_{\bigdot}$ is {\it complete} if the map $\delta: X_0 \rightarrow Z$
of Example \ref{hughlick} is a weak homotopy equivalence.
\end{definition}

Roughly speaking, a Segal space $Y_{\bigdot}$ is complete if every isomorphism in
the associated $(\infty,1)$-category $\calC$ arises from an essentially unique path in the space $Y_0$.
This allows us to identify the fundamental $\infty$-groupoid of $Y_0$ with the
$(\infty,0)$-category obtained by discarding the noninvertible $1$-morphisms in $\calC$.
In fact, it allows us to identify the fundamental $\infty$-groupoid of each $Y_n$ with
the underlying $(\infty,0)$-category of $\Fun( [n], \calC)$. In other words, Construction
\ref{jumpa} should establish an equivalence between the theory of complete Segal spaces
and the theory of $(\infty,1)$-categories. We can take this as a heuristic justification for the following
definition:

\begin{definition}\label{swish}
An {\it $(\infty,1)$-category} is a complete Segal space.
\end{definition}

\begin{remark}\label{compus}
Let $Y_{\bigdot}$ be a Segal space which is not complete. 
Then there exists a map $Y_{\bigdot} \rightarrow X_{\bigdot}$ in 
the homotopy category of simplicial spaces which is {\em universal} among maps from $Y_{\bigdot}$ to complete Segal spaces. In this case, we will say that $X_{\bigdot}$ is a {\it completion} of $Y_{\bigdot}$.
Informally, we can think of $X_{\bigdot}$ as the complete Segal space corresponding
to the $(\infty,1)$-category obtained from $Y_{\bigdot}$ via Construction \ref{jumpa}.
This construction will play an important role in what follows, because the higher
categories which arise in the bordism theory of manifolds are naturally obtained
from Segal categories which are {\em not} complete (see Warning \ref{sablin}).
\end{remark}

\begin{remark}
There are many alternatives to Definition \ref{swish} which give rise to essentially the
same theory. Were we to adopt such an alternative,
Thesis \ref{cope5} could be formulated as a theorem, which would be proved by
giving a precise implementation of Construction \ref{jumpa}. For a more detailed discussion of
the various models of the theory of $(\infty,1)$-categories, we refer the reader to \cite{bergner2}.
\end{remark}

\begin{remark}
Using a more rigorous version of the above arguments, To\"{e}n has proven a version
of Thesis \ref{cope5}. More precisely, he has proven that any homotopy theory satisfying
a short list of reasonable axioms is equivalent to the theory of complete Segal spaces (\cite{toenchar}).
\end{remark}

In the next section, we will need a variant of the theory of complete Segal spaces.

\begin{definition}
Let $\cDelta_0$ denote the subcategory of $\cDelta$ with the same objects,
where the morphisms from $[m]$ to $[n]$ are given by {\em strictly} increasing
maps of linearly ordered sets $\{ 0 < 1 < \ldots < m \} \rightarrow \{ 0 < 1 < \ldots < n \}$.
A {\it semisimplicial object} of a category $\bfA$ is a functor from $\cDelta_0^{op}$ into
$\bfA$. 
\end{definition}

\begin{definition}
Let $X_{\bigdot}$ be a semisimplicial space. We will say that $X_{\bigdot}$ is a 
{\it semiSegal space} if the following condition is satisfied:
\begin{itemize}
\item[$(\ast)$] For every pair of integers $m, n \geq 0$, the diagram
$$ \xymatrix{ X_{m+n} \ar[r] \ar[d]
& X_{m} \ar[d] \\
X_{n} \ar[r] & X_0 }$$
is a homotopy pullback square.
\end{itemize}
\end{definition}

\begin{example}
Every simplicial object of a category $\bfA$ determines a semisimplicial object of $\bfA$
by restriction. In particular, every simplicial space $X_{\bigdot}$ determines a semisimplicial
space $X'_{\bigdot}$; we observe that $X_{\bigdot}$ is a Segal space if and only if 
$X'_{\bigdot}$ is a semiSegal space.
\end{example}

\begin{example}
A {\it nonunital category} $\calC$ consists of the following data:
\begin{itemize}
\item[$(1)$] A collection of objects $X, Y, Z, \ldots \in \calC$.
\item[$(2)$] For every pair of objects $X,Y \in \calC$, a set $\Hom_{\calC}(X,Y)$.
\item[$(3)$] For every triple of objects $X,Y,Z \in \calC$, a composition map
$$ \Hom_{\calC}(X,Y) \times \Hom_{\calC}(Y,Z) \rightarrow \Hom_{\calC}(X,Z).$$
These composition maps are required to be associative in the obvious sense.
\end{itemize}
In other words, a nonunital category is like a category, except that we do not require
the existence of identity morphisms. Every category determines an underlying nonunital category, simply by forgetting the identity morphisms. 

To every nonunital category $\calC$, we can associate a semisimplicial set $\Nerve(\calC)$, the 
{\it nerve} of $\calC$: we let $\Nerve(\calC)_{n}$ denote the collection of all $n$-tuples
$$ x_0 \stackrel{f_1}{\rightarrow} x_1 \stackrel{f_2}{\rightarrow} \cdots \stackrel{f_n}{\rightarrow} x_n$$
of morphisms in $\calC$. A nonunital category is determined up to isomorphism by its nerve, and it is not difficult to characterize those semisimplicial sets which arise as nerves of nonunital categories
as in Exercise \ref{fif}.
\end{example}

If $X_{\bigdot}$ is a semiSegal space, then one can attempt to apply Construction \ref{jumpa}
to build an $(\infty,1)$-category. In general, this does not succeed: one can extract a
collection of objects, a topological space of morphisms between every pair of objects, and a coherently associative composition law, but there is no natural candidate for identity morphisms.
In other words, we can think of a semiSegal space $X_{\bigdot}$ as encoding a 
{\em nonunital $(\infty,1)$-category}. Just as in ordinary category theory, the existence of units
is merely a condition to be assumed: identity morphisms are unique (up to canonical isomorphism) when they exist. Formally, this translates into the following assertion:

\begin{claim}\label{amax}
Let $Y_{\bigdot}$ be a semiSegal space. Suppose that there exists a simplicial space
$X_{\bigdot}$ and a weak homotopy equivalence $X_{\bigdot} \rightarrow Y_{\bigdot}$ of
semisimplicial spaces. Then $X_{\bigdot}$ is a Segal space, and is uniquely determined
up to weak homotopy equivalence.
\end{claim}

In fact, one can be more precise: a semiSegal space $Y_{\bigdot}$ is equivalent to the
restriction of a Segal space if and only if it has ``identity morphisms up to homotopy''
in an appropriate sense. We will not pursue the matter in any further detail.

We conclude this section by sketching how the above definitions can be generalized to
the setting of $(\infty,n)$-categories for $n > 0$. Roughly speaking, one can think of an $(\infty,n)$-category $\calC$ has having an underlying $\infty$-groupoid $X_0$ (obtained by discarding all noninvertible $k$-morphisms in $\calC$ for $1 \leq k \leq n$), which we view as an $(\infty,n-1)$-category. For every pair of objects
$x,y \in X_0$, there is an $(\infty,n-1)$-category $\OHom_{\calC}(x,y)$ of $1$-morphisms
$f: x \rightarrow y$. We can organize the collection of all triples $(x,y,f)$ into an
$(\infty,n-1)$-category $X_1$ which is equipped with a pair of forgetful functors
$X_1 \rightarrow X_0$. Proceeding in this manner, we can encode the entirety of the structure
of $\calC$ into a {\it simplicial} $(\infty,n-1)$-category $X_{\bigdot}$. To describe the mathematical structures which arise via this procedure, we need to introduce a definition.

\begin{definition}
Let $n \geq 0$ an integer, and let $\bfA$ be a category. An {\it $n$-fold} simplicial object
of $\bfA$ is a functor 
$$ \cDelta^{op} \times \ldots \times \cDelta^{op} \rightarrow \bfA,$$
where the product on the left hand side has $n$ factors.
\end{definition}

\begin{example}
If $n=0$, then an $n$-fold simplicial object of $\bfA$ is just an object of $\bfA$.
If $n=1$, then an $n$-fold simplicial object of $\bfA$ is a simplicial object of
$\bfA$ in the sense of Definition \ref{cle}. In general, an $n$-fold simplicial
object of $\bfA$ consists of a collection of objects
$X_{ k_1, \ldots, k_n} \in \bfA$ indexed by $n$-tuples of nonnegative
integers $k_1, \ldots, k_n \geq 0$, which are related by a variety of ``face'' and ``degeneracy'' maps.
\end{example}

\begin{notation}
If $\bfA$ is a category, we will let $\bfA^{ (n)}$ denote the category of
$n$-fold simplicial objects of $\bfA$. For $m, n \geq 0$, we have an evident equivalence
of categories
$$ (\bfA^{(m)})^{(n)} \simeq \bfA^{(m+n)}.$$
In particular, we can identify $n$-fold simplicial objects of $\bfA$ with
simplicial objects $X_{\bigdot}$ of $\bfA^{(n-1)}$.
\end{notation}

We now specialize to the case where the target category $\bfA$ is the category
of topological spaces.

\begin{definition}
An {\it $n$-fold simplicial space} is an $n$-fold simplicial object in the category
of topological spaces and continuous maps.

We will say that a map $X \rightarrow Y$ of $n$-fold simplicial spaces is a 
{\it weak homotopy equivalence} if the induced map
$X_{k_1, \ldots, k_n} \rightarrow Y_{k_1, \ldots, k_n}$ is a weak homotopy equivalence
of topological spaces, for every sequence of nonnegative integers $k_1, \ldots, k_n \geq 0$.
A diagram
$$ \xymatrix{ X \ar[r] \ar[d] & Y \ar[d] \\
X' \ar[r] & Y' }$$
of $n$-fold simplicial spaces is a {\it homotopy pullback square} if,
for every sequence of nonnegative integers $k_1, \ldots, k_n \geq 0$, the 
induced square
$$ \xymatrix{ X_{k_1, \ldots, k_n} \ar[r] \ar[d] & Y_{k_1, \ldots, k_n} \ar[d] \\
X'_{k_1, \ldots, k_n} \ar[r] & Y'_{k_1, \ldots, k_n} }$$
is a homotopy pullback square of topological spaces (see Definition \ref{cmale}).

We will say that an $n$-fold simplicial space $X$ is {\it essentially constant} 
if there exists a weak homotopy equivalence of $n$-fold simplicial spaces $X' \rightarrow X$, 
where $X'$ is a constant functor.
\end{definition}

\begin{remark}
An $n$-fold simplicial space $X$ is constant if and only if, for every sequence
$k_1, \ldots, k_n \geq 0$, the canonical map $X_{0,\ldots, 0} \rightarrow
X_{k_1, \ldots, k_n}$ is a weak homotopy equivalence; in this case,
$X$ is weakly equivalent to the constant $n$-fold simplicial space associated
to $X_{0,\ldots, 0}$. 
\end{remark}

\begin{definition}\label{multiseg}
Let $n > 0$, and let $X$ be an $n$-fold simplicial space. We will regard
$X$ as a simplicial object $X_{\bigdot}$ in the category of $(n-1)$-fold simplicial spaces.
We will say that $X$ is an {\it $n$-fold Segal space} if the following conditions
are satisfied:
\begin{itemize}
\item[$(A1)$] For every $0 \leq k \leq m$, the diagram
$$ \xymatrix{ X_{m} \ar[r] \ar[d] & X_{k} \ar[d] \\
X_{m-k} \ar[r] & X_0}$$
of Definition \ref{psabat} is a homotopy pullback square (of $(n-1)$-fold simplicial spaces).
\item[$(A2)$] The $(n-1)$-fold simplicial space $X_0$ is essentially constant.
\item[$(A3)$] Each of the $(n-1)$-uple simplicial spaces $X_{k}$
is an $(n-1)$-dimensional Segal space.
\end{itemize}

We will say that an $n$-fold Segal space $X_{\bigdot}$ is {\it complete} if it
satisfies the following additional conditions:
\begin{itemize}
\item[$(A4)$] Each of the $(n-1)$-dimensional Segal spaces $X_{n}$ is complete
(we regard this condition as vacuous when $n=1$).
\item[$(A5)$] Let $Y_{\bigdot}$ be the simplicial space described by the formula
$Y_{k} = X_{k,0,\ldots,0}$; note that condition $(A1)$ guarantees that $Y_{\bigdot}$ is a
Segal space. Then $Y_{\bigdot}$ is complete.
\end{itemize}
\end{definition}

We now have the corresponding analogue of Definition \ref{swish}:

\begin{definition}\label{swish2}
An $(\infty,n)$-category is an $n$-fold complete Segal space.
\end{definition}

\begin{remark}
There are other reasonable approaches to the theory of $(\infty,n)$-categories which are equivalent to Definition \ref{swish2}. We refer the reader to \cite{bicat} for a discussion in the case $n=2$.
\end{remark}

\begin{remark}\label{kisster}
If $X$ is an $n$-fold Segal space, then there is a universal example of a map
$X \rightarrow X'$ in the homotopy category of $n$-fold simplicial spaces, such that
$X'$ is an $n$-fold complete Segal space. In this case, we will refer to $X'$ as the {\it completion} of $X$.
We can regard $X'$ as an $(\infty,n)$-category (Definition \ref{swish2}) whose structure is determined by $X$.
\end{remark}

\begin{remark}
There is an evident action of the symmetric group $\Sigma_n$ on the category
of $n$-fold simplicial spaces. Definition \ref{multiseg} is {\em not} invariant under
this action. For example, when $n=2$, the axioms demand that the simplicial space
$X_{0, \bigdot}$ be essentially constant, but there is no corresponding demand
on the simplicial space $X_{\bigdot, 0}$.
\end{remark}

\subsection{Bordism Categories as Segal Spaces}\label{swugg}

Let $n$ be a positive integer, which we regard as fixed throughout this section.
In \S \ref{sugar}, we argued that it is natural to replace the ordinary bordism
category $\Cob{}{n}{}$ with an $(\infty,1)$-category $\tCob(n)$, which encodes information about the homotopy types of diffeomorphism groups of $n$-manifolds. In \S \ref{bigseg}, we introduced the notion
of a {\it Segal space}, and argued that complete Segal spaces can be regarded as representatives
for $(\infty,1)$-categories. Our goal in this section is to unite these two lines of thought,
giving an explicit construction of $\tCob(n)$ in the language of Segal spaces. 
To simplify the exposition, we consider the {\em unoriented} version of the bordism
category, which we will denote by $\tunCob(n)$. At the end of this section, we will explain how the
construction of $\tunCob(n)$ can be generalized to the setting of $n$-fold Segal spaces to give
a precise definition of the $(\infty,n)$-category $\Bord_{n}$ described informally in Definition \ref{cabber2}.

Let us first outline the rough idea of the construction. We would like to produce a Segal
space $\SuntCob(n)_{\bigdot}$ which encodes the structure of the $(\infty,1)$-category $\tunCob(n)$.
Roughly speaking, we would like $\SuntCob(n)_{k}$ to be a classifying space for composable chains of
bordisms
$$ M_0 \stackrel{B_1}{\rightarrow} M_1 \stackrel{B_2}{\rightarrow} \cdots \stackrel{B_k}{\rightarrow} 
M_k$$
of length $k$: here each $M_i$ is a closed manifold of dimension $(n-1)$ and each
$B_i$ is a bordism from $M_{i-1}$ to $M_i$. There are two considerations to bear in mind:

\begin{itemize}
\item[$(a)$] As noted in Remark \ref{jirik}, composition of bordisms is not quite well-defined
without making some auxiliary choices. Fortunately, the formalism of Segal spaces comes to our rescue: we do not need the space $\SuntCob(n)_{k}$ to coincide with the iterated fiber product 
$$\SuntCob(n)_1 \times_{\SuntCob(n)_0}
\cdots \times_{\SuntCob(n)_0} \SuntCob(n)_1;$$ we only need $X_k$ to be weakly equivalent to the corresponding homotopy fiber product. We can therefore allow points of $X_k$ to encode more information than just
the chain of composable bordisms $\{ B_i \}_{1 \leq i \leq k}$ (for example, a smooth structure
on $B = B_1 \coprod_{ M_1} \cdots \coprod_{M_{k-1}} B_{k-1}$) so long as the inclusion of this information does not change the relevant homotopy type.

\item[$(b)$] The collection of all composable chains of bordisms as above is naturally
organized into a (topological) groupoid, where the morphisms are given by diffeomorphisms.
We would like $\SuntCob(n)_k$ to be a classifying space for this groupoid. To construct such a classifying space explicitly, we will choose some auxiliary data: namely, an embedding of the manifold $B$
into $V \times \R$, where $V$ is a real vector space of large dimension. As the dimension of
$V$ grows, the relevant space of embeddings becomes highly connected (by general position arguments), and in the limit we can identify the relevant classifying space with the collection
of embedded submanifolds.
\end{itemize}

\begin{notation}
Let $V$ be a real vector space of finite dimension $d$. We let
$\Sub_0(V)$ denote the collection of all smooth 
closed submanifolds $M \subseteq V$ of dimension $n-1$, and
$\Sub(V)$ the collection of all smooth compact $n$-manifolds
properly embedded in $V \times [0,1]$ (we say that an embedding
$M \hookrightarrow V \times [0,1]$ is {\it proper} if $\bd M = M \cap (V \times \{0,1\})$).
\end{notation}

\begin{remark}\label{sing}
For every finite dimensional real vector space $V$, the spaces $\Sub_0(V)$ and
$\Sub(V)$ admit topologies. We will describe this topology in the case
for $\Sub(V)$; the case of $\Sub_0(V)$ is similar but slightly easier.
Given an abstract $n$-manifold $M$, we can define a topological space $\Emb(M, V \times [0,1])$ of
smooth proper embeddings of $M$ into $V \times [0,1]$. The space
$\Emb(M, V \times [0,1])$ carries an action of the diffeomorphism group $\Diff(M)$, and
we have a canonical bijection
$$ \coprod_{M} \Emb(M, V \times [0,1] ) / \Diff(M) \rightarrow \Sub(V)$$
where the coproduct is taken over all diffeomorphism classes of $n$-manifolds.
We endow $\Sub(V)$ with the quotient topology:
a subset $U \subseteq \Sub(V)$ is open if and only if its
inverse image in $\Emb(M, V \times [0,1])$ is open, for every $n$-manifold $M$.
With respect to this topology, each of the quotient maps $\Emb(M, V \times [0,1]) \rightarrow \Sub(V)$ exhibits $\Emb(M, V \times [a,b])$ as a principal $\Diff(M)$-bundle over a suitable summand of $\Sub(V)$. 
\end{remark}

\begin{definition}
For each $k \geq 0$, let
$\untCob(n)^{V}_{k}$ denote the set of all pairs
$(t_0 < t_1 < \cdots < t_k; M)$ where $\{ t_i \}_{0 \leq i \leq k}$ is a
strictly increasing sequence of real numbers
and $M \subseteq V \times [t_0, t_k]$ is either a smooth submanifold
of dimension $(n-1)$ (if $k=0$) or a properly embedded submanifold of
dimension $n$ (if $k > 0$) which intersects each of the submanifolds
$V \times \{t_i\} \subseteq V \times [t_0, t_k]$ transversely. 
If $k=0$ we can identify $\untCob(n)_{k}^{V}$ with $\R \times \Sub_{0}(V)$
and if $k > 0$ we can identify $\untCob(n)_{k}^{V}$ with an open subset of
$\Sub(V) \times \{ t_0, \ldots t_k \in \R: t_0 < \cdots < t_k \}$ (using a linear change of coordinates to identify $V \times [t_0, t_k]$ with $V \times [0,1]$. In either case, the relevant identification endows
$\untCob(n)_{k}^{V}$ with the structure of a topological space.

Given a strictly increasing map $f: \{ 0 < 1 < \cdots < k \} \rightarrow
\{ 0 < 1 < \cdots < k' \}$, we obtain a continuous map of topological spaces
$f^{\ast}: \untCob(n)_{k'}^{V} \rightarrow \untCob(n)_{k}^{V}$, given
by the formula
$$ f^{\ast}( t_0 < \ldots < t_{k'}; M) =
( t_{f(0)} < \ldots < t_{f(k)}; M \cap (V \times [t_{f(0)}, t_{f(k)}]).$$
In this way, the collection of topological spaces $\{ \untCob(n)_{k}^{V} \}_{k \geq 0}$
can be organized into a semisimplicial space, which we will denote by
$\untCob(n)_{\bigdot}^{V}$. 

Let $\R^{\infty}$ denote an infinite dimensional real vector space.
We define the semisimplicial space
$\untCob(n)_{\bigdot}$ to be the direct limit
$\varinjlim \untCob(n)_{\bigdot}^{V}$, where $V$ ranges over the
collection of all finite dimensional subspaces of $\R^{\infty}$.
\end{definition}

\begin{remark}
Up to homotopy equivalence, the semisimplicial space
$\untCob(n)_{\bigdot}$ does not depend on the choice
of infinite dimensional real vector space $\R^{\infty}$. In fact,
if we require that $\R^{\infty}$ have countable dimension, then
$\untCob(n)_{\bigdot}$ is well-defined up to homeomorphism.
\end{remark}

\begin{claim}\label{ilmer}
The semisimplicial space $\untCob(n)_{\bigdot}$ is a semiSegal space.
\end{claim} 

Claim \ref{ilmer} expresses the idea that we can glue pairs of bordisms together, and the
result is well-defined up to a contractible space of choices. As we explained in
\S \ref{bigseg}, this allows us to view $\untCob(n)_{\bigdot}$ as encoding the structure of
a not necessarily unital $(\infty,1)$-category. According to Claim \ref{amax}, if
$\untCob(n)_{\bigdot}$ is weakly equivalent to the restriction of a Segal space, then
that Segal space is uniquely determined up to weak homotopy equivalence. The existence of such a Segal space can be deduced
formally from the fact that $\untCob(n)_{\bigdot}$ admits units ``up to homotopy'' (given an object
of $\untCob(n)_{\bigdot}^{V}$ represented by a submanifold $M \subseteq V$, the identity map
from this object to itself can be represented by the product $M \times [0,1] \subseteq V \times [0,1]$).
However, it is possible to give a direct construction of such a Segal space.

\begin{protodefinition}\label{slamm}
Let $V$ be a finite dimensional real vector space. For every nonnegative integer $k$, let
$\SuntCob(n)^{V}_{k}$ denote the collection of all pairs $(M, \{ t_0 \leq t_1\leq \cdots \leq t_k \})$, where the $t_i$ are real numbers, $M \subseteq V \times \R$ is a (possibly noncompact) $n$-dimensional submanifold, and the projection $M \rightarrow \R$ is a proper map whose critical values are
disjoint from $\{ t_0 \leq \cdots \leq t_k \}$. The set $\SuntCob(n)^{V}_{k}$ can be endowed with a topology. This topology generalizes that of Remark \ref{sing}, but is more complicated because we allow noncompact manifolds; we refer the reader to the appendix of \cite{soren2} for a definition. We can use
these topologies to endow $\SuntCob(n)^{V}_{\bigdot}$ with the structure of a simplicial space.
\end{protodefinition}

For each $k \geq 0$, let $\OSuntCob(n)^{V}_{k}$ denote the open subset of
$\SuntCob(n)^{V}_{k}$ consisting of pairs $(M, \{ t_0 \leq \ldots \leq t_k \})$ such that
$t_0 < \cdots < t_k$. We can regard $\OSuntCob(n)^{V}_{\bigdot}$ as a semisimplicial space
equipped with an evident inclusion (of semisimplicial spaces) $\OSuntCob(n)^{V}_{\bigdot}
\subseteq \SuntCob(n)^{V}_{\bigdot}$. There is also a natural map
$f: \OSuntCob(n)^{V}_{\bigdot} \rightarrow \untCob(n)^{V}_{\bigdot}$, which carries a pair
$(M, \{ t_0 < \cdots < t_k \})$ to the pair $(M \cap (V \times [t_0, t_k]), \{t_0 < \cdots < t_k \})$.
The topology on each $\SuntCob(n)^{V}_{k}$ is defined so that $f$ induces a weak homotopy
equivalence: roughly speaking, there is a canonical path joining any
two points $(M, \{ t_0 < \cdots < t_k \} ), (N, \{t_0 < \cdots < t_k \})$ such that
$M \cap (V \times [t_0, t_k]) = N \cap (V \times [t_0, t_k ])$, which is given by ``stretching to
infinity'' the parts of $M$ and $N$ which do not lie in $V \times [t_0, t_k]$. 

Let $\SuntCob(n)_{\bigdot}$ denote the simplicial space
$\varinjlim \SuntCob(n)^{V}_{\bigdot},$ where the direct limit is taken over all finite dimensional
subspaces of $\R^{\infty}$. Then the underlying semisimplicial
space of $\SuntCob(n)_{\bigdot}$ is weakly equivalent to
$\untCob(n)_{\bigdot}$. In particular, $\SuntCob(n)_{\bigdot}$ is a Segal space.

\begin{definition}
We let $\tunCob(n)$ denote the $(\infty,1)$-category associated to the Segal space
$\SuntCob(n)_{\bigdot}$ by Construction \ref{jumpa}. If we adopt Definition \ref{swish}, we can
be more precise: we let $\tunCob(n)$ denote the complete Segal space obtained by completing
$\SuntCob(n)_{\bigdot}$ (see Remark \ref{compus}).
\end{definition}

\begin{warning}\label{sablin}
The Segal space $\SuntCob(n)_{\bigdot}$ is usually {\em not} complete if
$n$ is large. To see why this is, we observe that
$$\SuntCob(n)_0 \simeq \varinjlim \SuntCob(n)_0^{V}
\simeq \varinjlim \R \times \Sub_{0}(V)$$ is a classifying space
for closed manifolds of dimension $(n-1)$ (this follows from general position
arguments: as the dimension of the vector space $V$ grows, the embedding
spaces $\Emb(M, V)$ of Remark \ref{sing} become highly connected, so the
homotopy type of the quotients $\Emb(M,V)/ \Diff(M)$ become good approximations
to the classifying spaces $\BDiff(M)$).
Consequently, we can think of paths in 
$\untCob(n)_{0}$ as corresponding to diffeomorphisms between
$(n-1)$-manifolds. By contrast, invertible $1$-morphisms in the
homotopy category $\h{\SuntCob(n)_{\bigdot}}$ are given by invertible
bordisms between $(n-1)$-manifolds. An invertible bordism $B: M \rightarrow N$ arises from a diffeomorphism of $M$ with $N$ if and only if $B$ is diffeomorphic to
a product $M \times [0,1]$. If $n \geq 6$, the {\it s-cobordism theorem} asserts that this is equivalent
to the vanishing of a certain algebraic obstruction, called the {\it Whitehead torsion} of $B$.
Since there exist bordisms with nontrivial Whitehead torsion, the Segal space
$\SuntCob(n)_{\bigdot}$ is not complete for $n \geq 6$.

The failure of the Segal space $\SuntCob(n)_{\bigdot}$ to be complete is not really problematic;
we can always pass to its completion using Construction \ref{jumpa}. For our purposes, this
is not even necessary: in this paper, we are interested in studying topological field theories,
which are given by functors from $\SuntCob(n)_{\bigdot}$ into other Segal spaces
$\calC_{\bigdot}$. If we assume that $\calC_{\bigdot}$ is complete to begin with, then
the classification of such maps does not change if we replace $\SuntCob(n)_{\bigdot}$
by its completion. In some cases, we can even obtain more refined information by
{\em not} passing to the completion.
\end{warning}

We can regard the above discussion as providing a precise definition of
$\tCob(n)$, which appeared more informally earlier (in its oriented incarnation) in
Definition \ref{spout2}. We can employ the same ideas to define the $(\infty,n)$-category
$\Bord_{n}$ using the language of $n$-fold Segal spaces.

\begin{definition}\label{cabinn}
Let $V$ be a vector space. For every $n$-tuple $k_1, \ldots, k_n$ of nonnegative integers, we
let $(\unBord_{n}^{V})_{k_1, \ldots, k_n}$ denote the collection of tuples
$(M, \{ t^1_{0} \leq \ldots \leq t^1_{k_1} \}, \ldots, \{ t^n_{0} \leq \ldots \leq t^{n}_{k_n} \} )$
where 
\begin{itemize}
\item[$(i)$] $M$ is a closed submanifold of $V \times \R^{n}$ of dimension $n$ (not necessarily compact).
\item[$(ii)$] The projection $M \rightarrow \R^{n}$ is a proper map.
\item[$(iii)$] For every subset $S \subseteq \{1, \ldots, n \}$ and every collection of integers
$\{ 0 \leq j_i \leq k_i \}_{i \in S}$, the projection $M \rightarrow \R^{n} \rightarrow \R^{S}$
does not have $( t_{j_i} )_{i \in S}$ as a critical value.
\end{itemize}
As in Definition \ref{slamm}, the set $(\unBord_{n}^{V})_{k_1, \ldots, k_n}$ can be endowed
with a topology (see \cite{soren2}) so that $\unBord_{n}^{V}$ becomes an $n$-fold simplicial space.
We let $\unBord_{n}$ denote the limit $\varinjlim \unBord_{n}^{V}$, as
$V$ ranges over the finite dimensional subspaces of $\R^{\infty}$.
\end{definition}

The $n$-fold simplicial space $\unBord_{n}$ of Definition \ref{cabinn} is {\em not} an
$n$-fold Segal space in the sense of Definition \ref{multiseg}, because we have not guaranteed that
the $(n-1)$-fold simplicial spaces $(\unBord_{n}^{V})_{0, \bigdot, \ldots, \bigdot}$ are
essentially constant. To remedy the situation, we can modify Definition \ref{cabinn}
by adding the following condition to $(i)$, $(ii)$, and $(iii)$:

\begin{itemize}
\item[$(iv)$] The projection map $M \rightarrow \R^{ \{i+1, \ldots, n \} }$ is submersive
at every point $x \in M$ whose image in $\R^{ \{i\} }$ belongs to the subset
$\{ t_{i_0}, t_{i_1}, \ldots, t_{i_{k_i}} \}$.
\end{itemize}

With this modification, we obtain an $n$-fold Segal space which we will denote by
$\untBord_{n}$. This $n$-fold Segal space is generally not complete (Warning \ref{sablin}),
but nevertheless determines an $(\infty,n)$-category:

\begin{definition}\label{que}
We let $\Bord_{n}$ denote the $(\infty,n)$-category associated to the
$n$-fold Segal space $\untBord_{n}$ defined above. More precisely, we
define $\Bord_{n}$ to be an $n$-fold complete Segal space which is obtained
from $\untBord_{n}$ by completion (see Remark \ref{kisster}).
\end{definition}

\begin{variant}
In the above discussion, we could impose the additional requirement that all manifolds
be endowed with some additional structure, such as an orientation or an $n$-framing.
In these cases, we obtain $n$-fold complete Segal spaces which we will denote by $\Bord_{n}^{\ori}$
and $\Bord_{n}^{\fr}$. 
\end{variant}

\begin{remark}
The $(\infty,n)$-category $\Bord_{n}$ produced by the construction outlined above is naturally written as a direct limit of $(\infty,n)$-categories $\Bord_{n}^{V}$, where $V$ ranges over finite dimensional subspaces of an infinite-dimensional vector space $\R^{\infty}$. The characterization of
$\Bord_{n}$ provided by the cobordism hypothesis has an unstable analogue for
the $(\infty,n)$-categories $\Bord_{n}^{V}$ (the {\it Baez-Dolan tangle hypothesis}).
We will return to the study of these embedded bordism categories in \S \ref{tangus}.
\end{remark}

\subsection{Fully Dualizable Objects}\label{capertown}

Let $\calC$ be a symmetric monoidal $(\infty,n)$-category with duals.
According to the cobordism hypothesis (Theorem \ref{swisher2}), a symmetric
monoidal functor $Z: \Bord_{n}^{\fr} \rightarrow \calC$ is determined (up to canonical isomorphism)
by the object $Z(\ast) \in \calC$. However, not every object of $\calC$ need arise in this way:
as we saw in Example \ref{1dim}, an object $V \in \Vect$ determines a $1$-dimensional topological field theory if and only if $V$ is finite dimensional. Our goal in this section is to formulate
an analogous finiteness condition in the setting of an arbitrary symmetric monoidal
$(\infty,n)$-category $\calC$. 

We begin by reformulating the condition that a vector space $V$ (over a field $k$) be finite-dimensional
in purely categorical terms. As we saw in Example \ref{1dim}, the essential feature
of finite-dimensional vector spaces is that they have a well-behaved duality theory.
If we let $V^{\vee}$ denote the dual space of $V$, then we have a canonical map
$\ev_{V}: V \otimes V^{\vee} \rightarrow k$. For any pair of vector spaces $W$ and $W'$, the map
$\ev_{V}$ determines a map
$$ \Hom( W, W' \otimes V) \rightarrow
\Hom( W \otimes V^{\vee}, W' \otimes V \otimes V^{\vee}) \rightarrow \Hom(W \otimes V^{\vee}, W').$$
If $V$ is finite-dimensional, then this map is an isomorphism. In fact, it has an inverse given by the composition
$$ \Hom( W \otimes V^{\vee}, W') \rightarrow \Hom( W \otimes V^{\vee} \otimes V, W' \otimes V)
\rightarrow \Hom(W, W' \otimes V),$$
where the second map is given by composition with the a {\it coevaluation}
$\coev_{V}: k \rightarrow V^{\vee} \otimes V$ (which is well-defined whenever $V$ is finite-dimensional).
The assertion that these constructions are inverse to one another rests on a compatibility between
the maps $\ev_{V}$ and $\coev_{V}$ which can be described axiomatically as follows:

\begin{definition}\label{frogman}
Let $\calC$ be a monoidal category: that is, a category equipped with a tensor product operation
$\otimes: \calC \times \calC \rightarrow \calC$ which is unital and associative (but not necessarily commutative) up to coherent isomorphism. Let $V$ be an object of $\calC$. We will say that
an object $V^{\vee}$ is a {\it right dual} of $V$ if there exist maps
$$\ev_{V}: V \otimes V^{\vee} \rightarrow {\bf 1} \quad \quad \coev_{V}: {\bf 1} \rightarrow V^{\vee} \otimes V$$
such that the compositions
$$ V \stackrel{ \id_{V} \otimes \coev_{V} }{\longrightarrow} V \otimes V^{\vee} \otimes V
\stackrel{\ev_{V} \otimes \id_{V}}{\longrightarrow} V$$
$$ V^{\vee}  \stackrel{ \coev_{V} \otimes \id_{V^{\vee}}}{\longrightarrow} V^{\vee} \otimes V \otimes V^{\vee} \stackrel{ \id_{V^{\vee}} \otimes \ev_{V}}{\rightarrow} V^{\vee}$$
coincide with $\id_{V}$ and $\id_{V^{\vee}}$, respectively. 
In this case, we will also say that $V$ is a {\it left dual} of $V^{\vee}$.
\end{definition}

\begin{remark}
If $\calC$ is a symmetric monoidal category, then the relationship described in
Definition \ref{frogman} is symmetric in $V$ and $V^{\vee}$; in this case, we will simply say that
$V^{\vee}$ is a {\it dual} of $V$.
\end{remark}

\begin{remark}
Let $V$ be an object of a monoidal category $\calC$. Then left and right duals of $V$ are uniquely
determined up to (unique) isomorphism if they exist. This is a consequence of a more general
assertion regarding adjoint morphisms in a $2$-category which we will explain below (see Example \ref{kist}).
\end{remark}

\begin{example}
An object $V \in \Vect(k)$ has a dual (in the sense of Definition \ref{frogman}) if and only if $V$
is finite-dimensional. For any vector space $V$, we can define a dual space $V^{\vee}$ and
an evaluation map $\ev_{V}: V \otimes V^{\vee} \rightarrow k$, but a compatible coevaluation
$\coev_{V}: k \rightarrow V^{\vee} \otimes V$ can only be defined in the finite-dimensional case.
\end{example}

\begin{definition}\label{dua}
Let $\calC$ be a symmetric monoidal $(\infty,n)$-category. We will say that
an object $C \in \calC$ is {\it dualizable} if it admits a dual when regarded as
an object of the homotopy category $\h{\calC}$.
\end{definition}

The requirement that an object $C$ of a symmetric monoidal $(\infty,n)$-category be dualizable is a natural finiteness condition which can be formulated in completely categorical terms. Moreover, it is obviously
a necessary condition for the existence of a field theory $Z: \Bord^{\fr}_{n} \rightarrow \calC$ with
$Z(\ast) = C$ (a dual of $C$ can be given by evaluating $Z$ on a point with a different orientation,
as we saw in \S \ref{kob}). If $n=1$, this condition is also sufficient: when $\calC$ is an ordinary category, this follows from the argument sketched in Example \ref{1dim} (the general case is more difficult and has interesting consequences, as we will explain in \S \ref{cost}). For $n > 1$, we will need a stronger condition on $C$ to guarantee the existence of $Z$. This condition has a similar flavor but is higher-categorical in nature, involving a demand for duals not only of objects in $\calC$ but also for morphisms.
Before we can formulate it, we need to embark on a brief digression.

One of the most basic examples of a $2$-category is the 
$2$-category $\Cat$ of (small) categories: the objects of
$\Cat$ are categories, the $1$-morphisms of $\Cat$ are functors, and
the $2$-morphisms of $\Cat$ are natural transformations between functors.
We can regard classical category theory as the study of the $2$-category $\Cat$.
It turns out that many of the fundamental concepts of category theory can be generalized
to arbitrary $2$-categories. We now describe one example of this phenomenon, coming from the theory of adjoint functors.

Let $\calC$ and $\calD$ be categories, and let $F: \calC \rightarrow \calD$
and $G: \calD \rightarrow \calC$ be functors. An {\it adjunction} between $F$
and $G$ is a collection of bijections
$$ \phi_{C,D}: \Hom_{\calD}( F(C), D) \simeq \Hom_{\calD}( C, G(D) ),$$
which depend functorially on $C \in \calC$ and $D \in \calD$. In this situation, we say that
$F$ is {\it left adjoint} to $G$ and that $G$ is {\it right adjoint} to $F$; by Yoneda's lemma, 
either $F$ or $G$ determines the other up to canonical isomorphism.

Suppose we are given an adjunction $\{ \phi_{C,D} \}_{C \in \calC, D \in \calD}$ between
$F$ and $G$. Taking $D = F(C)$ and applying
$\phi_{C,D}$ to the identity map from $D$ to itself, we get a canonical map
$u_C: C \rightarrow (G \circ F)(C)$, which depends functorially on $C$: we can
regard the collection of maps $\{ u_C \}_{C \in \calC}$ as a natural transformation
of functors $u: \id_{\calC} \rightarrow G \circ F$. Conversely, if we are given
a natural transformation $u: \id_{\calC} \rightarrow G \circ F$, we get a canonical map
$$ \Hom_{\calD}(F(C), D) \rightarrow
\Hom_{\calC}( (G \circ F)(C), G(D) ) \stackrel{ \circ u_{C}}{\rightarrow} 
\Hom_{\calC}(C, G(D) ).$$
for every pair of objects $C \in \calC$, $D \in \calD$. If each of these maps is bijective, then
we obtain an adjunction between $F$ and $G$; in this case, we will say that $u$ is
the {\it unit} of the adjunction.

We also have the dual notion of a {\it counit} for an adjunction: given a pair of functors
$F: \calC \rightarrow \calD$ and $G: \calD \rightarrow \calC$, a natural transformation
$v: F \circ G \rightarrow \id_{\calD}$ determines a map
$$ \Hom_{\calC}(C,G(D)) \rightarrow \Hom_{\calD}(F(C), (F \circ G)(D))
\stackrel{v_{D} \circ}{\rightarrow} \Hom_{\calD}(F(C), D)$$
for every pair of objects $C \in \calC$, $D \in \calD$. If each of these maps is bijective,
then we obtain an adjunction between 
between a pair of functors $F: \calC \rightarrow \calD$ and $G: \calD \rightarrow \calC$
determines a natural transformation $v: F \circ G \rightarrow \id_{\calD}$. In this case,
we say that $v$ is the {\it counit} of the corresponding adjunction.

If we are given an adjunction between a pair of functors $F: \calC \rightarrow \calD$
and $G: \calD \rightarrow \calC$, then the unit $u: \id_{\calC} \rightarrow G \circ F$
and $v: F \circ G \rightarrow \id_{\calC}$ are compatible in the following sense:
\begin{itemize}
\item[$(\ast)$]
The composite transformations
$$ F = F \circ \id_{\calC} \stackrel{ \id \times u}{\longrightarrow} F \circ G \circ F \stackrel{v \times \id}{\longrightarrow}
\id_{\calD} \circ F = F$$
$$ G = \id_{\calC} \circ G \stackrel{u \times \id}{\longrightarrow}
G \circ F \circ G \stackrel{\id \times v}{\longrightarrow} G \circ \id_{\calD} = G$$
coincide with the respective identity maps on $F$ and $G$.
\end{itemize} 
Conversely, if we are given an arbitrary
pair of natural transformations $u: \id_{\calC} \rightarrow G \circ F$,
$v: F \circ G \rightarrow \id_{\calD}$ satisfying $(\ast)$, then the maps
$$ \Hom_{\calD}(F(C), D) \rightarrow \Hom_{\calC}( C, G(D)) \quad \quad
\Hom_{\calC}(C,G(D)) \rightarrow \Hom_{\calD}( F(C), D)$$
are mutually inverse. Consequently, $u$ is the unit of an adjunction between
$F$ and $G$, and $v$ is the counit of the same adjunction. 

This motivates the following definition:

\begin{definition}\label{cughli}
Let $\calE$ be an arbitrary $2$-category. Suppose we are given a pair of objects
$X,Y \in \calE$ and a pair of $1$-morphisms $f: X \rightarrow Y$ and $g: Y \rightarrow X$.
We will say that a $2$-morphism $u: \id_{X} \rightarrow g \circ f$ is the
{\it unit of an adjunction} between $f$ and $g$ if there exists another $2$-morphism
$v: f \circ g \rightarrow \id_Y$ such that the compositions
$$ f \simeq f \circ \id_{X} \stackrel{ \id \times u}{\longrightarrow} f \circ g \circ f \stackrel{v \times \id}{\longrightarrow} 
\id_{Y} \circ f \simeq f$$
$$ g \simeq \id_{\calC} \circ g \stackrel{u \times \id}{\longrightarrow}
g \circ f \circ g \stackrel{\id \times v}{\longrightarrow} g \circ \id_{\calD} \simeq g$$
both coincide with the identity. In this case, we will also say that $v$ is the
{\it counit of an adjunction}, and that either $u$ or $v$ {\it exhibit $f$ as a left adjoint to
$g$} and {\it exhibit $g$ as a right adjoint to $f$}.
\end{definition}

To get a feeling for the meaning of Definition \ref{cughli}, it is helpful to consider some
examples of $2$-categories other than $\Cat$:

\begin{example}\label{kist}
Recall that a category with a single object is essentially the same thing as a monoid. This
observation can be generalized to higher category theory. For example, suppose that $\calC$ is a monoidal category. We can associate to $\calC$ a $2$-category $B \calC$ as follows:
\begin{itemize}
\item The $2$-category $B \calC$ has only a single object $\ast$.
\item The category of $1$-morphisms $\OHom_{B \calC}( \ast, \ast)$ is $\calC$.
\item The composition law 
$$\OHom_{B \calC}(\ast, \ast) \times \OHom_{B \calC}(\ast, \ast) \rightarrow \OHom_{B \calC}(\ast, \ast)$$
is given by the tensor product on $\calC$.
\end{itemize}
Conversely, if $\calE$ is {\em any} $2$-category with a distinguished object $\ast$, then
$\calC = \OHom_{\calE}(\ast, \ast)$ has the structure of a monoidal category, and there
is a canonical functor $B \calC \rightarrow \calE$, which is an equivalence of $2$-categories
if and only if every object of $\calE$ is equivalent to $\ast$. We may informally summarize this
discussion as follows: a monoidal category is essentially the same thing as a $2$-category
with a single (distinguished) object.

The above construction sets up a dictionary which allows us to translate concepts from the
theory of $2$-categories into concepts in the theory of monoidal categories. In particular,
we note that an object $X \in \calC$ is right dual to an object $Y \in \calC$ (in the sense of Definition \ref{frogman}) if and only if $X$ is right adjoint to $Y$ when both are viewed as $1$-morphisms in $B \calC$ (in the sense of Definition \ref{cughli}).
\end{example}

Suppose we are given a pair of $1$-morphisms $f: X \rightarrow Y$ and
$g: Y \rightarrow X$ in a $2$-category $\calC$, together with a $2$-morphism $u: \id_{X} \rightarrow g \circ f$. If $u$ is the unit of an adjunction between $f$ and $g$, then a compatible counit
$v: f \circ g \rightarrow \id_{Y}$ is uniquely determined. In fact, it is uniquely determined by either
one of the compatibilities demanded by Definition \ref{cughli}. To see this, it is convenient
to break Definition \ref{cughli} into two parts. We will say that a 
$2$-morphism $v: f \circ g \rightarrow \id_{Y}$ is {\it upper compatible} with
$u$ if the composition
$$ f \simeq f \circ \id_{X} \stackrel{ \id \times u}{\longrightarrow} f \circ g \circ f \stackrel{v \times \id}{\longrightarrow}
\id_{Y} \circ f \simeq f$$
coincides with $\id_{f}$. Similarly, we will say that $v$ is {\it lower compatible} 
with $u$ if the composition
$$ g \simeq \id_{\calC} \circ g \stackrel{u \times \id}{\longrightarrow}
g \circ f \circ g \stackrel{\id \times v}{\longrightarrow} g \circ \id_{\calD} \simeq g$$
coincides with $\id_{g}$. We then have the following result, which we will use in \S \ref{kuji}:

\begin{lemma}\label{simptom}
Let $f: X \rightarrow Y$ and $g: Y \rightarrow X$ be $1$-morphisms in a $2$-category
$\calC$, and let $u: \id_{X} \rightarrow g \circ f$ be a $2$-morphism. Suppose that
there exist a $2$-morphism $v: f \circ g \rightarrow \id_{Y}$ which is upper compatible with
$u$, and another $2$-morphism $v': f \circ g \rightarrow \id_{Y}$ which is lower compatible with
$u$. Then $v = v'$, so that $u$ is the unit of an adjunction between $f$ and $g$.
\end{lemma}

Lemma \ref{simptom} can be regarded as an analogue of the following classical observation:
let $G$ be an associative monoid containing an element $f$ which admits both a left inverse
$g$ and a right inverse $g'$. Then $g = g(fg') = (gf)g' = g'$
so that $f$ is invertible. The proof of Lemma \ref{simptom} is essentially the same, but slightly more notationally involved:

\begin{proof}
Let $w: f \circ g \rightarrow \id_{Y}$ be the $2$-morphism in $\calC$ defined by the composition
$$ f \circ g \simeq f \circ \id_{X} \circ g
\stackrel{u}{\rightarrow} f \circ (g \circ f) \circ g
\simeq (f \circ g) \circ (f \circ g) \stackrel{v \times v'}{\rightarrow} \id_{Y} \circ \id_{Y} \simeq \id_{Y}.$$
The $2$-morphism $v \times v'$ can be factored as a composition
$$ (f \circ g) \circ (f \circ g) \stackrel{ v \times \id}{\rightarrow} \id_{Y} \circ (f \circ g)
\stackrel{v'}{\rightarrow} \id_{Y}.$$ It follows that $w$ agrees with the composition
$v' \circ (w' \times \id_{g})$, where $w'$ is the composition
$$ f \simeq f \circ \id_{X} \stackrel{ \id \times u}{\rightarrow} f \circ g \circ f \stackrel{v \times \id}{\rightarrow}
\id_{Y} \circ f \simeq f.$$
Since $v$ is upper compatible with $u$, we conclude that $w' = \id_{f}$ so that $w = v'$.
The same reasoning shows that $w = v$, so that $v = v'$ by transitivity.
\end{proof}

\begin{example}
Let $f: X \rightarrow Y$ be an invertible $1$-morphism in a $2$-category $\calC$, and let
$g: Y \rightarrow X$ denote its inverse. Then we can choose isomorphisms
$$\id_X \simeq g \circ f \quad \quad f \circ g \simeq \id_{Y}$$
which form the unit and counit for an adjunction between $f$ and $g$.
In particular, we can identify $g$ with a right adjoint to $f$; the same argument allows us to identify $g$ with a left adjoint to $f$.
\end{example}

Conversely, suppose that $f$ and $g$ are adjoint $1$-morphisms in $\calC$ such that the unit
and counit maps $\id_{X} \rightarrow g \circ f$ and $f \circ g \rightarrow \id_Y$ are isomorphisms.
Then these maps exhibit $g$ as an inverse to $f$, up to isomorphism. This proves the following:

\begin{proposition}\label{spink}
Let $\calC$ be a $2$-category in which every $2$-morphism is invertible, and let
$f$ be a $1$-morphism in $\calC$. The following conditions are equivalent:
\begin{itemize}
\item[$(1)$] The morphism $f$ is invertible.
\item[$(2)$] The morphism $f$ admits a left adjoint.
\item[$(3)$] The morphism $f$ admits a right adjoint.
\end{itemize}
\end{proposition}

\begin{definition}\label{hotone}
We will say that a $2$-category $\calE$ {\it has adjoints for $1$-morphisms}
if the following conditions are satisfied:
\begin{itemize}
\item[$(1)$] For every $1$-morphism $f: X \rightarrow Y$ in $\calE$, there
exists another $1$-morphism $g: Y \rightarrow X$ and a $2$-morphism
$u: \id_{X} \rightarrow g \circ f$ which is the unit of an adjunction.
\item[$(2)$] For every $1$-morphism $g: Y \rightarrow X$ in $\calE$, there
exists another $1$-morphism $f: X \rightarrow Y$ and a $2$-morphism
$u: \id_{X} \rightarrow g \circ f$ which is the unit of an adjunction.
\end{itemize}
\end{definition}

\begin{example}\label{list}
Let $k$ be a field, and consider the category $\Vect(k)$ of vector spaces over $k$,
endowed with the monoidal structure given by tensor products of vector spaces.
Then linear map $e: V \otimes W \rightarrow k$ is the evaluation map for a duality
if and only if it determines a perfect pairing between $V$ and $W$: that is, $V$ and
$W$ are finite-dimensional, and $e$ induces an isomorphism of $V$ with dual space of $W$.
\end{example}

We now wish to generalize Definition \ref{hotone} to the setting of higher categories.

\begin{definition}\label{adjman}
Let $\calC$ be an $(\infty,n)$-category for $n \geq 2$, and let $\htt{\calC}$ denote its {\it homotopy $2$-category}, defined as follows:
\begin{itemize}
\item The objects of $\htt{\calC}$ are the objects of $\calC$.
\item The $1$-morphisms of $\htt{\calC}$ are the $1$-morphisms of $\calC$.
\item Given a pair of objects $X,Y \in \calC$ and a pair of $1$-morphisms
$f,g: X \rightarrow Y$, we define a $2$-morphism from $f$ to $g$ in $\htt{\calC}$ to be an isomorphism
class of $2$-morphisms from $f$ to $g$ in $\calC$.
\end{itemize}
We will say that $\calC$ {\it admits adjoints for $1$-morphisms} if the homotopy $2$-category $\htt{\calC}$ admits adjoints for $1$-morphisms, in the sense of Definition \ref{hotone}. For $1 < k < n$, we will say that $\calC$ {\it admits adjoints for $k$-morphisms} if, for any pair of objects $X,Y \in \calC$, the $(\infty,n-1)$-category $\OHom_{\calC}(X,Y)$ admits
adjoints for $(k-1)$-morphisms. We will say that an $(\infty,n)$-category $\calC$ {\it has adjoints} if $\calC$ admits adjoints for
$k$-morphisms for all $0 < k < n$.
\end{definition}

\begin{remark}\label{cospin}
Let $\calC$ be an $(\infty,n)$-category. If every $k$-morphism in $\calC$ is invertible, then
$\calC$ admits adjoints for $k$-morphisms. The converse holds provided that
every $(k+1)$-morphism in $\calC$ is invertible (this follows from Proposition \ref{spink}).
\end{remark}

\begin{warning}
The condition that an $(\infty,n)$-category $\calC$ have adjoints depends on the choice
of $n$. We can always choose to view $\calC$ as an $(\infty,n+1)$-category, in which all
$(n+1)$-morphisms are invertible. However, $\calC$ will never have adjoints for
$n$-morphisms unless $\calC$ is an $\infty$-groupoid.
\end{warning}

In the case of a monoidal $(\infty,n)$-category, we can demand slightly more:

\begin{definition}
Let $\calC$ be a monoidal category. We will say that $\calC$ {\it has duals for objects}
if the $2$-category $B \calC$ admits adjoints for $1$-morphisms: in other words, if every
object $X \in \calC$ has both a left and a right dual.

More generally, suppose that $\calC$ is an $(\infty,n)$-category equipped with a monoidal structure.
Let $\h{\calC}$ denote its homotopy category: the objects of $\h{\calC}$ are the objects of
$\calC$, and given a pair of objects $X,Y \in \calC$ we let $\Hom_{ \h{\calC}}(X,Y)$ denote the
set of isomorphism classes of objects in the $(\infty,n-1)$-category $\OHom_{\calC}(X,Y)$.
The homotopy category $\h{\calC}$ inherits a monoidal structure from the monoidal structure
on $\calC$. We will say that $\calC$ {\it has duals for objects} if the ordinary category
$\h{\calC}$ has duals for objects.

We will say that a monoidal $(\infty,n)$-category {\it has duals} if $\calC$ has duals for
objects and $\calC$ has adjoints in the sense of Definition \ref{adjman}.
\end{definition}

\begin{remark}
Let $\calC$ be a monoidal $(\infty,n)$-category. The construction of Example \ref{kist} can
be generalized to produce an $(\infty,n+1)$-category $B \calC$ having only a single object
(see Remark \ref{delop}). Then $\calC$ has duals if and only if $B \calC$ has adjoints.
\end{remark}

\begin{example}\label{stayn}
Let $\calC$ be a monoidal $(\infty,n)$-category. We say that an object
$X \in \calC$ is {\it invertible} if it is invertible when regarded as a $1$-morphism
in $B \calC$: in other words, if there exists another object $X^{-1} \in \calC$ such that
the tensor products $X \otimes X^{-1}$ and $X^{-1} \otimes X$ are isomorphic to the unit object of $\calC$. Every invertible object $X \in \calC$ admits left and right duals (both given by $X^{-1}$), and the converse holds if every $1$-morphism in $\calC$ is invertible (Proposition \ref{spink}). 

A {\it Picard $\infty$-groupoid} is a symmetric monoidal $(\infty,0)$-category $\calC$ such that
every object of $\calC$ is invertible. Using Remark \ref{cospin}, we deduce that a Picard
$\infty$-groupoid has duals when regarded as an $(\infty,n)$-category for {\em any} $n \geq 0$.
Conversely, if $\calC$ is a symmetric monoidal $(\infty,n)$-category which has duals when
regarded as an $(\infty,n+1)$-category, then $\calC$ is a Picard $\infty$-groupoid.
\end{example}

\begin{claim}\label{swist}
Let $\calC$ be a symmetric monoidal $(\infty,n)$-category. Then there exists another
symmetric monoidal $(\infty,n)$-category $\calC^{\fd}$ and a symmetric monoidal functor $i: \calC^{\fd} \rightarrow \calC$
with the following properties:
\begin{itemize}
\item[$(1)$] The symmetric monoidal $(\infty,n)$-category $\calC^{\fd}$ has duals.
\item[$(2)$] For any symmetric monoidal $(\infty,n)$-category $\calD$ with duals
and any symmetric monoidal functor $F: \calD \rightarrow \calC$, there exists
a symmetric monoidal functor $f: \calD \rightarrow \calC^{\fd}$ and an isomorphism $F \simeq i \circ f$;
moreover, $f$ is uniquely determined up to isomorphism.
\end{itemize}
\end{claim}

It is clear that the monoidal $(\infty,n)$-category $\calC^{\fd}$ is determined up to equivalence
by the properties required by Claim \ref{swist}. 

\begin{example}
Let $\calC$ be a symmetric monoidal $(\infty,1)$-category. Then we can identify
$\calC^{\fd}$ with the full subcategory of $\calC$ spanned by the dualizable objects of $\calC$.
\end{example}

In general, the passage from a symmetric monoidal $(\infty,n)$-category $\calC$ to its
fully dualizable part $\calC^{\fd}$ can be accomplished by repeatedly discarding $k$-morphisms which do not admit left and right adjoints (and all objects which do not admit duals).

\begin{definition}\label{justic}
Let $\calC$ be a monoidal $(\infty,n)$-category. We will say that an object
$X \in \calC$ is {\it fully dualizable} if it belongs to the essential image of the 
functor $\calC^{\fd} \rightarrow \calC$. 
\end{definition}

\begin{warning}
The terminology of Definition \ref{justic} is potentially ambiguous, because the notion of a fully dualizable object of an $(\infty,n)$-category $\calC$ depends on $n$. For example, a fully dualizable object of $\calC$ will almost never remain fully dualizable if we regard $\calC$ as an symmetric monoidal $(\infty,n+1)$-category.
\end{warning}

\begin{example}\label{caperr}
For each $n \geq 0$, the symmetric monoidal $(\infty,n)$-category $\Bord_{n}$ has duals.
Every $k$-morphism $f: X \rightarrow Y$ in $\Bord_{n}$ can be identified with an oriented $k$-manifold $M$, having boundary $\overline{X} \coprod_{ \bd X = \bd Y} Y$; here $\overline{X}$ denotes the
manifold $X$ with the opposite orientation. We note that $\overline{M}$
can be interpreted as a $k$-morphism $Y \rightarrow X$, which is both a right and a left
adjoint to $f$. In the case $k=0$ the analysis is similar but easier: for every object
$M \in \Bord_{n}$, the object $\overline{M}$ is both a right and left dual of $M$.
\end{example}

\begin{example}
Let $\calC$ be the category $\Vect(k)$ of vector spaces over a field $k$
(viewed as an $(\infty,1)$-category). Then an object $V \in \calC$ is fully dualizable
if and only if $V$ is finite-dimensional. More generally, an object of a symmetric monoidal
$(\infty,1)$-category $\calC$ is fully dualizable if and only if is dualizable, in the sense of
Definition \ref{dua}.
\end{example}

\begin{remark}
If $n > 1$, then the condition that an object $C$ of a symmetric monoidal $(\infty,n)$-category $\calC$ be fully dualizable is {\em much} stronger than the condition that $C$ be dualizable. The strength of this condition grows rapidly with $n$, and tends to be quite difficult to verify if $n$ is large. In \S \ref{cost}, we will give a simple criterion for testing full dualizability in the case $n=2$ (Proposition \ref{swip}).
\end{remark}

\subsection{The Cobordism Hypothesis}\label{ONACT}

Our goal in this section is to give a more precise formulation of
the Baez-Dolan cobordism hypothesis for framed manifolds (Theorem \ref{swisher2}), and to explain
how this statement generalizes to other types of manifolds. We first establish a bit of terminology.

\begin{notation}
Let $\calC$ and $\calD$ be $(\infty,n)$-categories. There exists another $(\infty,n)$-category
$\Fun( \calC, \calD)$ of {\it functors from $\calC$ to $\calD$}. The $(\infty,n)$-category
$\Fun(\calC, \calD)$ is characterized up to equivalence by the following universal property:
for every $(\infty,n)$-categoy $\calC'$, there is a bijection between the set of isomorphism
classes of functors $\calC' \rightarrow \Fun(\calC, \calD)$ and the set of isomorphism classes
of functors $\calC' \times \calC \rightarrow \calD$.
\end{notation}

\begin{remark}
The collection of all (small) $(\infty,n)$-categories can be organized into a (large)
$(\infty,n+1)$-category $\Cat_{(\infty,n)}$, with mapping objects given by
$\OHom_{ \Cat_{(\infty,n)}}( \calC, \calD) = \Fun(\calC, \calD).$
\end{remark}

\begin{variant}
Suppose that $\calC$ and $\calD$ are {\it symmetric monoidal} $(\infty,n)$-categories. Then
we can also define an $(\infty,n)$-category $\Fun^{\otimes}(\calC, \calD)$ of {\it symmetric monoidal}
functors from $\calC$ to $\calD$.
\end{variant}

\begin{notation}\label{tzee}
Let $\calC$ be an $(\infty,n)$-category. We let $\calC^{\sim}$ denote the underlying
$(\infty,0)$-category obtained by discarding all of the noninvertible morphisms in $\calC$.
\end{notation}

\begin{remark}\label{jumpi}
In the situation of Notation \ref{tzee}, the $(\infty,0)$-category $\calC^{\sim}$ can
be characterized by the following universal property: for any $(\infty,0)$-category $\calD$,
composition with the inclusion $\calC^{\sim} \subseteq \calC$ induces an equivalence
$$ \Fun( \calD, \calC^{\sim}) \rightarrow \Fun( \calD, \calC).$$
\end{remark}

We are now ready to formulate Theorem \ref{swisher2} more precisely:

\begin{theorem}[Cobordism Hypothesis: Framed Version]\label{swisher3}
Let $\calC$ be a symmetric monoidal $n$-category with duals. Then
the evaluation functor $Z \mapsto Z( \ast)$ induces an equivalence
$$ \Fun^{\otimes}( \Bord_{n}^{\fr}, \calC) \rightarrow \calC^{\sim}.$$
In particular, $\Fun^{\otimes}( \Bord_{n}^{\fr}, \calC)$ is an
$(\infty,0)$-category.
\end{theorem}

\begin{remark}
Theorem \ref{swisher3} is best regarded as comprised of two separate assertions:
\begin{itemize}
\item[$(a)$] The $(\infty,n)$-category $\Fun^{\otimes}( \Bord_{n}^{\fr},\calC)$
is an $(\infty,0)$-category. Consequently, Remark \ref{jumpi} implies that the evaluation
functor $\Fun^{\otimes}( \Bord_{n}^{\fr}, \calC) \rightarrow \calC$ factors through a functor
$\phi: \Fun^{\otimes}( \Bord_{n}^{\fr}, \calC) \rightarrow \calC^{\sim}$, which is well-defined up to isomorphism.

\item[$(b)$] The functor $\phi$ is an equivalence of $(\infty,0)$-categories.
\end{itemize}
Assertion $(a)$ asserts that every $k$-morphism in $\Fun^{\otimes}( \Bord_{n}^{\fr}, \calC)$ is invertible
for $0 < k \leq n$. This statement is relatively formal. For example, suppose that $k=1$, and that
$\alpha: Z \rightarrow Z'$ is a natural transformation of field theories $Z,Z': \Bord_{n}^{\fr} \rightarrow \calC$. We wish to show that $\alpha$ is invertible; in other words, we wish to show that for every object
$M \in \Bord_{n}^{\fr}$, the induced map $\alpha_M: Z(M) \rightarrow Z'(M)$ is an isomorphism in the homotopy category $\h{\calC}$. Let $\overline{M}$ be the same manifold equipped with an $n$-framing of the opposite orientation. Then $Z(M)$ and $Z'(M)$ are dual to $Z( \overline{M})$ and $Z'( \overline{M})$ in the homotopy category $\h{\calC}$. In particular, we have a map 
$$\alpha_{ \overline{M} }^{\vee}: Z'(M) \simeq Z'( \overline{M})^{\vee}
\rightarrow Z( \overline{M})^{\vee} \simeq Z(M).$$
It is not difficult to check that $\alpha_{ \overline{M}}^{\vee}$ is the desired homotopy inverse to
$\alpha_{M}$.

Assertion $(b)$ is much less trivial, and will occupy our attention for the bulk of this paper.
\end{remark}

\begin{remark}
In the statement of Theorem \ref{swisher3}, the assumption that $\calC$ has duals entails no
real loss of generality. For any symmetric monoidal $(\infty,n)$-category $\calC$, the
canonical map
$$ \Fun^{\otimes}( \Bord_{n}^{\fr}, \calC^{\fd} ) \rightarrow \Fun^{\otimes}( \Bord_{n}^{\fr}, \calC)$$
is an equivalence of $(\infty,n)$-categories, where $\calC^{\fd}$ is defined as in Claim \ref{swist}:
this follows from the fact that $\Bord_{n}^{\fr}$ has duals (Example \ref{caperr}). 
Applying Theorem \ref{swisher3}, we deduce that
$\Fun^{\otimes}( \Bord_{n}^{\fr},\calC)$ is equivalent to the underlying
$\infty$-groupoid of $\calC^{\fd}$: in other words, $\Fun^{\otimes}( \Bord_{n}^{\fr}, \calC)$
is a classifying space for fully dualizable objects of $\calC$.
\end{remark}

\begin{remark}
Theorem \ref{swisher3} may appear to be more precise than Theorem \ref{swisher2}:
it describes the entire $(\infty,n)$-category of functors $\Fun^{\otimes}( \Bord_{n}^{\fr}, \calC)$, rather
than just its isomorphism classes of objects. However, this generality is only apparent: to prove
that a functor $\calD \rightarrow \calD'$ of $\infty$-groupoids is an equivalence, it suffices to show that for every $\infty$-groupoid $\calE$, the induced functor
$\Fun( \calE, \calD) \rightarrow \Fun( \calE, \calD')$ induces a bijection between isomorphism
classes of objects (this follows from Yoneda's lemma). To deduce Theorem \ref{swisher3} from
Theorem \ref{swisher2}, it suffices to apply this observation to the commutative diagram
$$ \xymatrix{ \Fun^{\otimes}( \Bord_{n}^{\fr}, \Fun(\calE, \calC))
\ar[r] \ar[d] & \Fun( \calE, \calC)^{\sim} \ar[d] \\
\Fun( \calE, \Fun^{\otimes}( \Bord_{n}^{\fr}, \calC) ) \ar[r] & \Fun( \calE, \calC^{\sim}), }$$
in which the vertical functors are equivalences.
\end{remark}

We now explain how Theorem \ref{swisher3} generalizes to the case of manifolds with
other structure groups. We begin with the following observation: by definition, a framing
of an $n$-manifold $M$ is an isomorphism of the tangent bundle
$T_M$ with the trivial bundle $\underline{\R}^{n}$ of rank $n$. Consequently,
the collection of all framings of $M$ carries an action of the orthogonal group
$\OO(n)$. More generally, we have an action of $\OO(n)$ on the collection of all
$n$-framings of a manifold $M$ of dimension $\leq n$. These actions together
determine an action of $\OO(n)$ on the $(\infty,n)$-category $\Bord_{n}^{\fr}$, and therefore
on the $\infty$-groupoid $\Fun^{\otimes}( \Bord_{n}^{\fr}, \calC)$, for any
symmetric monoidal $(\infty,n)$-category $\calC$. Combining this observation
with Theorem \ref{swisher3}, we obtain the following:

\begin{corollary}\label{ail}
Let $\calC$ be a symmetric monoidal $(\infty,n)$-category with duals.
Then the underlying $\infty$-groupoid $\calC^{\sim}$ carries an action of the orthogonal
group $\OO(n)$.
\end{corollary}

\begin{remark}\label{aile}
According to Thesis \ref{cope3}, the $\infty$-groupoid $\calC^{\sim}$ can be identified
with the fundamental $\infty$-groupoid of a topological space $X$, which is well-defined up to homotopy equivalence. Corollary \ref{ail} can be formulated more precisely as follows: it is possible to choose the space $X$ so that it carries an action of the orthogonal group $\OO(n)$.
\end{remark}

\begin{example}\label{ilser}
Let $\calC$ be a symmetric monoidal $(\infty,1)$-category. Then to say that $\calC$ has duals
is to say that every object $X \in \calC$ admits a dual $X^{\vee}$ in the homotopy category
$\h{\calC}$. In this case, Corollary \ref{ail} asserts that the underlying $\infty$-groupoid
of $\calC$ admits an action of the orthogonal group $\OO(1) \simeq \Z / 2\Z$. The action
of this group corresponds to the involution $X \mapsto X^{\vee}$ on $\calC^{\sim}$.
In this case, we can interpret Corollary \ref{ail} as saying that the dual $X^{\vee}$ of an object
$X \in \calC$ is defined not only up to isomorphism in the homotopy category $\h{\calC}$, but up to a contractible space of choices in $\calC$ itself.
\end{example}

\begin{warning}
The action of $\OO(n)$ on $\calC^{\sim}$ in Corollary \ref{ail} is not the restriction of an
action of $\OO(n)$ on $\calC$ itself. For example, when $n=1$, the construction
$X \rightarrow X^{\vee}$ is a {\em contravariant} functor from $\calC$ to itself.
We will return to this point in Remark \ref{swing}.
\end{warning}

\begin{example}
Let $\calC$ be a symmetric monoidal $(\infty,2)$-category with duals. According
to Corollary \ref{ail}, the $\infty$-groupoid $\calC^{\sim}$ carries an action of
the orthogonal group $\OO(2)$. In particular, we have a canonical map
$\SO(2) \times \calC^{\sim} \rightarrow \calC^{\sim}$, which we can think of as
an {\em automorphism} of the identity functor from $\calC^{\sim}$ to itself.
This gives rise to an automorphism $S_X$ of every object $X \in \calC$.
We will refer to $S_X$ as the {\it Serre automorphism} of $X$. See Remark
\ref{umblecrown} for further discussion.
\end{example}

\begin{example}\label{slipper}
Let $\calC$ be a Picard $\infty$-groupoid (see Example \ref{stayn}).
Using Thesis \ref{cope3}, we can identify $\calC$ with a topological space $X$. The symmetric monoidal structure
on $\calC$ endows $X$ with the structure of an {\it $E_{\infty}$-space}: that is, it is equipped
with a multiplication operation which is commutative, associative, and unital up to coherent homotopy.
The assumption that every object of $\calC$ be invertible translates into the requirement that
$X$ be {\it grouplike}: that is, the commutative monoid $\pi_0 X$ is actually an abelian group.
It follows that $X$ has the structure of an infinite loop space: that is, there is a sequence of pointed spaces $\{ X(n) \}_{n \geq 0}$ such that $X(0) \simeq X$ and $X(n)$ is equivalent to the loop space
$\Omega X(n+1)$ for all $n \geq 0$. In particular, we can identify $X(0)$ with the $n$-fold loop
space $$\Omega^{n} X(n) = \{ f: D^n \rightarrow X(n): ( \forall x \in S^{n-1}) [ f(x) = \ast ] \}.$$
Here $\ast$ denotes the base point of $X(n)$, and $D^n$ the $n$-dimensional disk
with boundary $S^{n-1} \subseteq D^n$. There is a canonical action of $\OO(n)$ on the
disk $D^n$, which gives rise to an action of $\OO(n)$ on $X$ up to homotopy; this construction recovers the $\OO(n)$-action of Corollary \ref{ail}.

We can summarize the situation more succinctly using the language of algebraic topology.
We can identify a Picard $\infty$-groupoid $\calC$ with a (connective) spectrum. This
spectrum then carries an action of the direct limit $\OO = \varinjlim \OO(n)$ (this is a version
of the classical {\it $J$-homomorphism} in stable homotopy theory). Restricting
to the subgroups $\OO(n)$ for various $n$, we recover the action of $\OO(n)$ on
$\calC = \calC^{\sim}$ guaranteed by Corollary \ref{ail} (note that $\calC$ can be regarded
as an $(\infty,n)$-category with duals; see Example \ref{stayn}).
\end{example}

We now turn to the problem of describing the analogue of Theorem \ref{swisher3} when we endow our manifolds with structures other than that of an $n$-framing. We first need a general digression about tangential structures on manifolds.

\begin{notation}\label{cusper1}
Let $X$ be a topological space and let $\zeta$ be a real vector bundle on $X$ of rank $n$.
Let $M$ be a manifold of dimension $m \leq n$. An
{\it $(X,\zeta)$-structure} on $M$ consists of the following data:
\begin{itemize}
\item[$(1)$] A continuous map $f: M \rightarrow X$.
\item[$(2)$] An isomorphism of vector bundles
$$T_M \oplus \underline{\R}^{n-m} \simeq f^{\ast} \zeta.$$
\end{itemize}
\end{notation}

\begin{protodefinition}
Let $X$ be a topological space, and let $\zeta$ be an $n$-dimensional vector bundle
on $X$. The $(\infty,n)$-category $\Bord_{n}^{(X,\zeta)}$ is defined just as
$\Bord_{n}$ (see Definition \ref{que}), except that all of the manifolds involved are required to be equipped with an $(X, \zeta)$-structure.
\end{protodefinition}

\begin{theorem}[Cobordism Hypothesis for $(X, \zeta)$-Manifolds]\label{swisher5}
Let $\calC$ be a symmetric monoidal $(\infty,n)$-category with duals, let
$X$ be a CW complex, let $\zeta$ be an $n$-dimensional vector bundle over $X$ equipped with an inner product, and let $\widetilde{X} \rightarrow X$ be the associated principal $\OO(n)$-bundle
of orthonormal frames in $\zeta$. Then there is an equivalence of $(\infty,0)$-categories
$$ \Fun^{\otimes}( \Bord_{n}^{(X, \zeta)}, \calC) \simeq \Hom_{\OO(n)}( \widetilde{X}, \calC^{\sim}),$$
Here we identify $\calC^{\sim}$ with a topological space carrying an action of the group $\OO(n)$
$($see Corollary \ref{ail} and Remark \ref{aile}$)$.
\end{theorem}

\begin{remark}
Let $X$ and $\zeta$ be as in Theorem \ref{swisher5}. Note that every point $\widetilde{x} \in \widetilde{X}$ determines an $(X, \zeta)$-structure on the $0$-manifold $\ast$ consisting of a single point. The equivalence of Theorem \ref{swisher5} is implemented by restricting a symmetric monoidal functor $Z: \Bord_{n}^{(X, \zeta)} \rightarrow \calC$ to $(X, \zeta)$-manifolds which are obtained in this way.
\end{remark}

\begin{remark}\label{exor}
Let $X$, $\zeta$, and $\calC$ be as in Theorem \ref{swisher5}. For every map of
CW complexes $Y \rightarrow X$, let $\widetilde{Y} = Y \times_{X} \widetilde{X}$ be the associated
$\OO(n)$-bundle over $Y$. The functor $Y \mapsto \Hom_{\OO(n)}( \widetilde{Y}, \calC^{\sim})$
carries homotopy colimits in $Y$ to homotopy limits of spaces. It follows from Theorem \ref{swisher5}
that the functor $Y \mapsto \Bord_{n}^{(Y, \zeta|Y)}$ commutes with homotopy colimits in $Y$.
This can be regarded as a kind of excision property. For example, it implies that for any open
covering $\{ U_i \}$ of $X$, we can recover $\Bord_{n}^{(X, \zeta)}$ by gluing together the
$(\infty,n)$-categories $\Bord_{n}^{(U_i, \zeta|U_i)}$ in a suitable way. This is the reflection of a simple geometric idea: namely, that any manifold $M$ equipped with an $(X, \zeta)$-structure
$f: M \rightarrow X$ can be disassembled into pieces such that on each piece,
$f$ factors through one of the open sets $U_i$. We do not know a direct proof of this excision property:
we can only deduce it indirectly from Theorem \ref{swisher5}. However, the idea of decomposing
$X$ into pieces will feature in the proof of Theorem \ref{swisher5} that we present in \S \ref{forty1}.
\end{remark}

When the topological space $X$ is connected, Theorem \ref{swisher5} can be expressed in a slightly more conceptual way. To explain this, we need to introduce a bit more notation.

\begin{notation}\label{cusper2}
Suppose that $G$ is a topological group equipped with a continuous homomorphism
$\chi: G \rightarrow \OO(n)$, where $\OO(n)$ denotes the orthogonal group. We let
$EG$ denote a weakly contractible $G$-CW complex on which $G$ acts freely
(in other words, a space which is obtained by gluing together cells of the form
$G \times D^{n}$ for $n \geq 0$; such a space exists and is unique up to $G$-equivariant homotopy equivalence). Let $BG = EG/G$ denote a classifying space for $G$, and let
$\zeta_{\chi} = (\R^{n} \times EG)/G$ denote the vector bundle over $BG$ determined by $\chi$.
We will denote the $(\infty,n)$-category $\Bord_{n}^{(BG, \zeta_{\chi})}$ by $\Bord_{n}^{G}$ and
refer to an $(BG, \zeta)$-structure on a manifold $M$ as an
{\it $G$-structure on $M$}.
\end{notation}

\begin{example}
If the group $G$ is trivial, then a $G$-structure on a manifold $M$ is an
$n$-framing of $M$, as described in Variant \ref{barvar}. If $G = \SO(n)$, then giving a $G$-structure on a manifold $M$ is (up to contractible ambiguity) equivalent to choosing an orientation of $M$. If
$G = \OO(n)$, then a $G$-structure on a manifold $M$ consists of no structure at all. We therefore have equivalences 
$$ \Bord^{\fr}_{n} \simeq \Bord^{ \{1\} }_{n} \quad \Bord^{\ori}_{n} \simeq \Bord^{\SO(n)}_{n}
\quad \Bord_{n} \simeq \Bord_{n}^{\OO(n)}.$$
\end{example}

\begin{definition}\label{silk}
Let $G$ be a topological group acting continuously on a topological space $X$.
The {\it homotopy fixed set} $X^{hG}$ is defined to be the space of $G$-equivariant
maps $\Hom_{G}( EG, X)$, where $EG$ is as in Notation \ref{cusper2}.
\end{definition}

\begin{remark}
In order for Definition \ref{silk} to be sensible, we should require that $G$ and $EG$ are CW-complexes
(in practice, this is easy to arrange, since we will generally take $G$ to be a compact Lie group).
In this case, the homotopy type of $X^{hG}$ is independent of the choice of $EG$, and
the construction $X \mapsto X^{hG}$ preserves weak homotopy equivalences.
\end{remark}

\begin{notation}
Let $\calC$ be an $\infty$-groupoid carrying an action of the topological group $G$.
An analogue of Thesis \ref{cope3} asserts that $\calC$ is equivalent to the fundamental
$\infty$-groupoid of a topological space $X$ carrying an action of $G$. We let
$\calC^{hG}$ denote the fundamenal $\infty$-groupoid of the homotopy fixed set $X^{hG}$.
\end{notation}

\begin{theorem}[Cobordism Hypothesis for $G$-Manifolds]\label{swisher4}
Let $\calC$ be a symmetric monoidal $(\infty,n)$-category with duals, and let
$\chi: G \rightarrow \OO(n)$ be a continuous group homomorphism.
There is a canonical equivalence of $(\infty,n)$-categories
$$ \Fun^{\otimes}( \Bord_{n}^{G}, \calC) \rightarrow ( \calC^{\sim})^{hG}.$$
In particular, $\Fun^{\otimes}( \Bord_{n}^{G}, \calC)$ is an $\infty$-groupoid.
\end{theorem}

\begin{proof}
Let $\widetilde{BG}$ denote the $\OO(n)$-bundle $(EG \times \OO(n)) / G$
over $BG$ determined by the homomorphism $\chi$. 
According to Theorem \ref{swisher5}, we can identify $\Fun^{\otimes}( \Bord_{n}^{G}, \calC)$ with
$\Hom_{\OO(n)}(\widetilde{BG}, \calC^{\sim})$. The desired conclusion follows from
the evident equivalence
$$ \Hom_{ \OO(n)}( \widetilde{BG}, \calC^{\sim}) \simeq
\Hom_{G}( EG, \calC^{\sim}) \simeq (\calC^{\sim})^{hG}.$$
\end{proof}

\begin{example}
In the special case where the group $G$ is trivial, Theorem \ref{swisher4} reduces to
Theorem \ref{swisher3}.
\end{example}

\begin{example}
In the case $n=1$ and $G = \OO(1)$, Theorem \ref{swisher4} asserts that the data of a
$1$-dimensional {\em unoriented} field theory $Z: \Bord_{1} \rightarrow \calC$ is equivalent to the data of a homotopy fixed point of for the natural action of $\OO(1) \simeq \Z / 2 \Z$ on
$\calC^{\sim}$. As indicated in Example \ref{ilser}, we can think of this action as given by the involution on $\calC$ which carries every object $X \in \calC$ to its dual $X^{\vee}$. A homotopy fixed point
can therefore be identified with a {\em symmetrically} self-dual object of $\calC$: in other words,
an object $X \in \calC$ equipped with a symmetric map $X \otimes X \rightarrow {\bf 1}$ which exhibits $X$ as a dual of itself. For example, if $\calC$ is the (ordinary) category of vector spaces, then
$Z$ is determined by the finite dimensional vector space $V = Z(\ast)$, together with a nondegenerate symmetric bilinear form on $V$.
\end{example}

\begin{remark}
If $X$ is a nonempty path connected topological space, then there exists a topological group
$G$ and a weak homotopy equivalence $BG \simeq X$. A real vector bundle of rank $n$ on $X$ determines a continuous homomorphism $\chi: G \rightarrow \OO(n)$ (possibly after replacing $G$ by
a weakly equivalent topological group) so that the pairs $(BG, \zeta_{\chi})$ and $(X, \zeta)$ are weakly equivalent. The proof of Theorem \ref{swisher4} shows that Theorem \ref{swisher4} for
$\chi: G \rightarrow \OO(n)$ is equivalent to Theorem \ref{swisher5} for the pair $(X, \zeta)$. 
Consequently, we can regard Theorem \ref{swisher5} as a very mild generalization
of Theorem \ref{swisher4}: it is slightly stronger because it encompasses the situation where the space $X$ is not path connected.
\end{remark}

\begin{remark}\label{skilj}
Throughout this paper, we have considered only {\em smooth} manifolds. However,
it is possible to define analogues of the $(\infty,n)$-category $\Bord_{n}$ in the piecewise linear and topological settings. Let us denote these $(\infty,n)$-categories by $\Bord_{n}^{\PL}$ and
$\Bord_{n}^{\Top}$. These $(\infty,n)$-categories are related by functors
$$ \Bord_{n} \stackrel{\theta}{\rightarrow} \Bord_{n}^{\PL} \stackrel{\theta'}{\rightarrow} \Bord_{n}^{\Top},$$ where $\theta'$ is defined by forgetting piecewise linear structures, and $\theta$ is defined
by selecting Whitehead compatible triangulations of smooth manifolds (which are always unique
up to a contractible space of choices). These functors are related as follows:
\begin{itemize}
\item[$(1)$] For $n \leq 3$, the theories of smooth, topological, and piecewise linear manifolds
all essentially equivalent to one another, and the functors $\theta$ and $\theta'$ are equivalences.
\item[$(2)$] If we restrict our attention to $n$-framed manifolds, then the analogue of the functor
$\theta$ is an equivalence $\Bord_{n}^{\fr} \rightarrow \Bord_{n}^{\PL, \fr}$ (here the notion of
a {\em framing} in the piecewise linear setting involves trivialization of piecewise linear microbundles, rather than vector bundles): this can be proven using parametrized smoothing theory.
\item[$(3)$] Combining $(2)$ with the proof of Corollary \ref{ail}, we can obtain a stronger result: 
for any symmetric monoidal $(\infty,n)$-category $\calC$ with duals, the underlying
$\infty$-groupoid $\calC^{\sim}$ carries an action of the group $\PL(n)$ of piecewise-linear
homeomorphisms from $\R^{n}$ to itself.
\item[$(4)$] If $n \neq 4$, then assertion $(3)$ is in some sense optimal: the group $\PL(n)$
is homotopy equivalent to the automorphism group of $\Bord_{n}^{\PL, \fr} \simeq
\Bord_{n}^{\fr}$, and is therefore {\em universal} among groups which act on
$\calC^{\sim}$ for every symmetric monoidal $(\infty,n)$-category with duals $\calC$. 
We do not know if the analogous statement holds for $n=4$: it is equivalent to the
piecewise-linear Schoenflies conjecture.
\item[$(5)$] Using $(3)$, one can formulate an analogue of Theorem \ref{swisher5} for
piecewise linear $\R^{n}$-bundles $\zeta \rightarrow X$ (and Theorem \ref{swisher4} for
maps of groups $G \rightarrow \PL(n)$). This analogue is equivalent to Theorem \ref{swisher5}
if $\zeta \rightarrow X$ can be refined to a vector bundle. The proof in general is more difficult, and requires methods which we will not describe here.
\item[$(6)$] Assertion $(4)$ guarantees that the action of $\PL(n)$ on $\calC^{\sim}$ generally
does {\em not} extend to an action of the group $\Top( \R^{n})$ of topological homeomorphisms
of $\R^{n}$ with itself when $n \geq 5$ (here $\calC$ denotes a symmetric monoidal
$(\infty,n)$-category with duals). Consequently, there is no obvious way to formulate
Theorems \ref{swisher4} and \ref{swisher5} in the topological setting.

We do not know an analogue of the cobordism hypothesis which describes the topological bordism categories $\Bord_{n}^{\Top}$ for $n \geq 4$. Roughly speaking, the usual cobordism hypothesis (for smooth manifolds) can be regarded as an articulation of the idea that smooth manifolds can be constructed by a sequence of handle attachments: that is, every smooth manifold admits a handle decomposition. The handle decomposition of a smooth manifold is not unique. Nevertheless, any two handle decompositions can be related by a finite sequence of handle cancellation rules which have natural category-theoretic interpretations (we will explain this idea more precisely in \S \ref{kuji}). 
However, there are topological $4$-manifolds which do not admit handle decompositions (such as Freedman's $E_8$-manifold).
\end{itemize}
\end{remark}

\subsection{The Mumford Conjecture}\label{mumm}


Fix a closed oriented surface $\Sigma_g$ of genus $g \geq 0$.
Let $\Diff( \Sigma_g)$ denote the group
of orientation-preserving diffeomorphisms of $\Sigma_g$, let
$\EDiff(\Sigma_g)$ denote a contractible space with a free
action of $\Diff(\Sigma_g)$, and let $\BDiff( \Sigma_g ) = \EDiff(\Sigma_g)/ \Diff(\Sigma_g)$
denote a classifying space for $\Diff(\Sigma_g)$. Over the classifying space
$\BDiff(\Sigma_g)$ we have a canonical fiber bundle
$$ \pi: X = (E \Diff(\Sigma_g) \times \Sigma_g) / \Diff(\Sigma_g) \rightarrow \BDiff(\Sigma_g),$$
with fibers homeomorphic to $\Sigma_g$. Using the fact that the fibers of $\pi$ are
oriented surfaces, we deduce:

\begin{itemize}

\item[$(a)$] There is an oriented real vector bundle $V$ of rank $2$ on $X$, whose restriction
to every point $x \in X$ is given by the tangent space to the fiber $\pi^{-1}\{ \pi(x) \}$ at $x$.
This vector bundle $V$ has an Euler class $e(V) \in \HH^{2}(X; \Q)$.

\item[$(b)$] There is an integration map on cohomology
$$ \HH^{\ast+2}(X; \Q) \rightarrow \HH^{\ast}( \BDiff(\Sigma_g); \Q).$$
In particular, for each $n \geq 0$ we can evaluate this map on $e(V)^{n+1}$ to
obtain a class $$\kappa_{n} \in \HH^{2n}( \BDiff(\Sigma_g); \Q).$$

\end{itemize}

The classes $\{ \kappa_n \}_{n > 0}$ determine a homomorphism of graded rings
$$ \Q[ \kappa_1, \kappa_2, \ldots ] \rightarrow \HH^{\ast}( \BDiff(\Sigma_g); \Q),$$
where the grading on the left hand side is determined by letting each $\kappa_n$ have
degree $2n$. The following result was conjectured by Mumford:

\begin{theorem}[Mumford Conjecture]\label{camber}
Fix a positive integer $n$. Then for all sufficiently large $g$ $($depending on $n${}$)$, the map
$$\Q[ \kappa_1, \kappa_2, \ldots ] \rightarrow \HH^{\ast}( \BDiff( \Sigma_g); \Q)$$
defined above is an isomorphism in degrees $\leq n$.
\end{theorem}

\begin{remark}
For each $g \geq 0$, the group of connected components
$\pi_0 \Diff( \Sigma_g )$ is called the {\it mapping class group} of $\Sigma_g$ and denoted
by $\Gamma_g$.
If $g \geq 2$, then the projection map $\Diff(\Sigma_g) \rightarrow \pi_0 \Diff(\Sigma_g)$
is a homotopy equivalence. Consequently, we can identify
$\HH^{\ast}( \BDiff( \Sigma_g);\Q)$ with the rational cohomology of the discrete group
$\Gamma_g$.
\end{remark}

Theorem \ref{camber} was proven by Madsen and Weiss in \cite{madsenweiss}. Our goal in this section
is to explain the relationship between their proof (at least in its modern incarnation) and
the cobordism hypothesis. For this, we need to introduce a bit of terminology.

\begin{notation}
Let $X_{\bigdot}$ be a simplicial space. We define a new topological space
$| X_{\bigdot} |$, the {\it geometric realization} of $X_{\bigdot}$, as the coequalizer
of the diagram
$$\xymatrix{ \coprod_{ f: [m] \rightarrow [n]} X_{n} \times \Delta^m
 \ar@<.4ex>[r] \ar@<-.4ex>[r] & \coprod_{n} X_{n} \times \Delta^n }.$$
In other words, $| X_{\bigdot} |$ is the space obtained by gluing together the products
$X_{n} \times \Delta^n$ in the pattern specified by the simplicial structure of $X_{\bigdot}$.

More generally, suppose that $X$ is a $k$-fold simplicial space. We define
the geometric realization $|X|$ of $X$ to be the coequalizer
$$\xymatrix{ \coprod_{ \{ f_i: [m_i] \rightarrow [n_i] \}_{ 1 \leq i \leq k}} X_{n_1, \ldots, n_k} \times \Delta^{m_1} \times \ldots \times \Delta^{m_k} 
 \ar@<.4ex>[r] \ar@<-.4ex>[r] & \coprod_{n_1, \ldots, n_k} X_{n_1, \ldots, n_k} \times \Delta^{n_1} \times \ldots \times \Delta^{n_k} }.$$
\end{notation}

\begin{remark}\label{aio}
Let $X$ be a $k$-fold simplicial space. Then we can view $X$ as a diagram in the category of
topological spaces. The geometric realization $|X|$ can be identified with the {\it homotopy colimit} of this diagram. In other words, the construction $X \mapsto |X|$ is left adjoint (at the level of homotopy categories) to the functor which carries a topological space $Y$ to the constant $k$-fold simplicial space taking the value $Y$.
\end{remark}

\begin{remark}\label{spiela}
According to Thesis \ref{cope3}, we can identify $(\infty,0)$-categories with topological spaces.
Suppose that $X$ is an $n$-fold Segal space, which determines an $(\infty,n)$-category
$\calC$ as explained in \S \ref{bigseg}. Then the geometric realization $|X|$ can be viewed as an $(\infty,0)$-category, and Remark \ref{aio} implies that $|X|$ is {\it universal} among $(\infty,0)$-categories
equipped with a functor $\calC \rightarrow |X|$. In other words, we can think of the geometric
realization $|X|$ as encoding the $(\infty,0)$-category obtained from $\calC$ by formally inverting
all $k$-morphisms for $1 \leq k \leq n$.
\end{remark}

\begin{example}\label{camma}
Let $X$ denote the $n$-fold Segal space $\untBord_{n}^{\ori}$ used to define
$\Bord_{n}^{\ori}$ in \S \ref{swugg}. The geometric realization $|X|$ has a canonical base point $\ast$, given by the point of $X_{0,\ldots,0}$ supplied by the empty set. 

Suppose that $M$ is a closed oriented manifold of dimension $n$. Then there exists a tangential
embedding $f$ of $M$ into $(0,1)^{n} \times \R^{\infty} \times \BSO(n)$, which determines a
point of the space $X_{1,\ldots, 1}$. This point in turn determines a map of topological spaces
$$\Delta^1 \times \ldots \times \Delta^1 \rightarrow |X|,$$
which carries the boundary of the cube $\Delta^1 \times \ldots \times \Delta^1$ to the base point
of $|X|$: this determines an element of the homotopy group $\pi_n( |X|, \ast)$ which we will
denote by $[M]$; it is not difficult to see that the homotopy class $[M]$ is independent of the choice of $f$.

The construction $M \mapsto [M]$ is completely functorial, and makes sense for families of manifolds.
In particular, we can apply this construction to the {\em universal} bundle of manifolds with fiber
$M$, whose base is the classifying space $\BDiff(M)$. This classifying space can be identified with
a suitable path component of $X_{1, \ldots, 1}$, so we get a map
$$ \Delta^1 \times \ldots \times \Delta^1 \times \BDiff(M) \rightarrow |X|,$$
which carries the product of $\BDiff(M)$ with the boundary of the cube $\Delta^1 \times \ldots \times \Delta^1$ to the base point of $|X|$. This can also be interpreted as a map
of topological spaces $\BDiff(M) \rightarrow \Omega^n |X|$, where
$\Omega^n |X|$ denotes the $n$th loop space of $|X|$. Composing with the completion map
$X \rightarrow \Bord_{n}^{\ori}$, we get a map
$\BDiff(M) \rightarrow \Omega^{n} | \Bord_{n}^{\ori} |$. 
\end{example}

Returning to the case of surfaces, we observe that for every genus $g \geq 0$ Example
\ref{camma} provides a map $\eta_{g}: \BDiff( \Sigma_g) \rightarrow \Omega^2 | \Bord_2^{\ori} |$.
Let $Y_g$ denote the path component of $\Omega^2 | \Bord_2^{\ori} |$ containing the image
of $\eta_g$. To prove Theorem \ref{camber}, it suffices to do the following: 

\begin{itemize}
\item[$(i)$] Prove that the induced map on cohomology
$$ \HH^{n}( Y_g ; \Q) \rightarrow
\HH^{n}( \BDiff(\Sigma_g); \Q)$$
is an isomorphism for all sufficiently large $g$ (depending on $n$).
\item[$(ii)$] Compute the cohomology groups $\HH^{\ast}( \Omega^2 | \Bord^{\ori}_2 | ; \Q)$
(which contain the cohomology groups of each component $Y_g \subseteq \Omega^2 | \Bord^{\ori}_2 |$
as direct factors).
\end{itemize}

Step $(i)$ involves delicate geometric arguments which are very specific to
manifolds of dimension $2$ (such as the Harer stability theorem). However, step $(ii)$ has an analogue which is true in any dimension. Moreover, it is possible to be much more precise: we can describe not just the rational cohomology of the space $\Omega^2 | \Bord^{\ori}_2 |$, but the entire homotopy type of
the classifying space $| \Bord^{\ori}_2 |$ itself:

\begin{theorem}[Galatius-Madsen-Tillmann-Weiss, \cite{soren1}]\label{singsung}
Let $n \geq 0$ be an integer. Then the geometric realization $| \Bord^{\ori}_n |$ is homotopy
equivalent to the $0$th space of the spectrum $\Sigma^{n} \MTSO(n)$. Here
$\MTSO(n)$ denotes the Thom spectrum of the virtual bundle $- \zeta$, where
$\zeta$ is the universal rank $n$-vector bundle over the classifying space $\BSO(n)$.
\end{theorem}

\begin{remark}
The result of Galatius-Madsen-Tillmann-Weiss is actually somewhat more general than
Theorem \ref{singsung}; it can be formulated for manifolds with arbitrary structure group, as we will explain below.
\end{remark}

\begin{remark}\label{od}
Theorem \ref{singsung} can be regarded as a generalization of a classical result of 
Thom on the bordism groups of manifolds. Recall that a pair of closed oriented manifolds
$M$ and $N$ of the same dimension $d$ are said to be {\it cobordant} if there is a bordism
from $M$ to $N$: that is, an oriented manifold $B$ of dimension $(d+1)$ whose boundary is diffeomorphic with $\overline{M} \coprod N$. Cobordism is an equivalence relation on manifolds, and the set of equivalence classes $\Omega_{d}$ has the structure of an abelian group (given by disjoint unions of manifolds). In \cite{thom}, Thom showed that the calculation of the groups $\{ \Omega_{d} \}_{d \geq 0}$ could be reduced to a problem of homotopy theory. More precisely, there exists a pointed topological space $X$ and a sequence of isomorphisms $\Omega_{d} \simeq \pi_{d} X$. Moreover,
the space $X$ admits a direct construction in the language of algebraic topology: it is the $0$th space
of what is now callled the {\it Thom spectrum} $\MSO$. 

In the language of higher category theory, we might predict the existence of the space $X$ on the following grounds. Let $\calC$ be the higher category described as follows:
\begin{itemize}
\item The objects of $\calC$ are oriented $0$-manifolds.
\item The $1$-morphisms of $\calC$ are bordisms between oriented $0$-manifolds.
\item The $2$-morphisms of $\calC$ are bordisms between bordisms between oriented $0$-manifolds.
\item \ldots
\end{itemize}
Here it is sensible to view $\calC$ as an $(\infty,0)$-category: for every $k$-morphism $B: M \rightarrow N$ in $\calC$, the same manifold with the opposite orientation defines a bordism $\overline{B}: N \rightarrow M$ which can be taken as an inverse to $B$. According to Thesis \ref{cope3}, we should
expect the existence of a topological space $X$ whose fundamental $\infty$-groupoid
is equivalent to $\calC$. Unwinding the definitions, we learn that the homotopy groups $\pi_d X$
can be identified with the bordism groups $\Omega_{d}$.

The higher category $\calC$ can be viewed as the direct limit of the $(\infty,n)$-categories
$\Bord_{n}$ as $n$ grows. Consequently, the space $X$ can be constructed as the direct
limit of the classifying spaces $| \Bord_n |$. Invoking Theorem \ref{singsung}, we deduce that
$X$ is equivalent to the zeroth space of the spectrum given by the direct limit
$\varinjlim \Sigma^{n} \MTSO(n)$. This direct limit coincides with the Thom spectrum $\MSO$, essentially by definition: consequently, Thom's result can be recovered as a limiting case of
Theorem \ref{singsung}. 
\end{remark}

Our goal for the remainder of this section is to explain the relationship between Theorem \ref{singsung} and the cobordism hypothesis. To begin, suppose that we are given an $(\infty,n)$-category $\calD$. As explained in Remark \ref{spiela}, we can extract from $\calD$ a topological space $| \calD |$, whose fundamental $\infty$-groupoid can be viewed as the $(\infty,0)$-category obtained from $\calD$ by inverting all $k$-morphisms for $1 \leq k \leq n$. We can rephrase this universal property as follows: let $X$ be any topological space having fundamental $\infty$-groupoid $\calC$. Then isomorphism classes
of functors $F: \calD \rightarrow \calC$ can be identified with homotopy classes of continuous maps
$| \calD | \rightarrow X$.
 
Suppose now that the $(\infty,n)$-category $\calD$ is equipped with a symmetric monoidal tensor product operation $\otimes: \calD \times \calD \rightarrow \calD$. This operation induces a continuous map $| \calD | \times | \calD | \rightarrow | \calD |$, which is commutative, associative and unital up to coherent homotopy. Suppose that this multiplication induces a group structure on $\pi_0 | \calD |$
(in other words, that every point of $| \calD |$ has an inverse in $| \calD |$ up to homotopy): this is automatic, for example, if every object of $\calD$ has a dual. As in Example \ref{slipper}, we deduce
that $| \calD |$ is an infinite loop space: that is, there exists a sequence of pointed spaces
$Y(0) = | \calD|, Y(1), Y(2), \ldots$ together with homotopy equivalences $Y(n) \simeq \Omega Y(n+1)$.
Moreover, this infinite loop space can be again be characterized by a universal property.
Suppose that $X$ is another infinite loop space, so that the fundamental $\infty$-groupoid $\calC$
has the structure of a Picard $\infty$-groupoid (see Example \ref{stayn}). Then isomorphism
classes of {\em symmetric monoidal} functors $F: \calD \rightarrow \calC$ can be identified with homotopy classes of {\em infinite loop space} maps $| \calD | \rightarrow X$. Combining
this observation with Theorem \ref{swisher4}, we deduce the following description of the geometric realization of a bordism $(\infty,n)$-category:

\begin{theorem}[Cobordism Hypothesis, Group-Completed Version]\label{swisher5.5}
Let $G$ be a topological group equipped with a continuous homomorphism
$\chi: G \rightarrow \OO(n)$, and let $X$ be an infinite loop space $($so that $X$
carries an action of the group $G$ via the J-homomorphism, as explained in
Example \ref{slipper}$)$. Then the space of infinite loop maps
$\bHom( | \Bord_{n}^{G} |, X)$ is homotopy equivalent to the homotopy fixed point set
$X^{hG}$.
\end{theorem}

Theorem \ref{swisher5.5} completely determines the homotopy type of $| \Bord_{n}^{G} |$ as an infinite loop space. For example, if the group $G$ is trivial, then we deduce that $| \Bord_{n}^{G} |$ is
freely generated (as an infinite loop space) by a single point: this tells us that $| \Bord_{n}^{G} |$ is
equivalent to the stable sphere $Q S^0 \simeq \varinjlim_{k} \Omega^{k} S^k$. More generally,
we can identify $| \Bord_{n}^{G} |$ with the infinite loop space of homotopy {\em coinvariants}
$( Q S^0 )_{hG}$. This, in turn, can be identified with the $0$th space of a certain spectrum: namely, the $n$-fold suspension of the Thom spectrum of the virtue bundle $- \zeta_{\chi}$ on the classifying space $BG$. In the special case $G = \SO(n)$, we deduce that Theorems \ref{swisher5.5} and \ref{singsung} are equivalent to one another. In other words, Theorem \ref{singsung} can be regarded as a special case of the cobordism hypothesis.

\section{Proof of the Cobordism Hypothesis}\label{glob3}

Our objective in this section is to present a proof of the cobordism hypothesis (in its incarnation
as Theorem \ref{swisher5}). Because the argument is quite lengthy and requires a substantial amount of technology which we do not have the space to fully develop here, we will be content to give a sketch which highlights some of the main ideas; a detailed account will appear elsewhere.

For the reader's convenience, we begin by giving a basic summary of our strategy:
\begin{itemize}
\item[$(1)$] To prove the cobordism hypothesis, we need to show that the $(\infty,n)$-category
$\Bord_{n}$ and its variants can be characterized by universal properties. The first idea
is to try to establish these universal properties using induction on $n$. Roughly speaking,
instead of trying to describe $\Bord_{n}$ by generators and relations, we begin by assuming
that we have a similar presentation for $\Bord_{n-1}$; we are then reduced to describing
only the generators and relations which need to be adjoined to pass from $\Bord_{n-1}$ to $\Bord_{n}$. 
We will carry out this reduction in \S \ref{forty1}.

\item[$(2)$] Theorem \ref{swisher5} gives us a description of $\Bord_{n}^{(X,\zeta)}$ for any
topological space $X$ and any rank $n$ vector bundle $\zeta$ (with inner product) on $X$.
In \S \ref{slabun}, we will see that it suffices to treat only the universal case where $X$ is a classifying space $\BO(n)$ (and $\zeta$ is the tautological bundle on $X$). Roughly speaking, the idea
is to consider a topological field theory $Z: \Bord_{n}^{(X, \zeta)} \rightarrow \calC$ as
an unoriented topological field theory having a different target category, whose value
on a manifold $M$ is a collection of $\calC$-valued invariants parametrized by the space
of $(X, \zeta)$-structures on $M$. This reduction to the unoriented case is not logically necessary for the rest of the argument, but does result in some simplifications.

\item[$(3)$] In \S \ref{unf}, we will explain how the cobordism hypothesis (and many other assertions regarding symmetric monoidal $(\infty,n)$-categories with duals) can be reformulated entirely within the setting of $(\infty,1)$-categories. Again, this formulation is probably not logically necessary, but it does make the constructions of \S \ref{kuji} considerably more transparent.

\item[$(4)$] The bulk of the argument will be carried out in \S \ref{kuji}. Roughly speaking, we can view
$\Bord_{n}$ as obtained from $\Bord_{n-1}$ by adjoining new $n$-morphisms corresponding to
bordisms between $(n-1)$-manifolds. Using Morse theory, we can break any bordism up into a sequence of handle attachments, which give us ``generators'' for $\Bord_{n}$ relative to
$\Bord_{n-1}$. The ``relations'' are given by handle cancellations. The key geometric input for our argument is a theorem of Igusa, which asserts that the space of ``framed generalized Morse functions'' on a manifold $M$ is highly connected. This will allow us to prove the cobordism hypothesis
for a modified version of the $(\infty,n)$-category $\Bord_{n}$, which we will denote by
$\Bord_{n}^{\frun}$.

\item[$(5)$] In \S \ref{kumma}, we will complete the proof of the cobordism hypothesis by showing that
$\Bord_{n}^{\frun}$ is equivalent to $\Bord_{n}$. The key ingredients are a connectivity estimate
of Igusa (Theorem \ref{conig}) and an obstruction theoretic argument which relies on a cohomological calculation (Theorem \ref{swisher9}) generalizing the work of Galatius, Madsen, Tillmann, and Weiss.
\end{itemize}

\subsection{Inductive Formulation}\label{forty1}

Our original formulation of the cobordism hypothesis (Theorem \ref{swisher}) was stated
entirely in the setting of symmetric monoidal $n$-categories. In \S \ref{sugar}, we described a more general version (Theorem \ref{swisher2}), which describes bordism categories by a universal property
in the more general setting of $(\infty,n)$-categories. As we explained in Remark \ref{cultis}, this additional generality is crucial to our proof, which uses induction on $n$: even if we are ultimately only interested in understanding tensor functors $Z: \Bord_{n} \rightarrow \calC$ in the case where $\calC$ is an ordinary $n$-category, we will need to understand the restriction of $Z$ to $\Bord_{n-1}$, which takes values in the $(n,n-1)$-category obtained from $\calC$ by discarding the noninvertible $n$-morphisms. Our goal in this section is to outline the inductive step of the proof: namely, we will explain how to deduce the cobordism hypothesis in dimension $n$ from the cobordism hypothesis in dimension $n-1$, together with another statement (Theorem \ref{swisher6}) which describes the relationship between
$\Bord_{n}$ and $\Bord_{n-1}$. 

Throughout this section, we will fix an integer $n \geq 2$, a topological space $X$, and
a real vector bundle $\zeta$ of rank $n$ on $X$, equipped with an inner product (our discussion will apply also in the case $n=1$, but some of the notation needs to be modified). Let
$\widetilde{X} = \{ (x,f): x \in X, f: \R^n \simeq \zeta_x \}$ denote the bundle of orthonormal
frames of $\zeta$, so that $\widetilde{X}$ is a principal $\OO(n)$-bundle over $X$.
Let $X_0$ denote the unit sphere bundle
$\{ (x,v): x \in X, v \in \zeta_{x}, |v| = 1 \}$ of $\zeta$, and let
$\zeta_0 = \{ (x,v,w): (x,v) \in X_0, w \in \zeta_x, (v,w) = 0 \}$ denote the induced 
$(n-1)$-dimensional vector bundle on $X_0$. We observe that the bundle of orthonormal
frames of $\zeta_0$ can also be identified with $\widetilde{X}$, so we have a homeomorphism
$X_0 = \widetilde{X} / \OO(n-1)$.

Let $p: X_0 \rightarrow X$ denote the projection map. We have a canonical isomorphism of
vector bundles $\zeta_0 \oplus \underline{\R} \simeq p^{\ast} \zeta$. Moreover,
$(X_0, \zeta_0)$ is {\em universal} among vector bundles of rank $n-1$ with this property.
It follows that if $M$ is a manifold of dimension $< n$, then the data of an $(X_0, \zeta_0)$-structure on $M$ is equivalent to the data of an $(X, \zeta)$-structure on $M$.
We obtain a map of bordism categories
$$ i: \Bord_{n-1}^{(X_0, \zeta_0)} \rightarrow \Bord_{n}^{(X,\zeta)}.$$

\begin{remark}
Roughly speaking, we can think of $i$ as an inclusion functor: we have included the
$(\infty,n-1)$-category $\Bord_{n-1}^{(X_0, \zeta_0)}$ (in which $n$-morphisms are given by diffeomorphisms between $(n-1)$-manifolds) into a larger $(\infty,n)$-category $\Bord_{n}^{(X,\zeta)}$ (in which $n$-morphisms are given by bordisms between $(n-1)$-manifolds). It is tempting to assume that $\Bord_{n-1}^{(X_0, \zeta_0)}$ is obtained from $\Bord_{n}^{(X, \zeta)}$ by discarding the noninvertible $n$-morphisms. However, this is not always correct: for large values of $n$, the invertible $n$-morphisms in $\Bord_{n}^{(X, \zeta)}$ are given by
$h$-cobordisms between $(n-1)$-manifolds. Such an $h$-cobordism need not arise from a diffeomorphism between the underlying manifolds without assumptions of simple-connectivity.
\end{remark}

\begin{remark}
Strictly speaking, the map $i: \Bord_{n-1}^{(X_0, \zeta_0)} \rightarrow \Bord_{n}^{(X,\zeta)}$ depends
on a choice of isomorphism $\alpha: \zeta_0 \oplus \underline{\R} \simeq p^{\ast} \zeta$. Our
choice will be normalized by the following requirements:
\begin{itemize}
\item[$(i)$] The restriction of $\alpha$ to the factor $\zeta_0$ reduces to the canonical
inclusion of $\zeta_0 \simeq \{ (x,v,w): x \in X; v,w \in \zeta_x; (v,w) = 0 \}$ into
$p^{\ast} \zeta \simeq \{ (x,v,w): x \in X; v,w \in \zeta_x \}$.
\item[$(ii)$] The restriction of $\alpha$ to the factor $\underline{\R}$ is given by the global
section $(x,v) \mapsto (x,v,v)$ of $p^{\ast} \zeta$.
\end{itemize}
However, there is another canonical normalization, where $(ii)$ is replaced by the following:
\begin{itemize}
\item[$(ii')$] The restriction of $\alpha$ to the factor $\underline{\R}$ is given by the global
section $(x,v) \mapsto (x,v,-v)$ of $p^{\ast} \zeta$.
\end{itemize}
This choice determines a {\em different} functor $i': \Bord_{n-1}^{(X_0, \zeta_0)}
\rightarrow \Bord_{n}^{(X,\zeta)}$.
\end{remark}

Our goal in this section is to study the difference between $\Bord_{n}^{(X,\zeta)}$ and
$\Bord_{n-1}^{(X_0, \zeta_0)}$. More precisely, we wish to analyze the problem of extending
a symmetric monoidal functor $Z_0: \Bord_{n-1}^{(X_0, \zeta_0)} \rightarrow \calC$
to a symmetric monoidal functor $Z: \Bord_{n}^{(X,\zeta)} \rightarrow \calC$, where
$\calC$ is a symmetric monoidal $n$-category with duals. It turns out that extensions
of $Z_0$ are easy to classify: they are determined by the values of $Z$ on the class of
$n$-dimensional disks. To state this result more precisely, we need to introduce some terminology.

\begin{notation}
Let $\calC$ be a symmetric monoidal $(\infty,n)$-category, and let ${\bf 1}$ denote the unit object
of $\calC$. We let $\Omega \calC$ denote the symmetric monoidal $(\infty,n-1)$-category
$\OHom_{\calC}( {\bf 1}, {\bf 1})$. More generally, for each $k \leq n$, we let
$\Omega^{k} \calC$ denote the symmetric monoidal $(\infty,n-k)$-category $\Omega( \Omega^{k-1} \calC)$. We will refer to objects of $\Omega^{k} \calC$ as {\it closed $k$-morphisms} in $\calC$.
\end{notation}

\begin{example}
A closed $k$-morphism in $\Bord_{n}^{(X, \zeta)}$ is a closed $k$-manifold $M$ equipped with an $(X,\zeta)$-structure. In particular, for every point $x \in X$, the unit sphere $S^{\zeta_x} = \{ w \in \zeta_x : |v| = 1 \}$ comes equipped with a canonical $(X_0, \zeta_0)$-structure, and can therefore be regarded as a closed $(n-1)$-morphism in $\Bord_{n-1}^{(X_0, \zeta_0)}$. 
\end{example}

Suppose that
we are given a symmetric monoidal functor $Z_0: \Bord_{n-1}^{(X_0, \zeta_0)} \rightarrow \calC$.
The construction $x \mapsto Z_0(S^{\zeta_x})$ determines a functor from $X$ to $\Omega^{n-1} \calC$, which we will denote by $\ZS$. 

Given a point $\overline{x} = (x,v) \in X_0$, we obtain a decomposition
of the sphere $S^{\zeta_x}$ into upper and lower hemispheres
$$ S^{\zeta_{x}}_{v,+} = \{ w \in \zeta_x: (v,w) \geq 0, |w| = 1 \}
\quad \quad S^{\zeta_{x}}_{v,-} = \{ w \in \zeta_x: (v,w) \leq 0, |w| = 1 \},$$
so that $S^{\zeta_{x}} = S^{\zeta_{x}}_{v,+} \coprod_{ S^{\zeta_x}_{v,0}} S^{\zeta_x}_{v,-}$
where $S^{\zeta_x}_{v,0} = S^{\zeta_x}_{v,+} \cap S^{\zeta_x}_{v,-} = \{ w \in \zeta_x: (v,w) = 0,
|w| = 1 \}$. Consequently, the $(n-1)$-morphism $S^{\zeta_x}$ in $\Bord_{n-1}^{(X_0, \zeta_0)}$
can be written as the composition of a pair of morphisms 
$$ \emptyset \stackrel{ S^{\zeta_{x}}_{v,-}}{\longrightarrow} S^{\zeta_{x}}_{v,0}
\stackrel{ S^{\zeta_{x}}_{v,+}}{\longrightarrow} \emptyset.$$
Composing with $Z_0$, we obtain a pair of morphisms
$$ {\bf 1} \stackrel{ H_{-}( \overline{x})}{\longrightarrow} H_0( \overline{x})
\stackrel{ H_{+}( \overline{x} )}{\longrightarrow} {\bf 1}$$
in the $(\infty,2)$-category $\Omega^{n-2} \calC$, whose composition is 
$\ZS(x) \in \Omega^{n-1} \calC$.

\begin{definition}\label{nondeg}
Suppose given a symmetric monoidal functor $Z_0: \Bord_{n-1}^{(X_0, \zeta_0)} \rightarrow
\calC$, where $\calC$ is a symmetric monoidal $(\infty,n)$-category with duals.
Let $x$ be a point of $X$, and let $\overline{x} \in X_0$ be a lift of $x$. We will say that
an $2$-morphism
$\eta: {\bf 1} \rightarrow \ZS(x) = H_{+}( \overline{x}) \circ H_{-}( \overline{x} )$ in $\Omega^{n-2} \calC$ is {\it nondegenerate at $\overline{x}$} if $\eta$ exhibits $H_{+}( \overline{x})$ as a right adjoint
to $H_{-}( \overline{x})$.
\end{definition}

\begin{remark}
Let $Z_0: \Bord_{n-1}^{(X_0, \zeta_0)} \rightarrow
\calC$ and $x \in X$ be as in Definition \ref{nondeg}.
Our assumption $n \geq 2$ implies that the $(n-1)$-sphere $S^{ \zeta_x}$ is connected.
Consequently, if a map $\eta: {\bf 1} \rightarrow \ZS(x)$ is nondegenerate at
{\em any} point $\overline{x} \in S^{\zeta_x}$, then it is nondegenerate at
{\em every} point of $S^{\zeta_x}$. In this case, we will simply say that
$\eta$ is {\em nondegenerate at $x$}.

When $n=1$, these notions need to be slightly revised. In this case, the object
$\ZS(x) \in \calC$ factors as a tensor product $H_{+}( \overline{x}) \otimes H_{-}( \overline{x})$.
We will say that a $1$-morphism $\eta: {\bf 1} \rightarrow \ZS(x)$ in $\calC$ is
{\em nondegenerate} if it exhibits $H_{+}( \overline{x} )$ as a dual of $H_{-}(\overline{x})$. 
The sphere $S^{ \zeta_x}$ is disconnected in this case. However, the condition that
$\eta$ be nondegenerate is still independent of the choice of $\overline{x}$, since
a map ${\bf 1} \rightarrow X \otimes Y$ exhibits $X$ as a dual of $Y$ if and only if it exhibits
$Y$ as a dual of $X$.
\end{remark}

\begin{example}\label{higtt}
Let $\calC$ be a symmetric monoidal $(\infty,n)$-category with duals, let $Z_0: \Bord_{n-1}^{(X_0, \zeta_0)} \rightarrow \calC$ be a symmetric monoidal functor, and suppose that $Z_0$ can be extended to a symmetric monoidal functor $Z: \Bord_{n}^{(X, \zeta)} \rightarrow \calC$.
For each $x \in X$, we can regard the unit disk $D^{\zeta_x} = \{ v \in \zeta_x: |v| \leq 1 \}$ as
a bordism from the empty manifold to $S^{\zeta_x}$; it therefore defines an $n$-morphism
$D^{\zeta_x}: \emptyset \rightarrow S^{\zeta_x}$ in $\Bord_{n}^{(X,\zeta)}$. Applying the functor
$Z$, we obtain a nondegenerate $n$-morphism $\eta_{x} = Z( D^{\zeta_x}): {\bf 1}
\rightarrow \ZS(x)$. 
\end{example}

Example \ref{higtt} admits the following converse, which is the basis of our inductive approach to the cobordism hypothesis:

\begin{theorem}[Cobordism Hypothesis, Inductive Formulation]\label{swisher6}
Let $\calC$ be a symmetric monoidal $(\infty,n)$-category with duals, and let 
$Z_0: \Bord_{n-1}^{(X_0, \zeta_0)} \rightarrow \calC$ be a symmetric monoidal functor.
The following types of data are equivalent:
\begin{itemize}
\item[$(1)$] Symmetric monoidal functors $Z: \Bord_{n}^{(X, \zeta)} \rightarrow \calC$
extending $Z_0$.
\item[$(2)$] Families of nondegenerate $n$-morphisms
$\eta_x: {\bf 1} \rightarrow Z_0(S^{\zeta_x})$ in $\calC$, parametrized
by $x \in X$.
\end{itemize}
The equivalence is given by assigning a symmetric monoidal functor
$Z: \Bord_{n}^{(X, \zeta)} \rightarrow \calC$ the collection of nondegenerate
$n$-morphisms $\{ \eta_x = Z(D^{\zeta_x}) \}_{x \in X}$ of Example \ref{higtt}.
\end{theorem}

\begin{remark}
Theorem \ref{swisher6} can be stated a bit more simply in the case where the space
$X$ is path connected. In this case, we can assume without loss of generality that
$X$ is a classifying space $BG$, where $G$ is a topological group equipped with a continuous homomorphism $\chi: G \rightarrow \OO(n)$. Assume further that $\chi$ is a fibration, and let $G_0$
denote the subgroup $G \times_{\OO(n)} \OO(n-1) \subseteq G$ so that we can identify
$X_0$ with the classifying space $BG_0$. If $\calC$ is a symmetric monoidal $(\infty,n)$-category with duals, then Theorem \ref{swisher6} asserts that symmetric monoidal functors $Z: \Bord_{n}^{G} \rightarrow \calC$ are determined by two pieces of data:
\begin{itemize}
\item[$(i)$] The restriction $Z_0 = Z | \Bord_{n-1}^{G_0}$. In this case, we can evaluate
$Z_0$ on the $n$-sphere $S^{n-1}$ to obtain a closed $(n-1)$-morphism
$Z_0( S^{n-1} )$ of $\calC$. The orthogonal group acts by diffeomorphisms
on $S^{n-1}$, and the resulting action of $G$ on $S^{n-1}$ is compatible with the
$G$-structure on $S^{n-1}$ determined by the stable framing $T_{S^{n-1}} \oplus \underline{ \R} 
\simeq \underline{ \R}^n$. Consequently, the topological group $G$ acts on the object
$Z_0(S^{n-1}) \in \Omega^{n-1} \calC$.

\item[$(ii)$] A $G$-equivariant $n$-morphism $\eta \in \bHom_{\Omega^{n-1} \calC}( {\bf 1}, Z_0(S^{n-1}) )$ which satisfies the nondegeneracy condition described in Definition \ref{nondeg}.
This $n$-morphism is given by evaluating $Z$ on the the $n$-disk $D^{n} = \{ v \in \R^n: |v| \leq 1 \}$.
\end{itemize}
\end{remark}

We will outline the proof of Theorem \ref{swisher6} later in this paper.
Our goal for the remainder of this section is to explain how Theorem \ref{swisher6}
can be used to prove earlier incarnations of the cobordism hypothesis (Theorems
\ref{swisher3} and \ref{swisher5}). We will prove these results by a simultaneous induction on $n$.

\begin{remark}
It is necessary to discuss Theorems \ref{swisher3} and \ref{swisher5} individually, because
Theorem \ref{swisher5} cannot even be formulated without assuming Theorem \ref{swisher3}
(we need some form of the cobordism hypothesis to define the action of $\OO(n)$ on
the underlying
$\infty$-groupoid of a symmetric monoidal $(\infty,n)$-category with duals).
\end{remark}

\begin{proof}[Proof of Theorem \ref{swisher3}]
Our concern in this case is the framed bordism $(\infty,n)$-category $\Bord_{n}^{\fr}$, which
coincides with $\Bord_{n}^{(X, \zeta)}$ in the case where $X$ is a single point. Let
$\{ v_1, v_2, \ldots, v_n \}$ be an orthonormal basis for the vector space $\zeta$.
The choice of such a basis determines an identification of $\widetilde{X}$ with the orthogonal
group $\OO(n)$, and of $X_0$ with the standard $(n-1)$-sphere $S^{n-1} = \{ v \in \R^n: |v| = 1 \}$. 
Let $\calC$ be a symmetric monoidal $(\infty,n)$-category with duals.
We wish to prove that the groupoid $\Fun^{\otimes}( \Bord_{n}^{(X, \zeta)}, \calC)$ is
equivalent to $\calC^{\sim}$, the equivalence being implemented by the functor
$Z \mapsto Z(\ast)$. In proving this, we will assume that Theorem \ref{swisher6} holds in dimensions
$\leq n$, and that Theorems \ref{swisher3} and \ref{swisher5} hold in dimension $< n$. We now apply these assumptions as follows:

\begin{itemize}
\item[$(1)$] Applying Theorem \ref{swisher6}, we deduce that giving a symmetric monoidal
functor $Z: \Bord_{n}^{(X, \zeta)} \rightarrow \calC$ is equivalent to giving the following data:
\begin{itemize}
\item[$(a_0)$] A symmetric monoidal functor $Z_0: \Bord_{n-1}^{(X_0, \zeta_0)} \rightarrow \calC$.
\item[$(b_0)$] A nondegenerate $n$-morphism $\eta: {\bf 1} \rightarrow Z_0( S^{n-1} )$. 
\end{itemize} 
\item[$(2)$] Applying Theorems \ref{swisher3} and \ref{swisher5} in dimension $(n-1)$, we deduce that the $\infty$-groupoid $\calC^{\sim}$ carries an action of the orthogonal group $\OO(n-1)$. Moreover,
a symmetric monoidal functor $Z_0$ as in $(a_0)$ above is equivalent to the following data:
\begin{itemize}
\item[$(a_1)$] An $\OO(n-1)$-equivariant map $\OO(n) \rightarrow \calC^{\sim}$; here
$\OO(n-1) = \{ g \in \OO(n): g v_1 = v_1 \} \subseteq \OO(n)$, acting on $\OO(n)$ by left translations.
\end{itemize}
\item[$(3)$] Let $\gamma: [0,1] \rightarrow \OO(n)$ and $\epsilon \in \OO(n-1)$
be given by the formulas
$$ \gamma_{t} v_i = \begin{cases} \cos(\pi t) v_1 + \sin(\pi t) v_2 & \text{if } i = 1 \\
- \sin( \pi t) v_1 + \cos( \pi t) v_2 & \text{if } i = 2\\
v_i & \text{if } i > 2. \end{cases}$$
$$ \epsilon v_i = \begin{cases} - v_2 & \text{if } i = 2 \\
v_i & \text{otherwise.} \end{cases}$$
Consider the map
$\widetilde{q}: \OO(n-1) \times [0,1] \times \OO(n-1)$ defined by the formula
$\widetilde{q}(g,t,g') = g \gamma_t g'$. Let 
$K = \OO(n-1) \times \OO(n-1)$, and let $\OO(n-2) = \{ g \in \OO(n): gv_1 = v_1, gv_2 = v_2 \} \subseteq \OO(n)$ act on $K$ by the formula $h(g,g')=(gh^{-1}, hg')$. 
We observe that $\widetilde{q}$ determines an $\OO(n-2)$-equivariant map
$K \times [0,1] \rightarrow \OO(n)$ (where $\OO(n-2)$ acts trivially on $\OO(n)$), which induces a homeomorphism
$$ \OO(n-1) \coprod_{ K \times \{0\} }
(K \times [0,1]) \coprod_{ K \times \{1\} }
(\OO(n-1) \gamma_1) \simeq
 \OO(n).$$
It follows that $(a_1)$ is equivalent to the following data:
\begin{itemize}
\item[$(a_2)$] A pair of $\OO(n-1)$-equivariant maps $e_{-}, e_{+}: \OO(n-1) \rightarrow \calC^{\sim}$, together with an $\OO(n-2)$-equivariant homotopy from $e_{-}$ to the map
$g \mapsto e_{+}( \epsilon g \epsilon^{-1} )$. 
\end{itemize}

\item[$(4)$] Let $Y$ be a single point, and let $\zeta'$ be the vector bundle on
$Y$ with orthonormal basis $(v_2, v_3, \ldots, v_n)$. Let $Y_0 \simeq S^{n-2}$ denote the unit sphere bundle of $\zeta'$, and let $\zeta'_0$ be the tangent bundle of $Y_0$. Let $\calC^{op}$ denote
the $(\infty,n)$-category obtained from $\calC$ by taking the opposite category at the level
of $(n-1)$-morphisms. Applying Theorem \ref{swisher5} in dimensions $n-1$ and $n-2$, we deduce that $(a_2)$ is equivalent to the following data:
\begin{itemize}
\item[$(a_3)$] A pair of symmetric monoidal functors 
$$Z_{-}: \Bord^{(Y, \zeta')}_{n-1} \rightarrow \calC \quad \quad Z_{+}: \Bord^{(Y, \zeta')}_{n-1} \rightarrow \calC^{op}$$
together with an isomorphism an isomorphism $Z_{-} | \Bord^{(Y_0, \zeta'_0)}_{n-2} \simeq
Z_{+} | \Bord^{(Y_0, \zeta'_0)}_{n-2}$ (this data makes sense, since the underlying
$(\infty,n-2)$-categories of $\calC$ and $\calC^{op}$ are canonically equivalent). 
\end{itemize}

\item[$(5)$] Applying Theorem \ref{swisher6} in dimension $n-1$, we deduce that
$(a_3)$ is equivalent to the following data:
\begin{itemize}
\item[$(a_4)$] A functor $Z': \Bord^{(Y_0, \zeta'_0)}_{n-2} \rightarrow \calC$, together with
a pair of nondegenerate $(n-1)$-morphisms
$$ f: {\bf 1} \rightarrow Z'(S^{n-2}) \quad \quad g: {\bf 1} \rightarrow Z'( S^{n-2})$$
in $\calC$ and $\calC^{op}$, respectively.
\end{itemize}

\item[$(6)$] Suppose we are given the data of $(a_4)$. We can regard
$g$ as an $(n-1)$-morphism from $Z'(S^{n-2})$ to ${\bf 1}$ in the original
$(\infty,n)$-category $\calC$. Unwinding the definitions, we see that 
$Z_0(S^{n-1})$ is given by the composition $g \circ f$, and that an $n$-morphism
$\eta: {\bf 1} \rightarrow Z_0(S^{n-1})$ is nondegenerate if and only if it exhibits
$g$ as a right adjoint to $f$. Consequently, $(b_0)$ is equivalent to the following:
\begin{itemize}
\item[$(b_1)$] An $n$-morphism $\eta: \id_{\bf 1} \rightarrow g \circ f$ in $\calC$ which
exhibits $g$ as a right adjoint to $f$. 
\end{itemize}

\item[$(7)$] Since $\calC$ admits adjoints, there is an equivalence between the
underlying $(\infty,n-1)$-categories of $\calC$ and $\calC^{op}$, which is the identity
on $k$-morphisms for $k < n-1$ and carries each $(n-1)$-morphism of $\calC$ to its right adjoint.
Consequently, if $f$ and $g$ are adjoint $(n-1)$-morphisms as in $(a_4)$, then
$f$ is nondegenerate if and only if $g$ is nondegenerate. It follows that the data
of $(a_4)$ and $(b_1)$ together is equivalent to the following:
\begin{itemize}
\item[$(c_0)$] A symmetric monoidal functor $Z': \Bord^{(Y_0, \zeta'_0)}_{n-2} \rightarrow \calC$
together with a pair of $(n-1)$-morphisms
$$ f: { \bf 1} \rightarrow Z'(S^{n-2}) \quad \quad g: Z'( S^{n-2}) \rightarrow {\bf 1}$$
such that $f$ is nondegenerate, and an $n$-morphism $\eta: \id_{\bf 1} \rightarrow g \circ f$
which exhibits $g$ as a right adjoint to $f$.
\end{itemize}

\item[$(8)$] Let $f$ be any $(n-1)$-morphism in an $(\infty,n)$-category. If $f$ admits
a right adjoint $f^{R}$, then $f^{R}$ is determined up to canonical isomorphism.
Moreover, giving another $(n-1)$-morphism $g$ together with a map
$\eta: \id \rightarrow g \circ f$ which exhibits $g$ as a right adjoint to $f$ is equivalent
to giving an isomorphism $g \simeq f^{R}$. Consequently, we may neglect the
data of $g$ and $\eta$ and we arrive at the following reformulation of $(c_0)$:
\begin{itemize}
\item[$(c_1)$] A symmetric monoidal functor $Z': \Bord^{(Y_0, \zeta'_0)}_{n-2} \rightarrow \calC$
together with a nondegenerate $(n-1)$-morphism $ f: { \bf 1} \rightarrow Z'(S^{n-2})$.
\end{itemize}

\item[$(9)$] Invoking Theorem \ref{swisher6} again (in dimension $n-1$), we deduce
that $(c_1)$ is equivalent to the following:
\begin{itemize}
\item[$(c_2)$] A symmetric monoidal functor $Z_{-}: \Bord^{(Y, \zeta')}_{n-1} \rightarrow \calC$.
\end{itemize}
We may now invoke the cobordism hypothesis (Theorem \ref{swisher3}) in dimension
$(n-1)$ to conclude that $(c_2)$ is equivalent to the data of a single object of $\calC$,
as desired.
\end{itemize}
\end{proof}

We now prove Theorem \ref{swisher5}; the basic idea is to use Theorem \ref{swisher6}
to justify the excision principle described in Remark \ref{exor}.

\begin{proof}[Proof of Theorem \ref{swisher5}]
Let $\calC$ be a symmetric monoidal $(\infty,n)$-category with duals and $X$ a CW complex. It follows from Theorem \ref{swisher3} that $\calC^{\sim}$ carries an action of the orthogonal group $\OO(n)$, and we have a canonical map of $\infty$-groupoids
$\alpha: \Fun^{\otimes}( \Bord_{n}^{(X, \zeta)}, \calC) \rightarrow \bHom_{\OO(n)}( \widetilde{X}, \calC^{\sim})$. 
We wish to prove that this map is an equivalence. More generally, consider any continuous
map $f: Y \rightarrow X$, where $Y$ is a CW complex. Set $F(Y) = \Fun^{\otimes}( \Bord_{n}^{(Y, f^\ast \zeta)}, \calC)$ and $G(Y) = \bHom_{ \OO(n)}( \widetilde{X} \times_{X} Y, \calC^{\sim})$. We have a canonical map $\alpha_Y: F(Y) \rightarrow G(Y)$, which depends functorially on $Y$. Let
$S$ denote the collection of all CW complexes $Y$ for which $\alpha_Y$ is an equivalence,
for any map $f: Y \rightarrow X$.

The functor $Y \mapsto G(Y)$ carries homotopy colimits in $Y$ to homotopy limits
of $\infty$-groupoids. Theorem \ref{swisher6} implies that the functor $Y \mapsto F(Y)$ has the same property. It follows that the collection of spaces $S$ is closed under the formation of homotopy colimits.
Theorem \ref{swisher3} implies that $\alpha_Y$ is an equivalence when $Y$ consists of a single point, so that $\ast \in S$. Since every CW complex $Y$ can be obtained as a homotopy colimit of points,
we deduce that $S$ contains every CW complex. In particular, taking $Y = X$ and $f$ to be the identity map, we deduce that $\alpha$ is an equivalence as desired.
\end{proof}


\subsection{Reduction to the Unoriented Case}\label{slabun}

Theorem \ref{swisher4} asserts that for any continuous homomorphism of topological groups
$\chi: G \rightarrow \OO(n)$, the bordism $(\infty,n)$-category $\Bord^{G}_{n}$ of manifolds with structure group $G$ has a certain universal property. In the special case where the group $G$ is trivial,
we recover Theorem \ref{swisher3}, which describes the framed bordism $(\infty,n)$-category
$\Bord^{\fr}_{n}$ as the free symmetric monoidal $(\infty,n)$-category with duals generated by a single object. This special case is in some sense fundamental: it allows us to define an action
of the group $\OO(n)$ on the classifying space of objects for an arbitrary symmetric monoidal $(\infty,n)$-category with duals, without which we cannot even formulate the more general Theorem \ref{swisher4} (at least directly).
However, there is another special case of interest, when the map $\chi: G \rightarrow \OO(n)$ is homeomorphism. In this case, the $(\infty,n)$-category $\Bord^{G}_{n}$ can be identified with the
{\em unoriented} bordism $(\infty,n)$-category $\Bord_{n}$. While $\Bord^{\fr}_{n}$ is in some sense the ``smallest'' bordism category, $\Bord_{n}$ can be thought of as the ``largest'': for every
continuous homomorphism $\chi: G \rightarrow \OO(n)$, we have a forgetful functor $\alpha: \Bord^{G}_{n} \rightarrow \Bord_{n}$. Our goal in this section is to explain how to use this forgetful functor to deduce the cobordism hypothesis for $G$-manifolds to the cobordism hypothesis for unoriented manifolds.

Let us first outline the basic idea. Suppose we are given a symmetric monoidal functor
$Z: \Bord_{n}^{G} \rightarrow \calC$, and let $M$ be an object of $\Bord_{n}$ (that is, a finite collection of points). Let $X$ denote a classifying space for $G$-structures on $M$. For every $x \in X$,
we can apply the functor $Z$ to the pair $(M,x)$ to obtain an object $Z(M,x) \in \calC$.
We can regard the collection $\{ Z(M,x) \}_{x \in X}$ as a {\it local system} on the space $X$, taking
values in the $(\infty,n)$-category $\calC$. The collection of such local systems can itself
be regarded as an $(\infty,n)$-category which we will denote by $\Fam_n(\calC)$. 
We can regard the assignment $M \mapsto \{ Z(M,x) \}_{x \in X}$ as determining a symmetric monoidal functor $Z': \Bord_{n} \rightarrow \Fam_n(\calC)$. The functor $Z$ can be recovered from $Z'$ by extracting the fibers of the relevant local systems. This construction therefore gives a mechanism for describing the classification of symmetric monoidal functors $Z: \Bord_{n}^{G} \rightarrow \calC$ in terms of symmetric monoidal functors $Z': \Bord_{n} \rightarrow \Fam_n(\calC)$, which will allow us to reduce the proof of Theorem \ref{swisher6} to the unoriented case.

We begin by sketching the definition of the $(\infty,n)$-category $\Fam_n(\calC)$ in more detail, where
$\calC$ is an $(\infty,n)$-category. The construction uses induction on $n$.
\begin{itemize}
\item[$(1)$] The objects of $\Fam_n(\calC)$ are pairs
$(X,f)$, where $X$ is a topological space and $f$ is a functor from $X$ (regarded as an
$\infty$-groupoid) into $\calC$; we can think of $f$ as a {\em local system} on $X$ with values in $\calC$.
\item[$(2)$] If $n =0$, then a morphism from $(X,f)$ to $(X', f')$ in $\Fam_n(\calC)$ is a weak homotopy
equivalence $g: X \rightarrow X'$ and an equivalence of functors between $f' \circ g$ and $f$.
(Strictly speaking, this definition is only correct if $X$ is a CW complex; otherwise we may need to
adjoin additional morphisms corresponding to diagrams of weak homotopy equivalences
$X \leftarrow \widetilde{X} \rightarrow X'$; we will ignore this technical point.)

\item[$(3)$] Suppose that $n > 0$, and let $(X,f)$ and $(X',f')$ be objects of $\Fam_{n}(\calC)$. We then
have a local system $\calF$ of $(\infty,n-1)$-categories on $X \times X'$, given by the formula
$\calF_{x,x'} = \Fam_{n-1} \OHom_{\calC}( f(x), f'(x') )$. We let $\OHom_{\Fam_{n}(\calC)}( (X,f), (X',f') )$ denote the $(\infty,n-1)$-category of global sections of $\calF$.
\end{itemize}

\begin{notation}
Let $\ast$ denote the trivial $(\infty,n)$-category having a single object. We will denote
the $(\infty,n)$-category $\Fam_n(\ast)$ simply by $\Fam_n$. 
\end{notation}

\begin{example}\label{singal}
The $\infty$-category $\Fam_1$ can be described as follows:
\begin{itemize}
\item[$(a)$] The objects of $\Fam_1$ are topological spaces $X$.  
\item[$(b)$] Given a pair of topological spaces $X$ and $Y$, a morphism from $X$ to $Y$ in
$\Fam_1$ is another topological space $C$ equipped with a continuous map $C \rightarrow X \times Y$.
\item[$(c)$] The composition of a $1$-morphism $C: X \rightarrow Y$ and another one morphism
$C': Y \rightarrow Z$ is given by the homotopy fiber product $C \times^{R}_{Y} C'$, which is equipped with a canonical map $C \times_{Y} C' \rightarrow X \times Z$.
\end{itemize}
In other words, $\Fam_1$ is the $(\infty,1)$-category whose objects are topological spaces and whose morphisms are {\em correspondences} between them.
\end{example}

\begin{remark}\label{spitter}
Let $\calC$ be an $(\infty,n)$-category equipped with a symmetric monoidal structure.
Then the $\Fam_n(\calC)$ inherits a symmetric monoidal structure, given on objects by the formula
$(X,f) \otimes (Y,g) = (X \times Y, h)$ where $h(x,y) = f(x) \otimes g(y) \in \calC$.
One can show that if $\calC$ has duals, then $\Fam_n(\calC)$ also has duals. In particular,
for each $n > 0$, the $\infty$-category $\Fam_n$ has the structure of a symmetric monoidal
$(\infty,n)$-category with duals.

For example, suppose that $n = 1$, so that $\Fam_1$ can be identified with the $(\infty,1)$-category
of topological spaces and correspondences between them, with a symmetric monoidal structure given by the Cartesian product of spaces. Every object $X \in \Fam_1$ is dualizable:
the diagonal map $X \rightarrow X \times X$ determines correspondences
$$ \ev_{X}: {\bf 1} \rightarrow X \times X $$
$$ \coev_{X}: X \times X \rightarrow {\bf 1},$$
which exhibit $X$ as its own dual.
\end{remark}

Since $\Fam_n$ is a symmetric monoidal $(\infty,n)$-category with duals, the cobordism
hypothesis predicts that the underlying $\infty$-groupoid $\Fam_n^{\sim}$ carries an
action of the orthogonal group $\OO(n)$ (see Corollary \ref{ail}). In the case $n=1$, this action
should carry each object of $\Fam_1$ to its dual. As we saw in Remark \ref{spitter}, every object
of $\Fam_1$ is canonically self-dual, so the action of $\OO(1)$ on $\Fam_1^{\sim}$ is trivial.
The analogous statement is true for every $n$: the $\OO(n)$ action on $\Fam_n^{\sim}$ is trivial.
Since the objects of $\Fam_n^{\sim}$ are topological spaces, we can identify the objects of
$(\Fam_n^{\sim})^{h \OO(n)}$ with $\OO(n)$-equivariant spaces: that is, topological spaces
$\widetilde{X}$ equipped with a continuous action of the group $\OO(n)$. Any such 
CW complex $\widetilde{X}$ is (equivariantly) weakly homotopy equivalent to a space with
a free action of $\OO(n)$, which determines a space $X = \widetilde{X} / \OO(n)$ and
a vector bundle $\zeta = ( \widetilde{X} \times \R^n) / \OO(n)$ over $X$. Combining
this analysis with Theorem \ref{swisher4}, we arrive at the following prediction:

\begin{claim}\label{spiller}
For each $n \geq 0$, the following data are equivalent:
\begin{itemize}
\item[$(1)$] Symmetric monoidal functors $Z: \Bord_{n} \rightarrow \Fam_n$.
\item[$(2)$] Pairs $(X,\zeta)$, where $X$ is a topological space and $\zeta$ is an
$n$-dimensional vector bundle on $X$ with an inner product.
\end{itemize}
\end{claim}

We can justify Claim \ref{spiller} by invoking Theorem \ref{swisher4} in the case where the
structure group $G$ is the entire orthogonal group $\OO(n)$. However, the reasoning involved would be somewhat circular, since we first need to prove Theorem \ref{swisher4} in the framed case to justify the existence of an $\OO(n)$-action on $\Fam_{n}^{\sim}$ (and would then need to argue further that this $\OO(n)$ action was trivial). However, we will not need the full strength of Claim \ref{spiller} in the arguments which follow. We will only need to know the ``hard'' direction: namely, that we can associate to every pair $(X, \zeta)$ a tensor functor $Z_{(X, \zeta)}: \Bord_{n} \rightarrow \Fam_n$. This we can produce by direct construction. Indeed, $Z_{(X, \zeta)}$ can be defined as the functor which associates to every $k$-morphism $M$ in $\Bord_{n}$ (given by a $k$-manifold with corners) a $k$-morphism $Z_{(X,\zeta)}(M)$ in $\Bord_{n}$ (given by a topological space):
namely, we let $Z_{(X,\zeta)}(M)$ be a classifying space for $(X,\zeta)$-structures on $M$
(see Notation \ref{cusper1}). In other words, $Z_{(X,\zeta)}(M)$ can be identified with the collection of pairs $\{ f: M \rightarrow X, \alpha: f^{\ast} \zeta \simeq T_X \oplus \underline{ \R}^{n-k} \}$, topologized in a natural way.

There is a canonical way to recover the $(\infty,n)$-category
$\Bord^{(X, \zeta)}_{n}$ from the functor $Z_{(X, \zeta)}: \Bord_{n} \rightarrow \Fam_n$.
To make this precise, we need to introduce another bit of notation.

\begin{variant}\label{famvar}
We can describe $\Fam_n$ informally as follows:
\begin{itemize}
\item The objects of $\Fam_n$ are topological spaces.
\item The morphisms of $\Fam_n$ are correspondences between topological spaces.
\item The $2$-morphisms of $\Fam_n$ are correspondences between correspondences.
\item \ldots 
\item The $n$-morphisms of $\Fam_n$ are correspondences between
correspondences between $\ldots$
\item The $(n+1)$-morphisms of $\Fam_n$ are homotopy equivalences of correspondences.
\item The $(n+2)$-morphisms of $\Fam_n$ are homotopies between homotopy equivalences.
\item \ldots
\end{itemize}
We can obtain a new $(\infty,n)$-category by requiring all of the topological spaces in
the above description to be equipped with base points; we will denote this $(\infty,n)$-category
by $\Fam_{n}^{\ast}$. More generally, for any $(\infty,n)$-category $\calC$, we let
$\Fam_n^{\ast}(\calC)$ denote the homotopy fiber product
$\Fam_n(\calC) \times_{ \Fam_n}^{R} \Fam_n^{\ast}$. We can think of objects of
$\Fam_n^{\ast}(\calC)$ as triples $(X,f,x)$, where $X$ is a topological space, 
$f$ is a local system on $X$ with values in $\calC$, and $x$ is a point of $X$. Note that
there is a canonical evaluation functor $\chi: \Fam_n^{\ast}(\calC) \rightarrow \calC$, given
by the formula $\chi(X,f,x) = f(x) \in \calC$.
\end{variant}

\begin{proposition}\label{coalmine}
Let $X$ be a topological space, and let $\zeta$ be an $n$-dimensional vector bundle on $X$, equipped with an inner product. Then there is a homotopy pullback diagram of $($symmetric monoidal$)$
$(\infty,n)$-categories
$$ \xymatrix{ \Bord^{(X,\zeta)}_{n} \ar[r] \ar[d] & \Fam_n^{\ast} \ar[d] \\
\Bord_{n} \ar[r]^{ Z_{(X,\zeta)}} & \Fam_n. }$$
In other words, $\Bord^{(X,\zeta)}_{n}$ can be described as the homotopy fiber product
$\Bord_{n} \times_{ \Fam_n}^{R} \Fam_{n}^{\ast}.$
\end{proposition}

The proof of Proposition \ref{coalmine} is a simple unwinding of definitions.
A $k$-morphism in the homotopy fiber product $\Bord_{n} \times_{ \Fam_n}^{R} \Fam_{n}^{\ast}$
can be identified with a $k$-morphism $M$ of $\Bord_{n}$ together with a point of
$Z_{(X, \zeta)}(M)$: by definition, such a point provides an $(X, \zeta)$-structure on $M$, which allows us to view $M$ as a $k$-morphism in $\Bord_{n}^{(X, \zeta)}$. 

Our next goal is to use Proposition \ref{coalmine} to understand symmetric monoidal
functors from $\Bord^{(X,\zeta)}_{n}$ into another $(\infty,n)$-category $\calC$. Our main
tool is the following general principle:

\begin{proposition}\label{campwise}
Let $\calB$ and $\calC$ be symmetric monoidal $(\infty,n)$-categories, let
$Z: \calB \rightarrow \Fam_n$ be a symmetric monoidal functor, and let
$\calB^{\ast}$ denote the homotopy fiber product $\calB \times_{ \Fam_n}^{R} \Fam_n^{\ast}$.
The following types of data are equivalent:
\begin{itemize}
\item[$(1)$] Symmetric monoidal functors $\overline{Z}: \calB \rightarrow \Fam_n(\calC)$
lifting $Z$.
\item[$(2)$] Symmetric monoidal functors $\overline{Z}': \calB^{\ast} \rightarrow \calC$.
\end{itemize}
The equivalence is implemented by carrying a symmetric monoidal functor
$\overline{Z}$ to the composition
$$ \calB^{\ast} \simeq \calB \times_{ \Fam_n}^{R} \Fam_n^{\ast}
\rightarrow \Fam_n(\calC) \times_{\Fam_n}^{R} \Fam_n^{\ast} \simeq
\Fam_n^{\ast}(\calC) \rightarrow \calC$$
where the last functor is described in Variant \ref{famvar}.
\end{proposition}

The proof of this result again amounts to carefully unwinding the definitions. We must show that the values of a functor $\overline{Z}: \calB \rightarrow \Fam_n(\calC)$ can be recovered
from those of $\overline{Z}': \calB^{\ast} \rightarrow \calC$. Let $M$ be an object of $\calB$, and let
$\overline{Z}(M) = (X, f)$, where $X = Z(M)$ is a topological space and $f$ is a local system on $X$
with values in $\calC$. For every point $x \in X$, the pair $(M, x)$ determines an object of
$\calB^{\ast}$, and we have a canonical isomorphism $f(x) \simeq \overline{Z}'(M,x)$ in
the $(\infty,n)$-category $\calC$.

Combining Propositions \ref{campwise} and \ref{coalmine} in the special case where
$Z$ is the functor $Z_{(X,\zeta)}: \Bord_{n} \rightarrow \Fam_n$, we obtain the following result:

\begin{proposition}\label{cabil}
Let $\calC$ be a symmetric monoidal $(\infty,n)$-category, $X$ a topological space, and
$\zeta$ a real vector bundle of rank $n$ on $X$ with inner product. The following data are equivalent:

\begin{itemize}
\item[$(1)$] Symmetric monoidal functors $\overline{Z}: \Bord_{n} \rightarrow \Fam_n(\calC)$
lifting the functor $Z_{(X, \zeta)}: \Bord_{n} \rightarrow \Fam_n$.

\item[$(2)$] Symmetric monoidal functors $\overline{Z}': \Bord_{n}^{(X, \zeta)} \rightarrow \calC$.
\end{itemize}
\end{proposition}

We are now ready to sketch an argument reducing the proof of Theorem \ref{swisher6} to the
unoriented case. Fix a topological space $X$, a rank $n$ vector bundle
$\zeta$ on $X$ equipped with an inner product, and a symmetric monoidal
$(\infty,n)$-category $\calC$ with duals. We wish to classify symmetric monoidal
functors $\overline{Z}': \Bord_{n}^{(X, \zeta)} \rightarrow \calC$. We argue in several steps:

\begin{itemize}
\item[$(a)$] According to Proposition \ref{cabil}, the data of a symmetric monoidal functor
$\overline{Z}': \Bord_{n}^{(X,\zeta)} \rightarrow \calC$ is equivalent to the data of a symmetric
monoidal functor $\overline{Z}: \Bord_{n} \rightarrow \Fam_{n}(\calC)$ lifting
the functor $Z_{(X, \zeta)}: \Bord_{n} \rightarrow \Fam_n$.

\item[$(b)$] Let $X_0$ be the set of pairs $(x, v)$, where $x \in X$ and
$v \in \zeta_x$ is a vector unit length, and let $\zeta_0 = \{ (x,v,w): (x,v) \in X_0,
w \in \zeta_x, (v,w) = 0 \}$ be the induced vector bundle on $X_0$. We observe that the restriction of
$Z_{(X,\zeta)}$ to $\Bord_{n-1}$ can be identified with the composition
$\Bord_{n-1} \stackrel{Z_{(X_0, \zeta_0)}}{\rightarrow} \Fam_{n-1} \rightarrow \Fam_n$. 

Let $\EO(n)$ denote a contractible space with a free action of the orthogonal
group $\OO(n)$, let $\BO(n)$ denote the classifying space $\EO(n)/\OO(n)$, and let
$\zeta'$ denote the tautological vector bundle on $\BO(n)$. For each point
$y \in \BO(n)$, the nondegenerate $n$-morphism
$\eta_y = Z_{(X,\zeta)}( D^{\zeta'_y} )$ in $\Fam_n$ can be identified with the space of all
$(X,\zeta)$-structures on the $n$-dimensional ball $D^{\zeta'_y}$. Since $D^{\zeta'_{y}}$
is contractible, this can be identified with the space of pairs $(x, \alpha)$, where
$x \in X$ and $\alpha: \zeta_x \simeq \zeta'_y$ is an isometry. Applying Theorem \ref{swisher6} in the unoriented case, we conclude that giving a symmetric monoidal functor $\overline{Z}: \Bord_{n} \rightarrow \Fam_n(\calC)$ is equivalent to giving the following data:
\begin{itemize}
\item[$(i)$] A symmetric monoidal functor $\overline{Z}_0: \Bord_{n-1} \rightarrow \Fam_{n}(\calC)$
lifting $Z_{(X_0, \zeta_0)}$. 
\item[$(ii)$] For each $y \in \BO(n)$, a nondegenerate $n$-morphism
$\overline{\eta}_y: {\bf 1} \rightarrow \overline{Z}_0( S^{\zeta'_y})$ in 
$\Fam_n(\calC)$ lying over $\eta_y$.
\end{itemize}

\item[$(c)$] Invoking Proposition \ref{cabil} in dimension $(n-1)$ (and the fact that
$\Fam_{n-1}(\calC)$ and $\Fam_{n}(\calC)$ have the same underlying $(\infty,n-1)$-category),
we can replace $(i)$ by the following data:
\begin{itemize}
\item[$(i')$] A symmetric monoidal functor $\overline{Z}'_0: \Bord_{n-1}^{(X_0, \zeta_0)}
\rightarrow \calC$.
\end{itemize}

\item[$(d)$] Let $X'$ be the collection of triples $(x,y,\alpha)$, where
$x \in X$, $y \in \BO(n)$, and $\alpha: \zeta_x \simeq \zeta'_y$ is an isometry,
and let $\phi: X' \rightarrow X$ denote the projection map. For each
$y \in \BO(n)$, let $X'_y = X' \times_{\BO(n)} \{y\}$ denote the fiber of
$X'$ over the point $y$. As we saw above, $X'_y$ is homotopy equivalent to the topological
space $\eta_y = Z_{(X,\zeta)}( D^{\zeta'_y})$ (viewed as an $n$-morphism in $\Fam_n$). 
Unwinding the definitions, we see that lifting $\eta_y$ to a nondegenerate $n$-morphism
in $\overline{\eta}_y \in \Fam_n(\calC)$ is equivalent to giving a family of nondegenerate $n$-morphisms $\overline{\eta}'_z: {\bf 1} \rightarrow \overline{Z}'_0( S^{\zeta_{\phi(z)}})$ in
$\calC$ parametrized by $z \in X'_{y}$. Allowing $y$ to vary over $Y$, we see that $(ii)$ is equivalent to the following data:
\begin{itemize}
\item[$(ii')$] A family of nondegenerate $n$-morphisms
$$ \overline{\eta}'_z: {\bf 1} \rightarrow \overline{Z}'_0( S^{\zeta_{\phi(z)}})$$
in $\calC$, parametrized by $z \in X'$.
\end{itemize}

\item[$(e)$] Since the space $\EO(n)$ is contractible, the projection map
$\phi: X' \rightarrow X$ is a homotopy equivalence. We may therefore replace
$(ii')$ by the following data:
\begin{itemize}
\item[$(ii'')$] A family of nondegenerate $n$-morphisms
$\overline{\eta}''_x: {\bf 1} \rightarrow \overline{Z}'_0( S^{\zeta_{x}})$ in
$\calC$, parametrized by $x \in X$.
\end{itemize}
This argument shows that symmetric monoidal functors 
$\Bord_{n}^{(X,\zeta)} \rightarrow \calC$ are classified by the data
of $(i')$ and $(ii'')$, which is precisely the assertion of Theorem \ref{swisher6}. 
\end{itemize}

\begin{remark}
We can summarize the main theme of this section as follows: to deduce the cobordism hypothesis for
one class of manifolds, it suffices to prove the cobordism hypothesis for any larger class of manifolds.
In particular, if we can prove the cobordism hypothesis for the $(\infty,n)$-category
$\Bord_{n}$, then it will follow for any other bordism $(\infty,n)$-category
$\Bord_{n}^{(X, \zeta)}$ of {\em smooth} $n$-manifolds.
\end{remark}

\subsection{Unfolding of Higher Categories}\label{unf}

As we explained in \S \ref{forty1}, the basic skeleton of our proof of the cobordism hypothesis uses induction on the dimension $n$. More precisely, in order to prove Theorem \ref{swisher5} for the
$(\infty,n)$-category $\Bord_n$, we consider the filtration
$$ \Bord_1 \rightarrow \Bord_2 \rightarrow \ldots \rightarrow \Bord_n$$
and apply Theorem \ref{swisher6} repeatedly. One feature of the above filtration is that it is by increasing levels of complexity: each $\Bord_k$ is an $(\infty,k)$-category, and we have seen (\S \ref{higt}) that the theory of $(\infty,k)$-categories becomes increasingly complicated as $k$ grows. Our goal in this section is to show that we can assign to this filtration an ``associated graded object'' which is considerably simpler. This will, in principle, allow us to formulate the cobordism hypothesis entirely in the language of
$(\infty,1)$-categories. In practice this formulation is not so convenient, but the ideas described in this section are nevertheless useful for the proof we will outline in \S \ref{kuji}.

Let us begin by describing an analogue of the idea we wish to implement in a much simpler setting.
Recall the oriented bordism groups $\{ \Omega_{d} \}_{d \geq 0}$ defined in Remark \ref{od}:
the elements of $\Omega_{d}$ are represented by closed oriented $d$-manifolds, and
two such manifolds $M$ and $N$ represent the same element of $\Omega_{d}$ if there is an (oriented) bordism from $M$ to $N$. We can {\em almost} realize the groups $\Omega_{d}$ as the homology
groups of a chain complex. To see this, let $C_{d}$ denote the set of all isomorphism classes of
oriented $d$-manifolds with boundary. If $M \in C_{d}$, then the boundary $\bd M$ can be regarded
as a $(d-1)$-manifold with boundary, whose boundary happens to be empty. Consequently, we have a sequence of boundary maps
$$ \cdots \rightarrow C_2 \stackrel{\bd}{\rightarrow} C_1 \stackrel{\bd}{\rightarrow} C_0.$$
Each $C_d$ has the the structure of a commutative monoid (given by the formation of disjoint unions), and each of the compositions $\bd^2$ is trivial. We can therefore think of $C_{\bigdot}$ as a
chain complex in the category of commutative monoids. Let $Z_{d}$ denote the kernel of the map
$C_d \rightarrow C_{d-1}$, so that $Z_{d}$ can be identified with the set of all isomorphism classes
of closed $d$-manifolds. Then $\Omega_{d}$ can be identified with the quotient of $Z_{d}$
by the equivalence relation which identifies a pair of objects $M, N \in Z_{d}$ if the coproduct
$M \coprod \overline{N}$ lies in the image of the boundary map $\bd: Z_{d+1} \rightarrow Z_{d}$.

We would like to construct an analogous picture in the setting of higher category theory. The first step is to introduce an appropriate analogue of the commutative monoid $C_{d}$:

\begin{protodefinition}
Let $d \geq 1$ be an integer. We let $\bdCob(d)$ denote the $(\infty,1)$-category
which can be described informally as follows:
\begin{itemize}
\item[$(1)$] The objects of $\bdCob(d)$ are $(d-1)$-manifolds with boundary.
\item[$(2)$] Given a pair of objects $M, N \in \bdCob(d)$, we let $\bHom_{ \bdCob(d)}(M,N)$ denote a
classifying space for bordisms from $M$ to $N$ $($any such bordism determines a bordism from
$\bd M$ to $\bd N$; we do not require these bordisms to be trivial$)$.
\item[$(3)$] Composition of morphisms in $\bdCob(d)$ is given by gluing of bordisms.
\end{itemize}
\end{protodefinition}

The $(\infty,1)$-categories $\bdCob(d)$ admit symmetric monoidal structures, given by disjoint unions.
Moreover, passage to the boundary induces a symmetric monoidal functors
$$ \cdots \stackrel{\bd}{\rightarrow} \bdCob(3) \stackrel{\bd}{\rightarrow} \bdCob(2)
\stackrel{\bd}{\rightarrow} \bdCob(1).$$
We note that each of the compositions $\bd^2$ is trivial: more precisely, it is isomorphic to the constant functor $\bdCob(d) \rightarrow \bdCob(d-2)$ taking the value ${\bf 1} \in \bdCob(d-2)$ (here
${\bf 1}$ denotes the unit with respect to the symmetric monoidal structure on $\bdCob(d-2)$: that is,
the empty set). The sequence $\{ \bdCob(d) \}_{d \geq 1}$ is an example of a
{\it categorical chain complex} (see Definition \ref{ka} and Example \ref{matt} below).

The main goal of this section is to prove that the data provided by the chain complex
$$ \cdots \stackrel{\bd}{\rightarrow} \bdCob(3) \stackrel{\bd}{\rightarrow} \bdCob(2)
\stackrel{\bd}{\rightarrow} \bdCob(1).$$
is equivalent to the data provided by the sequence of symmetric monoidal functors
$$ \Bord_1 \rightarrow \Bord_2 \rightarrow \Bord_3 \rightarrow \cdots$$
in the sense that either can be used to reconstruct the other. This is a special case of a
general result which comparing {\it categorical chain complexes} (Definition \ref{ka}) with
{\it skeletal sequences} (Definition \ref{swipe1}). 

\begin{remark}
The above assertion might seem surprising, since the chain complex
$$ \cdots \stackrel{\bd}{\rightarrow} \bdCob(3) \stackrel{\bd}{\rightarrow} \bdCob(2)
\stackrel{\bd}{\rightarrow} \bdCob(1)$$
consists only of $(\infty,1)$-categorical data, while each $\Bord_{n}$ is an
$(\infty,n)$-category. The $(\infty,n)$-category $\Bord_{n}$ is a fairly complicated object:
it records information not only about bordisms between manifolds but also bordisms between bordisms, bordisms between bordisms between bordisms, and so forth. Consequently, the definition of
$\Bord_{n}$ involves manifolds with corners of arbitrary codimension. By contrast, the
$(\infty,1)$-categories $\bdCob(d)$ can be defined using manifolds having corners of codimension $\leq 2$. The fact that we can dispense with corners of higher codimension can be regarded as a reflection of the following geometric idea: any manifold with corners can be regarded as a manifold with boundary by
``smoothing'' the corners.
\end{remark}

Our first step is to axiomatize the properties of the sequence of
$(\infty,1)$-categories $\{ \bdCob(d) \}_{d \geq 1}$. 
Recall that a sequence of abelian groups and group
homomorphisms
$$ \cdots \stackrel{d_3}{\rightarrow} C_2 \stackrel{d_2 }{\rightarrow} C_1 \stackrel{d_1}{\rightarrow} C_0$$ is a {\it chain complex} if each composition $d_{n-1} \circ d_{n}: C_{n} \rightarrow C_{n-2}$
is the zero map. As we already noted above, the sequence 
$$ \cdots \stackrel{\bd}{\rightarrow} \bdCob(3) \stackrel{\bd}{\rightarrow} \bdCob(2)
\stackrel{\bd}{\rightarrow} \bdCob(1)$$
has an analogous property: each of the compositions $\bd^2: \bdCob(n) \rightarrow \bdCob(n-2)$
is trivial. However, when we work in the setting of higher category theory, we need to be more precise: what we should say is that there is a canonical isomorphism $\alpha_n: \bd^2 \simeq \underline{\bf 1}$,
where $\underline{ \bf 1}$ denotes the constant functor $\bdCob(n) \rightarrow \bdCob(n-2)$ taking the value ${\bf 1}$. We then encounter secondary phenomena: specifying the isomorphisms
$\{ \alpha_n \}_{n \geq 3}$ allows us to identify the composition
$\bd^3: \bdCob(n) \rightarrow \bdCob(n-3)$ with the constant functor
(taking the value ${\bf 1} \in \bdCob(n-3)$) in two different ways, depending on whether we use
$\alpha_{n}$ or $\alpha_{n-1}$. In our example, there is a canonical homotopy which relates these two isomorphisms. To extract a good theory of chain complexes, it is necessary to take into account these homotopies as well as the isomorphisms $\{ \alpha_n \}_{n \geq 3}$, together with additional coherence conditions satisfied by higher powers of $\bd$. 

We will sidestep these issues by defining the notion of chain complex in a different way. For motivation, let us begin with classical homological algebra. Suppose we are given a chain complex of abelian groups 
$$ \cdots \stackrel{d_3}{\rightarrow} C_2 \stackrel{d_2 }{\rightarrow} C_1 \stackrel{d_1}{\rightarrow} C_0 \rightarrow \cdots $$
For every integer $n$, let $Z_{n} \subseteq C_{n}$ denote the kernel of the differential $d_{n}$. The condition $d_{n-1} \circ d_{n} = 0$ implies that $d_{n}$ maps $C_{n}$ into $Z_{n-1}$. We therefore obtain a collection of short exact sequences
$$ 0 \rightarrow Z_{n} \rightarrow C_{n} \stackrel{d'_n}{\rightarrow} Z_{n-1},$$
each of which is determined (up to canonical isomorphism) by the homomorphism
$d'_{n}: C_{n} \rightarrow Z_{n-1}$. Conversely, given sequence of maps
$\{ d'_{n}: C_{n} \rightarrow Z_{n-1} \}_{n \in \Z}$ together with isomorphisms
$Z_{n} \simeq \ker d'_{n}$, we can construct a complex of abelian groups
$( C_{\bigdot}, d_{\bigdot})$ by letting $d_{n}$ denote the composition
$$ C_{n} \stackrel{d'_{n}}{\rightarrow} Z_{n-1} \simeq \ker(d'_{n-1}) \subseteq C_{n-1}.$$
We can summarize the above discussion as follows: every chain complex of abelian groups can be obtained by splicing together short exact sequences. 

We would like to export this idea to our higher-categorical setting, replacing the abelian groups $C_{n}$ by the $(\infty,1)$-categories
$\bdCob(n)$ and the abelian groups $Z_{n}$ by the $(\infty,1)$-categories $\tunCob(n)$.
For each $n > 1$, passage to the boundary defines a forgetful functor $\pi: \bdCob(n) \rightarrow \tunCob(n-1)$. To complete the analogy with our homological algebra discussion, we should identify 
$\tunCob(n)$ with the ``kernel'' of the map $\pi$. This kernel can be defined as the homotopy
fiber product  $\bdCob(n) \times_{ \tunCob(n-1)}^{R} \{ {\bf 1} \}$: in other words, the inverse image
$\pi^{-1} \{ \bf 1 \}$. The fact that the formation of this kernel is a well-behaved operation is a reflection of a special feature of the functor $\pi$ which we now describe.

Let $\pi: \calC \rightarrow \calD$ be an arbitrary functor between $(\infty,1)$-categories.
For each object $D \in \calD$, we let $\calC_{D}$ denote the homotopy fiber product
$\calC \times_{ \calD}^{R} \{D \}$. We can think of $\{ \calC_{D} \}_{D \in \calD}$ as a family of
$(\infty,1)$-categories parametrized by the objects of $\calD$. However, this intuition can be somewhat misleading: a morphism $D \rightarrow D'$ in $\calD$ need not induce a functor between fibers $\calC_{D} \rightarrow \calC_{D'}$. We can remedy the situation by introducing an assumption on
the functor $\pi$:

\begin{definition}\label{cabber1}
Let $\pi: \calC \rightarrow \calD$ be a functor between $(\infty,1)$-categories, and let
$\overline{f}: C \rightarrow C'$ be a morphism in $\calC$. We will say that $\overline{f}$ is {\it $\pi$-coCartesian}
if the following condition is satisfied:
\begin{itemize}
\item For every object $C'' \in \calC$, the diagram of $\infty$-groupoids
$$ \xymatrix{ \OHom_{\calC}( C', C'') \ar[r] \ar[d] & \OHom_{\calC}( C, C'') \ar[d] \\
\OHom_{\calD}( \pi(C'), \pi(C'') ) \ar[r] & \OHom_{\calD}( \pi(C), \pi(C'') ) }$$
is a homotopy pullback square.
\end{itemize}
We say that the functor $\pi$ is a {\it coCartesian fibration} if the following condition is satisfied:
for every object $C \in \calC$ and every morphism $f: \pi(C) \rightarrow D$ in $\calD$, there
exists a $\pi$-coCartesian morphism $\overline{f}: C \rightarrow \overline{D}$ lifting $f$.
\end{definition}

In the situation of Definition \ref{cabber1}, the morphism $\overline{f}$ is determined up
to isomorphism by $D$ and $f$, provided that $\overline{f}$ exists. It follows that the codomain
$\overline{D}$ of $\overline{f}$ is also determined by $C$ and $f$; we will often indicate this dependence by writing $\overline{D} = f_{!} C$. 

\begin{definition}
Let $\calC$ and $\calD$ be symmetric monoidal $(\infty,1)$-categories. A {\it symmetric monoidal coCartesian fibration} from $\calC$ to $\calD$ is a symmetric monoidal functor $\pi: \calC \rightarrow \calD$ which is a coCartesian fibration and such that the collection of $\pi$-coCartesian morphisms in $\calC$ is stable under tensor products.
\end{definition}

\begin{example}
For $n \geq 2$, passage to the boundary determines a symmetric monoidal coCartesian fibration
$\bdCob(n) \rightarrow \tunCob(n-1)$. For every closed $(n-2)$-manifold $M$, we can identify
the fiber $\bdCob(n)_{M}$ with an $(\infty,1)$-category whose objects are $(n-1)$-manifolds
with boundary $M$, with $\bHom_{ \bdCob(n)_{M}}( X, X')$ is a classifying space for bordisms
from $X$ to $X'$ which are trivial along $M$. A bordism $B$ from $M$ to $M'$ determines a functor
$\bdCob(n)_{M} \rightarrow \bdCob(n)_{M'}$, given by the formula $X \mapsto X \coprod_{M} B$.
\end{example}

We are now ready to define the basic objects of interest to us in this section:

\begin{definition}\label{ka}
A {\it categorical chain complex of length $n$} consists of the following data:
\begin{itemize}

\item[$(a)$] A sequence of symmetric monoidal coCartesian fibrations
$\{ \calC_{k} \rightarrow \calZ_{k-1} \}_{1 \leq k \leq n}$ between symmetric monoidal $(\infty,1)$-categories, where $\calZ_0 \simeq \ast$ is trivial and each $\calC_{k}$ has duals.

\item[$(b)$] For $1 \leq k < n$, a symmetric monoidal equivalence
of $\calZ_{k}$ with the homotopy fiber $\calC_{k} \times^{R}_{ \calZ_{k-1} } \{ {\bf 1} \}$.
\end{itemize}
\end{definition}

\begin{example}\label{matt}
The symmetric monoidal coCartesian fibrations
$\{ \bdCob(k) \rightarrow \tunCob(k-1) \}_{1 \leq k \leq n}$ determine a categorical chain complex of length $n$. 
\end{example}

Our goal is to compare the class of categorical chain complexes with another class of mathematical objects, which we call {\it skeletal sequences}. 

\begin{definition}
Let $f: \calC \rightarrow \calD$ be a functor between $(\infty,n)$-category, and let $k \geq 0$ be an integer. We will say that $f$ is {\it $k$-connective} if the following conditions are satisfied:
\begin{itemize}
\item[$(1)$] The functor $f$ is essentially surjective. That is, for every object $D \in \calD$, there
exists an object $C \in \calC$ and an isomorphism $D \simeq f(C)$.
\item[$(2)$] If $k > 0$, then for every pair of objects $C,C' \in \calC$, the induced functor
$$\OHom_{\calC}(C,C') \rightarrow \OHom_{\calD}( F(C), F(C') )$$ is $(k-1)$-connective.
\end{itemize}
\end{definition}

\begin{example}
Let $X \subseteq Y$ be a nice inclusion of topological spaces (for example, a cellular inclusion of
CW complexes). Then the induced functor $f: \pi_{\leq \infty} X \rightarrow \pi_{\leq \infty} Y$
is $k$-connective if and only if the inclusion $X \subseteq Y$ is $k$-connected: that is, if and only if the
homotopy groups $\pi_i(Y/X)$ vanish for $i \leq k$. 
\end{example}

\begin{example}
For every integer $k$, the evident functor $\Bord_{k} \rightarrow \Bord_{k+1}$ is
$k$-connective: in fact, $\Bord_{k}$ and $\Bord_{k+1}$ have the same $j$-morphisms for
$j \leq k$.
\end{example}

\begin{definition}\label{swipe1}
A {\it skeletal sequence} $(${\it of length $n$}{}$)$ is a diagram
$$ \calB_1 \stackrel{f_1}{\rightarrow} \calB_2 \stackrel{f_2}{\rightarrow} \ldots 
\stackrel{f_{n-1}}{\rightarrow} \calB_n$$
with the following properies:
\begin{itemize}
\item[$(1)$] Each $\calB_k$ is a symmetric monoidal $(\infty,k)$-category.
\item[$(2)$] Each $f_k$ is a $(k-1)$-connective symmetric monoidal functor.
\item[$(3)$] Every object of $\calB_1$ is dualizable.
\item[$(4)$] For $2 \leq k \leq n$, every morphism $\alpha: {\bf 1} \rightarrow X$ in
$\Omega^{k-2} \calB_k$ admits a left adjoint.
\end{itemize}
\end{definition}

\begin{example}\label{kap}
The sequence
$$ \Bord_1 \rightarrow \Bord_2 \rightarrow \ldots \rightarrow \Bord_n$$
is a skeletal sequence of length $n$. More generally, suppose that $X$ is a topological space
and $\zeta$ is an $n$-dimensional vector bundle on $X$ with inner product. For
$0 \leq k \leq n$, let $X_k$ denote the set of pairs $(x,\alpha)$, where $x \in X$ and
$\alpha: \R^{n-k} \rightarrow \zeta_{x}$ is an isometric embedding, and let 
$\zeta_{k}$ denote the vector bundle on $X_k$ whose fiber at $(x,\alpha)$ is the orthogonal
complement to the image of $\alpha$. Then we have a skeletal sequence
$$ \Bord_{1}^{(X_1, \zeta_1) } \rightarrow \ldots \rightarrow \Bord_{n}^{(X_n, \zeta_n)} = \Bord_{n}^{(X,\zeta)}.$$
\end{example}

\begin{example}\label{kaq}
Let $\calB$ be a symmetric monoidal $(\infty,n)$-category with duals. For $k \leq n$, let
$\calB_k$ denote the underlying $(\infty,k)$-category of $\calB$: that is, the
$(\infty,k)$-category obtained from $\calB$ by discarding noninvertible $m$-morphisms for
$m > k$. Then
$$ \calB_1 \rightarrow \calB_2 \rightarrow \cdots \rightarrow \calB_n =\calB$$
is a skeletal sequence, which we will call the {\it canonical skeletal sequence} associated to $\calB$.
\end{example}

\begin{remark}
The skeletal sequences described in Examples \ref{kap} and \ref{kaq} have a number of additional features:
\begin{itemize}
\item[$(i)$] For $1 \leq k \leq n$, the symmetric monoidal $(\infty,k)$-category $\calC_k$ has duals.
\item[$(ii)$] Each of the functors $\calC_{k} \rightarrow \calC_{k+1}$ is $k$-connective (rather
than merely $(k-1)$-connective).
\end{itemize}
We generally be interested only in skeletal sequences satisfying these conditions.
However, we will have no need of $(i)$ or $(ii)$ in the analysis below.
\end{remark}

\begin{warning}
For $n \leq 3$, the skeletal sequence $\Bord_1 \rightarrow \ldots \rightarrow \Bord_n$ of
Example \ref{kap} coincides with the canonical skeletal sequence of Example \ref{kaq}. However, this
is not true in general, due to the failure of the parametrized $h$-cobordism theorem.
\end{warning}

The main result of this section can be summarized as follows:

\begin{claim}\label{inker}
For each integer $n \geq 1$, the following types of data are equivalent:
\begin{itemize}
\item[$(1)$] Categorical chain complexes of length $n$.
\item[$(2)$] Skeletal sequences of length $n$.
\end{itemize}
\end{claim}

We will give a more precise formulation of Claim \ref{inker} below (see Theorem \ref{sobbus}). First, we need to introduce mild generalizations of Definitions \ref{swipe1} and \ref{ka}.

\begin{definition}\label{ka2}
A {\it weak categorical chain complex of length $n$} consists of the following data:
\begin{itemize}

\item[$(a)$] A sequence of symmetric monoidal coCartesian fibrations
$\{ \calC_{k} \rightarrow \calZ_{k-1} \}_{1 \leq k \leq n}$ between symmetric monoidal $(\infty,1)$-categories, where $\calZ_0 \simeq \ast$ is trivial and each $\calC_{k}$ has duals for $k < n$.

\item[$(b)$] For $1 \leq k < n$, a symmetric monoidal equivalence
of $\calZ_{k}$ with the homotopy fiber $\calC_{k} \times^{R}_{ \calZ_{k-1} } \{ {\bf 1} \}$.

\end{itemize}

A {\it weak skeletal sequence} $(${\it of length $n$}{}$)$ is a diagram
$$ \calB_1 \stackrel{f_1}{\rightarrow} \calB_2 \stackrel{f_2}{\rightarrow} \ldots 
\stackrel{f_{n-1}}{\rightarrow} \calB_n$$
with the following properies:
\begin{itemize}
\item[$(1)$] Each $\calB_k$ is a symmetric monoidal $(\infty,k)$-category.
\item[$(2)$] Each $f_k$ is a $(k-1)$-connective symmetric monoidal functor.
\item[$(3)$] Every object of $\calB_1$ is dualizable if $n > 1$.
\item[$(4)$] For $2 \leq k < n$, every morphism $\alpha: {\bf 1} \rightarrow X$ in
$\Omega^{k-2} \calB_k$ admits a left adjoint.
\end{itemize}
\end{definition}

\begin{example}\label{kubb} 
Let $\vect{\calC}$ denote a (weak) skeletal sequence
$\calC_1 \rightarrow \calC_2 \rightarrow \ldots \rightarrow \calC_n$ of length $n$. Then the induced sequence 
$\Omega \calC_2 \rightarrow \Omega \calC_3 \rightarrow \ldots \rightarrow \Omega \calC_n$
is a (weak) skeletal sequence of length $n-1$, which we will denote by $\Omega \vect{\calC}$.
The only nontrivial point is to verify that $\Omega \vect{\calC}$ satisfies condition
$(3)$ of Definition \ref{ka2}. This follows from the observation that the dual of an object
of $\Omega \calC_2$ can be identified with a left adjoint of the corresponding endomorphism
of the unit object ${\bf 1} \in \calC_2$.
\end{example}

Let us now analyze the notion of weak skeletal sequence of length $2$.
Let $F: \calB_1 \rightarrow \calB_2$ be a symmetric monoidal functor from a symmetric monoidal $(\infty,1)$-category $\calB_1$ to
a symmetric monoidal $(\infty,2)$-category $\calB_2$. We would like to
describe $\calB_2$ in terms of $\calB_1$, together with some additional data of an $(\infty,1)$-categorical nature. To accomplish this, we will need to know the mapping objects
$\OHom_{\calB_2}( F(X), F(Y) )$ for $C, D \in \calB_1$. If we assume that every object $X \in \calB_1$ has a dual $C^{\vee}$, we obtain a canonical equivalence (of $(\infty,1)$-categories) $$ \OHom_{\calB_2}( F(X), F(Y) ) \simeq \OHom_{\calB_2}( {\bf 1}, F(X^{\vee} \otimes Y));$$
here ${\bf 1}$ denotes the unit object of $\calB_2$. 

\begin{notation}\label{goob}
Let $F: \calB_1 \rightarrow \calB_2$ be a symmetric monoidal functor from
a symmetric monoidal $(\infty,1)$-category $\calB_1$ to a symmetric monoidal
$(\infty,2)$-category $\calB_2$. For each object $X \in \calB_1$, we let
$M_{F}(X)$ denote the $(\infty,1)$-category $\OHom_{\calB_2}( {\bf 1}, F(X) )$. 
\end{notation}

What sort of an object is $M_{F}$? We first observe that $M_{F}(X)$ depends functorially on the object $X \in \calB_1$. In other words, we can regard $M_{F}$ as a
functor from $\calB_1$ into $\Cat_{(\infty,1)}$, where $\Cat_{(\infty,1)}$ denotes the
(large) $(\infty,1)$-category of $(\infty,1)$-categories and functors between them.
Moreover, the $M_{F}$ interacts with the symmetric monoidal structure on $\calC_1$:
it is an example of a {\em lax symmetric monoidal functor}, which means that there is a collection
of maps $$ M_{F}(X) \times M_{F}(Y) \rightarrow M_{F}(X \otimes Y)$$
which are suitably compatible with the commutativity and associativity properties of the tensor product $\otimes$ on $\calB_1$.
We can attempt to recover the $(\infty,2)$-category $\calB_2$ from $\calB_1$ and the functor $M_F$, using the following general construction:

\begin{construction}\label{entering}
Let $\calB_1$ be a symmetric monoidal $(\infty,1)$-category, and let $M: \calB_1 \rightarrow \Cat_{(\infty,1)}$ be a lax symmetric monoidal functor. Suppose that every object of $\calB_1$ has a dual. We can then construct a symmetric monoidal $(\infty,2)$-category $\calB[M]$ which can be described informally as follows:
\begin{itemize}
\item[$(1)$] The objects of $\calB[M]$ are the objects of $\calB_1$.
\item[$(2)$] Given a pair of objects $X,Y \in \calB[M]$, we let
$\OHom_{\calB[M]}(X,Y) = M( X^{\vee} \otimes Y)$.
\item[$(3)$] Given a triple of objects $X,Y,Z \in \calB[M]$, the composition law
$\OHom_{\calB[M]}(X,Y) \times \OHom_{\calB[M]}(Y,Z)
\rightarrow \OHom_{\calB[M]}(X,Z)$
is given by the composition
$$ M(X^{\vee} \otimes Y) \times
M(Y^{\vee} \otimes Z) \rightarrow M( X^{\vee} \otimes Y \otimes {Y}^{\vee}
\otimes Z) \rightarrow M(X^{\vee} \otimes Y),$$
where the first map is given by the lax symmetric monoidal structure on the functor $M$ and the second is induced by the evaluation map $Y^{\vee} \otimes Y \rightarrow {\bf 1}$ in $\calB_1$.
\end{itemize}
\end{construction}

The fundamental properties of Construction \ref{entering} may be summarized as follows:

\begin{proposition}\label{kooma}
Let $\calB_1$ be a symmetric monoidal $(\infty,1)$-category with duals. Then
the construction $M \mapsto \calB[M]$ determines an equivalence between the following data:
\begin{itemize}
\item[$(1)$] Lax symmetric monoidal functors $M: \calB_1 \rightarrow \Cat_{(\infty,1)}$. 
\item[$(2)$] Symmetric monoidal $(\infty,2)$-categories $\calB_2$ equipped with an essentially surjective symmetric monoidal functor $\calB_1 \rightarrow \calB_2$.
\end{itemize}
\end{proposition}

\begin{warning}
Let $\calB_1$ and $M$ be as in Construction \ref{entering}. To define $\calB[M]$ as an
$(\infty,2)$-category, it is not necessary to assume that the tensor product operation on
$\calB_1$ is commutative: it suffices to assume that $\calB_1$ is a monoidal $(\infty,1)$-category, and that $M$ is a lax monoidal functor. In this case, $\calB[M]$ does {\em not} inherit a monoidal structure from the monoidal structure on $\calB_1$. However, it does inherit an action of the monoidal
category $\calB_1$. Proposition \ref{kooma} admits the following analogue: giving a lax monoidal functor $M: \calB_1 \rightarrow \Cat_{(\infty,1)}$ is equivalent to giving an $(\infty,2)$-category
$\calB_2$ with a distinguished object ${\bf 1}$ which is acted on by $\calB_1$, such that the action functor $\calB_1 \simeq \calB_1 \times \{ {\bf 1} \} \rightarrow \calB_1 \times \calB \rightarrow \calB$
is essentially surjective. 
\end{warning}

In order to apply Proposition \ref{kooma} in practice, we need a way of describing
lax symmetric monoidal functors $M: \calB_1 \rightarrow \Cat_{(\infty,1)}$. This can be achieved by a higher-categorical version of what is often called the {\it Grothendieck construction}:

\begin{construction}\label{groth}
Let $\calB$ be an $(\infty,1)$-category, and let $M: \calB \rightarrow \Cat_{(\infty,1)}$ be
a functor. Then $M$ associates to each object $X \in \calB$ an $(\infty,1)$-category
$M(X)$, and to each morphism $f: X \rightarrow Y$ in $\calB$ a functor
$f_{!}: M(X) \rightarrow M(Y)$. 

We can construct a new $(\infty,1)$-category $\Groth(\calB, M)$ which can be described informally as follows:
\begin{itemize}
\item The objects of $\Groth(\calB_1,M)$ are pairs $(X,\eta)$, where
$X$ is an object of $\calB$ and $\eta$ is an object of $M(X)$.
\item Given a pair of objects $(X,\eta), (X',\eta') \in \Groth(\calB,M)$, we define
$\OHom_{\Groth(\calB,M)}( (X,\eta), (Y',\eta') )$ to be a classifying space for pairs
$(f, \alpha)$, where $f \in \OHom_{\calB}(X,X')$ and $\alpha \in \OHom_{M(X')}( f_{!} \eta, \eta')$.
\item Composition of morphisms in $\Groth(\calB, M)$ is defined in a straightforward way.
\end{itemize}
\end{construction}

Let $\calB$ and $M$ be as in Construction \ref{groth}. There is a canonical projection functor $\pi:\Groth(\calB,M) \rightarrow \calB$, given on objects by the formula $(X,\eta) \mapsto X$.
This functor is a coCartesian fibration, and we have a canonical equivalence
$\Groth(\calB,M)_{X} \simeq M(X)$ for each $X \in \calB$. We can therefore
recover the functor $M$ from $\pi$: for example, if $f: X \rightarrow Y$ is a morphism
in $\calB$, then the induced functor $M(X) \rightarrow M(Y)$ can be identified with
the functor $f_{!}: \Groth(\calB, M)_{X} \rightarrow \Groth(\calB, M)_{Y}$ whose value
on an object $\overline{X} \in \Groth(\calB, M)_{X}$ is the codomain of a
$\pi$-coCartesian morphism $\overline{f}: \overline{X} \rightarrow \overline{Y}$ lifting
$f$. Elaborating on these ideas, one can prove the following:

\begin{proposition}\label{fatvat}
Let $\calB$ be an $(\infty,1)$-category. The construction $M \mapsto \Groth(\calB, M)$ determines an equivalence between the following types of data:
\begin{itemize}
\item[$(1)$] Functors from $\calB$ to $\Cat_{(\infty,1)}$.
\item[$(2)$] CoCartesian fibrations $\pi: \calC \rightarrow \calB$.
\end{itemize}
\end{proposition}

\begin{remark}
For a precise formulation and proof of Proposition \ref{fatvat} in the $(\infty,1)$-categorical context, we refer the reader to \cite{htt}.
\end{remark}

Suppose now that $\calB$ is a symmetric monoidal $(\infty,1)$-category, and that
$M: \calB \rightarrow \Cat_{(\infty,1)}$ is a lax symmetric monoidal structure.
Then for every pair of objects $X, Y \in \calB$, we have a canonical functor $\beta_{X,Y}: M(X) \times M(Y) \rightarrow M(X \otimes Y)$. The $(\infty,1)$-category $\Groth(\calB, M)$ inherits a symmetric monoidal structure, which is given on objects by the formula $(X, \eta) \otimes (Y, \eta')
\simeq (X \otimes Y, \beta_{X,Y}(\eta, \eta'))$. We observe that the projection functor
$\pi: \Groth(\calB, M) \rightarrow \calB$ is a symmetric monoidal coCartesian fibration.
In fact, we have the following converse, which can be regarded as a symmetric monoidal
analogue of Proposition \ref{fatvat}:

\begin{proposition}\label{fatvat2}
Let $\calB$ be a symmetric monoidal $(\infty,1)$-category. The construction $M \mapsto \Groth(\calB, M)$ determines an equivalence between the following data:
\begin{itemize}
\item[$(1)$] Lax symmetric monoidal functors from $\calB$ to $\Cat_{(\infty,1)}$.
\item[$(2)$] Symmetric monoidal coCartesian fibrations $\pi: \calC \rightarrow \calB$.
\end{itemize}
\end{proposition}

\begin{notation}\label{pf}
Let $F: \calB_1 \rightarrow \calB_2$ be a symmetric monoidal functor from a symmetric monoidal
$(\infty,1)$-category $\calB_1$ to a symmetric monoidal $(\infty,2)$-category $\calB_2$. We let
$\calC[F]$ denote the symmetric monoidal $(\infty,1)$-category
$\Groth(\calB_1, M_F)$, where $M_F$ is defined in Notation \ref{goob}.
More concretely, $\calC[F]$ is an $(\infty,1)$-category whose objects are pairs $(X, \eta)$, where
$X \in \calB_1$ and $\eta: {\bf 1} \rightarrow F(X)$ is a $1$-morphism in $\calB_2$.
\end{notation}

Combining Propositions \ref{kooma} and \ref{fatvat2}, we obtain the following result:

\begin{proposition}\label{kubbus}
Let $\calB_1$ be a symmetric monoidal $(\infty,1)$-category with duals.
The construction 
$$ (F: \calB_1 \rightarrow \calB_2) \mapsto (\pi: \calC[F] \rightarrow \calB_1)$$
determines an equivalence between the following types of data:
\begin{itemize}
\item[$(1)$] Essentially surjective symmetric monoidal functors $F: \calB_1 \rightarrow \calB_2$, where
$\calB_2$ is a symmetric monoidal $(\infty,2)$-category.

\item[$(2)$] Symmetric monoidal coCartesian fibrations $\calC \rightarrow \calB_1$.
\end{itemize}
\end{proposition}

It follows from Propositino \ref{kubbus} that properties of a symmetric monoidal coCartesian fibration
$\pi: \calC \rightarrow \calB_1$ can be translated into properties of the corresponding
weak skeletal sequence $\calB_1 \rightarrow \calB_2$. In particular, we have the following result:

\begin{proposition}\label{kamp}
Let $\calB_1$ be a symmetric monoidal $(\infty,1)$-category with duals, let
$\calB_2$ be a symmetric monoidal $(\infty,2)$-category, and let
$F: \calB_1 \rightarrow \calB_2$ be an essentially surjective symmetric monoidal functor.
The following conditions are equivalent:
\begin{itemize}
\item[$(1)$] Every morphism ${\bf 1} \rightarrow X$ in $\calB_2$ admits a left adjoint
$($in other words, $F: \calB_1 \rightarrow \calB_2$ is a skeletal sequence$)$.
\item[$(2)$] The symmetric monoidal $(\infty,1)$-category $\calC[F]$ has duals.
\end{itemize}
\end{proposition}

\begin{proof}
Since $F$ is essentially surjective, condition $(1)$ is equivalent to the following:
\begin{itemize}
\item[$(1')$] Let $X \in \calB_1$ be an object. Then every morphism $\eta: {\bf 1} \rightarrow F(X)$ in $\calB_2$ admits a left adjoint.
\end{itemize}
In the situation of $(1')$, the pair $(X, \eta)$ determines an object of $\calC[F]$. Unwinding the definitions, we see that $(X^{\vee}, \eta')$ is dual to $(X, \eta)$ if and only if
${\eta'}^{\vee}: X \rightarrow {\bf 1}$ is a left adjoint to $\eta$.
\end{proof}

Combining Propositions \ref{kubbus} and \ref{kamp}, we obtain the following result
(which proves Claim \ref{inker} in the case $n=2$):

\begin{proposition}\label{kubbis}
Let $\calB_1$ be a symmetric monoidal $(\infty,1)$-category with duals. The construction
$(F: \calB_1 \rightarrow \calB_2) \mapsto ( \pi: \calC[F] \rightarrow \calB_1)$ determines an equivalence
between the following types of data:
\begin{itemize}
\item[$(1)$] Skeletal sequences $F: \calB_1 \rightarrow \calB_2$ of length $2$.
\item[$(2)$] Symmetric monoidal coCartesian fibrations $\pi: \calC \rightarrow \calB_1$, where
$\calC$ is a symmetric monoidal $(\infty,1)$-category with duals.
\end{itemize} 
\end{proposition}




We now outline the modifications to the above constructions which are needed to
justify Claim \ref{inker} for $n > 2$. Consider a weak skeletal sequence
$\calB_1 \rightarrow \calB_2 \rightarrow \ldots \rightarrow \calB_n$ of length $n > 2$.
For each $i \geq 2$, let $F_{i}: \calB_1 \rightarrow \calB_i$ denote the induced functor. We can then define a lax symmetric monoidal functor $M_i: \calB_1 \rightarrow \Cat_{(\infty,i-1)}$ by the formula $M_i(X) = \OHom_{\calB_i}( {\bf 1}, F_i(X) )$, where $\Cat_{(\infty,i-1)}$ denotes the $(\infty,1)$-category of $(\infty,i-1)$-categories and functors between them. Applying a generalization of Construction \ref{groth}, we can convert the functors $M_i$ into symmetric monoidal
coCartesian fibrations $\pi[i]: \calC_i \rightarrow \calB_1$, where $\calC_i$ is an $(\infty,i-1)$-category.
We have a sequence of essentially surjective functors
$\calC_2 \rightarrow \calC_3 \rightarrow \ldots \rightarrow \calC_{n}.$
For $i > 2$, let $G_i: \calC_2 \rightarrow \calC_i$ denote the induced functor. We can then define
a new functor $N_i: \calC_2 \rightarrow \Cat_{(\infty,i-2)}$ by the formula $N_i(X) = 
\OHom_{ \calC_i}( {\bf 1}, G_i(X) )$. Since $\calC_2$ has duals (Proposition \ref{kamp}), 
we can recover $\calC_{i}$ from $N_i$, at least up to canonical equivalence.
For $i > 2$ and $X= (C,\eta) \in \calC_2$, define $N(X)$ to be the $\infty$-groupoid $\OHom_{ \calB_1}( {\bf 1}, C )$. Then $N$ is a functor from $\calC_2$ to the $(\infty,1)$-category $\Cat_{(\infty,0)}$ (whose objects we can view as topological spaces). For each $i > 2$, we have a natural transformation of functors $N_i \rightarrow N$. Given a morphism $f: {\bf 1} \rightarrow C$ in $N(X)$, the homotopy fiber product $N_i(X) \times^{R}_{N(X)} \{ f \}$ is given by
\begin{eqnarray*}
\OHom_{ \calC_i}( {\bf 1}, G_i(X)) \times^{R}_{ \OHom_{ \calB_1}( {\bf 1}, C)} \{f\} &
\simeq & \OHom_{ \calC_i }( G_i(X^{\vee}), {\bf 1}) 
\times^{R} _{ \OHom_{\calB_1}( C^{\vee}, {\bf 1}) } \{f^{\vee} \} \\
& \simeq & \OHom_{ \Omega \calB_i}( G_i(f^{\vee}_{!} X^{\vee}), {\bf 1}) \end{eqnarray*}
and can therefore be entirely reconstructed from the functor
$\overline{G_i}: \Omega \calB_2 \rightarrow \Omega \calB_i$ by passing to the fiber
over the unit object ${\bf 1} \in \calB_1$. Elaborating
on this reasoning, one can prove the following:

\begin{proposition}\label{longwife}
Fix a skeletal sequence $F: \calB_1 \rightarrow \calB_2$ of length $2$, and let $n > 2$. The construction
above establishes an equivalence between the following types of data:
\begin{itemize}
\item[$(1)$] Weak skeletal sequences $\calB_1 \rightarrow \calB_2 \rightarrow \calB_3 \rightarrow \cdots \rightarrow \calB_n$ of length $n$ which begin with $F$.
\item[$(2)$] Weak skeletal sequences $\Omega \calB_2 \rightarrow \calC_2 \rightarrow \ldots
\rightarrow \calC_{n-1}$ of length $(n-1)$ which begin with $\Omega \calB_2$.
\end{itemize}
\end{proposition}

\begin{construction}\label{sabber}
Suppose we are given a weak skeletal sequence
$$ \calB_1 \stackrel{F_2}{\rightarrow} \calB_2 \stackrel{F_3}{\rightarrow} \calB_3
\rightarrow \cdots \rightarrow \calB_n$$
of length $n$. We can associate to this weak skeletal sequence a weak categorical chain complex of
of length $n$ $\{ \calC_i \rightarrow \calZ_{i-1} \}_{1 \leq i \leq n}$ as follows:
\begin{itemize}
\item For $1 \leq k \leq n$, set $\calZ_k = \Omega^{k-1} \calB_{k}$.
\item Let $\calC_1 = \calZ_1$, and for $2 \leq k \leq n$ let
$\calC_k = \calC[ \Omega^{k-2} F_{k-1}]$.
\end{itemize}
\end{construction}

Claim \ref{inker} can now be formulated more precisely as follows:

\begin{theorem}\label{sobbus}
Construction \ref{sabber} determines an equivalence between the following types of data:
\begin{itemize}
\item[$(1)$] Weak skeletal sequences of length $n$.
\item[$(2)$] Weak categorical chain complexes of length $n$.
\end{itemize}
Under this equivalence, skeletal sequences of length $n$ correspond to categorical
chain complexes of length $n$.
\end{theorem}

\begin{proof}
Combine Propositions \ref{longwife} and \ref{kubbis}.
\end{proof}

\begin{example}\label{saymot}
Under the equivalence of Theorem \ref{sobbus}, the skeletal sequence
$$ \Bord_{1} \rightarrow \Bord_{2} \rightarrow \cdots \rightarrow \Bord_{n}$$
corresponds to the categorical chain complex
$\{ \bdCob(k) \rightarrow \tunCob(k-1) \}_{ 1 \leq k \leq n}$ of 
Example \ref{matt}.
\end{example}

Example \ref{saymot} and Theorem \ref{sobbus} allow us to reformulate the cobordism hypothesis
as a statement about the symmetric monoidal $(\infty,1)$-categories
$\{ \bdCob(k) \}_{1 \leq k \leq n}$. We will not make use of this maneuver directly. Nevertheless, in \S \ref{kuji} we will exploit the following consequence of Theorem \ref{sobbus}:

\begin{corollary}\label{isa}
Let $\calB_{n-1}$ be a symmetric monoidal $(\infty,n-1)$-category with duals.
The following types of data are equivalent:
\begin{itemize}
\item[$(1)$] Symmetric monoidal functors $\calB_{n-1} \rightarrow \calB_{n}$ which
are $(n-2)$-connective, where $\calB_n$ is a symmetric monoidal $(\infty,n)$-category.
\item[$(2)$] Lax symmetric monoidal functors $M: \Omega^{n-2} \calC_{n-1} \rightarrow
\Cat_{(\infty,1)}$.
\item[$(3)$] Symmetric monoidal coCartesian fibrations $\pi: \calC \rightarrow \Omega^{n-2} \calB_{n-1}$ . 
\end{itemize}
\end{corollary}

\begin{proof}
Let $\vect{\calB}: \calB_1 \rightarrow \calB_2 \rightarrow \ldots \rightarrow \calB_{n-1}$ be the canonical
skeletal sequence of length $(n-1)$ associated to $\calB_{n-1}$ (Example \ref{kaq}). Then
the data of described in any of $(1)$, $(2)$, of $(3)$ can be identified with
that of a weak skeletal sequence of length $n$ extending $\vect{\calB}$.
\end{proof}

\subsection{The Index Filtration}\label{kuji}

Our goal in this section is to present the core geometric arguments
underlying our proof of the cobordism hypothesis. Using the methods of
\S \ref{forty1} and \ref{slabun}, we are reduced to analyzing the symmetric monoidal
functor $i: \Bord_{n-1} \rightarrow \Bord_{n}$. According to Corollary \ref{isa},
the functor $i$ is classified by a lax symmetric monoidal functor
$\calB: \Omega^{n-2} \Bord_{n-1} \rightarrow \Cat_{(\infty,1)}$. We will therefore proceed by analyzing the functor $\calB$.

We begin by observing that $\Omega^{n-2} \Bord_{n-1}$ is a familiar object: it can
be identified with the $(\infty,1)$-category $\tunCob(n-1)$ described in \S \ref{swugg}, whose objects
are closed $(n-2)$-manifolds and morphism spaces are classifying spaces for bordisms between
closed $(n-2)$-manifolds. For every closed $(n-2)$-manifold $M$, the $(\infty,1)$-category
$\calB(M)$ can be described informally as follows:

\begin{itemize}
\item[$(i)$] The objects of $\calB(M)$ are $(n-1)$-manifolds $X$ equipped with a diffeomorphism
$\bd X \simeq M$.
\item[$(ii)$] Given a pair of objects $X, X' \in \calB(M)$, the $\infty$-groupoid $\OHom_{\calB(M)}(X,X')$
is a classifying space for bordisms $B$ from
$X$ to $X'$. We require such bordisms to be trivial along the common boundary
$\bd X \simeq M \simeq \bd X'$, so that we have an identification
$$\bd B \simeq X \coprod_{ M} (M \times [0,1] ) \coprod_{M} X'.$$ 
\end{itemize}

We regard $\calB(M)$ as a functor of $M$: for every bordism $Y: M \rightarrow M'$ 
of closed $(n-2)$-manifolds, $Y$ defines a functor $\calB(M) \rightarrow \calB(M')$ which is given on objects by the formula $X \mapsto X \coprod_{M} Y$. 

The cobordism hypothesis (see Theorem \ref{swisher6}) asserts roughly that $\Bord_{n}$ is freely generated from $\Bord_{n-1}$ by adjoining an $\OO(n)$-equivariant $n$-morphism
$\eta: \emptyset \rightarrow S^{n-1}$, corresponding to an $n$-dimensional disk. In the present terms, this amounts to a description of the functor $\calB$ by ``generators and relations.'' In this section, we will explain how to obtain such a description using Morse theory. We begin by recalling some basic definitions.

Let $M$ be a closed $(n-2)$-manifold, and let $B: X \rightarrow X'$ be a $1$-morphism in $\calB(M)$.
We will identify $X$, $X'$, and $M \times [0,1]$ with their images in $B$. A smooth function
$f: B \rightarrow \R$ is said to have a {\it critical point} at $b \in B$ if each of the first derivatives of
$f$ vanish at the point $b$; in this case, we also say that $f(b)$ is a {\it critical value} of $f$. 
If $b \in B$ is a critical point of $f$, the second derivatives of $f$ determine a symmetric bilinear
on the tangent space $T_{B,b}$, called the {\it Hessian} of $f$. A critical point $b$ of $f$ is said to be {\it nondegenerate} if the Hessian of $f$ at $b$ is a nondegenerate bilinear form.

Let $\Fun(B)$ denote the collection of all smooth functions $f: B \rightarrow [0,1]$ such that $f^{-1} \{0\} = X$, $f^{-1} \{1\} = X'$, $f(m,t) = t$ for $(m,t) \in M \times [0,1]$, and such that $f$ has no critical points on $\bd B$. A generic element of $f \in \Fun(B)$ has the following properties:
\begin{itemize}
\item[$(a)$] The function $f: B \rightarrow [0,1]$ is {\it Morse}: that is, every critical point of $f$ is nondegenerate.  
\item[$(b)$] The critical values of $f$ are distinct. That is, for every pair of distinct critical points
$b \neq b'$ of $f$, the we have $f(b) \neq f(b')$. 
\end{itemize}
More precisely, the set $\Fun(B)$ carries a natural topology and the collection of functions
$f \in \Fun(B)$ satisfying both $(a)$ and $(b)$ is open and dense. Let $f: B \rightarrow [0,1]$ be such a function. If $f$ has no critical points then it exhibits $B$ as a fiber bundle over the interval $[0,1]$, which determines an identification of $X$ with $X'$ (well-defined up to isotopy). Under this identification,
$B$ corresponds to the identity morphism $\id_{X}$ in $\calB(M)$. If $f$ has at least one critical point, then we can choose a sequence of real numbers $0 = t_0 < t_1 < \ldots < t_k = 1$ with the following properties:
\begin{itemize}
\item None of the real numbers $t_i$ is a critical value of $f$.

\item Each of the inverse images $f^{-1} [t_{i-1}, t_{i} ]$ contains exactly one critical point of
$f$.
\end{itemize}

This allows us to express $B$ as a composition $B_k \circ \ldots \circ B_1$ of $1$-morphisms in $\calB(M)$, where $B_i \simeq f^{-1} [t_{i-1}, t_i]$ is a morphism from $X_{i-1} = f^{-1} \{t_{i-1} \}$ to
$X_i = f^{-1} \{t_i \}$. The advantage of this representation is that each
of the morphisms $B_i$ is very simple, thanks to the following fundamental result:

\begin{lemma}[Morse Lemma]\label{ml}
Let $f: B \rightarrow [0,1]$ be a smooth function with a nondegenerate critical point
at a point $b \in B$. Then there exists a system of local coordinates $x_1, \ldots, x_n$ for
$B$ at $b$ such that $f$ can be written
$$ f(x_1, \ldots, x_n) = f(b) - x_1^2 - \ldots - x_m^2 + x_{m+1}^{2} + \ldots + x_n^2.$$
Here $m \leq n$ is a nonnegative integer, called the {\it index} of the critical point $b$.
\end{lemma}

Suppose that $B$ is a morphism in $\calB(M)$ and that $f: B \rightarrow [0,1]$ is a Morse function
with exactly one critical point $b$, which has index $m$. Choose local coordinates 
$x_1, \ldots, x_n$ for $B$ at $b$ as described in Lemma \ref{ml}. For a sufficiently small real number $\epsilon$, we can regard
$$U = \{ (x_1, \ldots, x_n) \in \R^{n}: (x_1^2 + \ldots + x_n^2 < 2 \epsilon)
\wedge ( - \epsilon < -x_1^2 - \ldots - x_m^2 + x_{m+1}^{2} + \ldots + x_n^2
< {\epsilon} ) \}$$
as an open subset of $B$. Let $B' = f^{-1} [ f(b) - \epsilon, f(b) + \epsilon ]$;
since $f$ is submersive outside of $B'$, we can identify $B$ with $B'$ as morphisms
in $\calB(M)$. We observe that $B' - U \simeq [ f(b) - \epsilon, f(b) + \epsilon] \times Y$, where
$Y$ is a bordism from $M$ to $S^{m-1} \times S^{n-m-1}$. The closure $\overline{U}$ of $U$
determines a morphism $\gamma_m$ from
$D^{m} \times S^{n-m-1}$ to $S^{m-1} \times D^{n-m}$ in the $(\infty,1)$-category
$\calB( S^{m-1} \times S^{n-m-1})$, and we can identify $B$ with the image of
$\eta$ under the functor $\calB( S^{m-1} \times S^{n-m-1}) \rightarrow \calB(M)$
determined by $Y$.

The upshot of this discussion is that we can regard the morphisms $\{ \gamma_m \}_{0 \leq m \leq n}$
as ``generators'' for the functor $\calB$ in the following sense: for every closed $(n-2)$-manifold $M$,
every morphism $B$ in $\calB(M)$ can be obtained as a composition of images of
the morphisms of $\gamma_m$ under functors $\calB( S^{m-1} \times S^{n-m-1}) \rightarrow \calB(M)$
induced by bordisms from $S^{m-1} \times S^{n-m-1}$ to $M$. This is just a categorical reformulation
of the classical fact that every $n$-manifold $B$ admits a handle decomposition: in other words, 
$B$ can be built by a finite sequence of handle attachments. 

Unfortunately, the above analysis is not nearly refined enough for our purposes. In order to prove the cobordism hypothesis, we need a more precise understanding $(\infty,1)$-categories $\calB(M)$. Every morphism in $\calB(M)$ corresponds to a bordism $B$ between
$(n-1)$-manifolds, which admits a handle decomposition. However, this handle decomposition is not
unique: it depends on a choice of Morse function $f: B \rightarrow [0,1]$. Two different Morse functions $f_0,f_1: B \rightarrow [0,1]$ will generally give rise to very different handle decompositions: for example, they can have different numbers of critical points. However, {\it Cerf theory} guarantees that $f_0$ and $f_1$ must be related to one another in a reasonably simple way. To see this, one chooses a smooth family of functions $\{ f_t \}_{0 \leq t \leq 1}$ in $\Fun(B)$ which interpolate between $f_0$ and $f_1$. It is generally impossible to ensure that all of the functions $f_t$ are Morse. However,
if $\{ f_t \}_{0 \leq t \leq 1}$ is suitably generic, then $f_{t}$ will be fail to be Morse for only finitely many values of $t$; moreover, each $f_t$ will fail to be Morse in a very mild (and well-understood) way.

\begin{definition}\label{swell}
Let $B$ be an $n$-manifold. A smooth function $f: B \rightarrow \R$ is a {\it generalized Morse function}
if, for every point $b \in B$, one of the following conditions holds:
\begin{itemize}
\item[$(1)$] The point $b$ is a regular point of the function $f$: in other words, the derivative
of $f$ does not vanish at $b$.
\item[$(2)$] The point $b$ is a nondegenerate critical point of the function $f$.
\item[$(3)$] The function $f$ has a {\it birth-death singularity} at $b$: that is, there exists
a system of local coordinates $x_1, \ldots, x_n$ for $B$ at $b$ such that $f$ can be written
in the form
$$ f(x_1, \ldots, x_n) = f(b) - x_1^2 - \ldots - x_{m}^2 + x_{m+1}^2 + \ldots + x_{n-1}^2 + x_n^3.$$
Here $m < n$ is an integer, called the {\it index} of the critical point $b \in B$.
\end{itemize}
Moreover, if $B$ is a manifold with boundary (or with corners), then we require all critical
points of $f$ to lie in the interior of $B$.
\end{definition}

\begin{remark}
It is possible to formulate Definition \ref{swell} in more invariant terms. A smooth function
$f: B \rightarrow \R$ has a birth-death singularity at $b \in B$ if and only if $b$ is
a critical point for $f$, the Hessian of $f$ at $b$ has a one-dimensional nullspace 
$V \subseteq T_{B,b}$, and the third derivative of $f$ along $V$ does not vanish.
\end{remark}

The theory of generalized Morse functions describes the generic behavior of $1$-parameter families of smooth functions on a manifold $B$. For our purposes, this is still not good enough: in order to understand $\calB(M)$ as an $(\infty,1)$-category, as opposed to an ordinary category, we need to be able to contemplate families with an arbitrary number of parameters.
One approach to the problem is to try to explicitly understand the generic behavior of several
parameter families of functions on $B$. This is feasible for small numbers of parameters
(and in fact this is sufficient for our purposes: we can use the methods of \S \ref{kumma} to
reduce to the problem of understanding the $(2,1)$-categories $\tau_{\leq 2} \calB(M)$)
using Thom's theory of {\it catastrophes}. However, it will be more convenient to address the issue in a different way, using Igusa's theory of {\it framed functions}.

\begin{definition}\label{cablus}
Let $M$ be a closed $(n-2)$-manifold, and let $B$ be a morphism in $\calB(M)$. A
{\it framed function} on $B$ consists of the following data:
\begin{itemize}
\item[$(1)$] A function $f \in \Fun(B)$ which is a generalized Morse function.
\item[$(2)$] For every critical point $b$ of $f$ having index $m$, a collection
of tangent vectors $v_1, \ldots, v_m \in T_{B,b}$ satisfying
$$ H(v_i, v_j) = \begin{cases} 0 & \text{if } i \neq j \\
-2 & \text{if } i = j,\end{cases}$$
where $H$ denotes the Hessian of $f$ at $b$.
\end{itemize}
\end{definition}

In the situation of Definition \ref{cablus}, we will generally abuse terminology and simply refer to
$f$ as a framed function on $B$; the data of type $(2)$ is implicitly understood to be specified as well.
In \cite{igusa}, Igusa explains how to introduce a topological space $\FrFun(B)$ of framed
functions on $B$ (more precisely, he defines a simplicial set whose vertices are framed functions; one can then define $\FrFun(B)$ to be the geometric realization of this simplicial set). We can use the spaces
$\FrFun(B)$ to assemble a new version of the $(\infty,n)$-category
$\Bord_{n}$. For every closed $(n-2)$-manifold $M$, let $\calB_n(M)$
denote the $(\infty,1)$-category which we may describe informally as follows:

\begin{itemize}
\item[$(i')$] The objects of $\calB_n(M)$ are $(n-1)$-manifolds $X$ equipped with
an identification $\bd X \simeq M$.
\item[$(ii')$] For every pair of objects $X, X' \in \calB_n(M)$, we let
$\OHom_{\calB_n}(X,X')$ be a classifying
space for pairs $(B, f)$, where $B$ is a bordism from $X$ to $X'$ (which is trivial along $M$,
as in the description of $\calB(M)$), and $f \in \FrFun(B)$ is a framed function.
\end{itemize}

We can regard $\calB_n$ as a lax symmetric monoidal functor from $\Omega^{n-2} \Bord_{n-1}$
to $\Cat_{(\infty,1)}$. In view of Corollary \ref{isa}, this functor determines an $(n-2)$-connective
symmetric monoidal functor $\Bord_{n-1} \rightarrow \Bord_{n}^{\frun}$. We can think of
$\Bord_{n}^{\frun}$ as a variation on the $(\infty,n)$-category $\Bord_{n}$, where we require that every $n$-manifold be decorated by a framed function. In particular, there is a forgetful functor
$j: \Bord_{n}^{\frun} \rightarrow \Bord_{n}$. 

\begin{warning}
The reader should not confused the framed bordism $(\infty,n)$-category
$\Bord_{n}^{\fr}$ with the $(\infty,n)$-category $\Bord_{n}^{\frun}$ introduced above.
In the former case, we endow all $n$-manifolds with framings; in the latter, we endow all
$n$-manifolds with generalized Morse functions together with framings on the negative eigenspaces of the Hessian at each critical point.
\end{warning}

There is an evident functor $\Bord_{n}^{\frun} \rightarrow \Bord_{n}$, which is obtained by forgetting the framed functions. In order to prove the cobordism hypothesis, it will suffice to verify the following results:

\begin{theorem}[Cobordism Hypothesis, Framed Function Version]\label{swisher11}
Let $\calC$ be a symmetric monoidal $(\infty,n)$-category with duals, and let 
$Z_0: \Bord_{n-1} \rightarrow \calC$ be a symmetric monoidal functor.
The following types of data are equivalent:
\begin{itemize}
\item[$(1)$] Symmetric monoidal functors $Z: \Bord_{n}^{\frun} \rightarrow \calC$
extending $Z_0$.
\item[$(2)$] Nondegenerate $\OO(n)$-equivariant $n$-morphisms
$\eta: {\bf 1} \rightarrow Z_0(S^{n-1})$ in $\calC$.
\end{itemize}
\end{theorem}

\begin{theorem}\label{postwish}
The forgetful functor $\Bord_{n}^{\frun} \rightarrow \Bord_{n}$ is an equivalence of
$(\infty,n)$-categories.
\end{theorem}

The remainder of this section is devoted to a proof of Theorem \ref{swisher11}; we will
defer the proof of Theorem \ref{postwish} until \S \ref{kumma}.

\begin{remark}
Theorem \ref{postwish} is equivalent to the assertion that for every morphism
$B \in \calB(M)$, the space $\FrFun(B)$ of framed functions on $B$ is contractible.
This was conjectured by Igusa, who proved the weaker result that $\FrFun(B)$ is highly connected
(Theorem \ref{conig}). In \S \ref{kumma} we will deduce Theorem \ref{postwish} by combining
Igusa's connectivity result with deformation-theoretic arguments. This provides a proof that
each $\FrFun(B)$ is contractible, but the proof is very indirect: it uses in an essential way that
the spaces $\FrFun(B)$ can be packaged together (as $M$ and $B$ vary) to form an
$(\infty,n)$-category $\Bord^{\frun}_{n}$. It is likely possible to verify the contractibility
of $\FrFun(B)$ in a more direct way (the contractibility can be regarded as an instance of
Gromov's {\it $h$-principle}), which would eliminate the need to consider the cohomological formalism
described in \S \ref{kumma}.
\end{remark}

The advantage of $\Bord_{n}^{\frun}$ over $\Bord_{n}$ is that the former $(\infty,n)$-category evidently admits a description by ``generators and relations'', corresponding to the possible behaviors of a generalized Morse function near a critical point. However, the situation is still fairly complicated because critical points can appear with arbitrary Morse index. It is therefore convenient to consider these indices one at a time.

\begin{definition}
Let $M$ be a closed $(n-2)$-manifold, let $B$ be a $1$-morphism in $\calB(M)$, and let $k$ be an integer. We will say that a framed function $f$ on $B$ is {\it $k$-typical} if, for every critical
point $b$ of $f$, one of the following conditions holds:
\begin{itemize}
\item[$(1)$] The function $f$ has a nondegenerate critical point at $b$ of index $\leq k$.
\item[$(2)$] The function $f$ has a birth-death singularity at $b$ of index $< k$.
\end{itemize}
\end{definition}

We now define an $(\infty,1)$-category $\calB(M)$ informally as follows: 
\begin{itemize}
\item[$(i)$] The objects of $\calB_k(M)$ are $(n-1)$-manifolds $X$ equipped with an identification
$\bd X \simeq M$.
\item[$(ii)$] For every pair of objects $X,X' \in \calB(M)$, the mapping object
$\OHom_{\calB_k(M)}(X,X')$ is a classifying space for pairs $(B,f)$, where $B$ is a bordism from $X$ to $X'$ (which is trivial on $M$) and $f$ is a $k$-typical framed function on $B$.
\end{itemize}
For every integer $k$, we can view $\calB_k$ as a lax symmetric monoidal functor
$\Omega^{n-2} \Bord_{n-1} \rightarrow \Cat_{(\infty,1)}$, which (by virtue of Corollary \ref{isa})
determines a symmetric monoidal functor $\Bord_{n-1} \rightarrow \calF_k$.
Since every $k$-typical framed function is also $k'$-typical for $k' \geq k$, we obtain a sequence of
symmetric monoidal $(\infty,n)$-categories and functors 
$$ \ldots \rightarrow \calF_{-1} \rightarrow \calF_0 \rightarrow \ldots \rightarrow \calF_{n-1}
\rightarrow \calF_{n} \rightarrow \ldots$$

\begin{example}\label{fab}
If $k \geq n$, then every framed function on an $n$-manifold $B$ is $k$-typical. Consequently, 
we obtain a canonical equivalence $\calF_{k} \simeq \Bord^{\frun}_{n}$.
\end{example}

\begin{example}\label{fba}
If $k < 0$, then a framed function $f: B \rightarrow [0,1]$ is $k$-typical if and only if $f$ has no critical points. In this case, $f$ exhibits $B$ as a fiber bundle over the interval $[0,1]$, and determines
a diffeomorphism of $X = f^{-1} \{0\}$ with $X' = f^{-1} \{1\}$ (which is well-defined up to isotopy).
It follows that $\calF_{k}$ can be identified with the $(\infty,n-1)$-category $\Bord_{n-1}$, in which
$n$-morphisms are given by diffeomorphisms.
\end{example}

We begin the proof of Theorem \ref{swisher11} by analyzing the $(\infty,n)$-category
$\calF_0$. We observe that a framed function $f: B \rightarrow [0,1]$ is
$0$-typical if and only if it is a Morse function, and every critical point of $f$ has index $0$.
The Morse Lemma implies that near each critical point $b$ of $B$, we can choose local coordinates
$x_1, \ldots, x_n$ such that $f$ is given by the formula
$$ f(x_1, \ldots, x_n ) = b + x_1^2 + \ldots + x_n^2.$$
This choice of coordinates is not unique: however, it is unique up to a contractible space of choices,
once we fix an orthonormal frame for the tangent space $T_{B,b}$. One can use this reasoning
to prove the following result:

\begin{claim}\label{caboose}
The functor $\calB_0: \Omega^{n-2} \Bord_{n-1} \rightarrow \Cat_{(\infty,1)}$ is freely
generated $($as a lax symmetric monoidal functor$)$ by a single $\OO(n)$-equivariant
$1$-morphism $\emptyset \rightarrow S^{n-1}$ in $\calB_0(\emptyset)$, corresponding
to the pair $(D^n, f)$ where $D^n = \{ (x_1, \ldots, x_n) \in \R^n: x_1^2 + \ldots + x_n^2 \leq \frac{1}{2} \}$
and $f: D^n \rightarrow [0,1]$ is given by the formula $f(x_1, \ldots, x_n) = \frac{1}{2} + x_1^2 + \ldots + x_n^2$.
\end{claim}

\begin{corollary}\label{preswish}
The $(\infty,n)$-category $\calF_0$ is freely generated $($as a symmetric monoidal $(\infty,n)$-category$)$
by $\Bord_{n-1}$ together with a single $\OO(n)$-equivariant $n$-morphism $\emptyset \rightarrow S^{n-1}$, corresponding to the disk $D^n$ $($equipped with the framed function
described in Claim \ref{caboose}$)$. In other words, if $\calC$ is any symmetric monoidal $(\infty,n)$-category, then giving a symmetric monoidal functor $Z: \calF_0 \rightarrow \calC$ is equivalent
to giving a symmetric monoidal functor $Z_0: \Bord_{n-1} \rightarrow \calC$ together
with an $\OO(n)$-equivariant $n$-morphism $\eta: {\bf 1} \rightarrow Z_0(S^{n-1})$.
\end{corollary}

Corollary \ref{preswish} bears a strong resemblance to Theorem \ref{swisher11}. However,
it differs in two important respects:
\begin{itemize}
\item[$(a)$] We do not need to assume that the target $(\infty,n)$-category $\calC$ has duals.
\item[$(b)$] The morphism $\eta: {\bf 1} \rightarrow Z_0( S^{n-1} )$
need not be nondegenerate.
\end{itemize}

To deduce Theorem \ref{swisher11} from Corollary \ref{preswish}, we will prove the following:

\begin{lemma}\label{lak}
Let $\calC$ be a symmetric monoidal $(\infty,n)$-category. Then the forgetful functor
$\Fun^{\otimes}( \calF_1, \calC) \rightarrow \Fun^{\otimes}(\calF_0, \calC)$ is fully faithful,
and its essential image consists of precisely those functors $Z: \calF_0 \rightarrow \calC$
such that the $n$-morphism $Z( D^n) : Z(\emptyset) \rightarrow Z(S^{n-1})$ is nondegenerate.
\end{lemma}

\begin{lemma}\label{lak2}
Let $\calC$ be a symmetric monoidal $(\infty,n)$-category with duals.
For $2 \leq k \leq n$, the forgetful functor $\Fun^{\otimes}( \calF_k, \calC)
\rightarrow \Fun^{\otimes}( \calF_{k-1}, \calC)$ is an equivalence.
\end{lemma}

In other words, a functor $Z: \calF_0 \rightarrow \calC$ can be extended (in an essentially
unique fashion) to $\calF_1$ if and only if corresponds to a nondegenerate
$n$-morphism $\eta: {\bf 1} \rightarrow Z(S^{n-1})$. Then, if $\calC$ has duals, we
can extend $Z$ uniquely over each successive step of the filtration
$\calF_1 \rightarrow \calF_2 \rightarrow \ldots \rightarrow \calF_n$.

To prove Lemmas \ref{lak} and \ref{lak2}, we need an analogue of Claim \ref{caboose}
which describes the passage from $\calF_{k-1}$ to $\calF_{k}$ for $0 <  k \leq n$.
In other words, we want to describe how to obtain the functor $\calB_k: \Omega^{n-2} \Bord_{n-1} \rightarrow \Cat_{(\infty,1)}$ is obtained from $\calB_{k-1}$ by adjoining ``generators and relations''.
Once again, this question can be addressed by describing the behavior
of a framed function $f: B \rightarrow [0,1]$ near points $b \in B$ where $f$ is $k$-typical,
but not $(k-1)$-typical. There are two possibilities for the behavior of $f$:
\begin{itemize}
\item[$(1)$] The function $f$ can have a nondegenerate critical point of index
$k$ near the point $b \in B$. In this case, we can choose local coordinates
$x_1, \ldots, x_n$ for $b$ at $B$ so that $f$ admits an expression
$$ f(x_1, \ldots, x_n) = f(b) - x_1^2 - \ldots - x_k^2 + x_{k+1}^{2} + \ldots + x_n^2.$$
Moreover, since $f$ is a {\it framed} function, it comes equipped with a collection
of tangent vectors $$v_1, \ldots, v_k \in T_{B,b}$$ satisfying
$$ H(v_i, v_j) = \begin{cases} 0 & \text{if } i \neq j \\
-2 & \text{if } i = j,\end{cases}$$
where $H$ denotes the Hessian of $f$ at $b$. Without loss of generality, we may assume
that $v_i = \frac{\bd}{\bd x_i}$. The choice of coordinates $x_1, \ldots, x_n$ is not unique.
For example, any element of $\OO(n-k)$ determines a linear coordinate change (fixing each $x_i$ for $i \leq k$) which leaves the function $f$ invariant. However, this is essentially the only source of nonuniqueness: one can show that the group of local coordinate changes which respect both
$f$ and the vectors $v_i$ is homotopy equivalent to $\OO(n-k)$.

\item[$(2)$] The function $f$ can have a birth-death singularity of index
$(k-1)$ near the point $b \in B$. In this case, we can choose local coordinates $x_1, \ldots, x_n$
near $b$ so that $f$ can be written
$$ f(x_1, \ldots, x_n) = f(b) - x_1^2 - \ldots - x_{k-1}^2 + x_k^3 + x_{k+1}^2 + \ldots + x_n^2.$$
Because $f$ is a framed function, it comes equipped with a collection of tangent vectors
$v_1, \ldots, v_{k-1} \in T_{B,b}$ satisfying
$$ H(v_i, v_j) = \begin{cases} 0 & \text{if } i \neq j \\
-2 & \text{if } i = j,\end{cases}$$
where $H$ denotes the Hessian of $f$ at $b$. Without loss of generality, we may assume that
these tangent vectors are given by $v_i = \frac{\bd}{\bd x_i}$. The choice of coordinates is
not unique, but the relevant group of coordinate changes is again homotopy equivalent
to $\OO(n-k)$.
\end{itemize}

We now translate the local pictures described in $(1)$ and $(2)$ into categorical terms.
If $f$ has a nondegenerate critical point at $b \in B$ then we can locally identify $(B,f)$ with the pair
$(U,f_0)$, where $U = \{ (x_1, \ldots, x_n) \in \R^{n}: (x_1^2 + \ldots + x_n^2 \leq 1)
\wedge ( \frac{-1}{2} \leq -x_1^2 - \ldots - x_k^2 + x_{k+1}^{2} + \ldots + x_n^2
\leq \frac{1}{2} ) \}$ and $f_0(x_1, \ldots, x_n) = \frac{1}{2} - x_1^2 - \ldots - x_k^2
+ x_{k+1}^2 + \ldots x_{n}^2$. The pair $(U, f_0)$ determines a $1$-morphism
from $S^{k-1} \times D^{n-k}$ to $D^k \times S^{n-k-1}$ in the $(\infty,1)$-category $\calB_k( S^{k-1} \times S^{n-k-1})$, which we will denote by $\alpha_k$. We can informally summarize the situation as follows: the existence of nondegenerate critical points of index $k$ contributes a $1$-morphism
$\alpha_k$ to $\calB_{k}( S^{k-1} \times S^{n-k-1})$, which is equivariant with respect to the orthogonal
group $\OO(n-k)$. 

\begin{notation}\label{swin}
To avoid unnecessarily cumbersome notation in the arguments which follow, we will identify
$\alpha_{k}$ with its image in $\calB_{k'}( S^{k-1} \times S^{n-k-1} )$ for $k' \geq k$.
We will use this notation also in the case $k=0$, in which case
$\alpha_k$ corresponds to the $1$-morphism described in Claim \ref{caboose}.
\end{notation}

The analogous assertion for birth-death critical points is more complicated. First of all, a generic
generalized Morse function $f: B \rightarrow \R$ will not admit any birth-death critical points at all. Instead birth-death critical points appear generically at isolated points $(b,t)$ for a {\em family} of
generalized Morse functions $\{ f_t: B \rightarrow \R \}_{t \in [-1,1]}$. Suppose we are given such
a family which has a birth-death singularity of index $(k-1)$ at the point $(b,0)$. One can show that
for a suitable choice of local coordinates, we have the formula
$$ f_{t}(x_1, \ldots, x_n) = f_0(b) - x_1^2 - x_2^2 - \ldots - x_{k-1}^2 + x_k^3 - t x_k
+ x_{k+1}^2 + \ldots + x_{n}^2.$$
For $t < 0$, this function has no critical points. For $t > 0$, it has two critical points:
namely, the points where $x_k = \pm \sqrt{ \frac{t}{3}}$ and the other coordinates vanish.
These critical points have index $(k-1)$ (when $x_k$ is positive and the value of the function
$f$ is less than $f(b)$) and index $k$ (when $x_k$ is negative and the value of the function
$f$ is greater than $f(b)$) respectively. We can interpret this family of functions as giving us a
$2$-morphism in $\calB_{k}(S^{n-2})$. To describe the situation more precisely, we need to introduce a bit of notation.

Choose small disjoint open balls $V_{+}, V_{-} \subseteq S^{k-1}$ and
$W_{+}, W_{-} \subseteq S^{n-k}$. Set $$X = (S^{k-1} \times S^{n-k}) - (V_{-} \times W_{-}),$$
so we can regard $X$ as an object of $\calB_{k}( S^{n-2})$ (here we implicitly smooth the corners
of the product disk $\overline{V}_{-} \times \overline{W}_{-}$). We note that
$X - (S^{k-1} \times W_{+})$ can be regarded as a bordism from $\bd X = S^{n-2}$ to
$\bd( S^{k-1} \times \overline{W}_{+}) \simeq S^{k-1} \times S^{n-k-1}$, which we identify with a
$1$-morphism $\delta: S^{k-1} \times S^{n-k-1} \rightarrow S^{n-2}$ in $\Omega^{n-2} \Bord_{n-1}$. This $1$-morphism induces a functor $\delta_{!}: \calB_{k}( S^{k-1} \times S^{n-k-1}) \rightarrow
\calB_{k}( S^{n-2})$. Similarly, we can regard $X - ( V_{+} \times S^{n-k})$ as defining
a $1$-morphism $\epsilon: S^{k-2} \times S^{n-k} \rightarrow S^{n-2}$ in $\Omega^{n-2} \Bord_{n-1}$, which determines a functor $\epsilon_{!}: \calB_{k}( S^{k-2} \times S^{n-k-1})
\rightarrow \calB( S^{n-2})$. In particular, we can apply $\delta_{!}$ to $\alpha_{k-1}$ and
$\epsilon_{!}$ to $\alpha_{k}$, to obtain a diagram
$$ D^{n-1} \stackrel{\delta_{!} \alpha_{k-1}}{\longrightarrow} X \stackrel{\epsilon_{!} \alpha_{k}}{\longrightarrow} D^{n-1}$$
in the $(\infty,1)$-category $\calB_{k}( S^{n-2})$. The composition of these $1$-morphisms
corresponds to a $1$-morphism $(D^{n}, g): D^{n-1} \rightarrow D^{n-1}$ in
$\calB_k( S^{n-2})$, where $g$ is a framed function which is given locally by  
$$ f_{t}(x_1, \ldots, x_n) = f_0(b) - x_1^2 - x_2^2 - \ldots - x_{k-1}^2 + x_k^3 - t x_k
+ x_{k+1}^2 + \ldots + x_{n}^2$$
for positive values of $t$. The existence of such a family connecting $g$ to a function without critical points implies that the composition $\epsilon_{!}(\alpha_{k}) \circ \delta_{!}(\alpha_{k-1})$
is isomorphic to the identity $\id_{ D^{n-1} }$ in $\calB_{k}( S^{n-2})$.

We can now state the higher-index analogue of Claim \ref{caboose}:

\begin{claim}\label{laboose}
For $0 < k \leq n$, the lax symmetric monoidal functor $\calB_{k}: \Omega^{n-2} \Bord_{n-1} \rightarrow \Cat_{(\infty,1)}$ is freely generated from $\calB_{k-1}$ by the following data:
\begin{itemize}
\item[$(1)$] An $\OO(n-k)$-equivariant $1$-morphism
$\alpha_{k}: S^{k-1} \times D^{n-k} \rightarrow D^{k} \times S^{n-k-1}$ in
$\calB_{k}( S^{k-1} \times S^{n-k-1})$ $($corresponding to a handle attachment of index $k${}$)$.
\item[$(2)$] An $\OO(n-k)$-equivariant $2$-morphism
$\beta_{k}: \id_{ D^{n-1}} \simeq \epsilon_{!}(\alpha_k) \circ \delta_{!}( \alpha_{k-1})$
in $\calB_{k}(S^{n-2})$ $($corresponding to cancellation of handles of indices $k$ and $k-1${}$)$.
\end{itemize}
\end{claim} 

We can regard Claim \ref{laboose} as giving descriptions of the lax symmetric monoidal
functors $\calB_k$ by generators and relations. Our next step is to translate these into
descriptions of the $(\infty,n)$-category $\calF_{k}$ by generators and relations. Of course,
there is a naive way to go about this: by construction, $1$-morphisms from $X$ to $X'$ in
$\calB_k(M)$ can be identified with $n$-morphisms from $X: {\bf 1} \rightarrow M$
to $X': {\bf 1} \rightarrow M$ in the $(\infty,n)$-category $\calF_k$. However, other translations
are also possible. The results described in \S \ref{unf}, which assert that
$\calF_k$ is completely determined by the lax symmetric monoidal functor
$\calB_{k}: \Omega^{n-2} \Bord_{n-1} \rightarrow \Cat_{(\infty,1)}$, are possible
precisely because the $(\infty,n)$-category $\calF_k$ packages a great deal of information
in a redundant way. For example, if $K$ is any closed $(n-3)$-manifolds bounding a pair
of $(n-2)$-manifolds $M_{+}$ and $M_{-}$, then $\calB_k(M_{-} \coprod_{K} M_{+})$ can
be identified with the $(\infty,1)$-category $\OHom_{\calC}( M_{-}, M_{+} )$
where $\calC= \bHom_{ \Omega^{n-2} \calF_k}(\emptyset, K)$. We will apply this observation to obtain a more subtle interpretation of the presentation of Claim \ref{laboose}. First, we need to introduce a bit more terminology.

\begin{notation}\label{scaz}
Let $\calC$ be a symmetric monoidal $(\infty,n)$-category, and let
$K$ be an object of $\Omega^{m-1} \calC$, where $m < n$. We
let $\Omega^{m}_{K} \calC$ denote the $(\infty, n-m)$-category
$\OHom_{\Omega^{m-1} \calC}( {\bf 1}, K )$. Note that if
$K$ is the unit object of $\Omega^{m-1} \calC$, then 
$\Omega^{m}_{K} \calC = \Omega^{m} \calC$. We will sometimes use this notation
even when $m = 0$: in this case, we will implicitly assume that $K$ is the unique object of
the delooping $B \calC$ and define $\Omega^{m}_{K} \calC = \calC$.

In particular, if $K$ is a closed $(n-3)$-manifold, then we can consider an
$(\infty,2)$-category $\Omega^{n-2}_{K} \calF_k$. 
We can identify objects of $\Omega^{n-2}_{K} \calF_k$ with $(n-2)$-manifolds having
boundary $K$, and morphisms of $\Omega^{n-2}_{K} \calF_k$ with bordisms between such $(n-2)$-manifolds.
\end{notation}

For $k \geq 1$, we have in particular two objects $$S^{k-2} \times D^{n-k},
D^{k-1} \times S^{n-k-1} \in \Omega^{n-2}_{ S^{k-2} \times S^{n-k-1} } \calF_k$$ which we will
denote by $x$ and $y$, respectively. The disk $D^{n-1} \simeq D^{k-1} \times D^{n-k}$ can be interpreted as both a morphism $f: x \rightarrow y$ and as a morphism $g: y \rightarrow x$ in $\Omega^{n-2}_{S^{k-2} \times S^{n-k-1}} \calF_k$. The $(\infty,1)$-category
$\OHom_{ \Omega^{n-2}_{S^{k-2} \times S^{n-k-1}} \calF_k }(x, x)$ can be identified with
$\calB_k( S^{k-2} \times S^{n-k})$. Under this identification, the identity $1$-morphism
$\id_{x}$ corresponds to the object $S^{k-2} \times D^{n-k+1} \in \calB_k( S^{k-2} \times S^{n-k-1})$, while the composition $g \circ f$ corresponds to the object $D^{k-1} \times S^{n-k}
\in \calB_{k}( S^{k-2} \times S^{n-k-1})$. We may therefore interpret the handle-attachment
$1$-morphism $\alpha_{k-1}: S^{k-2} \times D^{n-k+1} \rightarrow D^{k-1} \times S^{n-k}$
in $\calB_k( S^{k-2} \times S^{n-k-1})$ as giving us a $2$-morphism
$u: \id_{x} \rightarrow g \circ f$ in the $(\infty,2)$-category $\Omega^{n-2}_{S^{k-2} \times S^{n-k-1}} \calF_k$.

Using the same reasoning, we obtain an equivalence of $(\infty,1)$-categories
$$\OHom_{ \Omega^{n-2}_{S^{k-2} \times S^{n-k-1}} \calF_k }(y, y) \simeq \calB_k( S^{k-1} \times S^{n-k}).$$ Under this equivalence, the identity map
$\id_y$ corresponds to the object $D^{k} \times S^{n-k-1}$, while the
composition $f \circ g$ corresponds to the object $S^{k-1} \times D^{n-k}$.
The $1$-morphism $\alpha_k: S^{k-1} \times D^{n-k} \rightarrow
D^{k} \times S^{k-1}$ in $\calB_k( S^{k-1} \times S^{n-k})$ corresponds to
a $2$-morphism $v: f \circ g \rightarrow \id_{y}$ in the $(\infty,2)$-category
$\Omega^{n-2}_{ S^{k-2} \times S^{n-k-1}} \calF_k$. 

Finally, we observe that there is an equivalence of $(\infty,1)$-categories
$$\OHom_{ \Omega^{n-2}_{S^{k-2} \times S^{n-k-1}} \calF_k }(x, y) \simeq \calB_{k}( S^{n-2}).$$ Under this equivalence, the map $f$ corresponds to the
$(n-1)$-disk $D^{n-1} \in \calB_{k}( S^{n-2})$, while the composition $f \circ g \circ f$
corresponds to the $(n-1)$-manifold $X = S^{k-1} \times S^{n-k} - V_{-} \times W_{-}$.
The $1$-morphisms $\epsilon_{!}( \alpha_k)$ and $\delta_{!}( \alpha_{k-1})$ appearing
in the statement of Claim \ref{laboose} correspond to the $2$-morphisms
$$f \stackrel{\id_f \times u}{\longrightarrow} f \circ g \circ f \quad \quad f \circ g \circ f \stackrel{v \times \id_f}{\longrightarrow} f$$
induces by $u$ and $v$. Consequently, the $2$-morphism $\beta_{k}: \id_{ D^{n-1}} \simeq \epsilon_{!}(\alpha_k) \circ \delta_{!}( \alpha_{k-1})$ can be regarded as an isomorphism 
$\gamma: \id_{f} \simeq (v \times \id_f) \circ (\id_f \times u)$ between $2$-morphisms of 
$\Omega^{n-2}_{S^{k-2} \times S^{n-k-1}} \calF_k$. In particular, the existence of
$\gamma$ implies that $v$ is upper compatible with $u$ in the homotopy $2$-category
$\tau_{\leq 2} \Omega^{n-2}_{ S^{k-2} \times S^{n-k-1}} \calF_k$ (see the discussion preceding
Lemma \ref{simptom}). 

Applying the same reasoning with the roles of $x$ and $y$ switched, we
deduce that there exists another $2$-morphism $v': f \circ g \rightarrow \id_y$ which is lower
compatible with $u$. It is not clear immediately that $v = v'$ (this is a bit subtle if we keep careful track of the framed functions), but Lemma \ref{simptom} implies that $v$ is isomorphic to $v'$, so
that $u$ is the unit of an adjunction between $u$ and $v$ in $\Omega^{n-2}_{ S^{k-2} \times S^{n-k-1}} \calF_k$.

Now, if we assume that $u: \id_{x} \rightarrow g \circ f$ is the unit of an adjunction
in $\Omega^{n-2}_{ S^{k-2} \times S^{n-k-1}} \calF_{k}$, then it
has a compatible counit $v_0: f \circ g \rightarrow \id_{y}$. For any map
$v: f \circ g \rightarrow \id_{y}$, giving an isomorphism 
$\gamma: \id_{f} \simeq (v \times \id_f) \circ (\id_f \times u)$ is equivalent to giving an
isomorphism $v_0 \simeq v$. Consequently, the pair $(v,\gamma)$ is uniquely determined
up to isomorphism. We can summarize our discussion as follows:

\begin{proposition}\label{simp}
Let $\calC$ be a symmetric monoidal $(\infty,n)$-category, let $1 \leq k \leq n$ and let $Z_0: \calF_{k-1} \rightarrow \calC$ be a symmetric monoidal functor. Let $C = Z_0( S^{k-2} \times S^{n-k-1})$
and let $\calC'$ denote the $(\infty,2)$-category $\Omega^{n-2}_{C} \calC$. Applying
$Z_0$ to $D^{k-1} \times S^{n-k-1}$, $S^{k-2} \times D^{n-k}$, and $D^{n-1}$, we obtain
objects $x,y \in \calC'$ and $1$-morphisms $f: x \rightarrow y$, $g: y \rightarrow x$. Moreover,
applying $Z_0$ to $\alpha_{k-1}$, we obtain a $2$-morphism $u: \id_{x} \rightarrow g \circ f$.
The functor $Z_0$ can be extended to a symmetric monoidal functor $Z: \calF_k \rightarrow \calC$
if and only if $u$ is the unit for an adjunction between $g$ and $f$ in $\calC'$. Moreover, if this
extension exists, then it is unique up to canonical isomorphism.
\end{proposition}

In the special case where $k=1$, Proposition \ref{simp} reduces to the statement of Lemma \ref{lak}.
Lemma \ref{lak2} follows immediately from Proposition \ref{simp} and the following:

\begin{proposition}\label{lak3}
Let $\calC$ be a symmetric monoidal $(\infty,n)$-category with duals, let
$2 \leq k \leq n$, and let $Z_0: \calF_{k-1} \rightarrow \calC$ be a symmetric monoidal
functor. Let $\calC'$, $x$, $y$, $f$, $g$, and $u: \id_{x} \rightarrow g \circ f$ be as
in Proposition \ref{simp}. Then $u$ is the unit of an adjunction between $g$ and $f$ in the
$(\infty,2)$-category $\calC'$. 
\end{proposition}

The proof of Proposition \ref{lak3} rests on the following bit of category theory:

\begin{lemma}[Exchange Principle]\label{swath}
Let $f: x \rightarrow y$ and $f^{\dag}: y \rightarrow x$ be $1$-morphisms in a $3$-category $\calD$.
Let $u: \id_{x} \rightarrow f^{\dag} \circ f$ and $u': \id_{y} \rightarrow f \circ f^{\dag}$ be
$2$-morphisms in $\calD$, and let $\alpha: (\id_{f^{\dag}} \times u') \rightarrow (u \times \id_{f^{\dag}})$
be a $3$-morphism between the $2$-morphisms
$(\id_{f^{\dag}} \times u'), (u \times \id_{f^{\dag}}): f^{\dag} \rightarrow f^{\dag} \circ f \circ f^{\dag}$. Assume that:
\begin{itemize}
\item[$(1)$] The $2$-morphism $u$ exhibits $f^{\dag}$ as a right adjoint of $f$. In particular, there
exists a compatible counit map $v: f \circ f^{\dag} \rightarrow \id_{y}$ which determines an
bijection $\Hom(\id_{f^{\dag}} \times u', u \times \id_{f^{\dag}}) \simeq
\Hom( u' \circ v, \id_{f \circ f^{\dag}})$; we let $\beta$ denote the image of $\alpha$ under this bijection.

\item[$(2)$] The $2$-morphism $u'$ exhibits $f^{\dag}$ as a left adjoint of $f$. In particular, there exists a compatible counit map $v': f^{\dag} \circ f \rightarrow \id_{x}$ which determines a bijection $\Hom(  \id_{f^{\dag}} \times u', u \times \id_{f^{\dag}}) \simeq
\Hom( \id_{f^{\dag} \circ f}, u \circ v')$; let $\gamma$ denote the image of $\alpha$ under this bijection.

\item[$(3)$] The $2$-morphisms $u$ and $v$ both admit left adjoints.
\end{itemize}
Then $\beta: u' \circ v \rightarrow \id_{f \circ f^{\dag}}$ is the counit of an adjunction between
$u'$ and $v$ if and only if $\gamma: \id_{ f^{\dag} \circ f} \rightarrow u \circ v'$ is the unit
of an adjunction between $u$ and $v'$. 
\end{lemma}

\begin{remark}
In the statement of Lemma \ref{swath}, the assumption that $f^{\dag}$ is both a right and a left adjoint to $f$ is not as strong as it might first appear. Suppose that $f$ admits a right adjoint
$f^{R}$, so we have unit and counit maps $u: \id_{x} \rightarrow f^{R} \circ f$
and $v: f \circ f^{R} \rightarrow \id_{y}$. If $u$ and $v$ admit left adjoints $u^{L}$ and
$v^{L}$, then $u^{L}$ and $v^{L}$ exhibit $f^{R}$ also as a left adjoint to $f$.
\end{remark}

\begin{proof}[Proof of Lemma \ref{swath}]
Let $u^{L}: f^{\dag} \circ f \rightarrow \id_{x}$ and $v^{L}: \id_{y} \rightarrow f \circ f^{\dag}$ be
left adjoints to $u$ and $v$, respectively. Assumption $(1)$ guarantees that $u$ and $v$ are compatible unit and counit maps which exhibit $f^{\dag}$ as a right adjoint to $f$, we conclude that $v^{L}$ and $u^{L}$ are compatible unit and counit maps which exhibit $f^{\dag}$ as left adjoint to $f$. 
We can therefore factor the map $u'$ as a composition
$$ \id_{y} \stackrel{ v^{L} }{\longrightarrow} f \circ f^{\dag} \stackrel{S \times \id_{f^{\dag}}}{\longrightarrow}
f \circ f^{\dag},$$ where $S: f \rightarrow f$ is a $2$-morphism in $\calD$ (which is well-defined up to canonical isomorphism). Assumption $(2)$ guarantees that $S$ is an isomorphism.
The identification $u' \simeq (S \times \id_{f^{\dag}}) \circ v^{L}$ induces an identification
$v' \simeq u^{L} \circ (\id_{f^{\dag}} \times S^{-1})$.

The $3$-morphism $\beta: u' \circ v \rightarrow \id_{ f \circ f^{\dag}}$ is classified by a $3$-morphism
$(S \times \id_{f^{\dag}}) \circ v^{L} \simeq u' \rightarrow v^{L}$, which is in turn determined by a
$3$-morphism $\beta': S \rightarrow \id_{f}$. Similarly, $\gamma: \id_{ f^{\dag} \circ f} \rightarrow u \circ v'$
is classified by a $3$-morphism $u^{L} \rightarrow v' \simeq u^{L} \circ ( \id_{f^{\dag}} \times S^{-1} )$, which is turn determined by a $3$-morphism $\gamma': \id_{f} \rightarrow S^{-1}$. We wish to prove
that $\beta'$ is an isomorphism if and only if $\gamma'$ is an isomorphism. To see this, it suffices
to observe that $\beta'$ is the image of $\gamma'$ under the equivalence of $2$-categories
$\OHom_{\calD}(f,f) \rightarrow \OHom_{\calD}(f,f)$ given by composition with $S$.
\end{proof}

\begin{proof}[Proof of Proposition \ref{lak3}]
Let $\calC$ be an $(\infty,n)$-category with duals, and let $Z_0: \calF_{k-1} \rightarrow \calC$
be a symmetric monoidal functor, where $2 \leq k \leq n$. For the sake of simplicity, we will assume
that $n \geq 4$ (the cases $n=2$ and $n=3$ can be handled by the same method, but require slight changes of notation). Let $C = Z_0( S^{k-3} \times S^{n-k-1}) \in \Omega^{n-4} \calC$, and let
$\calD = \tau_{\leq 3} \Omega^{n-3}_{C} \calC$. Let $x = Z_0( D^{k-2} \times S^{n-k-1})$ and $y = Z_0( S^{k-3} \times D^{n-k})$, regarded as objects of $\calD$. We observe that
$D^{k-2} \times D^{n-k}$ determines morphisms $f: x \rightarrow y$ and $f^{\dag}: y \rightarrow x$
in $\calD$. We have canonical identifications $\id_{x} \simeq Z_0( D^{k-1} \times S^{n-k-1})$
and $f^{\dag} \circ f \simeq Z_0( S^{k-2}, D^{n-k})$, so the product $D^{k-1} \times D^{n-k}$
determines a $2$-morphism $u: \id_{x} \rightarrow f^{\dag} \circ f$ in $\calD$. This
$2$-morphism is the unit an adjunction (the analogous statement is already true in the $3$-category $\Omega^{n-3}_{ S^{k-3} \times S^{n-k-1}} \calF_{k-1}$); let $v$ denote a compatible counit. Similarly, the product $D^{k-2} \times D^{n-k+1}$ determines a $2$-morphism $u': \id_{y} \rightarrow f \circ f^{\dag}$, which is again the unit of an adjunction and therefore has a compatible counit $v'$. Finally, we note that the $1$-morphism $\alpha_{k-1}$ in $\calB_{k-1}( S^{k-2} \times S^{n-k})$ (see the discussion preceding Notation \ref{swin}) determines a $3$-morphism $\alpha: (\id_{f^{\dag}} \times u') \rightarrow (u \times \id_{f^{\dag}} )$ in $\calD$, which induces $2$-morphisms
$\beta: u' \circ v \rightarrow \id_{f \circ f^{\dag}}$ and $\gamma: \id_{f^{\dag} \circ f} \rightarrow u \circ v'$ as in the statement of Lemma \ref{swath}. Since $\calC$ has duals,
every $2$-morphism in $\calD$ has a left adjoint. The discussion preceding Proposition \ref{simp} shows that $\beta$ is the counit of an adjunction (the analogous statement is already true
in the $(\infty,2)$-category $\Omega^{n-2}_{ S^{k-3} \times S^{n-k}} \calF_{k-1}$). It follows from Lemma \ref{swath} that $\beta$ is the unit of an adjunction, as desired. 
\end{proof}

\subsection{Obstruction Theory}\label{kumma}

Our goal in this section is to complete the proof of the cobordism hypothesis by
establishing Theorem \ref{postwish}, which asserts that the forgetful functor
$\Bord_{n}^{\frun} \rightarrow \Bord_{n}$ is an equivalence of (symmetric monoidal)
$(\infty,n)$-categories. 
To prove this, we need a method for testing when a functor between
$(\infty,n)$-categories is an equivalence. In the case $n=0$, we have the following criterion for detecting homotopy equivalences:

\begin{proposition}\label{curwise}
Let $f: X \rightarrow Y$ be a continuous map of CW complexes. Then
$f$ is a homotopy equivalence if and only if the following conditions are satisfied:
\begin{itemize}
\item[$(i)$] The map $f$ induces an equivalence of fundamental groupoids
$\pi_{\leq 1} X \rightarrow \pi_{\leq 1} Y$.
\item[$(ii)$] For every local system of abelian groups $\calA$ on $Y$ and every $n \geq 0$, the induced map on cohomology $\HH^{n}(Y; \calA) \rightarrow \HH^{n}(X; f^{\ast} \calA)$ is an isomorphism.
\end{itemize}
\end{proposition}

In this section we will describe an analogue of Proposition \ref{curwise} in the $(\infty,n)$-categorical setting and apply this analogue to prove Theorem \ref{postwish}.
We begin by sketching a proof of Proposition \ref{curwise} itself.

Suppose that $f: X \rightarrow Y$ is a map of CW complexes satisfying conditions $(i)$ and
$(ii)$ of Proposition \ref{curwise}. We wish to prove that $f$ is a homotopy equivalence. To prove this, we will show the following: for every topological space $Z$, composition with $f$ induces a weak homotopy equivalence of mapping spaces $\phi: \bHom( Y, Z) \rightarrow \bHom(X,Z)$
(in particular, it will follow that $\pi_0 \bHom(Y,Z) \simeq \pi_0 \bHom(X,Z)$, so that
$X$ and $Y$ corepresent the same functor on the homotopy category of topological spaces).
The idea is to break $Z$ up into simple pieces for which the map $\phi$ is easy to analyze.
First, we need to review a few basic ideas from homotopy theory.

\begin{definition}\label{toptrunk}
Let $K$ be a topological space and $n$ a nonnegative integer. We say that
$K$ is {\it $n$-truncated} if $\pi_{i}(K,x)$ vanishes, for every point $x \in K$ and
all $i > n$. We will say that a continuous map $p: Z \rightarrow K$ {\it exhibits $K$ as an $n$-truncation
of $X$} if $K$ is $n$-truncated, and $p$ induces isomorphisms $\pi_i(Z,z) \rightarrow \pi_i(K,f(z))$
for all $z \in Z$ and all $i \leq n$.
\end{definition}

For every topological space $Z$ and every $n \geq 0$, there exists an $n$-truncation
$p: Z \rightarrow K$ of $Z$. Moreover, $K$ is uniquely determined up to weak homotopy equivalence.
We can construct $K$ functorially in $Z$: it can be obtained by successively gluing on
cells of each dimension $m > n+1$ to kill all of the higher homotopy groups of $Z$.
We will generally denote an $n$-truncation of $Z$ by $\tau_{\leq n} Z$.

\begin{example}
Let $X$ be a CW complex. Then there is a continuous map $f: X \rightarrow \pi_0 X$, which collapses every connected component of $X$ to a point. This map exhibits $\pi_0 X$ as a $0$-truncation of $X$.
\end{example}

Allowing the integer $n$ to vary, we can associate to every topological space $Z$ its
{\it Postnikov tower}
$$ \ldots \rightarrow \tau_{\leq n} Z \rightarrow \tau_{\leq n-1} Z \rightarrow \ldots \rightarrow \tau_{\leq 0} Z.$$
The truncations appearing in this tower can be regarded as successively better approximations to the space $Z$. The space $Z$ itself can be recovered (up to weak homotopy equivalence) by forming the homotopy inverse limit of the tower. If $X$ and $Y$ are CW complexes, then the mapping spaces
$\bHom(X,Z)$ and $\bHom(Y,Z)$ can also be recovered (again up to weak homotopy equivalence)
as the homotopy inverse limits of the towers $\{ \bHom(X, \tau_{\leq n} Z) \}_{n \geq 1}$ and
$\{ \bHom(Y, \tau_{\leq n} Z) \}_{n \geq 1}$. Consequently, to prove that a map $f: X \rightarrow Y$
induces a weak homotopy equivalence $\bHom(Y,Z) \rightarrow \bHom(X,Z)$, it will suffice to show that
$f$ induces a weak homotopy equivalence $\bHom(Y, \tau_{\leq n} Z) \rightarrow \bHom( X, \tau_{\leq n}  Z)$ for each $n \geq 1$.

The proof of Proposition \ref{curwise} now proceeds by induction on $n$. When $n=1$, the space
$\tau_{\leq n} Z$ is $1$-truncated: in other words, it is completely determined (up to weak homotopy equivalence) by its fundamental groupoid. The mapping spaces $\bHom(X, \tau_{\leq 1} Z)$ is likewise $1$-truncated: it can be identified with the classifying space of the groupoid of functors 
$\pi_{\leq 1} X \rightarrow \pi_{\leq 1} Z$. Similarly, $\bHom( Y, \tau_{\leq 1} Z)$ is equivalent to the
classifying space of the groupoid of functors from $\pi_{\leq 1} Y$ into $\pi_{\leq 1} Z$. 
Hypothesis $(i)$ of Proposition \ref{curwise} guarantees that $f$ induces an equivalence of fundamental groupoids $\pi_{\leq 1} X \rightarrow \pi_{\leq 1} Y$; it follows that the induced map
$\bHom( Y, \tau_{\leq 1} Z) \rightarrow \bHom(X, \tau_{\leq 1} Z)$ is a weak homotopy equivalence.

Now suppose that $n > 1$, and suppose that the map $\bHom( Y, \tau_{\leq n-1} Z) \rightarrow
\bHom(X, \tau_{\leq n-1} Z)$ is a weak homotopy equivalence. We would like to prove that the map
$\bHom(Y, \tau_{\leq n} Z) \rightarrow \bHom(X, \tau_{\leq n} Z)$ is also a weak homotopy equivalence.
The idea is to take advantage of the fact that the spaces $\tau_{\leq n} Z$ and $\tau_{\leq n-1} Z$
are very similar. Without loss of generality, we may suppose that the map
$\tau_{\leq n} Z \rightarrow \tau_{\leq n-1} Z$ is a fibration; let $F_z$ denote the fiber of this map
taken over a point $z \in \tau_{\leq n-1} Z$. For every point $\overline{z} \in \tau_{\leq n} Z$ lying
over $z$, we obtain a long exact sequence of homotopy groups
$$ \ldots \rightarrow \pi_{k}(F_z, \overline{z}) \rightarrow \pi_{k}( \tau_{\leq n} Z, \overline{z} )
\rightarrow \pi_{k}( \tau_{\leq n-1} Z, z) \rightarrow \pi_{k-1}( F_z, \overline{z}) \rightarrow \ldots$$
It follows that the homotopy groups of $F_z$ are given by the formula
$$ \pi_{k}( F_z, \overline{z} ) \simeq \begin{cases} 0 & \text{if } k \neq n \\
\pi_n( \tau_{\leq n} Z, \overline{z} ) & \text{if } k = n. \end{cases}$$
In particular, each fiber $F_z$ is an Eilenberg-MacLane space $K(A_z,n)$ for some abelian group
$A_z$. The abelian group $A_z$ generally depends on the choice of point $z \in \tau_{\leq n-1} Z$.
However, this dependence is functorial: the construction $z \mapsto A_z$ determines a functor
from the fundamental groupoid $\pi_{\leq 1}(\tau_{\leq n-1} Z) \simeq \pi_{\leq 1} Z$
into the category of abelian groups. In other words, we can view the collection of abelian groups
$\{ A_z \}_{ z \in \tau_{\leq n-1} Z}$ as defining a {\em local system} of abelian groups on
the space $\tau_{\leq n-1} Z$.

The fibration $\tau_{\leq n} Z \rightarrow \tau_{\leq n-1} Z$ induces a fibration
$q: \bHom(Y, \tau_{\leq n} Z) \rightarrow \bHom( Y, \tau_{\leq n-1} Z)$. We can try to use this fibration
to compute the homotopy groups of the mapping space $\bHom( Y, \tau_{\leq n} Z)$. 
However, we must be careful in doing so, because the fibers of $q$ over different elements
of $\bHom(Y, \tau_{\leq n-1} Z)$ are generally different from one another. Let
$\overline{g}: Y \rightarrow \tau_{\leq n} Z$ be a continuous map, and let
$g: Y \rightarrow \tau_{\leq n-1} Z$ be the induced map. Then the fibration $q$ determines
a long exact sequence of homotopy groups
$$ \cdots \rightarrow \pi_{k+1}( \bHom(Y, \tau_{\leq n-1} Z),g) \rightarrow \HH^{n-k}(Y; \calA') \rightarrow
\pi_{k}(\bHom(Y, \tau_{\leq n} Z), \overline{g})
\rightarrow \pi_{k}( \bHom(Y, \tau_{\leq n-1} Z), g) \rightarrow \cdots$$
where $\calA' = g^{\ast} \calA$ denotes the pullback of the local system $\calA$ along
the map $g$. 

The morphism $f: X \rightarrow Y$ induces a map of long exact sequences
$$ \xymatrix{ \pi_{k+1}( \bHom(Y, \tau_{\leq n-1} Z), g) \ar[d] \ar[r]^{\phi_1} & \pi_{k+1}( \bHom(X, \tau_{\leq n-1} Z), g \circ f) \ar[d] \\
\HH^{n-k}(Y; \calA') \ar[d] \ar[r]^{\phi_2} & \HH^{n-k}(X; f^{\ast} \calA') \ar[d] \\
\pi_{k}(\bHom(Y, \tau_{\leq n} Z), \overline{g}) \ar[d] \ar[r]^{\phi_3} & \pi_{k}(\bHom(X, \tau_{\leq n} Z), \overline{g} \circ f) \ar[d] \\
\pi_{k}( \bHom(Y, \tau_{\leq n-1} Z), g) \ar[d] \ar[r]^{\phi_4} & \pi_{k}( \bHom(X, \tau_{\leq n-1} Z), g \circ f) \ar[d]  \\
\HH^{n+1-k}(Y; \calA') \ar[r]^{\phi_5}  & 
\HH^{n+1-k}(X; f^{\ast} \calA').} $$
The inductive hypothesis guarantees that $\phi_1$ and $\phi_4$ are isomorphisms, and
hypothesis $(ii)$ guarantees that $\phi_2$ and $\phi_5$ are isomorphisms. It follows
from the ``five lemma'' that $\phi_3$ is an isomorphism, so that the map
$\bHom( Y, \tau_{\leq n} Z) \rightarrow \bHom(X, \tau_{\leq n} Z)$ is a weak homotopy equivalence as desired.

\begin{warning}
The argument sketched above is not quite complete: we need to take special care with the above long exact sequence for small values of $k$ (where the relevant homotopy groups do not admit group structures). 
\end{warning}

We would now like to prove an analogue of Proposition \ref{curwise} in the setting of
higher category theory. The first step is to find the appropriate generalization of the theory of Postnikov towers.

\begin{definition}\label{sap}
Let $\calC$ be an $(\infty,n)$-category, and let $m \geq n$. We will say that a functor
$f: \calC \rightarrow \calD$ {\it exhibits $\calD$ as an $m$-truncation of $\calC$} if
$\calD$ is an $(m,n)$-category and the following condition is satisfied:
\begin{itemize}
\item[$(\ast)$] For any $(m,n)$-category $\calE$, composition with
$f$ induces an equivalence
$$ \Fun( \calD, \calE) \rightarrow \Fun( \calC, \calE).$$
\end{itemize}
\end{definition}

In other words, Definition \ref{sap} requires that $\calD$ be {\em universal} among
$(m,n)$-categories which admit a functor $\calC \rightarrow \calD$. It is clear that
if an $(\infty,n)$-category $\calC$ admits an $m$-truncation $\calD$, then $\calD$ is uniquely determined up to equivalence. We will denote this $m$-truncation by $\tau_{\leq m} \calC$.
To verify the existence of $\tau_{\leq m} \calC$, we use the following recursive construction:

\begin{construction}
Let $\calC$ be an $(\infty,n)$-category, and let $m \geq n$ be an integer. We will define
an $(m,n)$-category $\tau_{\leq m} \calC$ as follows:
\begin{itemize}
\item[$(a)$] The objects of $\tau_{\leq m} \calC$ are the objects of $\calC$.
\item[$(b)$] Given a pair of objects $X,Y \in \calC$, we define
$\OHom_{\tau_{\leq m} \calC}(X,Y) = \tau_{\leq m-1} \OHom_{\calC}(X,Y).$
\item[$(c)$] The composition of morphisms in $\tau_{\leq m} \calC$ is induced by the
composition of morphisms in $\calC$.
\end{itemize}
\end{construction}

\begin{remark}
More informally, we can describe the truncation $\tau_{\leq m} \calC$ of an $(\infty,n)$-category $\calC$ as follows. For $k < m$, the $k$-morphisms in $\tau_{\leq m} \calC$ are the same as the $k$-morphisms in $\calC$. For $k = m$, the $k$-morphisms in $\tau_{\leq m} \calC$ are {\em isomorphism classes} of
$k$-morphisms in $\calC$. For $k > m$, $\tau_{\leq m} \calC$ has only identity $k$-morphisms.
\end{remark}

\begin{remark}
In the case $n=0$, the notion of $m$-truncation of a topological space (Definition \ref{toptrunk})
and the notion of $m$-truncation of an $(\infty,n)$-category (Definition \ref{sap}) correspond
to one another, under the equivalence of Thesis \ref{cope3}.
\end{remark}

\begin{example}
Let $\calC$ be an $(\infty,n)$-category. Then the truncation $\tau_{\leq n} \calC$ coincides with
the homotopy $n$-category $\hn{\calC}$ described in Remark \ref{camblis}.
\end{example}

It follows from the above discussion that every $(\infty,n)$-category $\calC$ determines a 
{\it Postnikov tower}
$$ \ldots \rightarrow \tau_{\leq n+2} \calC \rightarrow \tau_{\leq n+1} \calC
\rightarrow \tau_{\leq n} \calC \simeq \hn{\calC}.$$ 
As in the topological case, we can recover $\calC$ (up to equivalence) as the homotopy
inverse limit of this tower. In practice, this Postnikov tower is a useful tool because it allows us to reduce questions about the $(\infty,n)$-category $\calC$ to questions about
the $n$-category $\hn{\calC}$ (a much less sophisticated object) and questions about the individual maps $\psi_m: \tau_{\leq m} \calC \rightarrow \tau_{\leq m-1} \calC$. To address the questions of the latter type, we would like to articulate a sense in which $\psi_m$ is close to being an isomorphism. In the case $n=0$, we saw that for $m \geq 2$, the homotopy fibers of $\psi_m$ were Eilenberg-MacLane spaces $K(A_z, m)$, where the functor $z \mapsto A_z$ determines a local system of abelian groups on the
base $\tau_{\leq m-1} \calC$. Our next goal is to formulate the appropriate higher categorical generalizations of these statements.

We will define, for each $(\infty,n)$-category $\calC$, an abelian category
$\Loc(\calC)$ of {\it local systems} (of abelian groups) on $\calC$. This construction will be
functorial in $\calC$: every functor $f: \calC \rightarrow \calD$ between $(\infty,n)$-categories
will induce a pullback functor $f^{\ast}: \Loc(\calD) \rightarrow \Loc(\calC)$. The definition
uses induction on $n$.

\begin{definition}
Let $\calC$ be an $(\infty,n)$-category. If $n=0$, then a {\it local system of abelian groups}
on $\calC$ is a functor from $\calC$ to the $($ordinary$)$ category of abelian groups.
If $n > 0$, then a {\it local system of abelian groups} on $\calC$ consists of the following data:
\begin{itemize}
\item[$(i)$] For every pair of objects $x,y \in \calC$, a local system $\calA_{x,y}$ of abelian
groups on the $(\infty,n-1)$-category $\OHom_{\calC}(x,y)$.
\item[$(ii)$] For every triple of objects $x,y,z \in \calC$, a map of local systems
$$ m_{x,y,z}: p_0^{\ast} \calA_{x,y} \times p_1^{\ast} \calA_{y,z} \rightarrow c^{\ast} \calA_{x,z}.$$
Here $p_0$ and $p_1$ denote the projection maps of 
$\OHom_{\calC}(x,y) \times \OHom_{\calC}(y,z)$ onto $\OHom_{\calC}(x,y)$ and
$\OHom_{\calC}(y,z)$, respectively, and $c$ the composition map
$\OHom_{\calC}(x,y) \times \OHom_{\calC}(y,z) \rightarrow \OHom_{\calC}(x,z)$.
The collection of maps $\{ m_{x,y,z} \}_{x,y,z \in \calC}$ is required to satisfy some natural
associativity conditions which we will not make explicit.
\end{itemize}
\end{definition}

Our next step is to define the cohomology of an $(\infty,n)$-category $\calC$ with coefficients in a local system $\calA \in \Loc(\calC)$. We begin by reviewing the classical case. If $X$ is a CW complex and
$A$ is an abelian group, then the cohomology group $\HH^{m}(X;A)$ can be described as the set
$[ X, K(A,m) ]$ of homotopy classes of maps from $X$ into an Eilenberg-MacLane space $K(A,m)$.
Equivalently, we can describe $\HH^{m}(X;A)$ as the set of homotopy classes of sections of the projection map $p: X \times K(A,m) \rightarrow X$. There is a generalization of this assertion
to the case of cohomology with coefficients in a local system $\calA$ on $X$: in this case, we
need to replace $p$ by a twisted fibration $q: K( \calA, m) \rightarrow X$ with the following properties:

\begin{itemize}
\item[$(a)$] There exists a section $s$ of $q$.
\item[$(b)$] For each $x \in X$, let $F_x$ denote the fiber of the map $q$ over the point $x$.
There exists a collection of isomorphisms
$$ \pi_{k}( F_x, s(x)) \simeq \begin{cases} 0 & \text{if } k \neq m \\
\calA_x & \text{if } k = m,\end{cases}$$
depending naturally on $x$.
\end{itemize}
The fibration $q$ always exists and is determined up to homotopy equivalence by these requirements, and the twisted cohomology group $\HH^{m}(X; \calA)$ can be identified with the set of homotopy classes of sections of $q$. We now present a generalization of this picture to higher category theory:

\begin{definition}\label{swu}
Let $\calC$ be an $(\infty,n)$-category, and let $\calA$ be a local system of abelian groups on $\calC$.
For each $m \geq n$, we will define a new $(\infty,n)$-category $K( \calA, m)$.
Our construction proceeds by induction on $n$. In the case $n=0$, we let $K(\calA, m)$ be defined
as in the preceding discussion. If $n > 0$, then we define $K( \calA, m)$ as follows:
\begin{itemize}
\item[$(1)$] The objects of $K( \calA, m)$ are the objects of $\calC$.
\item[$(2)$] Let $x$ and $y$ be objects of $\calC$, so that $\calA$ determines a local
system of abelian groups $\calA_{x,y}$ on the $(\infty,n-1)$-category
$\OHom_{\calC}(x,y)$. We now define $\OHom_{ K(\calA, m)}(x,y) = K( \calA_{x,y}, m-1)$.
\item[$(3)$] The composition law for morphisms in $K(\calA,m)$ is determined by the composition
of morphisms in $\calC$ (and the structure of $\calA$ as a local system).
\end{itemize}

By construction, the $(\infty,n)$-category $K(\calA, m)$ comes equipped with a forgetful functor
$q: K( \calA, n) \rightarrow \calC$. We let $\HH^{m}( \calC; \calA)$ denote the set of isomorphism
classes of sections of $q$; we refer to $\HH^{m}(\calC; \calA)$ as the {\it $m$th cohomology group of $\calC$ with values in $\calA$}.
\end{definition}

\begin{remark}
It is also possible to define the cohomology groups $\HH^{m}(\calC; \calA)$ for $m < n$.
For example, $\HH^{m-k}(\calC; \calA)$ can be identified with the $k$th homotopy group
of the classifying space for sections of the projection $K( \calA, M) \rightarrow \calC$.
\end{remark}

As the terminology suggests, the set $\HH^{m}(\calC; \calA)$ admits a natural (commutative) group structure, which is induced by the map of local systems $\calA \times \calA \rightarrow \calA$.
In particular, there is a canonical zero object $\HH^{m}(\calC; \calA)$, which corresponds to a
section $s_0: \calC \rightarrow K(\calA, m)$ of the projection map $q: K(\calA, m) \rightarrow \calC$.

\begin{variant}
Let $\calC$ be a symmetric monoidal $(\infty,n)$-category, so that we have a tensor product functor
$T: \calC \times \calC \rightarrow \calC$ which is commutative and associative, up to isomorphism.
We will say that a local system of abelian groups $\calA$ on $\calC$ is {\it multiplicative} if we are
provided with a map of local systems
$$ m: p_0^{\ast} \calA \times p_1^{\ast} \calA \rightarrow T^{\ast} \calA$$
on $\calC \times \calC$, which satisfies some natural commutativity and associativity properties (here $p_0, p_1: \calC \times \calC \rightarrow \calC$ denote the projection
maps). If $\calA$ is a multiplicative local system, then the $(\infty,n)$-categories
$K( \calA, m)$ inherit a symmetric monoidal structure, and the forgetful functor
$q: K(\calA, m) \rightarrow \calC$ preserves this symmetric monoidal structure. We let
$\HH^{m}_{\otimes}(\calC; \calA)$ denote the collection of all isomorphism classes of
{\em symmetric monoidal} sections of $q$. We will refer to $\HH^{m}_{\otimes}(\calC; \calA)$
as the {\it $m$th multiplicative cohomology group of $\calC$ with values in $\calA$}. 
By forgetting the symmetric monoidal structure, we obtain a canonical map of cohomology groups
$\HH^{m}_{\otimes}(\calC; \calA) \rightarrow \HH^{m}(\calC; \calA)$; this map is generally not an isomorphism.
\end{variant}

\begin{example}\label{swinging}
Let $A$ be an abelian group. Then we can regard $A$ as a local system $\calA$ of abelian groups on the trivial $(\infty,n)$-category $\ast$, having only a single object. For each $m \geq n$, the associated local system $K(\calA, m)$ can be identified with the fundamental groupoid of
an Eilenberg-MacLane space $K(A,m)$. 

For any symmetric
monoidal $(\infty,n)$-category $\calC$, we have a canonical functor $f: \calC \rightarrow \ast$, so that
$A$ determines a (multiplicative) local system $f^{\ast} \calA$ on $\calC$. We will refer to local systems on $\calC$ that arise via this construction as {\em constant} local systems on $\calC$. Unwinding the definitions, we deduce that the cohomology groups $\HH^{m}( \calC; f^{\ast} \calA)$ can be identified with the set of isomorphism classes of functors from $\calC$ into $K( \calA, m)$: in other words, the set of homotopy classes of maps of topological spaces from the geometric
realization $| \calC |$ into the Eilenberg-MacLane space $K(A,m)$. Consequently, we recover a canonical isomorphism
$$ \HH^{m}( \calC; f^{\ast} \calA) \simeq \HH^{m}( | \calC |; A).$$

If $\calC$ is endowed with a symmetric monoidal structure, then $f^{\ast} \calA$ is a multiplicative local system on $\calC$, and we can also consider the multiplicative cohomology
$\HH^{m}_{\otimes}( \calC; f^{\ast} \calA)$. If we suppose that every object in $\calC$ is dualizable, then this symmetric monoidal structure allows us to realize the geometric realization
$| \calC |$ as an infinite loop space, so that there is a sequence of topological spaces
$\{ X(n) \}_{n \geq 0}$ such that $X(0) \simeq | \calC |$, $X(n)$ is homotopy equivalent to the loop space of $X(n+1)$ for each $n \geq 0$, and each $X(n)$ is $(n-1)$-connected. Then
$\HH^{m}_{\otimes}( \calC; f^{\ast} \calA)$ can be identified with the set of homotopy classes of
infinite loop maps from $| \calC |$ into $K(A,m)$: in other words, the cohomology groups of
the {\em spectrum} $\{ X(n) \}_{n \geq 0}$ with coefficients in $A$.
\end{example}

Let us now return to our discussion of the Postnikov tower
$$ \ldots \rightarrow \tau_{\leq n+2} \calC \rightarrow \tau_{\leq n+1} \calC
\rightarrow \tau_{\leq n} \calC \simeq \hn{\calC}$$ 
of an $(\infty,n)$-category $\calC$. The maps
$q: \tau_{\leq m} \calC \rightarrow \tau_{\leq m-1} \calC$ bear a resemblance to the projection maps $K( \calA, m) \rightarrow \tau_{\leq m-1} \calC$ described in Definition \ref{swu}: for example, the fiber of
$q$ over any object of $\tau_{\leq m-1} \calC$ can be identified with an Eilenberg-MacLane space
$K(A,m)$, for some abelian group $A$. However, there is one crucial difference: the map $q$
does not necessarily admit a section. We can account for this discrepancy by introducing a
``twisted'' variant on Definition \ref{swu}:

\begin{definition}\label{swu2}
Let $\calC$ be an $(\infty,n)$-category, let $\calA$ be a local system of abelian groups on
$\calC$, and let $m \geq n$. Suppose we are given a pair of sections
$s,s': \calC \rightarrow K( \calA, m+1)$ of the projection map $K(\calA, m+1) \rightarrow \calC$.
We define a new $(\infty,n)$-category $\widetilde{\calC}$ by forming a homotopy pullback square
$$ \xymatrix{ \widetilde{\calC} \ar[r] \ar[d] & \calC \ar[d]^{s} \\
\calC \ar[r]^-{s'} & K(\calA, m+1). }$$
Note that $s$ and $s'$ determine cohomology classes $[s], [s'] \in \HH^{m+1}( \calC; \calA)$.
Up to equivalence, the fiber product $\widetilde{\calC}$ depends only on the difference
$\eta = [s] - [s'] \in \HH^{m+1}(\calC; \calA)$. We will refer to $\widetilde{\calC}$ as the
{\it small extension of $\calC$ determined by $\eta \in \HH^{m+1}(\calC; \calA)$}.
Note that $\widetilde{\calC}$ comes equipped with a canonical forgetful functor
$\widetilde{\calC} \rightarrow \calC$.
\end{definition}

\begin{example}
Let $\calC$ be an $(\infty,n)$-category and $\calA$ a local system of abelian groups on $\calC$, and let $m \geq n$. The zero element $0 \in \HH^{m+1}(\calC; \calA)$ determines a small extension
$\widetilde{\calC}$ of $\calC$, which can be identified with the $(\infty,n)$-category
$K(\calA, m)$ described in Definition \ref{swu}.
\end{example}

\begin{variant}
Suppose that $\calC$ is a symmetric monoidal $(\infty,n)$-category, and that $\calA$ is
a multiplicative local system of abelian groups on $\calC$. Let $m \geq n$, let
$\eta \in \HH^{m+1}_{\otimes}(\calC; \calA)$ be a multiplicative cohomology class, and let
$\overline{\eta}$ denote the image of $\eta$ in $\HH^{m+1}(\calC; \calA)$. Let
$\widetilde{\calC}$ denote the small extension of $\calC$ determines by $\overline{\eta}$. Then
$\widetilde{\calC}$ can be described as a homotopy fiber product of {\em symmetric monoidal}
$(\infty,n)$-categories, and therefore inherits a symmetric monoidal structure (which depends on
$\eta$). 
\end{variant}

The following result guarantees a sufficiently large class of small extensions:

\begin{claim}\label{kanter}
Let $\calC$ be an $(\infty,n)$-category, and let 
$$ \ldots \rightarrow \tau_{\leq n+2} \calC \rightarrow \tau_{\leq n+1} \calC
\rightarrow \tau_{\leq n} \calC \simeq \hn{\calC}$$ 
be its Postnikov tower. Then for each $m > n$, there exists a local system
of abelian groups $\calA_m$ on $\tau_{\leq m} \calC$ and a cohomology class
$\eta_m \in \HH^{m+2}( \calC; \calA_m)$ such that $\tau_{\leq m+1} \calC$ can be identified
with the small extension of $\tau_{\leq m} \calC$ determined by $\eta_m$.

Suppose furthermore that $\calC$ is equipped with a symmetric monoidal structure, and let $m > n$.
Then:
\begin{itemize}
\item[$(1)$] The truncation $\tau_{\leq m} \calC$ inherits a symmetric monoidal structure.
\item[$(2)$] The local system $\calA_m$ inherits a multiplicative structure.
\item[$(3)$] The cohomology class $\eta_m$ has a natural lift to multiplicative cohomology class $\widetilde{\eta}_m \in \HH^{m+2}_{\otimes}( \calC; \calA_m)$.
\item[$(4)$] The identification of $\tau_{\leq m+1} \calC$ with the small extension of 
$\tau_{\leq m} \calC$ determined by $\eta_m$ is compatible with the symmetric monoidal
structure provided by the multiplicative lift $\widetilde{\eta}_m$.
\end{itemize}
\end{claim}

Using Claim \ref{kanter}, one can mimic our proof of Proposition \ref{curwise} to obtain the following result:

\begin{proposition}\label{korder}
Let $f: \calD \rightarrow \calD'$ be a symmetric monoidal functor between symmetric monoidal $(\infty,n)$-categories. Then $f$ is an equivalence if and only if the following conditions are satisfied:
\begin{itemize}
\item[$(i)$] The functor $f$ induces an equivalence $\tau_{\leq n+1} \calD \rightarrow \tau_{\leq n+1} \calD'$.  
\item[$(ii)$] For every local system of abelian groups $\calA$ on $\calD'$ and every integer $m$, the functor $f$ induces an isomorphism of multiplicative cohomology groups $\HH^{m}_{\otimes}( \calD'; \calA) \rightarrow \HH^{m}_{\otimes}( \calD; f^{\ast} \calA)$.
\end{itemize}
\end{proposition}

\begin{remark}
It is convenient to restate hypothesis $(ii)$ of Proposition \ref{korder} in terms of 
{\em relative} cohomology groups. Suppose given a symmetric monoidal functor
$f: \calD \rightarrow \calD'$ between symmetric monoidal $(\infty,n)$-categories, and let
$\calA$ be a multiplicative local system on $\calD$. For $m \geq n$, we let
$\HH^{m}_{\otimes}( \calD', \calD; \calA)$ denote the set of isomorphism classes of symmetric monoidal sections $s$
of the projection $K( \calA, m) \rightarrow \calD'$ such that $s \circ f$ is identified with the zero section.
These relative cohomology groups can in fact be defined for {\em all} integers $m$, and fit into a long exact sequence
$$ \ldots \rightarrow \HH^{m-1}_{\otimes}( \calD; f^{\ast} \calA)
\rightarrow \HH^{m}_{\otimes}( \calD', \calD; \calA) \rightarrow
\HH^{m}_{\otimes}( \calD'; \calA) \rightarrow \HH^{m}_{\otimes}( \calD; f^{\ast} \calA)
\rightarrow \HH^{m+1}_{\otimes}( \calD', \calD; \calA) \rightarrow \ldots$$
Consequently, hypothesis $(ii)$ of Proposition \ref{korder} is equivalent to the vanishing
of the relative cohomology groups $\HH^{m}_{\otimes}( \calD', \calD; \calA)$, for every
integer $m$ and every multiplicative local system $\calA$ on $\calD'$.
\end{remark}

To prove Theorem \ref{postwish} from Proposition \ref{korder}, we need two things: a connectivity
estimate for the forgetful functor $\Bord^{\frun}_{n} \rightarrow \Bord_{n}$, and a calculation
of the relevant (multiplicative) cohomology groups. We will obtain the estimate
from the following theorem of Igusa (see \cite{igusa}):

\begin{theorem}[Igusa]\label{conig}
Let $M$ be a closed $(n-2)$-manifold and let $B$ be a $1$-morphism in $\calB(M)$. 
If $n = 1$, then $\FrFun(B)$ is contractible. For $n > 1$, the space $\FrFun(B)$ is $(n-1)$-connected.
In particular, the spaces $\FrFun(B)$ are always simply connected.
\end{theorem}

\begin{corollary}\label{sussk}
The forgetful functor $f: \Bord_{n}^{\frun} \rightarrow \Bord_{n}$ is $(n+2)$-connective. 
In particular, the induced map $\tau_{\leq n+1} \Bord _{n}^{\frun}
\rightarrow \tau_{\leq n+1} \Bord_{n}$ is an equivalence of $(n+1, n)$-categories.
\end{corollary}

To apply Proposition \ref{korder} to our situation, we also need to know that the
relative cohomology groups
$\HH^{m}_{\otimes}( \Bord_{n}, \Bord_{n}^{\frun}; \calA)$ vanish for every
integer $m$ and every multiplicative local system $\calA$ on $\Bord_{n}$. 
These relative cohomology groups fit into a long exact sequence
$$ \xymatrix{ 
\HH^{m}_{\otimes}( \Bord_{n}, \Bord_{n}^{\frun}; \calA) \ar[r] &  
\HH^{m}_{\otimes}( \Bord_{n}, \Bord_{n-1}; \calA) \ar[r]^{\theta_m} & \HH^{m}_{\otimes}( \Bord_{n}^{\frun}, \Bord_{n-1}; f^{\ast} \calA) \ar[dll] \\
\HH^{m+1}_{\otimes}( \Bord_{n}, \Bord_{n}^{\frun}; \calA) \ar[r] &  
\HH^{m+1}_{\otimes}( \Bord_{n}, \Bord_{n-1}; \calA) \ar[r]^{\theta_{m+1}} & \HH^{m+1}_{\otimes}( \Bord_{n}^{\frun}, \Bord_{n-1}; f^{\ast} \calA) }$$
Consequently, the vanishing of the cohomology groups
$\HH^{m}_{\otimes}( \Bord_{n}, \Bord_{n}^{\frun}; \calA)$ is equivalent to the assertion
that each of the maps $\theta_{m}$ is an isomorphism.
The relative cohomology group
$\HH^{m}_{\otimes}( \Bord_{n}^{\frun}, \Bord_{n-1}, f^{\ast} \calA)$ can be
identified with the collection of isomorphism classes of symmetric monoidal
sections $s$ of the projection $K( f^{\ast} \calA, m) \rightarrow \Bord_{n}^{\frun}$
which restrict to the zero section on $\Bord_{n-1}$. According to Theorem \ref{swisher11},
such sections are classified by their restriction to the $\OO(n)$-equivariant
$n$-morphism $D^{n}: {\bf 1} \rightarrow S^{n-1}$. For each $x \in \BO(n)$, we can
evaluate the local system $\calA$ on the corresponding $n$-morphism to obtain
an abelian group $\calB_x$. The collection of abelian groups
$\eta_x$ to obtain an abelian group $\calB_x$. The collection of abelian groups
$\{ \calB_{x} \}_{x \in \BO(n)}$ and the value of $s$ on the $n$-morphism $\eta_x$ can be identified
with a point of the Eilenberg-MacLane space $K( m-n, \calB_x)$. Allowing $x$ to vary, we obtain a canonical isomorphism
$$ \HH^{m}_{\otimes}( \Bord_{n}^{\frun}, \Bord_{n-1}, f^{\ast} \calA) \simeq
\HH^{m-n}( \BO(n); \calB).$$
The requisite cohomological calculation can therefore be formulated as follows:

\begin{theorem}[Cobordism Hypothesis, Infinitesimal Version]\label{swisher9}
Let $\calA$ be a multiplicative local system of abelian groups on $\Bord_{n}$, and
let $\calB$ be the induced local system of abelian groups on $\BO(n)$. Then for every integer $m$, the canonical map
$$ \HH^{m}_{\otimes}( \Bord_{n}, \Bord_{n-1}; \calA)
\rightarrow \HH^{m-n}(\BO(n); \calB)$$
is an isomorphism.
\end{theorem}

\begin{remark}\label{ilt}
In the situation of Theorem \ref{swisher9}, suppose that $\calA$
is a constant local system associated to an abelian group $A$ (see Example \ref{swinging}). 
According to the unoriented version of Theorem \ref{singsung}, the classifying spaces $| \Bord_{n} |$ and $| \Bord_{n-1} |$ can be identified with the zeroth spaces of the (connective) spectra $\Sigma^n \MTO(n)$ and $\Sigma^{n-1} \MTO(n-1)$, respectively.
As explained in Example \ref{swinging}, the relative multiplicative cohomology groups
$\HH^{m}_{\otimes}( \Bord_{n}, \Bord_{n-1}; \calA)$ can in this case be realized as the spectrum cohomology $\HH^{m}( \Sigma^{n} \MTO(n), \Sigma^{n-1} \MTO(n-1); A)$. 
The isomorphism of Theorem \ref{swisher9} in this case results from the existence of a cofiber sequence of spectra
$$ \Sigma^{n-1} \MTO(n-1) \rightarrow \Sigma^{n} \MTO(n) \rightarrow
\Sigma^{\infty+n}_{+} \BO(n).$$

If we assume Theorem \ref{swisher9} holds for every constant local system $\calA$, then we can deduce the unoriented version of Theorem \ref{singsung} using induction on $n$. For each $n$, the geometric realization $| \Bord_n |$ can be identified with
the zeroth space of {\em some} connective spectrum $Y(n)$ equipped with a canonical map
$f_n: Y(n) \rightarrow \Sigma^{n} \MTO(n)$. Theorem \ref{singsung} asserts that $f_n$
is a homotopy equivalence of spectra. If we assume that $f_{n-1}$ is a homotopy equivalence,
then Theorem \ref{swisher9} implies that $f_n$ induces an isomorphism on cohomology groups
$\HH^{m}( \Sigma^{n} \MTO(n); A) \rightarrow \HH^{m}( Y(n); A)$ for every abelian group $A$ and every integer $m$. Since the domain and codomain of $f_n$ are connective, this implies that
$f_n$ is a homotopy equivalence.
\end{remark}

We can summarize Remark \ref{ilt} as follows: the unoriented version of Theorem \ref{singsung} is equivalent to a special case of Theorem \ref{swisher9}, in which we assume that the local system $\calA$ is constant. It is possible to prove the general case Theorem \ref{swisher9} using the methods developed by Galatius, Madsen, Tillmann, and Weiss to prove Theorem \ref{singsung} (note that Theorem \ref{swisher9} is essentially calculational in nature; it can therefore be formulated in a purely homotopy-theoretic way that makes no mention of higher category theory). We will not describe the details any further here.

\section{Beyond the Cobordism Hypothesis}\label{glob4}

In this section, we will present some applications and extensions of the ideas developed earlier in this paper. We will begin in \S \ref{topchir} by describing a class of topological field theories which can
be produced by a very explicit homotopy-theoretic construction which we call {\it topological chiral homology}. In \S \ref{cost} we will discuss out the cobordism hypothesis in detail in dimensions $\leq 2$. In particular, we will formulate a ``noncompact'' analogue of the cobordism hypothesis
(Theorem \ref{swisher17}) and explain its relationship to earlier work of Costello (\cite{costello}) and to the string topology operations introduced by Chas and Sullivan (\cite{chassullivan}).
In \S \ref{mans}, we will describe a generalization of the cobordism hypothesis in which we work with bordism categories of (stratified) singular spaces, rather than smooth manifolds. We will
apply this generalization in \S \ref{tangus} to sketch a proof of a version of the Baez-Dolan {\it tangle hypothesis}, which characterizes $(\infty,n)$-categories of embedded bordisms and can be regarded as an ``unstable'' version of the cobordism hypothesis.

\subsection{Topological Chiral Homology}\label{topchir}

Let $\calC$ be a symmetric monoidal $(\infty,n)$-category with duals.
According to Theorem \ref{swisher3}, every object $C \in \calC$ determines
an symmetric monoidal functor $Z_C: \Bord_{n}^{\fr} \rightarrow \calC$, which is characterized
by the existence of an isomorphism $Z_C( \ast) \simeq C$. Though these invariants are formally determined by $C$ in principle, they can be very difficult to compute in practice. In this section,
we would like to illustrate the cobordism hypothesis in a special case where $Z_C$ can be described in completely explicit terms.

Let $\bfS$ be a symmetric monoidal $(\infty,1)$-category. It is sensible to talk about
{\it associative algebra objects} of $\bfS$: that is, objects $A \in \bfS$ which are endowed with a unit map and a multiplication
$${\bf 1} \rightarrow A \quad \quad A \otimes A \rightarrow A$$
which satisfy all of the usual associativity properties up to coherent isomorphism. The collection
of such algebra objects can itself be organized into an $(\infty,1)$-category, which we will denote by
$\Alg(\bfS)$. The tensor product $\otimes$ on $\bfS$ determines a tensor product on
$\Alg(\bfS)$, and endows $\Alg(\bfS)$ with a symmetric monoidal structure.

\begin{definition}\label{slapp}
Let $\bfS$ be a symmetric monoidal $(\infty,1)$-category. We will define a sequence
of symmetric monoidal $(\infty,1)$-categories $\Alg^{(n)}( \bfS)$ using induction as follows:
\begin{itemize}
\item If $n=1$, we let $\Alg^{(n)}(\bfS) = \Alg(\bfS)$ be the $(\infty,1)$-category of associative algebra objects of $\bfS$.
\item If $n > 1$, we let $\Alg^{(n)}(\bfS) = \Alg( \Alg^{(n-1)}(\bfS))$ be the $(\infty,1)$-category
of associative algebra objects in $\Alg^{(n-1)}(\bfS)$.
\end{itemize}
We will refer to objects of $\Alg^{(n)}(\bfS)$ as {\it $E_{n}$-algebras} in $\bfS$.
\end{definition}

\begin{remark}
It is convenient to extend Definition \ref{slapp} to the case $n=0$; we will agree to the convention
that an $E_0$-algebra in $\bfS$ is an object $A \in \bfS$ equipped with a unit map
${\bf 1} \rightarrow A$.
\end{remark}

\begin{remark}\label{kilrow}
More informally, we can think of an $E_n$-algebra in $\bfS$ as an object $A \in \bfS$ equipped
with $n$ associative algebra structures $\{ m_{i}: A \otimes A \rightarrow A \}_{1 \leq i \leq n}$,
which are compatible with one another in the following sense: if $i \neq j$, then
$m_{i}: A \otimes A \rightarrow A$ is a homomorphism with respect to the algebra structures determined by $m_{j}$.
\end{remark}

\begin{example}\label{slip}
Suppose that $\bfS$ is an ordinary symmetric monoidal category. In that case, Definition
\ref{slapp} reduces to the following:
\begin{itemize}
\item If $n = 1$, then $\Alg^{(n)}( \bfS)$ is the category of associative algebra objects of
$\bfS$.
\item If $n > 1$, then $\Alg^{(n)}(\bfS)$ is the category of commutative algebra objects of $\bfS$.
\end{itemize}
This is a consequence of the following general observation: let $m_{1}$ and
$m_{2}$ be associative multiplications on an object $A \in \bfS$ which are
compatible in the sense described in Remark \ref{kilrow}. Then $m_{1} = m_{2}$, and
both products are commutative. For example, suppose that $\bfS$ is the category of sets
(with symmetric monoidal structure given by the Cartesian product). Then we can
view $A$ as a set endowed with two associative multiplications $\times_{1}$ and
$\times_{2}$, having identity elements $e_1$ and $e_2$. We observe that
$$ e_2 = e_2 \times_2 e_2 = (e_2 \times_1 e_1)
\times_2 (e_1 \times_1 e_2) = (e_2 \times_2 e_1) \times_1 (e_1 \times_2 e_2)
= e_1 \times_1 e_1 = e_1,$$
so that $e_1$ and $e_2$ are equal to a common element $e \in A$.
The chain of equalities
$$ x \times_1 y = (x \times_2 e) \times_1 (e \times_2 y)
= (x \times_1 e) \times_2 (e \times_1 y) = x \times_2 y$$
shows that $\times_1 = \times_2$. The same argument shows that
$\times_1$ is the opposite of the multiplication given by $\times_2$, so that
the product $\times_1 = \times_2$ is commutative.
\end{example}

\begin{example}\label{swik}
Let $\bfS$ be the (large) $(\infty,1)$-category $\Cat_{(\infty,k)}$, whose objects are
(small) $(\infty,k)$-categories and whose morphisms are given by functors (here we discard
information about noninvertible natural transformations of functors). Then
$\bfS$ admits a symmetric monoidal structure, given by the Cartesian product.
We will refer to an $E_{n}$-algebra in $\Cat_{(\infty,k)}$ as an
{\it $E_{n}$-monoidal $(\infty,k)$-category}.
When $n = 1$, we recover the notion of a monoidal $(\infty,k)$-category; 
in the limiting case $n = \infty$, we recover the notion of a symmetric monoidal
$(\infty,k)$-category.
\end{example}

\begin{remark}
In order to define the notion of an $E_{n}$-algebra in an $(\infty,1)$-category $\bfS$, 
it suffices to assume that $\bfS$ has an $E_{n}$-monoidal structure: we do not need $\bfS$ to
be symmetric. For example, we can talk about associative algebra objects of an arbitrary monoidal $(\infty,1)$-category.
\end{remark}

Let $\bfS$ be a symmetric monoidal $(\infty,1)$-category, and let $A$ and $B$ be algebra objects of $\bfS$. We
can then define a new $(\infty,1)$-category $\Bimod_{A,B}(\bfS)$ of
$A$-$B$ bimodules in $\bfS$: that is, an $(\infty,1)$-category whose objects are objects
of $\bfS$ equipped with a left action of $A$ and a commuting right action of $B$.
We would like to regard $\Bimod_{A,B}(\bfS)$ as a collection of $1$-morphisms
in an $(\infty,2)$-category, where composition of bimodules is given by the formation of relative tensor products $(M,N) \mapsto M \otimes_{B} N$. To define this $(\infty,2)$-category, we need to introduce a technical assumption on $\bfS$.

\begin{definition}\label{bgood}
We will say that a monoidal $(\infty,1)$-category $\bfS$ is {\it good} if
$\bfS$ admits small sifted colimits, and the tensor product functor
$\otimes: \bfS \times \bfS \rightarrow \bfS$ preserves small sifted colimits
(see \cite{htt} for an explanation of this terminology).
\end{definition}

\begin{remark}
Let $\bfS$ be a monoidal $(\infty,1)$-category, let $A$ be an algebra object
of $\bfS$, let $M$ be a right $A$-module and $N$ a left $A$-module.
We would like to define the relative tensor product $M \otimes_{A} N$.
If $\bfS$ is an ordinary category, we can define this tensor product to be the
coequalizer of a pair of maps $f,g: M \otimes A \otimes N \rightarrow M \otimes N$.
In the general case, we need a more elaborate definition using the {\it two-sided bar construction}.
The assumption that $\bfS$ is good guarantees that this construction exists and is well-behaved;
we refer the reader to \cite{dag2} for more details.
\end{remark}

\begin{remark}
More generally, we will say that an $E_{n}$-monoidal $(\infty,1)$-category $\bfS$ is {\it good}
if it is good when regarded as a monoidal $(\infty,1)$-category, by neglecting
all but one of the $n$ compatible monoidal structures on $\bfS$ (an elaboration of the argument presented in Example \ref{slip} can be used to show that these monoidal structures are all equivalent
to one another, so it does not matter which one we choose). By convention, we will say that
an $E_{0}$-monoidal $(\infty,1)$-category is good if it admits sifted colimits.
\end{remark}

The construction $(A,B) \mapsto \Bimod_{A,B}(\bfS)$ can be regarded as a monoidal
functor of $A$ and $B$. For example, if we are given maps of algebra objects $A \otimes A' \rightarrow A''$ and $B \otimes B' \rightarrow B''$, then there is an induced bifunctor
$$ \Bimod_{A,B}(\bfS) \times \Bimod_{A',B'}(\bfS) \rightarrow \Bimod_{A'',B''}(\bfS).$$
In particular, if $A$ and $B$ are algebra objects of $\Alg(\bfS)$, then 
$\Bimod_{A,B}(\bfS)$ inherits a monoidal structure. Amplifying on this observation, we obtain the following:

\begin{claim}
Let $\bfS$ be a good $E_{n}$-monoidal $(\infty,1)$-category for $n \geq 1$, and let $A$ and $B$ be $E_{n}$-algebras in $\bfS$. Then the $(\infty,1)$-category $\Bimod_{A,B}(\bfS)$ admits the structure of
an $E_{n-1}$-category.
\end{claim}

\begin{definition}\label{hut}
Let $\bfS$ be a good $E_{n}$-monoidal $(\infty,1)$-category.
We can construct a new $(\infty,n+1)$-category $\Alg_{(n)}(\bfS)$ using induction on $n$ as follows:
\begin{itemize}
\item If $n=0$, then $\Alg_{(n)}(\bfS) = \bfS$.
\item If $n > 0$, then the objects of $\Alg_{(n)}(\bfS)$ are $E_{n}$-algebras in $\bfS$.
\item If $n > 0$ and $A,B \in \Alg_{(n)}(\bfS)$, then we set
$$ \OHom_{ \Alg_n(\bfS)}( A,B) = \Alg_{(n-1)}( \Bimod_{A,B}(\bfS)).$$
\end{itemize}
We let $\Alg^{\degree}_{(n)}(\bfS)$ denote the $(\infty,n)$-category obtained from $\Alg_{n}(\bfS)$ by discarding the noninvertible $(n+1)$-morphisms.
\end{definition}

\begin{example}
Let $\bfS$ be a good monoidal $(\infty,1)$-category. Then
$\Alg_{(1)}(\bfS)$ can be regarded as an $(\infty,2)$-category whose objects are 
algebras in $\bfS$ and whose $1$-morphisms are given by bimodules, with composition
given by tensor product of bimodules.
\end{example}

\begin{remark}
Let $\bfS$ be a good symmetric monoidal $(\infty,1)$-category.
Then for each $n \geq 0$, the $(\infty,n+1)$-category $\Alg_{(n)}(\bfS)$ and
the $(\infty,n)$-category $\Alg_{(n)}^{\degree}(\bfS)$ inherit symmetric monoidal structures.
\end{remark}

Let $\bfS$ be a good symmetric monoidal $(\infty,1)$-category. For every algebra object
$A \in \Alg_{(1)}(\bfS)$, the opposite algebra $A^{op}$ can be regarded as a dual of
$A$ in the symmetric monoidal $(\infty,1)$-category $\Alg_{(1)}^{\degree}$: we have evaluation and coevaluation maps 
$$ A \otimes A^{op} \rightarrow {\bf 1} \quad \quad {\bf 1} \rightarrow A \otimes A^{op}$$
given by $A$ itself, regarded as an $A \otimes A^{op}$-module. It follows that
the symmetric monoidal $(\infty,1)$-category $\Alg_{(1)}(\bfS)$ has duals. This observation admits the following generalization:

\begin{claim}\label{jilk}
Let $\bfS$ be a good symmetric monoidal $(\infty,1)$-category.
Then the symmetric monoidal $(\infty,n)$-category $\Alg_{(n)}^{\degree}(\bfS)$ has duals.
\end{claim}

Combining Claim \ref{jilk} with Theorem \ref{swisher3}, we conclude that
every $E_{n}$-algebra $A$ in $\bfS$ determines a symmetric monoidal
functor $Z_{A}: \Bord_{n}^{\fr} \rightarrow \Alg_{(n)}^{\degree}(\bfS) \rightarrow \Alg_{(n)}(\bfS)$
such that $Z_A(\ast) \simeq A$. In particular,
we get an induced functor $$\Omega^{n} \Bord_{n}^{\fr} \rightarrow \Omega^{n} \Alg_{(n)}(\bfS) \simeq \bfS,$$ which associates to every closed framed $n$-manifold $M$ an invariant $Z_A(M) \in \bfS$. 
Our goal in this section is to give an explicit construction of these invariants. We first review another approach to the theory of $E_{n}$-algebras.

\begin{notation}\label{diskop}
Fix $n \geq 0$, and let $D^{n}$ denote the (open) unit disk in $\R^{n}$. We will say that
an open embedding $D \rightarrow D$ is {\it rectilinear} if it can be extended to a linear map
$\R^{n} \rightarrow \R^{n}$, given by the formula $v \mapsto \lambda v + v_0$
for some $\lambda > 0$ and some $v_0 \in \R^{n}$. For each $k \geq 0$, we let
$\calE_{n}(k)$ denote the space of all $k$-tuples of rectilinear embeddings
$e_1, \ldots, e_k: D \rightarrow D$ whose images are disjoint. 

The collection of spaces $\{ \calE_{n}(k) \}$ can be organized into an {\it operad}:
that is, there are natural composition maps
$$ \calE_{n}(m) \times \calE_{n}(k_1) \times \cdots \times \calE_{n}(k_m) \rightarrow
\calE_{n}(k_1 + \ldots + k_m)$$
satisfying an appropriate associativity formula (we refer the reader to \cite{may} for
a more careful definition, and for a discussion of operads in general). We will refer to this
operad as the {\it little $n$-disks operad}, and denote it by $\calE_{n}$. 
\end{notation}

If $\bfS$ is any symmetric monoidal $(\infty,1)$-category, then it makes sense to talk about
$\calE_n$-algebras in $\bfS$: that is, objects $A \in \bfS$ which are equipped with
maps $\calE_{n}(k) \rightarrow \OHom_{\bfS}( A^{\otimes k}, A)$ for each $k \geq 0$, 
which are compatible with the composition on $\calE_{n}$ (up to coherent homotopy).

If $n=1$, then $\calE_{1}(k)$ is homotopy equivalent to a discrete space for every
$k$: namely, the discrete space of all linear orderings of a $k$-element set.
It follows that $\calE_1$ is homotopy equivalent to the usual associative operad, and
$\calE_1$-algebras can be identified with associative algebras in $\bfS$.
This observation admits the following amplification:

\begin{claim}\label{jilker}
Let $\bfS$ be a symmetric monoidal $(\infty,1)$-category. Then $\calE_{n}$-algebras in
$\bfS$ can be identified with $E_{n}$-algebras in $\bfS$.
\end{claim}

\begin{remark}
It follows from Claim \ref{jilk} and Corollary \ref{ail} that if $\bfS$ is a good symmetric monoidal $(\infty,1)$-category, then the $\infty$-groupoid $\Alg_{n}^{\sim}(\bfS)$ carries
an action of the group $\OO(n)$. Claim \ref{jilker} makes this action more evident, since the group $\OO(n)$ acts
naturally on the little disks operad $\calE_{n}$ itself.
\end{remark}

We now come to the main idea of this section:

\begin{construction}\label{seaward}
We define an $(\infty,1)$-category $\calA$ as follows:
\begin{itemize}
\item The objects of $\calA$ are finite disjoint unions
$D^{n} \coprod D^n \coprod \cdots \coprod D^{n}$, where $D^n$ denotes the open unit
disk in $\R^n$.
\item Given a pair of objects $X, Y \in \calA$, we let $\OHom_{\calA}(X,Y)$ be the space of open embeddings $X \rightarrow Y$ which are rectilinear on each connected component component.
\end{itemize}

If $M$ is a framed $n$-manifold (not necessarily compact), we can define a functor $f_{M}$ from $\calA$ into the $(\infty,1)$-category of topological spaces as follows: for every object $X \in \calA$, we let
$f_M(X)$ denote the space of framed open embeddings $X \rightarrow M$: that is, the space
of pairs $(j, h)$ where $j: X \rightarrow M$ is an open embedding and $h$ is a homotopy
between the canonical framing on $X$ and the framing obtained by pulling back the framing of $M$.
Let $\calA_{M}$ denote the $(\infty,1)$-category obtained by applying the Grothendieck
construction (Construction \ref{groth}) to $f_{M}$: in other words, $\calA_{M}$ is the
$(\infty,1)$-category whose objects are pairs $(X, \eta)$ where $X \in \calA$ and
$\eta: X \rightarrow M$ is a framed embedding. We observe that there is a canonical
forgetful functor $\calA_{M} \rightarrow \calA$.

Let $\bfS$ be a good symmetric monoidal $(\infty,1)$-category, and let $A$ be an
$E_{n}$-algebra in $\bfS$. Then $A$ carries an action of the little disks operad
$\calE_{n}$, and in particular determines a functor $g: \calA \rightarrow \bfS$
which carries a disjoint union of $k$ copies of $D^n$ into the tensor power $A^{\otimes k}$.
For every framed $n$-manifold $M$, we let $\int_{M} A$ denote a homotopy colimit of the composite functor $$ \calA_{M} \rightarrow \calA \stackrel{g}{\rightarrow} \bfS.$$
(Such a colimit always exists, provided that $\bfS$ is good.)
We will refer to $\int_{M} A$ as the {\it topological chiral homology of $M$ with coefficients in $A$}.
\end{construction}

\begin{remark}
One can think of the topological chiral homology $\int_{M} A$ as a kind of continuous tensor product
$\otimes_{x \in M} A$ indexed by points of the manifold $M$.
\end{remark}

\begin{remark}
The terminology of Construction \ref{seaward} is intended to invoke an analogy with the theory
of chiral homology introduced by Beilinson and Drinfeld (see \cite{beilinson}). The basic idea of our construction is the same, except that we use constant $\bfS$-valued sheaves on framed manifolds
in place of $\calD$-modules on algebraic varieties.
\end{remark}

\begin{example}
Suppose that $\bfS$ is the $(\infty,1)$-category of topological spaces. 
Let $X$ be a pointed topological space which is $n$-connective (that is, the homotopy groups
$\pi_i X$ vanish for $i < n$), and let $A$ denote the $n$th loop space $\Omega^{n} X$. Then
$A$ carries an action of the little disks operad $\calE_{n}$, and can therefore be regarded
as an $E_{n}$-algebra. For any framed (possibly noncompact) framed $n$-manifold $M$, the
integral $\int_{M} A$ can be identified with the space $C_c(M,X)$ of compactly supported
functions $M \rightarrow X$ (that is, functions from $M$ to $X$ which carry $M - K$ to the base point of $M$ for some compact subset $K \subseteq M$). 
\end{example}

\begin{example}\label{hochhom}
Let $\bfS$ be a good symmetric monoidal $(\infty,1)$-category, and let
$A$ be an associative algebra object of $\bfS$. Then $\int_{S^1} A$ can be identified with
the {\it Hochschild homology} of $A$, which is given by the relative tensor product
$$ A \otimes_{ A \otimes A^{op} } A.$$
\end{example}

\begin{example}
Let $\bfS$ be a good symmetric monoidal $(\infty,1)$-category, and suppose that
$A$ is a commutative algebra object of $\bfS$.
In this case, the topological chiral homology $\int_{M} A$ is again a commutative
algebra, and can be characterized by the following universal mapping property:
$$ \OHom_{ \CAlg(\bfS)}( \int_{M} A, B) \simeq \OHom_{ \Cat_{(\infty,0)}}( M, \OHom_{ \CAlg(\bfS)}(A,B) )$$
Here $\CAlg(\bfS)$ denotes the $(\infty,1)$-category of commutative algebra objects in $\bfS$.
\end{example}

Fix a good symmetric monoidal $(\infty,1)$-category $\bfS$ and
an $E_{n}$-algebra $A$ in $\bfS$.
The topological chiral homology $\int_{M} A$ is covariant with respect to
open inclusions of (framed) $n$-manifolds: an open inclusion $M_0 \subseteq M$
induces a functor $\calA_{M_0} \rightarrow \calA_{M}$, which in turn determines a map
of homotopy colimits $\int_{M_0} A \rightarrow \int_{M} A$. Moreover, one can show that the functor
$M \mapsto \int_{M} A$ carries disjoint unions of framed $n$-manifolds to tensor products in $\bfS$.

Suppose now that $M$ is an $n$-framed manifold of dimension $m \leq n$, and let
$D^{n-m}$ denote the open unit disk in $\R^{n-m}$. Then $M \times D^{n-m}$ can
be regarded as a framed $n$-manifold, and we define
$\int_{M} A$ to be $\int_{M \times D^{n-m}} A$. It follows from the above remarks that
$\int_{M} A$ carries an action of the operad $\calE_{n-m}$, and can therefore be regarded as
an $E_{n-m}$-algebra in $\bfS$. In the special case where $M$ consists of a single point, we have a canonical isomorphism $\int_{M} A \simeq A$ of $E_{n}$-algebras in $\bfS$. 

Let $M$ be a framed $n$-manifold with boundary. Let $M^{0}$ denote the interior of $M$.
Choosing a collar of the boundary, we obtain a bijection $M \simeq M^{0} \coprod ([0,1] \times \bd M)$,
which induces an open embedding $M^{0} \coprod ( D^1 \times \bd M) \rightarrow M^{0}$.
Passing to topological chiral homology, we obtain a map
$$ (\int_{M^{0}} A) \otimes ( \int_{\bd M} A) \rightarrow \int_{M^{0}} A,$$
which exhibits the object $\int_{M^{0}} A$ as a right module over the associative algebra object $\int_{\bd M} A$. More generally, if $M$ is a bordism from an $n$-framed $(n-1)$-manifolds
$N$ and $N'$, then we can regard $\int_{M^{0}} A$ as an $(\int_{N} A,\int_{N'} A)$-bimodule.
Elaborating on these constructions, one can prove the following:

\begin{theorem}\label{explicit}
Let $A$ be an $E_{n}$-algebra in a good symmetric monoidal $(\infty,1)$-category $\bfS$. Then the construction $M \mapsto \int_{M} A$ can be extended to a symmetric monoidal functor
$Z: \Bord_{n}^{\fr} \rightarrow \Alg^{\degree}_{(n)}(\bfS).$ In particular, we have an isomorphism
of $E_{n}$-algebras $Z(\ast) \simeq A$.
\end{theorem}

Theorem \ref{explicit} provides an explicit construction of the topological field theory
$Z: \Bord_{n}^{\fr} \rightarrow \Alg_{n}(\bfS)$ associated to an object
$A \in \Alg_{n}(\bfS)$. Namely, the value of $Z$ on a framed $n$-manifold $M$ is
given by $\int_{M} A$, which can in turn be described as a certain homotopy colimit.

\begin{remark}
Construction \ref{seaward} actually gives quite a bit more than the field theory
$\Bord_{n}^{\fr} \rightarrow \Alg_n(\bfS)$: the topological chiral homology
$\int_{M} A$ can be defined for {\em any} framed $n$-manifold $M$, whether or not $M$ is compact.
It can also be defined on a larger class of manifolds: for example, in dimension $4$, we can use topological manifolds equipped with a trivialization of their tangent microbundles.
\end{remark}

\begin{remark}
Theorem \ref{explicit} can be regarded a concrete version of the cobordism hypothesis for
framed manifolds. One can use similar ideas to produce concrete analogues of the more
exotic forms of the cobordism hypothesis. For example, suppose we are given a continuous homomorphism of topological groups $G \rightarrow \OO(n)$. Then $G$ acts on the
operad $\calE_{n}$, and it makes sense to talk about a $\calE_{n}$-algebra in $\bfS$ with
a compatible action of $G$. In this case, we obtain a symmetric monoidal functor
$Z: \Bord_{n}^{G} \rightarrow \Alg_{n}(\bfS)$ which we can think of as carrying
a $G$-manifold $M$ to the {\em twisted} topological chiral homology
$\otimes_{x \in M} A'_{x}$, where $A'$ denotes the bundle of $E_{n}$-algebras
on $M$ determined by the $G$-structure on $M$ and the action of $G$ on $A$.
\end{remark}

\begin{remark}
Theorem \ref{explicit} is usually not very satisfying, because it describes a functor
$Z: \Bord_{n}^{\fr} \rightarrow \Alg^{\degree}_{(n)}(\bfS)$ whose values on closed framed $n$-manifolds
are objects of an $(\infty,1)$-category $\bfS$, rather than concrete invariants like numbers.
We might attempt to remedy this by contemplating $(n+1)$-dimensional topological field theories taking values in the $(\infty,n+1)$-category $\Alg_{(n)}(\bfS)$. This turns out to be somewhat more difficult, because
$\Alg_{(n)}(\bfS)$ does not have duals in general (in other words, there are
$n$-morphisms in $\Alg_{(n)}(\bfS)$ which do not admit left or right adjoints). Consequently,
not every object of $\Alg_{(n)}(\bfS)$ is fully dualizable. In general, the condition that
an $E_{n}$-algebra $A \in \bfS$ be fully dualizable as an object of $\Alg_{(n)}(\bfS)$ amounts
to a very strong finiteness condition on $A$. By unwinding the proof of the cobordism hypothesis, one can formulate this finiteness condition in reasonably concrete terms: it amounts to the requirement that $A \simeq \int_{ D^{k} } A$ be dualizable as a module over $\int_{ S^{k-1} } A$ for
$0 \leq k \leq n$. For example, when $n = 1$, we must require that $A$ admits a dual
both as an object of $\bfS$ and as an $A \otimes A^{op}$-module. When
$\bfS$ is the $(\infty,1)$-category $\tChain(k)$ described in Definition \ref{spout2}, then we
can identify algebra objects $A$ of $\bfS$ with differential graded algebras over $k$; such an
object is fully dualizable in $\Alg_{(1)}(\bfS)$ if and only if $A$ is a {\it smooth and proper} differential graded algebra (see, for example, \cite{toensmooth}).
\end{remark}

\subsection{The Cobordism Hypothesis in Low Dimensions}\label{cost}

Our goal in this section is to discuss some consequences of the cobordism hypothesis and related results in the case of manifolds of dimension $1$ and $2$. In particular, we will relate the contents of this paper to the work of Costello (\cite{costello}) and to the Chas-Sullivan theory of string topology operations on the homology of loop spaces of manifolds (\cite{chassullivan}).

We begin by studying topological field theories in dimension $1$. Let $\calC$ be a symmetric monoidal $(\infty,1)$-category, and let $X \in \calC$ be a dualizable object. According to Theorem \ref{swisher3}, there is an essentially unique symmetric monoidal functor $Z: \Bord_{1}^{\ori} \simeq \Bord_{1}^{\fr} \rightarrow \calC$ satisfying
$Z(\ast) \simeq X$. In Example \ref{1dim}, we sketched a direct proof of this fact in the special case
where $\calC$ is the ordinary category of vector spaces. This proof exploits the fact that
manifolds of dimension $1$ are very simple (consisting only of intervals and circles), and
can be applied more generally whenever $\calC$ is an ordinary category. However, the
$(\infty,1)$-categorical case is substantially subtle. The $(\infty,1)$-category
$\Bord_{1}^{\ori}$ is {\em not} equivalent to an ordinary category. For example, the mapping space
$\OHom_{ \Bord_{1}^{\ori} }( \emptyset, \emptyset)$ can be identified with a classifying space
for oriented closed $1$-manifolds. In particular, it contains as a connected component
a classifying space $\CP^{\infty} \simeq \BSO(2)$ for oriented circle bundles. If
$Z: \Bord_{1}^{\ori} \rightarrow \calC$ is a symmetric monoidal functor, then
$Z$ induces a map $f: \CP^{\infty} \rightarrow \OHom_{\calC}( {\bf 1}, {\bf 1})$.
Roughly speaking, $f$ is determined by the values of $Z$ on circles. Given $X = Z(\ast)$, we can
compute $Z(S^1)$ as in Example \ref{1dim}, by breaking the circle $S^1$ into two half-circles.
The result is that we can identify $Z(S^1)$ with a $1$-morphism $\dim(X): {\bf 1} \rightarrow {\bf 1}$
which is given by composing the evaluation map $\ev_{X}: X \otimes X^{\vee} \rightarrow {\bf 1}$ with the coevaluation map $\coev_{X}: {\bf 1} \rightarrow X \otimes X^{\vee}$. We refer to
$\dim(X)$ as the {\it dimension of $X$} (this is motivated by the example where
$\calC$ is the category of vector spaces over a field, where we recover the classical notion of the
dimension of a vector space). In the $(\infty,1)$-categorical case, this above
argument does not determine the map $f$: it only determines the value of $f$ on a single
point of the classifying space $\CP^{\infty}$. The map $f$ encodes the idea that the object $Z(S^1) \in \OHom_{\calC}( {\bf 1}, {\bf 1})$ carries an action of the symmetry group $\SO(2)$. However, our calculation of $Z(S^1)$ proceeds by breaking the circle into two pieces, and thereby destroying its symmetry. Consequently, Theorem \ref{swisher3} has an interesting consequence even in dimension $1$:

\begin{proposition}\label{kij}
Let $\calC$ be a symmetric monoidal $(\infty,1)$-category, and let $X$ be a dualizable object
of $\calC$. Then the object $\dim_{X} \in \OHom_{\calC}( {\bf 1}, {\bf 1})$ carries a canonical
action of the circle group $S^1 = \SO(2)$.
\end{proposition}

\begin{example}
Let $\bfS$ be a symmetric monoidal $(\infty,1)$-category, and let
$A$ be an associative algebra object of $\bfS$. Then we can regard
$A$ as a (dualizable) object of $\Alg_{(1)}^{\degree}(\bfS)$, and thereby obtain
an object $\dim(A) \in \Omega \Alg_{(1)}^{\degree}(\bfS) \subseteq \Omega \Alg_{(1)}(\bfS) \simeq \bfS$.
In this case, we can identify $\dim(A)$ with the Hochschild homology
$\int_{S^1} A \simeq A \otimes_{ A \otimes A^{op} } A$ (see Example \ref{hochhom}), and the circle action of Proposition \ref{kij} recovers the classical circle action on Hochschild homology
(which can be obtained by computing the relative tensor product using a cyclic bar resolution).
\end{example}

Let us now analyze the cobordism hypothesis in dimension $2$. Our first step is to give
a simple criterion for full dualizability.

\begin{proposition}\label{swip}
Let $\calC$ be a symmetric monoidal $(\infty,2)$-category, and let $X \in \calC$ be an object.
Then $X$ is fully dualizable if and only if the following conditions are satisfied:
\begin{itemize}
\item[$(1)$] The object $X$ admits a dual $X^{\vee}$.
\item[$(2)$] The evaluation map $\ev_X: X \otimes X^{\vee} \rightarrow {\bf 1}$ admits both
a right and a left adjoint.
\end{itemize}
\end{proposition}

\begin{proof}
Conditions $(1)$ and $(2)$ are obviously necessary. To prove the converse, let us suppose
that $(1)$ and $(2)$ are satisfied. Then $\ev_{X}$ admits right and left adjoints
$$\ev_{X}^{R}, \ev_{X}^{L}: {\bf 1} \rightarrow X \otimes X^{\vee}.$$
Then there exist $1$-morphisms $S,T: X \rightarrow X$ such that
$\ev_{X}^{R} = (S \otimes \id_{X^{\vee}}) \circ \coev_{X}$ and
$\ev_{X}^{L} = (T \otimes \id_{X^{\vee}}) \circ \coev_{X}$. It is not difficult to show that the endomorphisms
$S$ and $T$ are inverse to each other, and in particular adjoints of one another.
Consequently, we deduce that for every integer $n$, the morphism
$\ev_{X} \circ (S^{n} \otimes \id_{X^{\vee}})$ has a right adjoint
(given by $(S^{1-n} \otimes \id_{X^{\vee}}) \circ \coev_{X}$) and
a left adjoint (given by $(S^{-1-n} \otimes \id_{X^{\vee}}) \circ \coev_{X}$). 
These formulas show that $(S^{n} \otimes \id_{X^{\vee}}) \circ \coev_{X}$ also
admits both right and left adjoints. Let $\calC_0$ denote the largest subcategory
of $\calC$ such that every $1$-morphism in $\calC_0$ admits both a right and a left adjoint.
Then $\ev_{X}$ and $\coev_{X}$ belong to $\calC_0$, so that $X$ is dualizable in
$\calC_0$ and therefore a fully dualizable object of $\calC$.
\end{proof}

\begin{remark}\label{umblecrown}
In the situation of Proposition \ref{swip}, we will refer to the map
$S: X \rightarrow X$ as the {\it Serre automorphism} of $X$. This terminology is motivated by the situation
where $\calC$ is the $(\infty,2)$-category of cocomplete differential graded categories over a field $k$ (suitably defined). If $\calD$ is a fully dualizable object of $\calC$ which is generated by compact objects, then there is a canonical endofunctor
$S: \calD \rightarrow \calD$ (called the {\it Serre functor} on $\calD$), which is characterized
by the existence of natural quasi-isomorphisms
$$ \bHom_{\calD}( C, S(D)) \simeq \bHom_{\calD}( D, C)^{\vee}$$ 
for every pair of compact objects $C$ and $D$. In the special case where
$\calD$ is the differential graded category of quasi-coherent complexes on
a smooth projective variety $X$, the Serre functor $S: \calD \rightarrow \calD$ is given by tensoring
with the canonical line bundle $\omega_{X}$ on $X$ and shifting by the dimension of $X$, and the quasi-isomorphism is provided by Serre duality.
\end{remark}

\begin{remark}\label{cui}
Let $\calC$ be a symmetric monoidal $(\infty,2)$-category. According to
Corollary \ref{ail}, there is an action of the group $\OO(2)$ on the $\infty$-groupoid of fully
dualizable objects of $\calC$. In particular, for every fully dualizable object $X \in \calC$,
we obtain a map $S^1 \simeq \SO(2) \times \{X\} \rightarrow \calC^{\sim}$ which carries the base
point of $S^{1}$ to the object $X \in \calC$, which gives rise to an automorphism of $X$. This automorphism coincides with the Serre automorphism constructed more explicitly in Proposition \ref{swip}. To prove this, it suffices to consider the universal case where $\calC$ is freely generated by a fully dualizable object $X$: that is, we may assume that $\calC = \Bord_{2}^{\fr}$; we leave this as an elementary exercise for the reader.
\end{remark}

It follows from Remark \ref{cui} that if $X$ is an $\SO(2)$-fixed point in the $\infty$-groupoid
of fully dualizable objects of a symmetric monoidal $(\infty,2)$-category $\calC$, then the
Serre automorphism $S: X \rightarrow X$ is the identity. However, we can formulate the condition
of being an $\SO(2)$-fixed point without the full strength of the assumption that $X$ is fully dualizable.

\begin{definition}
Let $\calC$ be a symmetric monoidal $(\infty,2)$-category. A {\it Calabi-Yau} object of
$\calC$ consists of the following data:
\begin{itemize}
\item[$(1)$] A dualizable object $X \in \calC$.
\item[$(2)$] A morphism $\eta: \dim(X) = \ev_{X} \circ \coev{X} \rightarrow {\bf 1}$ in $\Omega \calC$,
which is equivariant with respect to the action of $\SO(2)$ on $\dim(X)$ (see Proposition \ref{kij}) and
is the counit for an adjunction between $\ev_{X}$ and $\coev_{X}$.
\end{itemize}
\end{definition}

\begin{remark}
If $\calC$ is a symmetric monoidal $(\infty,2)$-category with duals, then
Theorem \ref{swisher5} and Theorem \ref{swisher6} together imply that
Calabi-Yau objects of $\calC$ can be identified with (homotopy) fixed points for the action
of $\SO(2)$ on $\calC^{\sim}$ (because both can be identified with symmetric monoidal
functors $\Bord_{2}^{\ori} \rightarrow \calC$). 
\end{remark}

\begin{example}
Let $\bfS$ be a good symmetric monoidal $(\infty,1)$-category (see Definition \ref{bgood}), and let
$\Alg_{(1)}(\bfS)$ be the $(\infty,2)$-category of Definition \ref{hut}. The objects of
$\Alg_{(1)}(\bfS)$ are associative algebras $A \in \bfS$, and are all dualizable objects of
$\Alg_{(1)}(\bfS)$ (the dual of an algebra $A$ is the opposite algebra $A^{op}$). By definition,
a Calabi-Yau object of $\Alg_{(1)}(\bfS)$ consists of an associative algebra $A$ together
with an $\SO(2)$-equivariant map
$$ \tr: \int_{S^1} A \rightarrow {\bf 1}$$
satisfying the following condition: the composite map
$$A \otimes A \simeq \int_{S^0} A \rightarrow \int_{S^1} A \stackrel{\tr}{\rightarrow} {\bf 1}$$
induces an identification of $A$ with its dual $A^{\vee}$ in $\bfS$.
We will refer to such a structure as a {\it Calabi-Yau algebra} in $\bfS$.
\end{example}

\begin{remark}
The notion of a Calabi-Yau algebra makes sense in an arbitrary symmetric monoidal
$(\infty,1)$-category $\bfS$ (in other words, $\bfS$ need not be good). Although
the Hochschild homology $\int_{S^1} A \simeq A \otimes_{ A \otimes A^{op} } A$
is generally not well-defined as an object of $\bfS$, it can be defined formally
as a colimit of objects of $\bfS$ so it still makes sense to talk about a map
$\tr: \int_{S^1} A \rightarrow {\bf 1}$.
\end{remark}

If $\calC$ is a general symmetric monoidal $(\infty,2)$-category, then Calabi-Yau objects of $\calC$ need not be fully dualizable. In fact, Calabi-Yau objects fail to be fully dualizable in a number
of interesting cases (see Example \ref{stringtop} below). Consequently, it will be convenient to
characterize Calabi-Yau objects in terms of topological field theories. We can extract
such a characterization from our proof of the cobordism hypothesis. Recall that our proof
of the cobordism hypothesis for $\Bord_{n}$ proceeds by analyzing a filtration
$$ \calF_{-1} \rightarrow \calF_{0} \rightarrow \ldots \rightarrow \calF_{n} = \Bord_{n}$$
of the $(\infty,n)$-category $\Bord_{n}$; roughly speaking, we can think of $\calF_{k}$ as
an $(\infty,n)$-category of bordisms where all $n$-manifolds are equipped with a decomposition into handles of index $\leq k$. As a by-product of the proof, we obtain a characterization of
each $\calF_{i}$ by a universal property. In particular, when $n=2$, we deduce that the oriented
version of $\calF_{1}$ can be described as the free symmetric monoidal $(\infty,2)$-category generated by a single Calabi-Yau object. It turns out that this $(\infty,2)$-category can be described more concretely, without making reference to the theory of framed functions.

\begin{protodefinition}
We define a symmetric monoidal $(\infty,2)$-category $\Bord_{2}^{\non}$ informally as follows:
\begin{itemize}
\item The objects of $\Bord_{2}^{\non}$ are oriented $0$-manifolds.
\item Given a pair of objects $X, Y \in \Bord_{2}^{\non}$, a $1$-morphism
from $X$ to $Y$ is an oriented bordism $B: X \rightarrow Y$.
\item Given a pair of $1$-morphsims $B, B': X \rightarrow Y$ in $\Bord_{2}^{\ori, \non}$, a
$2$-morphism from $B$ to $B'$ in $\Bord_{2}^{\non}$ is an oriented
bordism $\Sigma: B \rightarrow B'$ (which is trivial along $X$ and $Y$) with the following property: 
every connected component of $\Sigma$ has nonempty intersection with $B$.
\item Higher morphisms in $\Bord_{2}^{\non}$ are given by (orientation preserving) diffeomorphisms, isotopies between diffeomorphisms, and so forth.
\item The symmetric monoidal structure on $\Bord_{2}^{\non}$ is given by the formation of disjoint unions.
\end{itemize}
\end{protodefinition}

The $(\infty,2)$-category $\Bord_{2}^{\non}$ is characterized by the following analogue of the cobordism hypothesis:

\begin{theorem}[Cobordism Hypothesis, Noncompact Version]\label{swisher17}
Let $\calC$ be a symmetric monoidal $(\infty,2)$-category. The following types of data are equivalent:
\begin{itemize}
\item[$(1)$] Symmetric monoidal functors $Z: \Bord_{2}^{\non} \rightarrow \calC$.
\item[$(2)$] Calabi-Yau objects of $\calC$.
\end{itemize}
The equivalence is implemented by carrying a functor $Z$ to the Calabi-Yau object
$Z(\ast)$.
\end{theorem}

\begin{remark}
In particular, the $0$-manifold consisting of a single point can be regarded as a Calabi-Yau object of $\Bord_{2}^{\non}$.
\end{remark}

Using the methods of \S \ref{unf}, we can translate Theorem \ref{swisher17} into a statement
in the language of symmetric monoidal $(\infty,1)$-categories. Let
$\calC$ be a symmetric monoidal $(\infty,2)$-category, and let $\calC_{1}$ be the
symmetric monoidal $(\infty,1)$-category obtained by discarding the noninvertible $2$-morphisms in $\calC$. Using Proposition \ref{kubbus}, we can convert the inclusion $\calC_{1} \rightarrow \calC$
into a symmertic monoidal coCartesian fibration $\widetilde{\calC}_{1} \rightarrow \calC_1$ of $(\infty,1)$-categories. Similarly, we can convert the inclusion $\Bord_{1}^{\ori} \rightarrow \Bord_{2}^{\non}$ into a coCartesian fibration $\pi: \OC \rightarrow \Bord_{1}^{\ori}$. Here
$\OC$ is a symmetric monoidal $(\infty,1)$-category which can be described as follows:
\begin{itemize}
\item The objects of $\OC$ are oriented $1$-manifolds with boundary.
\item Given a pair of objects $I, J \in \OC$, a $1$-morphism from $I$ to $J$ in
$\OC$ is an oriented bordism $B$ from $I$ to $J$, satisfying the following condition:
every connected component of $B$ has nonempty intersection with $J$. 
\item Higher morphisms in $\OC$ are given by (orientation-preserving) diffeomorphisms, isotopies between diffeomorphisms, and so forth.
\end{itemize}
The forgetful functor $\pi: \OC \rightarrow \Bord_{1}^{\ori}$ is given by sending a $1$-manifold
$J$ to its boundary $\bd J$.

In the above situation, we can identify symmetric monoidal functors $Z: \Bord_{2}^{\non} \rightarrow \calC$ with diagrams of symmetric monoidal functors
$$ \xymatrix{ \OC \ar[r]^{Z_2} \ar[d] & \widetilde{\calC}_1 \ar[d] \\
\Bord_{1}^{\ori} \ar[r]^{Z_1} & \calC_1 }$$
satisfying the technical condition that $Z_2$ preserves coCartesian morphisms.
Applying the cobordism hypothesis in dimension $1$, we deduce that giving the symmetric monoidal functor $Z_1$ is equivalent to giving a dualizable object $X \in \calC$. Consequently, we may reformulate Theorem \ref{swisher17} as follows:

\begin{theorem}[Cobordism Hypothesis, Noncompact Unfolded Version]\label{swisher18}
Let $\calC$ be a symmetric monoidal $(\infty,2)$-category, let $\widetilde{\calC}_1 \rightarrow \calC_1$ be the coCartesian fibration defined as above, and let $X \in \calC$ be a dualizable object.
Let $Z_0$ denote the composition
$$ \OC \rightarrow \Bord_{1}^{\ori} \stackrel{Z_1}{\rightarrow} \calC_1$$
where $Z_1$ is the symmetric monoidal functor determined by $X$. The following
types of data are equivalent:
\begin{itemize}
\item[$(1)$] Symmetric monoidal functors $Z_2: \OC \rightarrow \widetilde{\calC}_1$
which are {\it coCartesian} in the following sense: they carry coCartesian morphisms for the projection $\OC \rightarrow \Bord_{1}^{\ori}$ to coCartesian morphisms for the projection $\widetilde{\calC}_1 \rightarrow \calC_1$.
\item[$(2)$] Morphisms $\eta: \dim(X) \rightarrow {\bf 1}$ in $\Omega \calC$ which are the
counit for an adjunction between $\ev_{X}$ and $\coev_{X}$ in $\calC$.
\end{itemize}
\end{theorem}

It is possible to give a proof of Theorem \ref{swisher18} which is completely independent of the methods presented earlier in this paper. Instead, it relies on a classification of symmetric monoidal
functors with domain $\OC$ which arises from the work of Kevin Costello. In order to state this classification, we need a bit of notation: let $\calO$ denote the full subcategory of $\OC$ whose objects
are finite unions of intervals (in other words, we disallow any components which are circles).

\begin{theorem}[Costello \cite{costello}]\label{costelloA}
Let $\bfS$ be a symmetric monoidal $(\infty,1)$-category.
The following types of data are equivalent:
\begin{itemize}
\item[$(1)$] Symmetric monoidal functors $Z: \calO \rightarrow \bfS$.
\item[$(2)$] Calabi-Yau algebras in $\bfS$.
\end{itemize}
The equivalence is implemented by carrying a functor $Z: \calO \rightarrow \bfS$
to the Calabi-Yau algebra $Z( [0,1])$.
\end{theorem}

\begin{theorem}[Costello \cite{costello}]\label{costelloB}
Let $\bfS$ be a good symmetric monoidal $(\infty,1)$-category, and let $Z_0:
\calO \rightarrow \bfS$ be a symmetric monoidal functor. Then
there is another symmetric monoidal functor
$Z: \OC \rightarrow \bfS$ such that $Z_0 = Z | \calO$, and $Z$ is universal with respect to this property $($more precisely, $Z$ is obtained from $Z_0$ by left Kan extension; see \cite{htt}$)$.
\end{theorem}

Theorems \ref{costelloA} and \ref{costelloB} were proven by Costello in the special case where $\bfS$ is
the $(\infty,1)$-category $\tChain(k)$ of Definition \ref{spout2} when $k$ is a field of characteristic zero. However, his methods are quite general, and can be adapted without essential change to prove the versions given above.

Let us briefly sketch how Theorems \ref{costelloA} and \ref{costelloB} can be used to prove
Theorem \ref{swisher18}. Consider first the symmetric monoidal functor
$Z_0: \OC \rightarrow \calC_1$ determined by a dualizable object $X \in \calC$. According to Theorem \ref{costelloA}, the restriction $Z_0| \calO$ is classified
by a Calabi-Yau algebra in $\calC_1$. Unwinding the definitions, we learn that
this algebra is $\End(X) \simeq X \otimes X^{\vee}$, with Calabi-Yau structure determined by
the evaluation map $\End(X) \rightarrow {\bf 1}$. Using Theorem \ref{costelloA}, we deduce that
lifting $Z_0 | \calO$ to a symmetric monoidal functor
$Z': \calO \rightarrow \widetilde{\calC}_1$
is equivalent to lifting $\End(X)$ to a Calabi-Yau algebra in $\widetilde{\calC}_1$.
Recall that the objects of $\widetilde{\calC}_{1}$ can be identified with morphisms
$\eta: {\bf 1} \rightarrow C$ in $\calC$. The requirement that $Z'$ preserve coCartesian
morphisms determines the lift $\widetilde{\End(X)}$ of $\End(X)$ as an algebra: it must be given by
the coevaluation map ${\bf 1} \rightarrow \End(X)$. Unwinding the definitions, one can show that 
a Calabi-Yau structure on the algebra $\widetilde{\End(X)}$ lifting the Calabi-Yau structure on
$\End(X)$ is equivalent to a Calabi-Yau structure on the object $X \in \calC$. It remains to prove that there is an essentially unique symmetric monoidal functor $Z_2: \OC \rightarrow \widetilde{\calC}_1$
which preserves coCartesian morphisms and is compatible with both $Z_0$ and $Z'$.
This can be deduced from a {\it relative} version of Theorem \ref{costelloB}
(where the notion of left Kan extension is replaced by the notion of relative left Kan extension with respect to the projection $\widetilde{\calC}_1 \rightarrow \calC_1$; see \cite{htt}). We will not describe the details here.


\begin{example}[String Topology (see \cite{chassullivan})]\label{stringtop}
Let $M$ be an oriented manifold of even dimension $2k$ and assume that $M$ is simply-connected. Let $R$ denote the graded algebra $\Q[x,x^{-1}]$ where $x$ has degree $2k$, and regard $R$ as a differential graded algebra with trivial differential. The collection of differential graded
$R$-modules can be organized into a symmetric monoidal $(\infty,1)$-category, which we will
denote by $\bfS$. The cochain complex $C^{\ast}( M; R)$ can be regarded as an algebra object
(even a commutative algebra object) of $\bfS$. Using the fact that $M$ is simply connected, one can show that the Hochschild homology $\int_{S^{1}} C^{\ast}(M;R)$ is quasi-isomorphic to the cochain
complex $C^{\ast}(LM, R)$, where $LM = M^{S^1}$ denotes the free loop space of $M$; moreover,
this identification is $\SO(2)$-equivariant. In particular, we have a canonical $\SO(2)$-equivariant map
$$ \tr: \int_{S^{1}} C^{\ast}(M;R) \simeq C^{\ast}( LM;R) \stackrel{t'}{\rightarrow} C^{\ast}(M;R)
\stackrel{t''}{\rightarrow} R,$$
where $t'$ is induced by the diagonal embedding $M \rightarrow LM$ and
$t''$ is given by evaluation on the fundamental cycle of $M$. The pair
$( C^{\ast}(LM;R), \tr)$ is a Calabi-Yau object of the symmetric monoidal $(\infty,2)$-category $\Alg_1( \bfS)$ (the nondegeneracy of $\tr$ follows from Poincare duality) and therefore determines
a topological field theory $Z: \Bord_{2}^{\non} \rightarrow \Alg_1(\bfS)$. We can identify $Z(S^1)$ with the complex of $R$-valued cochains $C^{\ast}( LM; R)$ on the free loop space of $M$. We can view surfaces with boundary as giving rise to operations on $C^{\ast}(LM;R)$, called {\it string topology operations} (see \cite{chassullivan}).
\end{example}

\begin{remark}
Example \ref{stringtop} can be refined in various ways. First, we can replace the cochain complex
$C^{\ast}(M; R)$ of $M$ with the chain complex $C_{\ast}( \Omega M;R)$ of the based loop space
$\Omega M$, which has the structure of an associative (but not commutative) algebra in
$\bfS$. The Hochschild homology $\int_{S^1} C_{\ast}(\Omega M; R)$ can be identified with
the chain complex $C_{\ast}(LM; R)$.
The algebra $C_{\ast}( \Omega M; R)$ is generally not a Calabi-Yau object of $\Alg_1(\bfS)$,
because it is not even dualizable as an object of $\bfS$ (the loop space $LM$ generally has homology in infinitely many degrees). However, it can be regarded as a Calabi-Yau object in the $(\infty,2)$-category $\Alg_1(\bfS)^{op}$ obtained by reversing the direction of $2$-morphisms in $\Alg_1(\bfS)$:
the fundamental cycle of $M$ gives rise to a nondegenerate, $\SO(2)$-invariant {\it cotrace}
$R \rightarrow C_{\ast}( LM; R)$. It therefore determines a symmetric monoidal functor
$Z: \Bord_{2}^{\ori, \non} \rightarrow \Alg_1(\bfS)^{\op}$ whose value on a circle
can be identified with the chain complex $C_{\ast}( LM; R)$. Evaluating $Z$ on manifolds of higher dimension, we obtain operations on $C_{\ast}(LM; R)$ which are {\em preduals} of the operations
of Example \ref{stringtop}, and are defined even if we do not assume that $M$ is simply connected.

It is also possible to drop the assumption that $M$ is even-dimensional, and to work over
the field $\Q$ (or other coefficient rings) rather than the periodic algebra $R$. However,
we encounter a new complication: the fundamental cycle of $M$ gives rise to a trace map
$\tr: C^{\ast}(LM; \Q) \rightarrow \Q$ which is not of degree zero, but involves a shift by the dimension of $M$. Nevertheless, we can view the pair $( C^{\ast}(M; \Q), \tr)$ as a Calabi-Yau object of an
appropriately defined elaboration of the $(\infty,2)$-category $\Alg_1(\bfS)$, where we allow twistings
by $2$-gerbes over $\Q$. In concrete terms, this means that the operations
$$C_{\ast}(LM; \Q)^{\otimes m} \rightarrow C_{\ast}(LM; \Q)^{\otimes n}$$
associated to a surface $\Sigma$ with $m$ incoming and $n$ outgoing boundary circles
is not of degree zero, but involves a homological shift whose magnitude depends on the
dimension of $M$ and the genus of $\Sigma$.
\end{remark}

\subsection{Manifolds with Singularities}\label{mans}

The cobordism hypothesis (Theorem \ref{swisher2}) asserts that
the higher category of framed bordisms $\Bord^{\fr}_{n}$ is freely generated, as a symmetric monoidal
$(\infty,n)$-category with duals, by a single object (corresponding to the $0$-manifold with a single point). In this section, we will describe a generalization of the cobordism hypothesis, which gives a geometric description of symmetric monoidal $(\infty,n)$-categories (again assumed to have duals) having more complicated presentations. 

To explain the basic idea, suppose that we are given an object $Y \in \Omega^{k-1} \Bord_{n}$, corresponding to a closed $(k-1)$-manifold. By definition, giving a $k$-morphism
$\emptyset \rightarrow Y$ in $\Bord_{n}$ is equivalent to giving a $k$-manifold whose boundary is identified with $Y$. Suppose that we wish to enlarge the $(\infty,n)$-category $\Bord_{n}$, to obtain a new $(\infty,n)$-category $\calC$ which contains a $k$-morphism $\alpha: \emptyset \rightarrow Y$.
We might then try to think of the $k$-morphisms in $\calC$ as given by some kind of ``generalized
$k$-manifolds''; in particular, we can try to think of $\alpha$ as a ``generalized $k$-manifold''
with boundary $Y$. In general, it is not possible to realize $Y$ as the boundary of a
smooth manifold of dimension $k$. However, there is always a canonical way to realize
$Y$ as the ``boundary'' of a $k$-dimensional topological space $Y'$. Namely, let $Y'$
denote the cone $C(Y) = ( Y \times [0,1] ) \coprod_{ Y \times \{1\} } \{v \}$, and set
$\bd Y' = Y \times \{0\} \subseteq Y'$. Then $Y'$ is a $k$-dimensional topological space
containing $Y$ as a closed subset, which is a manifold except possibly at a single point:
the vertex $v$ of the cone. The space $C(Y)$ is an example of a {\it manifold with singularities}:
it admits a decomposition $C(Y) = (C(Y) - \{v\}) \coprod \{v\}$ into locally closed subsets
which are manifolds and which fit together in a reasonably nice way. 

More generally, we can consider pairs $(M, M_0)$ where $M$ is a topological space of dimension $m$, 
$M_0 \subseteq M$ is a closed subset which is a smooth $(n-k)$-framed manifold of
dimension $(m-k)$ (so that $M_0$ is empty if $m < k$), the complement $M - M_0$ is a smooth manifold of dimension $m$, and we have a homeomorphism $U \simeq M_0 \times C(Y)$ for
some open neighborhood $U$ of $M_0$.
We can think of the pair $(M, M_0)$ as a kind of generalized $m$-manifold. Using these generalized manifolds in place of ordinary smooth manifolds, we can define an analogue of the $(\infty,n)$-category
$\Bord_{n}$; let us denote this analogue by $\Bord'_{n}$. As in the smooth case, one can show that
$\Bord'_{n}$ is a symmetric monoidal $(\infty,n)$-category with duals (the symmetric monoidal structure is given, as usual, by disjoint union). Every smooth $m$-manifold $M$ can be
regarded as a generalized $m$-manifold by taking $M_0 = \emptyset$. This construction determines
a symmetric monoidal functor $\Bord_{n} \rightarrow \Bord'_{n}$. In particular, we can regard $X$ and $\emptyset$ as objects of $\Omega^{k-1} \Bord'_{n}$. By construction, $C(Y)$ defines a $k$-morphism
$\alpha: \emptyset \rightarrow Y$ in $\Bord'_{n}$. In fact, $\Bord'_{n}$ is universal with respect
to these properties:

\begin{proposition}\label{cooba}
Let $\calC$ be a symmetric monoidal $(\infty,n)$-category with duals, let
$Z_0: \Bord_{n} \rightarrow \calC$ be a symmetric monoidal functor, and let $Y$ be a closed
$(k-1)$-manifold. The following types of data are equivalent:
\begin{itemize}
\item[$(1)$] Symmetric monoidal functors $Z: \Bord'_{n} \rightarrow \calC$ extending
$Z_0$, where $\Bord'_{n}$ is defined as above.
\item[$(2)$] Morphisms $\alpha: {\bf 1} \rightarrow Z_0(Y)$ in $\Omega^{k-1} \calC$.
\end{itemize}
\end{proposition}

In view of the description of $\Bord_{n}$ given by Theorem \ref{swisher4}, we can restate
Proposition \ref{cooba} more informally as follows: as a symmetric monoidal $(\infty,n)$-category with
duals, $\Bord'_{n}$ is freely generated by a single $\OO(n)$-equivariant object (corresponding
to a point) together with a single $k$-morphism (corresponding to the cone $C(Y)$). Our goal in this section is to explain a more general form of Proposition \ref{cooba}, which describes the
free $(\infty,n)$-category with duals generated by an arbitrary collection of objects, $1$-morphisms,
$2$-morphisms, and so forth, stopping at the level of $n$-morphisms. The description will be given in geometric terms: roughly speaking, the free $(\infty,n)$-category in question can be described
in terms of bordisms between manifolds with singularities, where we allow singularities whose local structure is determined by the pattern of generators. In order to make a more precise statement, we need to introduce a somewhat elaborate definition.

\begin{protodefinition}\label{singdat}
Fix an integer $n \geq 0$. Using a simultaneous induction on $0 \leq k \leq n$, we will define
the following:
\begin{itemize}
\item The notion of an {\it $n$-dimensional singularity datum of length $k$}.
\item If $\vect{X}$ is a singularity datum of length $k$ and $V$ is a real vector
space of dimension $\leq n-k$, the notion of a {\it $\vect{X}$-manifold of codimension $V$}.
\end{itemize}

The definitions are given as follows:
\begin{itemize}

\item[$(a)$] An $n$-dimensional singularity datum of length $0$ consists of
a pair $(X_0, \zeta_0)$, where $X_0$ is a topological space and $\zeta_0$ is a real vector
bundle of dimension $n$ on $X_0$ which is endowed with an inner product.

\item[$(b)$] If $0 < k \leq n$, then an $n$-dimensional singularity datum of length
$k$ is given by a quadruple $( \vect{X}, X_k, \zeta_k, p: E_k \rightarrow X_k)$, where
$\vect{X}$ is an $n$-dimensional singularity datum of length $k-1$, 
$X_k$ is a topological space, $\zeta_k$ is a real vector bundle of dimension $n-k$
on $X_k$ endowed with an inner product, and $p: E_k \rightarrow X_{k}$ is a fiber bundle whose fiber over each point $x \in X_{k}$ is a compact $\vect{X}$-manifold of codimension $\zeta_x \oplus \R$.

\item[$(c)$] Let $\vect{X}$ be an $n$-dimensional singularity datum of length
$k$, given by a quadruple $( \vect{X}', X_{k}, \zeta_k, p: E_k \rightarrow X_k)$.
Let $V$ be a real vector space of dimension $m \leq n-k$. A {\it $\vect{X}$-manifold of codimension $V$} consists of the following data:
\begin{itemize}
\item[$(i)$] A topological space $M$.
\item[$(ii)$] A closed subspace $M_k \subseteq M$, which is endowed with the structure of a smooth
manifold of dimension $n-m-k$, having tangent bundle $T$. 
\item[$(iii)$] A map $q: M_k \rightarrow X_k$ and an isomorphism of vector bundles
$T \oplus \underline{V} \simeq q^{\ast} \zeta$, where $\underline{V}$ denotes the constant
vector bundle on $M_k$ associated to $V$. This data endows the pullback
$q^{\ast} E = E \times_{ X_k} M_k$ with the structure of a $\vect{X}'$-manifold of
codimension $V \oplus \R$, so that $q^{\ast} E \times (0,1)$ has the structure of
a $\vect{X}'$-manifold of codimension $V$.
\item[$(iv)$] A structure of $\vect{X}'$-manifold of codimension $V$ on the open subset
$M - M_k \subseteq M$.
\item[$(v)$] An open neighborhood $U$ of $M_k$ and a continuous
quotient map $f: (0,1] \times q^{\ast} E \rightarrow U$ whose
restriction to $(0,1) \times q^{\ast} E$ is an open embedding of
$(0,1) \times q^{\ast} E \rightarrow M - M_k$ of $\vect{X}'$-manifolds of codimension
$V$ and whose restriction to $\{1\} \times q^{\ast} E$ coincides with the projection
$q^{\ast} E \rightarrow M_k$.
\end{itemize}
\end{itemize}
We will refer to an $n$-dimensional singularity datum of length $n$ simply as an
{\it $n$-dimensional singularity datum}. If $\vect{X}$ is an $n$-dimensional singularity datum and
$m \leq n$, then we define a {\it $\vect{X}$-manifold of dimension $m$} to be a $\vect{X}$-manifold
of codimension $\R^{n-m}$. 
\end{protodefinition}

\begin{remark}
Unwinding the induction, we see that an $n$-dimensional singularity datum
of length $k$ consists of a sequence of topological spaces $\{ X_i \}_{0 \leq i \leq k}$, 
a sequence of vector bundles $\{ \zeta_i \}_{0 \leq i \leq k}$ where each $\zeta_i$ has 
rank $n-i$ on $X_i$, and a sequence of fiber bundles $\{ E_i \rightarrow X_i \}_{0 \leq i \leq k}$.
\end{remark}

\begin{remark}
Any $n$-dimensional singularity datum $\vect{X}'$ of length $k$ can be completed to an
$n$-dimensional singularity datum $\vect{X}$ by taking the spaces $X_{i}$ to be
empty for $i > k$. In this situation, we will not distinguish between $\vect{X}'$ and
$\vect{X}$. In other words, we will think of $n$-dimensional singularity data of length
$k$ as $n$-dimensional singularity data for which the spaces $X_i$ are empty for $i > k$.
\end{remark}

\begin{remark}
Part $(a)$ of Definition \ref{singdat} can be regarded as a special case of
part $(b)$ if we make use the following conventions: 
\begin{itemize}
\item There is a unique $n$-dimensional singularity datum $\vect{X}$ of length $-1$.
\item Every $\vect{X}$-manifold is empty.
\end{itemize}
\end{remark}

\begin{remark}
Let $\vect{X} = ( \{ X_i \}_{0 \leq i \leq n}, \{ \zeta_i \}_{0 \leq i \leq n}, \{ p_i: E_i \rightarrow X_i \}_{0 \leq i \leq n})$ be an $n$-dimensional singularity datum. A $\vect{X}$-manifold of dimension $m$
consists of a topological space $M$ equipped with a stratification
$$ M_n \subseteq M_{n-1} \subseteq M_{n-2} \subseteq \ldots \subseteq M_0 = M,$$
where each open stratum $M_{k} - M_{k-1}$ is a smooth manifold of dimension $m-k$ (which
is empty if $m < k$) equipped with an $(X_k, \zeta_k)$-structure. Moreover, these smooth manifolds are required to ``fit together'' in a manner which is prescribed by the fiber bundles $p_i: E_i \rightarrow X_i$. 
\end{remark}

\begin{remark}
In Definition \ref{singdat}, we did not include any requirement that an $\vect{X}$-manifold
$M$ be compact. However, all of the $\vect{X}$-manifolds which we subsequently discuss will be assumed compact unless otherwise specified.
\end{remark}

\begin{example}
An $n$-dimensional singularity datum of length $0$ consists of a pair $(X, \zeta)$, where
$X$ is a topological space and $\zeta$ is a vector bundle of rank $n$ on $X$. The notion of
$(X, \zeta)$-manifold of dimension $m \leq n$ (appearing in Definition \ref{singdat}) agrees with
the notion of a smooth manifold with $(X, \zeta)$-structure (Notation \ref{cusper2}).
\end{example}

Let $\vect{X}$ be an $n$-dimensional singularity datum.
By elaborating on Definition \ref{singdat}, one can define the notions
{\it $\vect{X}$-manifold with boundary} and {\it bordism between $\vect{X}$-manifolds}.
Using $\vect{X}$-manifolds in place of ordinary manifolds, we can define an analogue of
the $(\infty,n)$-category $\Bord_n$, which we will denote by $\Bord_{n}^{\vect{X}}$.

\begin{remark}
The notion of a $\vect{X}$-manifold with boundary is described very naturally in the language of Definition \ref{singdat}: it is just a $\vect{X}'$-manifold, where $\vect{X}'$ is a singularity datum which can be extracted from $\vect{X}$. For example, the usual notion of a manifold with boundary 
(or, more precisely, of a manifold with a {\em collared} boundary) arises as a special case of
Definition \ref{singdat}; see Example \ref{bra} below.
\end{remark}

\begin{example}
Let $k \leq n$ be positive integers, and let $Y$ be a closed $(k-1)$-manifold.
We define an $n$-dimensional singularity datum 
$\vect{X} = ( \{ X_i \}_{0 \leq i \leq n}, \{ \zeta_i \}_{0 \leq i \leq n}, \{p_i: E_i \rightarrow X_i \}_{0 \leq i \leq n })$ as follows:
\begin{itemize}
\item[$(1)$] The topological space $X_0$ is a classifying space $\BO(n)$, the
topological space $X_k$ consists of a single point, and the topological space
$X_i$ is empty for $i \notin \{ 0, k \}$. 
\item[$(2)$] The vector bundle $\zeta_0$ is the tautological vector bundle of rank $n$
on $\BO(n)$, and the vector bundle $\zeta_k$ corresponds to the vector space $\R^{n-k}$.
\item[$(3)$] The fiber bundle $p_k: E_k \rightarrow X_k$ is given by the projection $Y \rightarrow \ast$.
\end{itemize}
The $(\infty,n)$-category $\Bord_{n}^{\vect{X}}$ can be identified with
the $(\infty,n)$-category $\Bord'_n$ appearing in Proposition \ref{cooba}.
\end{example}

We would now like to generalize Proposition \ref{cooba} to obtain a description of
the $(\infty,n)$-category $\Bord_{n}^{\vect{X}}$ for {\em any} $n$-dimensional singularity
datum $\vect{X}$. First, we need to introduce a bit of additional terminology. 
Recall that if $\calC$ is a symmetric monoidal $(\infty,n)$-category with duals, then
the underlying $(\infty,0)$-category $\calC^{\sim}$ carries an action of the orthogonal group
$\OO(n)$ (Corollary \ref{ail}). The group $\OO(n)$ does not act on the $(\infty,n)$-category
$\calC$ itself. For example, if $n=1$, then the nontrivial element in $\OO(n)$ acts by carrying
every object $X$ in $\calC$ to its dual $X^{\vee}$. A morphism $f: X \rightarrow Y$ does not
generally induce a morphism $X^{\vee} \rightarrow Y^{\vee}$ (unless $f$ is an isomorphism);
instead, it induces a dual map $f^{\vee}: Y^{\vee} \rightarrow X^{\vee}$. Nevertheless, the subgroup
$\OO(n-1) \subseteq \OO(n)$ naturally acts on the collection of $1$-morphisms in $\calC$.
To see this, let $[1]$ denote the ordinary category associated to the linearly ordered set
$\{ 0 < 1 \}$. One can endow the collection $\Fun( [1], \calC)$ of functors
$[1] \rightarrow \calC$ with the structure of a symmetric monoidal $(\infty,n-1)$-category with
duals (the appropriate construction is a bit subtle, since it is not an internal $\Hom$-object with respect
to the Cartesian product of higher categories), so that the $\infty$-groupoid $\Fun( [1], \calC)^{\sim}$
carries an action of the orthogonal group $\OO(n-1)$. More generally, the $\infty$-groupoid of $1$-morphisms in $\Omega^{k-1} \calC$ carries an action of the orthogonal group $\OO(n-k)$, which is
compatible with the action of $\OO(n+1-k)$ on the $\infty$-groupoid of objects of $\Omega^{k-1} \calC$.

\begin{theorem}[Cobordism Hypothesis with Singularities]\label{swisher12}
Let $0 < k \leq n$ be integers. Suppose we are given
an $n$-dimensional singularity datum $\vect{X}$ of length $k$, corresponding
to a quadruple $( \vect{X}', X, \zeta, p: E \rightarrow X)$ as in Definition
\ref{singdat}. Let $\widetilde{X} \rightarrow X$ denote the bundle of orthonormal frames in
$\zeta$ $($so that $\widetilde{X}$ is a principal $\OO(n-k)$-bundle over $X${}$)$.
For every point $\widetilde{x} = (x, \alpha: \zeta_{x} \simeq \R^{n-k})$, we 
can use $\alpha$ to view the fiber $p^{-1} \{x\}$ as a
$\vect{X}'$-manifold of codimension $\R^{n+1-k}$, which determines an object
$E_{ \widetilde{x} } \rightarrow \Omega^{k-1} \Bord_{n}^{\vect{X}'}$. 
We note that the assignment $\widetilde{x} \mapsto E_{ \widetilde{x} }$ is equivariant
with respect to the action of the orthogonal group $\OO(n-k)$.

Let $\calC$ be a symmetric monoidal $(\infty,n)$-category with duals.
Then giving a symmetric monoidal functor $Z: \Bord_{n}^{\vect{X}} \rightarrow \calC$
is equivalent to giving the following data:
\begin{itemize}
\item[$(1)$] A symmetric monoidal functor $Z_0: \Bord_{n}^{\vect{X}'} \rightarrow \calC$.
\item[$(2)$] A family of $1$-morphisms $\eta_{ \widetilde{x} }: {\bf 1} \rightarrow Z_0( E_{ \widetilde{x}})$
in $\Omega^{k-1} \calC$ parametrized by $\widetilde{x} \in \widetilde{X}$, such that
the assignment $\widetilde{x} \mapsto \eta_{\widetilde{x}}$ is $\OO(n-k)$-equivariant.
\end{itemize}
\end{theorem}

\begin{remark}
With the appropriate conventions, we can regard Theorem \ref{swisher5} as corresponding to the degenerate case of Theorem \ref{swisher12} where we take $k=0$.
\end{remark}

\begin{remark}
Let $\vect{X}$ be an $n$-dimensional singularity datum. Very roughly, Theorem \ref{swisher12} can be stated as follows: as a symmetric monoidal $(\infty,n)$-category with duals, 
$\Bord_{n}^{\vect{X}}$ is freely generated by adjoining a $k$-morphism for
every point of $X_k$, for $0 \leq k \leq n$. The source and target of these $k$-morphisms
are dictated by the details of the singularity datum $\vect{X}$.
\end{remark}

\begin{remark}
Theorem \ref{swisher12} is more general than it might first appear: it implies that
{\em any} symmetric monoidal $(\infty,n)$-category with duals which is freely generated by
adjoining $k$-morphisms for $0 \leq k \leq n$ can be realized as $\Bord^{\vect{X}}_{n}$ for a suitable
singularity datum $\vect{X}$. This is a consequence of the following general observation:
if $\calC$ is a symmetric monoidal $(\infty,n)$-category with duals and $k \leq n$, then for
any pair of $(k-1)$-morphisms $f$ and $g$ with the same source and target, the data of
a $k$-morphism from $f$ to $g$ is equivalent to the data of a $k$-morphism from
${\bf 1}$ to $h$ in $\Omega^{k-1} \calC$, for an appropriately chosen closed $(k-1)$-morphism
$h$. For example, when $k=1$, we note that $\bHom_{\calC}(f,g) \simeq \bHom_{\calC}( {\bf 1}, f^{\vee} \otimes g)$. The general assertion reflects the geometric idea that any $k$-manifold with corners can be regarded as a $k$-manifold with boundary by smoothing the corners in an appropriate way.
\end{remark}

\begin{remark}
Let $\vect{X} = ( \{ X_i \}_{0 \leq i \leq k}, \{ \zeta_i \}_{0 \leq i \leq k}, \{ p_i: E_i \rightarrow X_i \}_{0 \leq i \leq k})$ be an $n$-dimensional singularity datum of length $k$.
For every integer $m$, the triple
$( \{ X_i \}_{0 \leq i \leq k}, \{ \zeta_i \oplus \underline{ \R^{m} } \}_{0 \leq i \leq k}, \{ p_i: E_i \rightarrow X_i \}_{0 \leq i \leq k })$ corresponds to an $(n+m)$-dimensional singularty datum of length $k$,
which we will denote by $\vect{X}[\R^{m}]$. Implicit in this assertion is the following observation:
every $\vect{X}$-manifold can be regarded as a $\vect{X}[\R^{m}]$-manifold in a natural way.
Passing to the limit as $m \mapsto \infty$, we obtain the notion of a {\it stable} $\vect{X}$-manifold.
For example, suppose that $X_0$ consists of a single point and $X_i = \emptyset$ for $i > 0$.
In this case, a $\vect{X}$-manifold is a manifold equipped with an $n$-framing (see Variant \ref{barvar}), and a stable $\vect{X}$-manifold is a manifold $M$ equipped with a stable framing (that is, a trivialization of $T_M \oplus \underline{\R}^{m}$ for $m \gg 0$).

We can also consider the direct limit of the bordism categories $\Bord^{\vect{X}[\R^{m}]}_{n+m}$
as $m \rightarrow \infty$ to obtain a bordism theory of stable $\vect{X}$-manifolds. The classifying
space $\varinjlim | \Bord^{\vect{X}[\R^{m}]}_{n+m}|$ is an infinite loop space, and therefore represents a cohomology theory. The corresponding homology theory admits a geometric interpretation in terms of the bordism theories of manifolds with singularities. These homology theories were originally introduced by Baas and Sullivan (see \cite{baas}). We can therefore regard the $(\infty,n)$-categories $\Bord^{\vect{X}}_{n}$ as a providing a refinement of the Baas-Sullivan theory.
\end{remark}

The proof of Theorem \ref{swisher12} uses a general categorical construction.
Fix an integer $n > 0$, and suppose that $f: \calC \rightarrow \calD$ is a symmetric monoidal
functor, where $\calC$ is a symmetric monoidal $(\infty,n-1)$-category and $\calD$ a symmetric
monoidal $(\infty,n)$-category. We can associate to this data a new symmetric monoidal $(\infty,n)$-category $\Cone(f)$, which we will call the {\it cone} of $f$. This $(\infty,n)$-category can be described informally as follows:
\begin{itemize}
\item[$(a)$] The objects of $\Cone(f)$ are objects $D \in \calD$.
\item[$(b)$] Given a pair of objects $D, D' \in \calD$, a $1$-morphism from
$D$ to $D'$ is given by a pair $(C, \eta)$, where $C \in \calC$ and
$\eta \in \OHom_{\calD}( f(C) \otimes D, D')$ are objects (the collection of such pairs
$(C, \eta)$ can be organized into an $(\infty,n-1)$-category in a natural way which we will not describe in detail). 
\end{itemize}

\begin{remark}\label{usi}
Let $f: \calC \rightarrow \calD$ be as above. There is a symmetric monoidal
functor $\calD \rightarrow \Cone(f)$ which is the identity on objects, and carries
a morphism $\eta: D \rightarrow D'$ in $\calD$ to the morphism $( {\bf 1}, \eta)$ in
$\Cone(f)$.
\end{remark}

\begin{remark}
Suppose that $n = 2$, let $f: \calC \rightarrow \calD$ be as above. 
Then $\Omega \Cone(f)$ is a symmetric monoidal $(\infty,1)$-category
whose objects are pairs $(C, \eta)$, where $C \in \calC$ and $\eta: f(C) \rightarrow {\bf 1}$ is
a $1$-morphism in $\calD$. This is a mild variation on the $(\infty,1)$-category
$\calC[f]$ described in Notation \ref{pf}.
\end{remark}

\begin{remark}\label{corem}
Let $\calC$ be a symmetric monoidal $(\infty,n)$-category, let
$\ast$ denote the trivial $(\infty,n)$-category comprised of a single object, and let
$f: \calC \rightarrow \ast$ be the unique (symmetric monoidal) functor. We
will denote the symmetric monoidal $(\infty, n+1)$-category $\Cone(f)$ by $B \calC$,
and refer to it as the {\it connected delooping} of $\calC$. It is characterized up to equivalence
by the following properties:
\begin{itemize}
\item[$(i)$] The $(\infty,n+1)$-category $B \calC$ has only a single object (up to isomorphism).
\item[$(ii)$] There is a symmetric monoidal equivalence $\Omega( B \calC) \simeq \calC$.
\end{itemize}

More generally, if $k \geq 0$, we let $B^{k} \calC$ denote the $(\infty,n+k)$-category obtained by iterating the above construction $k$ times.
\end{remark}

The basic ingredient needed for the proof of Theorem \ref{swisher12} is the following:

\begin{theorem}[Cobordism Hypothesis, Relative Version]\label{swisher13}
Fix integers $0 < k \leq n$. Let $X$ be a topological space, $\zeta$ a vector
bundle of rank $n-k$ on $X$ endowed with an inner product, and
$\widetilde{X}$ the corresponding $\OO(n-k)$-bundle on $X$. Every point
$\widetilde{x} \in \widetilde{X}$ determines an object of $\Bord^{(X, \zeta)}_{n-k}$, which
we will denote by $P_{ \widetilde{x}}$. 

Let $Z_0: \calD \rightarrow \calC$ be a symmetric monoidal functor between symmetric monoidal
$(\infty,n)$-categories, and let $f: B^{k-1} \Bord_{n-k}^{(X, \zeta)} \rightarrow \calD$ be
a symmetric monoidal functor. Assume that $\calC$ has duals. The following types of data are equivalent:
\begin{itemize}
\item[$(1)$] Symmetric monoidal functors $Z: \Cone(f) \rightarrow \calC$ extending
$Z_0$ $($see Remark \ref{usi}$)$.
\item[$(2)$] A family of $1$-morphisms
$\eta_{ \widetilde{x}}: {\bf 1} \rightarrow Z_0( P_{\widetilde{x}})$ in
$\Omega^{k-1} \calC$ indexed by
$\widetilde{x} \in \widetilde{X}$, such that the assignment
$\widetilde{x} \rightarrow \eta_{ \widetilde{x} }$ is $\OO(n-k)$-equivariant.
\end{itemize}
\end{theorem}

\begin{example}
Suppose that $\calD$ is the trivial
$(\infty,n)$-category consisting of a single object. Then we can identify
$\Cone(f)$ with the $k$-fold delooping $B^{k} \Bord_{n-k}^{(X, \zeta)}$ and
Theorem \ref{swisher13} follows by applying Theorem \ref{swisher5} to
$\Omega^{k} \calC$.
\end{example}

The general case of Theorem \ref{swisher13} can also be reduced to Theorem \ref{swisher5}, using
formal properties of the cone construction $f \mapsto \Cone(f)$. We will omit the details, since they
involve a more detailed excursion into higher category theory. Instead, we explain how
Theorem \ref{swisher13} can be applied to the study of manifolds with singularities:

\begin{proof}[Proof of Theorem \ref{swisher12}]
Let $\vect{X} = ( \vect{X}', X, \zeta, p: E \rightarrow X)$ be an $n$-dimensional singularity datum
of length $k > 0$. Assume that $\zeta$ is equipped with an inner product, and let
$\widetilde{X} \rightarrow X$ be the $\OO(n-k)$-bundle of orthonormal frames associated to $\zeta$.
As explained in the statement of Theorem \ref{swisher12}, we can view the assignment
$\widetilde{x} \mapsto E_{\widetilde{x}}$ as defining an $\OO(n-k)$-equivariant map from
$\widetilde{X}$ into $\Omega^{k-1} \Bord_{n}^{\vect{X}'}$. According to Theorem \ref{swisher5}, such
a map is classified by a symmetric monoidal functor
$\Bord_{n-k}^{(X, \zeta)} \rightarrow \Omega^{k-1} \Bord_{n}^{\vect{X}'}$.
which can be ``delooped'' to obtain another symmetric monoidal functor
$f: B^{k-1} \Bord_{n-k}^{(X, \zeta)} \rightarrow \Bord_{n}^{\vect{X}'}$. To deduce
Theorem \ref{swisher12} from Theorem \ref{swisher13}, it suffices to observe that
the cone $\Cone(f)$ is canonically equivalent to $\Bord_{n}^{\vect{X}}$
(the definition of a $\vect{X}$-manifold is essentially rigged to produce this result).
\end{proof}

\begin{example}[Feynman Diagrams]
Let $\vect{X}$ be a $1$-dimensional singularity datum. Then $\vect{X}$ consists of the following data:
\begin{itemize}
\item A pair of topological spaces $X_0$ and $X_1$.
\item A rank $1$ vector bundle $\zeta_0$ on $X_0$, equipped with an inner product. Let $\widetilde{X}_0$ denote the associated double cover of $X_0$.
\item A covering space $E \rightarrow X_1$ with finite fibers, equipped with a continuous map
$E \rightarrow \widetilde{X}_0$.
\end{itemize}

For simplicity, let us assume that the homotopy groups $\pi_i \widetilde{X}_0$ vanish for
$i > 0$. Then the (weak) homotopy type of $\widetilde{X}_0$ is determined by the set
$P = \pi_0 \widetilde{X}_0$. We will refer to the elements of $P$ as {\it particles}. Since
$\widetilde{X}_0$ is a double covering of $X_0$, the set $P$ is equipped with a canonical involution
$p \mapsto \overline{p}$, which carries each particle to its corresponding {\it antiparticle}.

We will refer to the points of $X_1$ as {\it interactions}. For every point $x \in X_1$, the fiber
$E_x = E \times_{X_1} \{x\}$ is a finite set equipped with a map $\sigma_x: E_x \rightarrow P$. In other words,
we can think of $E_x$ as a finite set of particles (some of which might appear with multiplicity); these are the particles which participate in the relevant interaction.

By definition, a closed $\vect{X}$-manifold is given by the following:
\begin{itemize}
\item A compact topological space $G$, which is a smooth $1$-manifold away
from a specified finite subset $G_0 \subseteq G$.
\item Every oriented connected component $C$ of $G - G_0$ is labelled by a particle
$p \in P$. This particle depends on a choice of orientation of $C$, and is replaced by the corresponding antiparticle if the orientation is changed.
\item Every point $g \in G_0$ is labelled by an interaction $x \in X_1$.
Moreover, $g$ has a neighborhood in $G$ which is can be identified with the open cone
$( E_{x} \times (0,1] ) \coprod_{ E_x \times \{1\} } \{g \}$. Moreover, if $e \in E_x$ and
$C \subseteq G - G_0$ is the (oriented) connected component containing the interval $\{e\} \times (0,1)$
with its standard orientation, then $C$ is labelled with the particle $\sigma_x(e)$.
\end{itemize}

More informally, a (closed) $\vect{X}$-manifold consists of a graph $G$ (possibly with loops)
whose edges are labelled by elements of $P$ and whose vertices are labelled by points of $X_1$.
Such a graph is often called a {\it Feynman diagram}. We can informally summarize the situation by saying that there is an $(\infty,1)$-category $\Bord_{1}^{\vect{X}}$ whose objects are given
by finite sets labelled by elements of $P$ (in other words, finite collections of particles) and whose morphisms are given by Feynman diagrams. Theorem \ref{swisher12} asserts that this
$(\infty,1)$-category can be described by a universal mapping property. In particular, let
$\Vect(k)$ denote the category of vector spaces over a field $k$.
Using Theorem \ref{swisher12}, we deduce that symmetric monoidal functors
$Z: \Bord_{1}^{\vect{X}} \rightarrow \Vect_{\C}^{\fd}$ are classified by the following data:
\begin{itemize}
\item[$(i)$] A vector space $V_{p}$ for each particle $p \in P$. These vector
spaces should be endowed with a perfect pairing $V_{p} \otimes V_{ \overline{p} } \rightarrow k$, which is symmetric in the case where $p$ is its own antiparticle (in particular, each $V_p$ is finite-dimensional).
\item[$(ii)$] A vector $v_{x} \in \bigotimes_{ e \in E_x} V_{ \sigma(e) }$ for every interaction
$x \in X_1$ (which is continuous in the sense that it depends only on the path component of $x$ in $X_1$).
\end{itemize}
Given the data of $(i)$ and $(ii)$, we can construct a symmetric monoidal functor
$Z: \Bord_{1}^{\vect{X}} \rightarrow \Vect(k)$. 
In particular, if $G$ is a closed $\vect{X}$-manifold, we can evaluate
$Z$ on $G$ to obtain an invariant $Z(G) \in k$. It is easy to describe this invariant by a concrete procedure. Assume for simplicity that $G$ has no loops, let $G_0 \subseteq G$ be the finite subset
labelled by interactions $\{ x(g) \}_{g \in G_0}$, and set 
$$v = \otimes_{g \in G_0} v_{x(g)} \in \bigotimes_{g \in G_0, e \in E_{x(g)}} V_{\sigma(e)}.$$
Then $Z(G) \in k$ is the image of $v$ under the map
$$ \bigotimes_{g \in G_0, e \in E_{x(g)}} V_{\sigma(e)} \rightarrow k$$
induced by the pairings $V_{p} \otimes V_{ \overline{p}} \rightarrow k$ given by $(i)$
(note that the set of pairs $\{ (g,e): g \in G_0, e \in E_{x(g)} \}$ can be identified with
the collection of oriented connected components of $G - G_0$; in particular, this set
has two elements for each component of $G - G_0$, and the corresponding vector spaces are canonically dual to one another). 
\end{example}

\begin{example}[Boundary Conditions]\label{bra}
For any integer $n$, we can define an $n$-dimensional singularity datum
$\vect{X} = ( \{ X_i \}_{0 \leq i \leq n}, \{ \zeta_i \}_{0 \leq i \leq n}, \{ E_i \rightarrow X_i \}_{0 \leq i \leq n})$
as follows:
\begin{itemize}
\item The space $X_0$ is a classifying space $\BO(n)$, the space
$X_1$ is a classifying space $\BO(n-1)$, and the spaces $X_i$ are empty for $i > 0$.
\item The vector bundles $\zeta_0$ and $\zeta_1$ are the tautological vector bundles
on $\BO(n)$ and $\BO(n-1)$.
\item The map $E_1 \rightarrow X_1$ is a homeomorphism.
\end{itemize}
Unwinding the definitions, we see that a $\vect{X}$-manifold $M$ is just a manifold with boundary (more precisely, it is a manifold with boundary together with a specified collar of the boundary). It is sensible to talk about bordisms between manifolds with boundary (here we do not require our bordisms to be trivial on the boundary: a bordism from a manifold with boundary $M$ to another manifold with boundary $M'$ determines, in particular, a bordism from $\bd M$ to $\bd M'$ in the usual sense), bordisms between
bordisms between manifolds with boundary, and so forth: we thereby obtain a symmetric monoidal
$(\infty,n)$-category $\Bord_{n}^{\vect{X} }$. There is a canonical functor 
$:\Bord_{n} \rightarrow \Bord_{n}^{\vect{X}}$ which reflects the fact that any closed manifold can be regarded as a manifold with boundary by taking the boundary to be empty.

Let $\calC$ be a symmetric monoidal $(\infty,n)$-category and let
$Z_0: \Bord_{n} \rightarrow \calC$ be a symmetric monoidal functor. We can think
of $Z_0$ as a topological field theory; in particular, it assigns to
every closed $n$-manifold $M$ an invariant $Z_0(M) \in \Omega^{n} \calC$. In practice, the
target $(\infty,n)$-category $\calC$ will often have a linear-algebraic flavor, and 
$\Omega^{n} \calC$ can be identified with the set of complex numbers. In this case, 
we can think of $Z_0(M)$ as a complex number, which is often given heuristically
as the value of some integral over a space of maps from $M$ into a target $T$.
If $Z_0$ can be extended to a functor $Z: \Bord_{n}^{\vect{X}} \rightarrow \calC$, then the assignment
$M \mapsto Z_0(M)$ can be extended to assign invariants not only to closed $n$-manifolds but also $n$-manifolds with boundary. In this case, we can think of $Z(M) \in \Omega^{n} \calC$ as again
given by an integral: this time not over the space of all maps from $M$ into $T$, but instead
over the collection of maps which have some specified behavior at the boundary $\bd M$. 
It is typical to speak of the extension $Z$ of $Z_0$ as corresponding to a
{\it boundary condition}, or a {\it D-brane}.

If $\calC$ has duals, then Theorem \ref{swisher12} provides a classification for symmetric monoidal
functors $Z: \Bord_{n}^{\vect{X}} \rightarrow \calC$. They are determined (up to canonical isomorphism) by the following data:
\begin{itemize}
\item[$(1)$] An object $C \in \calC$, which is a (homotopy) fixed point with respect to the action of
$\OO(n)$ on $\calC^{\sim}$ (and determines the restriction $Z_0: \Bord_{n} \rightarrow \calC$ of
$Z$). 
\item[$(2)$] A $1$-morphism ${\bf 1} \rightarrow C$ in $\calC$, which is equivariant with
respect to the action of the group $\OO(n-1)$ (which encodes the relevant boundary condition).
\end{itemize}

Of course, there are a number of variations on this example. For example, we could replace
each of the spaces $\BO(n)$ and $\BO(n-1)$ by a single point in the definition of $\vect{X}$. In this case, a $\vect{X}$-manifold would consist of an $n$-framed manifold with an $(n-1)$-framed boundary, and we 
should drop the equivariance requirements in $(1)$ and $(2)$ above.
\end{example}

\begin{example}[Domain Walls]
If $\vect{X}$ is an $n$-dimensional singularity datum, we have been loosely referring to
$\vect{X}$-manifolds as ``manifolds with singularities''. However, this description is sometimes misleading: 
it is possible to give nontrivial examples of singularity data $\vect{X}$ such that the underlying
topological space of every $\vect{X}$-manifold is a smooth manifold. Roughly speaking, this corresponds to the situation where the fibers of each bundle $E_i \rightarrow X_i$ are spheres
$S^{i-1}$ (and the structure group of the bundle can be reduced to $\OO(i)$). We will discuss the simplest nontrivial example here; a more sophisticated variation will appear in our discussion of the tangle hypothesis in \S \ref{tangus}.

Consider the $n$-dimensional singularity datum
$\vect{X} = ( \{ X_i \}_{0 \leq i \leq n}, \{ \zeta_i \}_{0 \leq i \leq n}, \{ E_i \rightarrow X_i \}_{0 \leq i \leq n})$
defined as follows:
\begin{itemize}
\item The space $X_0$ is a disjoint union $\BO(n) \coprod \BO(n)$, the space
$X_1$ is a classifying space $\BO(n-1)$, and the spaces $X_i$ are empty for $i > 0$.
\item The vector bundles $\zeta_0$ and $\zeta_1$ are the tautological vector bundles
on $\BO(n)$ and $\BO(n-1)$. Note that an $(X_0, \zeta_0)$-structure on a manifold
$M$ of dimension $m \leq n$ consists of a decomposition $M \simeq M_{-} \coprod M_{+}$ of
$M$ into two disjoint open subsets.
\item We have $E_1 = X_1 \times \{ x,y \}$, where we regard $\{ x, y\}$ as a $0$-dimensional
$(X_0, \zeta_0)$-manifold via the decomposition $\{ x,y \} = \{x \} \coprod \{y\}$. 
\end{itemize}

Unwinding the definitions, we see that a closed $\vect{X}$-manifold of dimension $m \leq n$ consists
of a closed $m$-manifold $M$ equipped with a decomposition $M \simeq M_{-} \coprod_{M_0} M_{+}$, where $M_{-}$ and $M_{+}$ are codimension $0$ submanifolds of $M$ which meet along their
$M_0 = \bd M_{-} = \bd M_{+} = M_{-} \cap M_{+}$.

Every manifold $M$ of dimension $\leq n$ admits the structure of a $\vect{X}$-manifold in several different ways. For example, we can obtain a decomposition $M \simeq M_{-} \coprod_{M_0} M_{+}$ by
taking either $M_{-}$ or $M_{+}$ to be empty. These two recipes give rise to symmetric monoidal functors $j_{+}, j_{-}: \Bord_{n} \rightarrow \Bord_{n}^{\vect{X}}$. In particular, every symmetric monoidal
functor $Z: \Bord_{n}^{\vect{X}} \rightarrow \calC$ determines {\em two} topological field
theories $Z_{+}, Z_{-}: \Bord_{n} \rightarrow \calC$ via composition with $j_{+}$ and $j_{-}$, respectively.
A choice of symmetric monoidal functor $Z: \Bord_{n}^{\vect{X}} \rightarrow \calC$ giving rise
to $Z_{+} = Z \circ j_{+}$ and $Z_{-} = Z \circ j_{-}$ reflects a certain relationship between
$Z_{+}$ and $Z_{-}$. Theorem \ref{swisher12} allows us to describe the nature of this relationship in reasonably simple terms. Assuming that $\calC$ has duals, it asserts that symmetric monoidal functors
$Z: \Bord_{n}^{\vect{X}} \rightarrow \calC$ are classified by the following data:
\begin{itemize}
\item[$(1)$] A pair of objects $C, D \in \calC$, each of which is a (homotopy) fixed point for the action of
$\OO(n)$ on $\calC^{\sim}$ (these objects determine the topological field theories
$Z_{+}$ and $Z_{-}$, respectively).
\item[$(2)$] A $1$-morphism ${\bf 1} \rightarrow C \otimes D$ in $\calC$, which is equivariant
with respect to the action of $\OO(n-1)$.
\end{itemize}
Note that if $C$ and $D$ are as in $(1)$, then the $\OO(1) \times \OO(n-1)$
invariance of $C$ implies that there is an $\OO(n-1)$-equivariant equivalence
$C \simeq C^{\vee}$, so that the data of $(2)$ is equivalent to the data of an
$\OO(n-1)$-equivariant $1$-morphism $C \rightarrow D$ in $\calC$. 
\end{example}

\subsection{The Tangle Hypothesis}\label{tangus}

Fix an integer $n \geq 1$, and recall that $\tunCob(n)$ denotes the
$(\infty,1)$-category which may be described informally as follows:

\begin{itemize}
\item[$(i)$] The objects of $\tunCob(n)$ are closed manifolds of dimension
$(n-1)$.
\item[$(ii)$] Given a pair of closed $(n-1)$-manifolds $M$ and $N$, 
$\OHom_{ \tunCob(n)}(M,N)$ is a classifying
space for bordisms from $M$ to $N$.
\end{itemize}

In \S \ref{swugg}, we gave a more precise description of $\tunCob(n)$ using the language of Segal
spaces. In order to do so, we introduced a bit of auxiliary data: rather than taking arbitrary
$(n-1)$-manifolds as objects, we instead considered $(n-1)$-manifolds $M$ equipped with an embedding $M \hookrightarrow \R^{\infty}$. There is essentially no cost in doing so, since the space of embeddings $M \hookrightarrow \R^{\infty}$ is contractible (by general position arguments).
Nevertheless, it gives a bit of additional information: namely, it allows us to realize
$\tunCob(n)$ as the direct limit of $(\infty,1)$-categories $\tunCob(n)^{V}$, where
$V$ ranges over the finite dimensional subspaces of $\R^{\infty}$. These $(\infty,1)$-categories
are {\em not} equivalent to $\tunCob(n)$, because the space of embeddings $M \hookrightarrow V$
is generally not contractible when $V$ has finite dimension (it can even be empty if the dimension of $V$ is too small). Our goal in this section is to discuss the analogue of the cobordism hypothesis for
these (higher) categories of {\em embedded} cobordisms.

Our first step is to define a more elaborate version of the $(\infty,1)$-category
$\tunCob(n)^{V}$, which includes information about all manifolds of dimension $\leq n$.
In order to simplify the discussion, we will restrict our attention to the {\em framed} case.

\begin{definition}\label{kfr}
Let $0 \leq k \leq n$ be integers, and let $M$ be an $m$-manifold equipped with an $n$-framing.
A {\it $k$-framed submanifold} of $M$ consists of the following data:
\begin{itemize}
\item[$(1)$] A submanifold $M_0 \subseteq M$ of codimension $(n-k)$. Together
with the $n$-framing of $M$, the normal bundle to $M_0$ in $M$ determines
a ``Gauss map'' $g: M_0 \rightarrow \Gr_{n-k, n}$, where $\Gr_{n-k,n}$ denote the
real Grassmannian $\OO(n)/( \OO(k) \times \OO(n-k))$.
\item[$(2)$] A nullhomotopy of the map $g$.
\end{itemize}
If $M$ is equipped with boundary (or with corners), then we will assume that $M_0$ intersections
the boundary $\bd M$ (or the corners) transversely.
\end{definition}

\begin{protodefinition}
Fix integers $0 \leq k \leq n$, and let $V$ be a framed $(n-k)$-manifold.
We define an $(\infty,k)$-category $\Tang^{V}_{k,n}$ as follows:
\begin{itemize}
\item[$(a)$] The objects of $\Tang^{V}_{k,n}$ are (compact) $k$-framed submanifolds of $V$.
\item[$(b)$] Given a pair of objects $M_0, M_1 \in \Tang^{V}_{k,n}$, a $1$-morphism
from $M_0$ to $M_1$ is a (compact) $k$-framed submanifold $M \subseteq V \times [0,1]$ whose
intersection with $V \times \{i\}$ coincides with $M_i$.
\item[$(c)$] More generally, if $j \leq k$, we can identify $j$-morphisms $f$ in $\Tang^{V}_{k,n}$ with
$k$-framed submanifolds of $V \times [0,1]^{j}$, satisfying certain boundary conditions (corresponding to a choice of domain and codomain for $f$). If $j = k$, we regard the collection of such $j$-morphisms
as a topological space (so that higher morphisms in $\Tang^{V}_{k,n}$ correspond to homotopies, paths between homotopies, and so forth).
\end{itemize}
In the case where $V$ is the open unit disk in $\R^{n-k}$, we will simply denote $\Tang^{V}_{k,n}$ by $\Tang_{k,n}$.
\end{protodefinition}

\begin{remark}
For every pair of integers $k \leq n$, a linear inclusion $\R^{n-k} \subseteq \R^{n-k+1}$
induces a functor $\Tang_{k,n} \rightarrow \Tang_{k,n+1}$ between $(\infty,k)$-categories.
In particular, we can form the direct limit $\varinjlim_{n} \Tang_{k,n}$: this direct limit
is canonically equivalent to the framed bordism $(\infty,k)$-category $\Bord^{\fr}_{k}$.
\end{remark}

According to the cobordism hypothesis, $\Bord^{\fr}_{k}$ can be regarded as the free symmetric monoidal $(\infty,k)$-category with duals generated by a single object. We would like to formulate an analogous assertion for the $(\infty,k)$-categories $\Tang_{k,n}$. Our first observation
is that the $(\infty,k)$-categories $\Tang_{k,n}^{V}$ are generally not symmetric monoidal:
given a pair of submanifolds $M, M' \subseteq V$, there is generally no natural way to embed the disjoint union $M \coprod M'$ into $V$. Nevertheless, there is a good substitute for the symmetric monoidal
structure in the case where $V$ is an open disk. Given an open embedding
$$ V \coprod \cdots \coprod \cdots \coprod V \rightarrow V$$
which is rectilinear on each component, we get a functor
$$ \Tang_{k,n} \times \cdots \times \Tang_{k,n} \rightarrow \Tang_{k,n}.$$
In other words, $\Tang_{k,n}$ carries an action of the little $(n-k)$-disks
operad (Notation \ref{diskop}) and can therefore be regarded as an $E_{n-k}$-monoidal
$(\infty,k)$-category. This $E_{n-k}$-monoidal $(\infty,k)$-category can be characterized by
a universal property (a version of which was conjectured by Baez and Dolan in \cite{baezdolan}):

\begin{theorem}[Baez-Dolan Tangle Hypothesis]\label{swasher}
Fix integers $0 \leq k \leq n$. Let $\calC$ be an $E_{n-k}$-monoidal $(\infty,k)$-category with duals
$($if $k < n${}$)$ or with adjoints $($if $k=n${}$)$, and let $\ast$ denote the object of
$\Tang_{k,n}$ corresponding to the origin $0 \in \R^{n-k}$ $($regarded as a framed submanifold of the unit disk$)$. Evaluation at $\ast$ induces an equivalence of $(\infty,k)$-categories
$$ \Fun^{\otimes}( \Tang_{k,n}, \calC) \rightarrow \calC^{\sim}.$$
\end{theorem}

If $\calC$ is {\em any} $(\infty,k)$-category with an $E_{n-k}$-structure, then there
exists a maximal subcategory $\calC_0 \subseteq \calC$ satisfying the hypothesis of Theorem \ref{swasher}. We will say that an object $C \in \calC$ is {\it fully dualizable} if it belongs to this subcategory (if $n > k$, this is equivalent to the notion introduced in Definition \ref{justic};
if $n = k$ the condition is vacuous). We can informally summarize Theorem \ref{swasher}
by saying that $\Tang_{k,n}$ is freely generated as an $E_{n-k}$-monoidal $(\infty,k)$-category by a single fully dualizable object.

\begin{example}
Suppose that $k = n$. Then $\Tang_{k,n}$ is an $\infty$-groupoid which we can realize as the
fundamental $\infty$-groupoid of a topological space: namely, the configuration space of finite
subsets of the open unit disk in $\R^{n}$. Theorem \ref{swasher} asserts that this configuration space
is freely generated by a single point, as a representation of the $n$-disks operad $\calE_{n}$.
\end{example}

Let $\calC$ be a symmetric monoidal $(\infty,k)$-category. Using the fact that $\Bord_{k}^{\fr}$ can be obtained as a direct limit $\varinjlim_{n} \Tang_{k,n}$, we deduce that the $(\infty,k)$-category
of symmetric monoidal functors
$\Fun^{\otimes}( \Bord_{k}^{\fr}, \calC)$ is equivalent to the (homotopy) inverse limit of
the $(\infty,k)$-categories $\Fun^{\otimes}( \Tang_{k,n}, \calC)$ appearing in Theorem \ref{swasher}
(note that the symmetric monoidal $(\infty,k)$-category $\calC$ can be regarded as endowed with
an $E_{m}$-structure for every $m \geq 0$). Consequently, Theorem \ref{swasher} immediately
implies Theorem \ref{swisher3}. In other words, we can regard the tangle hypothesis as a
{\em refinement} of the cobordism hypothesis. Our goal in this section is to show that the converse is true as well: we can deduce the tangle hypothesis from the cobordism hypothesis, provided that the cobordism hypothesis is formulated in a sufficiently general form (namely, Theorem \ref{swisher12}). 

The first step is to rephrase Theorem \ref{swasher} in a way that does not mention $E_{n-k}$-monoidal structures. For this, we need the following general observation (which already appears implicitly in
\S \ref{mans}; see Remark \ref{corem}):

\begin{remark}\label{delop}
Let $\calC$ be a an $(\infty,n)$-category containing a distinguished object ${\bf 1}$.
We let $\Omega \calC$ denote the $(\infty,n-1)$-category $\OHom_{\calC}( {\bf 1}, {\bf 1})$.
Composition in $\calC$ endows the $(\infty,n-1)$-category $\Omega \calC$ with a monoidal
structure. Conversely, suppose that $\calD$ is a monoidal $(\infty,n-1)$-category. We can then construct
an $(\infty,n)$-category $B \calD$ having a single object ${\bf 1}$, with
$\OHom_{ B \calD}( {\bf 1}, {\bf 1}) \simeq \calD$ and with composition in 
$B \calD$ given by the monoidal structure on $\calD$.
These two constructions are adjoint to one another, and determine an equivalence between the following types of data:
\begin{itemize}
\item[$(1)$] Monoidal $(\infty,n-1)$-categories $\calD$.
\item[$(2)$] $(\infty,n)$-categories $\calC$ having only a single (distinguished) object,
up to isomorphism.
\end{itemize}

More generally, if $k \leq n$, then by applying the above constructions iteratively we obtain an equivalence between the following types of data:
\begin{itemize}
\item[$(1')$] $E_{n-k}$-monoidal $(\infty,k)$-categories $\calD$.
\item[$(2')$] $(\infty,n)$-categories $\calC$ having only a single (distinguished) $j$-morphism for
$j < n-k$.
\end{itemize}
Under this equivalence, an $E_{n-k}$ $(\infty,k)$-category $\calD$ has duals
if and only if the $(\infty,n)$-category $\calC$ has adjoints.
\end{remark}

We can use Remark \ref{delop} to reformulate Theorem \ref{swasher} as follows:

\begin{theorem}[Tangle Hypothesis, Framed $(\infty,n)$-Category Version]\label{swasher2}
Fix integers $0 \leq k \leq n$ and let $\calC$ be an $(\infty,n)$-category with adjoints.
Then $\Fun( B^{n-k} \Tang_{k,n}, \calC)$ is an $\infty$-groupoid which classifies pairs
$( {\bf 1}, \eta)$ where ${\bf 1}$ is an object of $\calC$ and $\eta$ is an object of
$\Omega^{n-k} \calC$.
\end{theorem}

\begin{remark}\label{cabbis}
Given a functor $Z: B^{n-k} \Tang_{k,n} \rightarrow \calC$, it is easy to extract the corresponding pair
$( {\bf 1}, \eta)$ in Theorem \ref{swasher2}: we obtain ${\bf 1} \in \calC$ by evaluating
$Z$ on the (unique) distinguished object $B^{n-k} \Tang_{k,n}$, and $\eta$ by
evaluating $\Omega^{n-k} Z$ on the object $\{ 0 \} \subseteq \R^{n-k}$.
\end{remark}

The difficult part is to go in the other direction: that is, to construct a functor $Z: B^{n-k} \Tang_{k,n} \rightarrow \calC$ given the pair $({ \bf 1}, \eta)$. We would like to obtain such a construction
applying the cobordism hypothesis with singularities (Theorem \ref{swisher12}). First, we need
to embark on a brief digression to describe the appropriate symmetric monoidal $(\infty,n)$-category
to use as a target.

Let $\Cat_{(\infty,1)}$ denote the (large) $(\infty,1)$-category whose objects are (small)
$(\infty,1)$-categories and whose morphisms are given by functors. This $(\infty,1)$-category
admits a symmetric monoidal structure given by the formation of Cartesian products. Moreover,
there is a canonical involution on $\Cat_{(\infty,1)}$, which carries each $(\infty,1)$-category
$\calC$ to its opposite $\calC^{op}$. This involution resembles a duality functor. However,
it does not correspond to duality in $\Cat_{(\infty,1)}$, because there are no candidates for evaluation and coevaluation functors
$$ \ev: \calC \times \calC^{op} \rightarrow {\bf 1} \quad \quad \coev: {\bf 1} \rightarrow \calC \times \calC^{op}$$
We can remedy the situation by passing to an enlargement of $\Cat_{(\infty,1)}$ which
has a more general class of morphisms. More precisely, we can introduce a new $(\infty,1)$-category $\Cat_{(\infty,1)}^{\Adj}$ whose objects are (small) $(\infty,1)$-categories and whose morphisms are given by {\it correspondences} between $(\infty,1)$-categories:
that is, we define a $1$-morphism from $\calC$ to $\calD$ in $\Cat^{\Adj}_{(\infty,1)}$ to be
a functor $\calC \times \calD^{op} \rightarrow \Cat_{(\infty,0)}$. In this setting, there is a natural candidate for the maps $\ev$ and $\coev$ indicated above: namely, we can take both to be correspondences
described by the functor $\calC^{op} \times \calC \rightarrow \Cat_{(\infty,0)}$ given by the formula
$(C,D) \mapsto \OHom_{\calC}(C,D)$. It is not difficult to check that the evaluation and coevaluation
maps are compatible with one another, which shows that $\Cat^{\Adj}_{(\infty,1)}$ is a
symmetric monoidal $(\infty,1)$-category with duals.

The above construction can be generalized (with some effort) to the setting of
$(\infty,n)$-categories, for every $n \geq 0$. More precisely, it is possible to introduce
a symmetric monoidal $(\infty,n)$-category $\Cat_{(\infty,n)}^{\Adj}$ with duals whose objects
are $(\infty,n)$-categories with adjoints and whose $k$-morphisms for $1 \leq k \leq n$
are given by a suitably general notion of correspondence.

\begin{warning}\label{twofa}
There is some danger of confusion when thinking of an $(\infty,n)$-category with adjoints
$\calC$ as an object of $\Cat_{(\infty,n)}^{\Adj}$, because there are morphisms
in $\Cat_{(\infty,n)}^{\Adj}$ which do not correspond to actual functors. It is possible for two
$(\infty,n)$-categories with adjoints to be equivalent as objects of $\Cat_{(\infty,n)}^{\Adj}$ without being equivalent as $(\infty,n)$-categories. For example, if $\calC$ and $\calD$ are $(\infty,1)$-categories,
then $\calC$ and $\calD$ are equivalent in $\Cat_{(\infty,1)}^{\Adj}$ if and only if they
are {\it Morita equivalent}: that is, if and only if they have equivalent idempotent completions
(see \cite{htt}).
\end{warning}   

\begin{remark}\label{swing}
The assertion that $\Cat^{\Adj}_{(\infty,n)}$ has duals has a remarkable consequence
when combined with Corollary \ref{ail}: it implies that there is an action of the orthogonal
group $\OO(n)$ on the theory of $(\infty,n)$-categories with adjoints. This action has the following features:
\begin{itemize}
\item[$(i)$] When restricted to the subgroup $\OO(1) \times \cdots \times \OO(1) \subseteq \OO(n)$,
the action of $\OO(n)$ on $\Cat^{\Adj, \sim}_{(\infty,n)}$ is given by a collection of involutions, each of which acts by replacing an $(\infty,n)$-category $\calC$ by the opposite $(\infty,n)$-category at the level of $k$-morphisms for $1 \leq k \leq n$. In particular, this restricted action can be defined without the
assumption that our $(\infty,n)$-categories have adjoints.

\item[$(ii)$] Let $n=2$, and let $\calC$ be an $(\infty,2)$-category with adjoints.
Then the action of the circle $S^1 \simeq \SO(2)$ on the object $\calC \in \Cat^{\Adj, \sim}_{(\infty,n)}$ determines a self-functor $A: \calC \rightarrow \calC$. It is possible to describe this
self-functor in concrete terms: it is the identity on objects, and carries every
$1$-morphism $f \in \calC$ to $f^{LL}$, the left adjoint of the left adjoint of $f$.

\item[$(iii)$] Let $\calC$ be a symmetric monoidal $(\infty,n)$-category with duals.
As we observed in \S \ref{mans}, the action of $\OO(n)$ on $\calC^{\sim}$ provided
by Corollary \ref{ail} does not extend to an action of $\OO(n)$ on $\calC$ itself.
However, it does extend to a {\em twisted} action of $\OO(n)$ on $\calC$ where
the twist is provided by the action of $\OO(n)$ on $\Cat^{\Adj, \sim}_{(\infty,n)}$ itself.
For example, when $n = 1$, the action of the nontrivial element $\eta \in \OO(1)$
on $\calC^{\sim}$ is given by carrying every object $X$ to its dual $X^{\vee}$.
The construction $X \mapsto X^{\vee}$ is a contravariant functor from $\calC$ to itself, and determines
an equivalence $\calC \simeq \calC^{op}$, where $\calC^{op}$ is the $(\infty,1)$-category
obtained by applying $\eta$ to $\calC$ (by virtue of $(i)$).
\end{itemize}
\end{remark}

\begin{remark}
The existence of an action of $\OO(n)$ on $\Cat^{\Adj,\sim}_{(\infty,n)}$ can be phrased another way:
the notion of an $(\infty,n)$-category with adjoints can be taken to depend not on a choice
of nonnegative integer $n$, but instead on a choice of finite dimensional inner product space
$V$ having dimension $n$.
\end{remark}

\begin{remark}
The relationship between $\Cat_{(\infty,n)}$ and $\Cat_{(\infty,n)}^{\Adj}$ is
analogous to the relationship between the higher categories $\Alg^{(n)}(\bfS)$ and
$\Alg_{(n)}^{\degree}(\bfS)$ introduced in \S \ref{topchir}. In fact, $\Alg_{(n)}^{\degree}(\bfS)$ and
$\Cat_{(\infty,n)}^{\Adj}$ admit a common generalization.
If $\bfS$ is a good symmetric monoidal $(\infty,1)$-category, then we can introduce a theory
of {\it $\bfS$-enriched $(\infty,n)$-categories}: that is, $(\infty,n)$-categories in which the
collections of $n$-morphisms are regarded as objects of $\bfS$. When $\bfS$ is the $(\infty,1)$-category of spaces, we recover the usual notion of $(\infty,n)$-category; when $\bfS$ is the ordinary category of sets, we recover the notion of $n$-category. The collection of $\bfS$-enriched $(\infty,n)$-categories
can be organized into a symmetric monoidal $(\infty,n)$-category $\Cat_{(\infty,n)}^{\Adj, \bfS}$ with
duals, which reduces to $\Cat_{(\infty,n)}^{\Adj}$ when $\bfS$ is the $(\infty,1)$-category of spaces.
We can regard $\Cat_{(\infty,n)}^{\Adj, \bfS}$ as an {\em enlargment} of the
$(\infty,n)$-category $\Alg_{(n)}^{\degree}$, because an $E_{n}$-algebra in $\bfS$
can be viewed as a $\bfS$-enriched $(\infty,n)$-category which has only a single $k$-morphism
for $k < n$.
\end{remark}

The $(\infty,n)$-category $\Cat^{\Adj}_{(\infty,n)}$ can also be regarded as an enlargement of
the $(\infty,n)$-category $\Fam_{n}$ described in \S \ref{slabun}: instead of merely considering
topological spaces and correspondences between topological spaces, we consider
$(\infty,n)$-categories and correspondences between $(\infty,n)$-categories. There
is also an analogue of the $(\infty,n)$-category $\Fam_n^{\ast}$ of Variant \ref{famvar},
which we will denote by $\Cat^{\Adj, \ast}_{(\infty,n)}$, whose objects can be viewed
as $(\infty,n)$-categories $\calC$ which have adjoints and are equipped with a distinguished
object ${\bf 1} \in \calC$. There is a forgetful functor 
$\pi: \Cat^{\Adj,\ast}_{(\infty,n)} \rightarrow \Cat^{\Adj}_{(\infty,n)}$.

\begin{warning}\label{threefa}
Let $\calC$ be an $(\infty,n)$-category with adjoints. There is a canonical functor from
$\calC$ to the (homotopy) fiber product $\Cat^{\Adj,\ast}_{(\infty,n)} \times^{R}_{ \Cat^{\Adj}_{(\infty,n)}} \{ \calC \},$ but it
is not generally an equivalence. An object of the fiber product on the right hand side can be identified
with a triple $(\calD, D, \phi)$ where $\calD$ is an $(\infty,n)$-category with adjoints,
$D \in \calD$ is an object, and $\phi: \calD \simeq \calC$ is an isomorphism in
$\Cat^{\Adj}_{(\infty,n)}$. As noted in Warning \ref{twofa}, $\phi$ need not arise from an
equivalence of $(\infty,n)$-categories, so we cannot necessarily regard $D$ as an object of $\calC$.

For example, when $n=1$, the fiber product $\Cat^{\Adj,\ast}_{(\infty,n)} \times_{ \Cat^{\Adj}_{(\infty,n)}} \{ \calC \}$ can be identified with the {\it idempotent completion} of the $(\infty,1)$-category $\calC$
(in other words, the $(\infty,1)$-category obtained from $\calC$ by adjoining an ``image'' for
every coherently idempotent $1$-morphism $f: C \rightarrow C$ in $\calC$; see \cite{htt} for a more detailed discussion). We will say that an $(\infty,n)$-category with adjoints $\calC$ is
{\it Morita complete} if the functor $\calC \rightarrow \Cat^{\Adj,\ast}_{(\infty,n)} \times_{ \Cat^{\Adj}_{(\infty,n)}} \{ \calC \}$ is an equivalence of $(\infty,n)$-categories.
\end{warning}

We are now ready to sketch the proof of Theorem \ref{swasher2}, at least in the case where
$\calC$ is Morita complete.

\begin{proof}[Proof of Theorem \ref{swasher2}]
Let $\vect{X} = ( \{ X_i \}_{0 \leq i \leq n}, \{ \zeta_i \}_{0 \leq i \leq n}, \{p_i: E_i \rightarrow X_i \}_{0 \leq i \leq n})$ denote the singularity datum characterized by the fact that
$X_i$ is empty for $i \notin \{0, n-k\}$, $X_0 \simeq X_{n-k} \simeq \{ \ast \}$, and
the space $E_{n-k}$ is equivalent (as an $n$-framed manifold) to the unit sphere in
$\R^{n-k}$. Unwinding the definitions, we see that a $\vect{X}$-manifold consists
of a pair $M_0 \subseteq M$, where $M$ is an $n$-framed manifold and
$M_0$ is a $k$-framed submanifold of $M$ (in the sense of Definition \ref{kfr}).
In particular, every $\vect{X}$-manifold can be regarded as an $n$-framed manifold by forgetting the submanifold $M_0$; this construction determines a symmetric monoidal functor
$\Bord^{\vect{X}}_{n} \rightarrow \Bord^{\fr}_{n}$. By construction, we have a diagram of
$(\infty,n)$-categories
$$ \xymatrix{ B^{n-k} \Tang_{k,n} \ar[r] \ar[d] & \Bord_{n}^{\vect{X}} \ar[d] \\
\ast \ar[r] & \Bord_{n}^{\fr} }$$
which commutes up to canonical isomorphism.

Let $\calC$ be an $(\infty,n)$-category with adjoints, which we regard
as an object of $\Cat^{\Adj}_{(\infty,n)}$. Let ${\bf 1} \in \calC$ and $\eta \in \Omega^{n-k} \calC$ be objects. Applying Theorem \ref{swisher3},
we deduce that this object determines a symmetric monoidal functor $Z_0: \Bord_{n}^{\fr}
\rightarrow \Cat^{\Adj}_{(\infty,n)}$, which is characterized up to equivalence by the
existence of an equivalence $Z_0(\ast) \simeq \calC$. Let $Z'_0$ denote the composition
$\Bord^{\vect{X}}_{n} \rightarrow \Bord^{\fr}_{n} \stackrel{Z_0}{\rightarrow} \Cat^{\Adj}_{(\infty,n)}.$
Applying Theorem \ref{swisher12}, we see that the pair $({\bf 1}, \eta)$ allow us to lift
$Z'_0$ to a symmetric monoidal functor $Z_1: \Bord^{\vect{X}}_{n} \rightarrow
\Cat^{\Adj}_{(\infty,n)}$. We obtain a rectangular diagram
$$ \xymatrix{ B^{n-k} \Tang_{k,n} \ar[r] \ar[d] & \Bord^{ \vect{X}}_{n} \ar[r] \ar[d] & \Cat^{\Adj, \ast}_{(\infty,n)} \ar[d] \\
\ast \ar[r] & \Bord^{\fr}_{n} \ar[r] & \Cat^{\Adj}_{(\infty,n)} }$$
which commutes up to isomorphism. This induces a functor
$$Z: B^{n-k} \Tang_{k,n} \rightarrow \Cat^{\Adj,\ast}_{(\infty,n)} \times_{
\Cat^{\Adj}_{(\infty,n)} } \{ \calC \}.$$
If $\calC$ is Morita complete (see Warning \ref{threefa}), we can regard
$Z$ as a functor with codomain $\calC$; it then suffices to check that the construction
$( {\bf 1}, \eta) \mapsto Z$ is homotopy inverse to the more evident construction described in Remark \ref{cabbis}. If $\calC$ is not assumed to be Morita complete, then a more elaborate argument (working directly with the $n$-fold simplicial spaces described in \S \ref{swugg}, rather than their underlying $(\infty,n)$-categories) is needed;
we will not present the details here.
\end{proof}

Like the cobordism hypothesis, the tangle hypothesis can be generalized in many ways.
For example, one can consider embedded submanifolds with tangential structure more complicated than that of a $k$-framing and embedded submanifolds with singularities. In these cases, one can still use the argument sketched above to establish a universal property of the relevant $(\infty,n)$-category.

\begin{example}[The Graphical Calculus]
According to Remark \ref{swing}, the group $\OO(3)$ acts on the $\infty$-groupoid of $(\infty,3)$-categories with adjoints. We can combine this with Remark \ref{delop} to obtain an $\OO(3)$ action on the $\infty$-groupoid $\Brd$ of braided monoidal $(\infty,1)$-categories $\calC$ which are rigid, in the sense that
every object has a dual. In particular, it makes sense to talk about a (homotopy) fixed point for
the action of $\SO(3)$ on $\Brd$: we will refer to such a fixed point as a {\it ribbon $(\infty,1)$-category}.
If we restrict our attention to ordinary categories (rather than $(\infty,1)$-categories), this recovers the usual notion of a ribbon structure on a braided monoidal category.

If we take $n=3$ and replace $\Bord_{n}^{\fr}$ by $\Bord_{n}^{\ori}$ in the proof of Theorem \ref{swasher2}, then we obtain a description of the free ribbon $(\infty,1)$-category $\calC$ on a single
generator. Namely, $\calC$ can be described as an $(\infty,1)$-category where the morphisms
are given by framed tangles (that is, $1$-dimensional submanifolds embedded in $\R^{3}$ together with a trivialization of their normal bundles). Passing to the truncation $\tau_{\leq 1} \calC$, we obtain
a well-known description of the free (ordinary) ribbon category on one generator in terms of framed tangles; see, for example, \cite{kirillov}.
\end{example}

\end{document}